\documentclass[12pt]{amsart}
\usepackage{amsfonts}
\usepackage{amssymb}
\usepackage{amsmath}
\usepackage{amsthm}
\usepackage[dvipdfmx]{graphicx,color}
\usepackage{moreverb}
\usepackage{fancybox}
\usepackage{fancyvrb}
\usepackage{comment}
\usepackage[all]{xy}
\usepackage[colorlinks=true, urlcolor=rltblue, citecolor=drkgreen, linkcolor=drkred] {hyperref}
\usepackage{color}
\usepackage{bm}
\definecolor{rltblue}{rgb}{0,0,0.4}
\definecolor{drkgreen}{rgb}{0,0.4,0}
\definecolor{drkred}{rgb}{0.5,0,0}
\usepackage{pdfsync}
\theoremstyle{plain}
\newtheorem{theorem}{Theorem}[section]
\newtheorem{lemma}[theorem]{Lemma}
\newtheorem{cor}[theorem]{Corollary}
\newtheorem{prop}[theorem]{Proposition}

\newtheorem{question}{Question}

\newtheorem{fact}{Fact}
\newtheorem{obs}[theorem]{Observation}
\newtheorem*{lemma*}{Lemma}
\newtheorem*{prop*}{Proposition}
\newtheorem*{theorem*}{Theorem}
\newtheorem*{obs*}{Observation}

\theoremstyle{definition}
\newtheorem{definition}[theorem]{Definition}
\newtheorem{example}[theorem]{Example}

\theoremstyle{definition}
\newtheorem{remark}[theorem]{Remark}
\newtheorem*{claim}{Claim}

\newtheorem*{ack}{Acknowledgement}

\newcommand{\upto}{\upharpoonright}
\newcommand{\fr}{\mbox{}^\smallfrown}
\newcommand{\N}{\mathbb{N}}
\newcommand{\om}{\omega}
\newcommand{\ep}{\varepsilon}
\newcommand{\nbase}[1]{{{\rm Nbase}( #1 )}}
\newcommand{\nbaseb}[2]{{{\rm Nbase}_{ #1 }( #2 )}}

\newcommand{\nbaseem}{{\rm Nbase}}
\newcommand{\name}[1]{{{\rm Name}( #1 )}}
\newcommand{\nameb}[2]{{{\rm Name}_{ #1 }( #2 )}}

\newcommand{\pt}[2]{{#1}\colon{#2}}

\newcommand{\cmp}[1]{{ #1 }^\mathsf{c}}

\newcommand{\reduce}{\leq_{\mathbf{T}}}
\newcommand{\eqreduce}{\equiv_{\mathbf{T}}}

\newcommand{\xx}{\mathcal{X}}
\newcommand{\yy}{\mathcal{Y}}
\newcommand{\zz}{\mathcal{Z}}

\newcommand{\mm}{\mathcal{M}}
\newcommand{\nn}{\mathcal{N}}
\newcommand{\nnp}{\mathcal{N}\sqcup\overline{\mathcal{N}}}
\newcommand{\dmm}{\delta_\mathcal{M}}
\newcommand{\dnn}{\delta_\mathcal{N}}
\newcommand{\dnnp}{\delta_{\nnp}}

\newcommand{\bvec}[1]{\bm{#1}}

\newcommand{\ex}[1]{\mathtt{Ex}(#1)}
\newcommand{\exc}[1]{\mathtt{Ex}^c(#1)}
\newcommand{\ab}{\{a,b,0\}}
\newcommand{\rs}[1]{\upharpoonright{#1}}

\newcommand{\dom}{\operatorname{dom}}

\makeindex

\newcommand{\mylem}[2]{
\begin{lemma*}[Lemma \ref{#1}]
#2
\end{lemma*}
\expandafter\newcommand\csname #1\endcsname{#2}}

\newcommand{\lemproof}[2]{
\begin{lemma}
\label{#1}
\expandafter\csname #1\endcsname
\end{lemma}
\begin{proof}
#2
\end{proof}}

\newcommand{\myprop}[2]{
\begin{prop*}[Proposition \ref{#1}]
#2
\end{prop*}
\expandafter\newcommand\csname #1\endcsname{#2}}

\newcommand{\propproof}[2]{
\begin{prop}
\label{#1}
\expandafter\csname #1\endcsname
\end{prop}
\begin{proof}
#2
\end{proof}}

\newcommand{\mythm}[2]{
\begin{theorem*}[Theorem \ref{#1}]
#2
\end{theorem*}
\expandafter\newcommand\csname #1\endcsname{#2}}

\newcommand{\thmproof}[2]{
\begin{theorem}
\label{#1}
\expandafter\csname #1\endcsname
\end{theorem}
\begin{proof}
#2
\end{proof}}

\newcommand{\myobs}[2]{
\begin{obs*}[Observation \ref{#1}]
#2
\end{obs*}
\expandafter\newcommand\csname #1\endcsname{#2}}

\newcommand{\obsproof}[2]{
\begin{obs}
\label{#1}
\expandafter\csname #1\endcsname
\end{obs}
\begin{proof}
#2
\end{proof}}

\title{Enumeration degrees and non-metrizable topology}
\author{Takayuki Kihara, Keng Meng Ng, and Arno Pauly}

\address[Takayuki Kihara]{Department of Mathematical Informatics, Graduate School of Informatics, Nagoya University, Japan}
\email{kihara@i.nagoya-u.ac.jp}
\urladdr{\href{http://www.math.mi.i.nagoya-u.ac.jp/~kihara/index.html}{www.math.mi.i.nagoya-u.ac.jp/~kihara}}

\address[Keng Meng Ng]{Division of Mathematical Sciences, School of Physical and Mathematical Sciences, Nanyang Technological University, Singapore}
\email{kmng@ntu.edu.sg}
\urladdr{\href{http://www.ntu.edu.sg/home/kmng/}{www.ntu.edu.sg/home/kmng/}}

\address[Arno Pauly]{Department of Computer Science, Swansea University, UK}
\email{Arno.M.Pauly@gmail.com}
\urladdr{\href{http://www.cs.swan.ac.uk/~cspauly/}{www.cs.swan.ac.uk/~cspauly/}}

\begin{document}

\begin{abstract}
The enumeration degrees of sets of natural numbers can be identified with the degrees of difficulty of enumerating neighborhood bases of points in a universal second-countable $T_0$-space (e.g.\ the $\om$-power of the Sierpi\'nski space).
Hence, every represented second-countable $T_0$-space determines a collection of enumeration degrees.
For instance, Cantor space captures the total degrees, and the Hilbert cube captures the continuous degrees by definition.
Based on these observations, we utilize general topology (particularly non-metrizable topology) to establish a classification theory of enumeration degrees of sets of natural numbers.

\end{abstract}

\maketitle

\setcounter{tocdepth}{2}
\tableofcontents

\section{Introduction}

\subsection{Context}

The notion of an enumeration degree was introduced by Friedberg and Rogers \cite{friedberg} in 1950s to estimate the degree of difficulty of enumerating a given set of natural numbers.
Roughly speaking, given sets $A,B\subseteq\om$, $A$ is {\em enumeration reducible to $B$} (written $A\leq_eB$) if there is a computable procedure that, given an enumeration of $B$, returns an enumeration of $A$.
Since then, the study of enumeration degrees have been one of the most important subjects in computability theory.

Nevertheless, only a few subcollections of enumeration degrees have been isolated.
Some prominent isolated properties are totality, semirecursiveness \cite{Joc68}, and quasi-minimality \cite{medvedev}.
Recently, the notion of cototality has also been found to be important and robust; see Andrews et al.\ \cite{cototal}, Jeandel \cite{Jea17}, and McCarthy \cite{McC17}.
Our aim is to understand the profound structure of enumeration degrees by isolating further subcollections of enumeration degrees and then establish a ``zoo'' of enumeration degrees.

To achieve our objective, we pay attention to a topological perspective of enumeration degrees.
The enumeration degrees can be identified with the degrees of difficulty of enumerating neighborhood bases of points in a universal second-countable $T_0$-space (e.g.\ the $\om$-power of the Sierpi\'nski space).
Hence, every represented second-countable $T_0$-space determines a collection of enumeration degrees.

This is exactly what Miller \cite{miller2} did for metric spaces.
Miller introduced the notion of {\em continuous degrees} as the degree structure of a universal separable metric space, and he described how this new notion can be understood as a substructure of enumeration degrees.
Subsequently, Kihara-Pauly \cite{KP} noticed that the {\em total degrees} are the enumeration degrees of neighborhood bases of points in (sufficiently effective) countable dimensional separable metric spaces.
This observation eventually led them to an application of continuous degrees in other areas outside of computability theory, such as descriptive set theory and infinite dimensional topology.
Other applications of continuous degrees can also be found in Day-Miller \cite{daymiller} and Gregoriades-Kihara-Ng \cite{GreKih}.

The connection to topology, and the results obtained through it, have already led to a flurry of recent activity in the study of enumeration degrees \cite{miller4,miller5,soskova,McC17}.

In this article, we further develop and deepen these connections.
We utilize general topology (particularly non-metrizable topology) to establish a classification theory of enumeration degrees.
For instance, we will examine which enumeration degrees can be realized as points in $T_0$ (Kolmogorov), $T_1$ (Fr\'echet), $T_2$ (Hausdorff), $T_{2.5}$ (Urysohn), and submetrizable spaces.
Furthermore, we will discuss the notion of $T_i$-quasi-minimality.
We will also provide a characterization of the notion of cototality in terms of computable topology.

Our work reveals that general topology (non-metrizable topology) is extremely useful to understand the highly intricate structure of subsets of the natural numbers.

\subsection{Summary}

In Section \ref{sec:zoo-list}, we reveal what substructures are captured by the degree structures of individual represented cb$_0$ spaces (some of which are quasi-Polish).
For instance, we define various subcollections of $e$-degrees, and then show the following.
\begin{itemize}
\item We construct a represented, decidable, $T_1$, non-$T_2$, quasi-Polish space $\mathcal{X}$ such that the $\mathcal{X}$-degrees are precisely the telograph-cototal degrees (Proposition \ref{prop:telophase-is-graph-cototal}).
\item We construct a represented, decidable, $T_2$, non-$T_{2.5}$, quasi-Polish space $\mathcal{X}$ such that the $\mathcal{X}$-degrees are precisely the doubled co-$d$-CEA degrees (Theorem \ref{thm:double-origin-degree}).
\item We construct a represented, decidable, $T_{2.5}$, non-submetrizable, quasi-Polish space $\mathcal{X}$ such that the $\mathcal{X}$-degrees are precisely the Arens co-$d$-CEA (the Roy halfgraph-above) degrees (Theorems \ref{thm:Arens-square} and \ref{thm:chained-co-d-CEA}).
\item We construct a represented, decidable, submetrizable, non-metrizable, quasi-Polish space $\mathcal{X}$ such that the $\mathcal{X}$-degrees are precisely the co-$d$-CEA degrees (Proposition \ref{prop:irregular-lattice-degree}).
\item Given a countable pointclass $\Gamma$, there is a computable extension $\gamma$ of the standard representation of Cantor space (hence, it induces a submetrizable topology) such that the $(2^\om,\gamma)$-degrees are exactly the $\Gamma$-above degrees (Proposition \ref{prop:Gamma-above}).
\item Every $e$-degree is an $\xx$-degree for some decidable, submetrizable, cb$_0$ space $\xx$ (Theorem \ref{thm:e-degree-realize}).
In particular, every $e$-degree is the degree of a point of a decidable $T_{2.5}$ space.
\end{itemize}

For the details of the above results, see Section \ref{sec:zoo-list}.
In Section \ref{sec:network}, we emphasize the importance of the notion of a network (which has been extensively studied in general topology).
A $G_\delta$-space is a topological space in which every closed set is $G_\delta$.
A second-countable $T_0$-space $\xx$ is a $G_\delta$-space if and only if $\xx$ has a countable closed network (Proposition \ref{prop:G_delta-equal-closed-network}).
The following is an unexpected characterization of cototality.

\begin{itemize}
\item An $e$-degree is cototal if and only if it is an $\xx$-degree of a computably $G_\delta$, cb$_0$ space $\xx$ (Theorem \ref{thm:cototal-Gdelta-space}).
\item There exists a decidable, computably $G_\delta$, cb$_0$ space $A_{\rm max}^{\rm co}$ such that the $A_{\rm max}^{\rm co}$-degrees are exactly the cototal $e$-degrees (Theorem \ref{thm:cototal-Gdelta-space2}).
\end{itemize}

We also show several separation results for specific degree-notions.
For instance,

\begin{itemize}
\item There are an $n$-semirecursive $e$-degree $\mathbf{c}\leq\mathbf{0}''$ and a total $e$-degree $\mathbf{d}\leq\mathbf{0}''$ such that the join $\mathbf{c}\oplus\mathbf{d}$ is not $(n+1)$-semirecursive (Theorem \ref{thm:rrn-semirec}).
\item For any $n\in\om$, an $n$-semirecursive $e$-degree is either total or a strong quasi-minimal cover of a total $e$-degree (Theorem \ref{thm:hprincipal}).
\item For any $n$, there is an $(n+1)$-cylinder-cototal $e$-degree which is not $n$-cylinder-cototal (Theorem \ref{thm:coBaire-hierarchy}).
\item There is a co-d-CEA set $A\subseteq\om$ such that $A$ is not cylinder-cototal (Proposition \ref{prop:non-cylinder-cototal}).
\item Every semirecursive, non-$\Delta^0_2$ $e$-degree is quasi-minimal w.r.t.\ telograph-cototal $e$-degrees (Theorem \ref{telograph-quasiminimal1}).
\item There is a semirecursive set $A\subseteq\om$ which is quasi-minimal, but not quasi-minimal w.r.t.\ telograph-cototal $e$-degrees (Theorem \ref{telograph-quasiminimal3}).
\item There is a cylinder-cototal $e$-degree which is quasi-minimal w.r.t.\ telograph-cototal $e$-degrees (Theorem \ref{telograph-quasiminimal2}).
\item By $(\om^\om)_{GH(n)}$ we denote the set $\om^\om$ endowed with the $\Sigma^1_n$-Gandy-Harrington topology.
For any distinct numbers $n,m\in\om$, there is no $e$-degree which is both an $(\om^\om)_{GH(n)}$-degree and an $(\om^\om)_{GH(m)}$-degree (Theorem \ref{thm:Gandy-Harrington-hierarchy}).
\item There is a continuous degree which is neither telograph-cototal nor cylinder-cototal (Proposition \ref{thm:continuous-cospec}).
\end{itemize}

Moreover, we introduce the notion of a regular-like network, and give a characterization in the context of a closure representation, which plays a key role in Section \ref{sec:quasiminimaltec}.
By using these notions, we will show the following separation results.

\begin{itemize}
\item Let $\mathcal{T}$ be a countable collection of second-countable $T_1$ spaces.
Then, there is a $\mathcal{T}$-quasi-minimal semirecursive $e$-degree (Theorem \ref{thm:countable-T_1-quasiminimal}).
\item Let $\mathcal{T}$ be a countable collection of second-countable $T_1$ spaces.
Then, there is an $(n+1)$-semirecursive $e$-degree which cannot be written as the join of an $n$-semirecursive $e$-degree and an $\xx$-degree for $\xx\in\mathcal{T}$ (Theorem \ref{thmseparation1}).
\item For any represented Hausdorff space $\xx$, there is a cylinder-cototal $e$-degree which is not an $\xx$-degree (Theorem \ref{thm:T-1-degree-not-T-2}).
\item There is a cylinder-cototal $e$-degree which is $\mathbb{N}^{\mathbb{N}^\mathbb{N}}$-quasi-minimal (Theorem \ref{thm:T-1-degree-NNN-quasiminimal}).
\item Given any countable collection $\{S_i\}_{i\in\omega}$ of effective $T_2$ spaces, there is a telograph-cototal $e$-degree which is $S_i$-quasi-minimal for any $i\in\om$ (Theorem \ref{thm:T_2-quasiminimal}).
\item For any represented $T_{2.5}$-space $\xx$, there is an $(\mathbb{N}_{\rm rp})^\om$-degree which is not an $\xx$-degree (Theorem \ref{thm:T-2-degree-not-T-25}).
\item There is an $(\mathbb{N}_{\rm rp})^\om$-degree which is $\mathbb{N}^{\mathbb{N}^\mathbb{N}}$-quasi-minimal (Theorem \ref{thm:T-2-degree-NNN-quasiminimal}).
\item Let $\xx=(X,\nn)$ be a regular Hausdorff space with a countable cs-network.
Then there is an $(\om^\om)_{GH}$-degree which is not an $\xx$-degree (Theorem \ref{thm:Gandy-Harrington-closed-neighborhood}).
\item The Gandy-Harrington space has no point of $\mathbb{N}^{\mathbb{N}^\mathbb{N}}$-degree (Theorem \ref{thm:GH-not-NNN}).
\end{itemize}

\subsection{Structure of the article}
To ease the reading, we are not giving proofs in Sections \ref{sec:zoo-list} and \ref{sec:further}. Instead, we focus on the statements of the theorems and accompanying narratives and explanations. The proofs omitted in these Sections are given in Sections \ref{sec:zoo-proofs} and \ref{sec:further-proofs}. The statements of the theorems are repeated in the latter sections. Theorem numbers always refer to the place where the theorems are given with proofs. In Section \ref{sec:cs-networks} narrative and proofs are not separated.


\section{Preliminaries}

\subsection{Notations}
We use $\cmp{A}$ to denote the complement of $A$, and $\overline{U}$ always means the topological closure of $U$.
We also use the standard notations in computability theory:
For $X,Y\subseteq\om$, let $X\oplus Y$ be the Turing join of $X$ and $Y$; that is, $(X\oplus Y)(2n)=X(n)$ and $(X\oplus Y)(2n+1)=Y(n)$.
For strings $\sigma,\tau\in\om^{<\om}$, let $\sigma\fr\tau$ be the concatenation of $\sigma$ and $\tau$, and we mean by $\sigma\preceq\tau$ that $\sigma$ is an initial segment of $\tau$.
For a string $\sigma\in\om^{<\om}$, let $|\sigma|$ be the length of $|\sigma|$, and $[\sigma]=\{X\in\om^\om:X\succ\sigma\}$.
By $x\upto s$ we mean the restriction of $x$ up to $s$.

\subsection{General topology}

We first review some basic concepts from general topology (see also Steen-Seebach \cite{CTopBook}).
In most parts of this paper, 
we only deal with second-countable $T_0$-spaces. However in Section \ref{sec:cs-networks} we also consider (non-second-countable) spaces which have a countable cs-network, e.g.\ the Kleene-Kreisel space $\N^{\N^\N}:=C(\N^\N,\N)$.
A normal space having a countable cs-network is known as an $\aleph_0$-space (see \cite{Mic66,Guth}).
\index{Separation axioms}
A space $X$ is $T_0$ ({\em Kolmogorov}) if any two distinct points are topologically distinguishable. We are only concerned with $T_0$ spaces in this paper.
A space $X$ is $T_1$ ({\em Fr\'echet}) if every singleton is closed.
A space $X$ is $T_2$ ({\em Hausdorff}) if the diagonal is closed.
A space $X$ is $T_{2.5}$ ({\em Urysohn}) if any two distinct points are separated by their closed neighborhoods.
A space $X$ is $T_D$ if every singleton is the intersection of an open set and a closed set.
A space $X$ is {\em completely Hausdorff} if any two distinct points are separated by a continuous $[0,1]$-valued function.
A space $X$ is {\em submetrizable} if it admits a continuous metric.
We have the following implications.
\[\mbox{metrizable}\Rightarrow\mbox{submetrizable}\Leftrightarrow\footnotemark\mbox{completely Hausdorff}\Rightarrow T_{2.5}\Rightarrow T_2\Rightarrow T_1\Rightarrow T_D\Rightarrow T_0.\]
\footnotetext{Being submetrizable and being completely Hausdorff coincides for spaces with hereditarily Lindel\"of squares, which includes all spaces with countable networks, hence all spaces relevant for our purposes. We are grateful to Taras Banakh for pointing this out to us on mathoverflow (\url{https://mathoverflow.net/questions/280359/does-second-countable-and-functionally-hausdorff-imply-submetrizable}).}
A space $X$ is {\em regular} if the closed neighborhoods of a point $x$ form a local network at the point $x$, that is, every neighborhood of a point contains a closed neighborhood of the same point.
In the category of second-countable $T_0$ spaces, by Urysohn's metrization theorem, the property being $T_3$ (regular Hausdorff) is equivalent to metrizability.
A space in which every closed set is $G_\delta$ is called a {\em $G_\delta$-space}.
Every metrizable space is $G_\delta$ (see \cite[Part III]{CTopBook}), and every $T_0$, $G_\delta$-space is $T_1$ (see Section \ref{sec:G-delta-space}).

\subsection{Computability theory}
\subsubsection{Enumeration and Medvedev reducibility}

We review the definition of enumeration reducibility (see also Odifreddi \cite[Chapter XIV]{OdiBook1}, Cooper \cite{CooperE} \& \cite[Chapter 11]{cooperbook}).
Let $(D_e)_{e\in\om}$ be a computable enumeration of all finite subsets of $\om$.
Given $A,B\subseteq\om$, we say that {\em $A$ is enumeration reducible to $B$} (written $A\leq_eB$) if there is a c.e.\ set $\Phi$ such that
\[n\in A\iff(\exists e)\;[\langle n,e\rangle\in\Phi\mbox{ and }D_e\subseteq B].\]

The $\Phi$ in the above definition is called an {\em enumeration operator}.
An enumeration operator induces a computable function on $\om^\om$, and indeed, $A\leq_eB$ iff there is a computable function $f:\om^\om\to\om^\om$ such that given an enumeration $p$ of $A$, $f(p)$ returns an enumeration of $B$, where we say that $p\in\om^\om$ is an enumeration of $A$ if $A=\{p(n)-1:p(n)>0\}$ ($p(n)=0$ indicates that we enumerate nothing at the $n$-th step).

Each equivalence class under the $e$-equivalence $\equiv_e:=\leq_e\cap\geq_e$ is called an {\em enumeration degree} or simply {\em $e$-degree}.
The $e$-degree of a set $A\subseteq\om$ is written as $\deg_e(A)$.
The $e$-degree structure forms a upper semilattice, where the join is given by the disjoint union $A\oplus B=\{2n:n\in A\}\cup\{2n+1:n\in B\}$.
We use the symbol $\mathcal{D}_e$ to denote the set of all $e$-degrees.


For $P,Q\subseteq\om^\om$, we say that {\em $P$ is Medvedev reducible to $Q$} (written $P\leq_MQ$, \cite{medvedev}) if there is a partial computable function $\Psi:\subseteq\om^\om\to\om^\om$ such that for any $q\in Q$, $\Psi(q)\in P$. There is a natural embedding of the enumeration degrees into the Medvedev degrees of the Baire space, by taking a set $A$ to the class of all enumerations of $A$.

\subsection{Represented spaces}\label{intro:rep-sp-00}
The central objects of study in computable analysis are the represented spaces, which allow us to make sense of computability for most space of interest in everyday mathematics.

\begin{definition}
A \emph{represented space} is a set $X$ together with a partial surjection $\delta:\subseteq\om^\om\to X$. We often write $\xx$ for a represented space.
\end{definition}

We say that $p\in\om^\om$ is a {\em $\delta$-name of $x$} if $x=\delta(p)$. We use $\nameb{\delta}{x}$ to denote the set of all $\delta$-names of $x$, or just write $\name{x}$, if the space is clear from the context. Hereafter, by a {\em point}, we mean a pair of a point $x\in X$ and the underlying represented space $\xx=(X,\delta)$, denoted by $\pt{x}{\xx}$ or simply $\pt{x}{\delta}$.

A partial function $F : \subseteq \om^\om \to \om^\om$ is called a \emph{realizer} of a partial function $f : \subseteq \xx \to \yy$, if $\delta_\yy(F(p)) = f(\delta_\xx(p))$ for any $p \in \dom(f\delta_\xx)$. We then say that $f$ is computable (respectively continuous), if $f$ has a computable (respectively continuous) realizer.

If $\gamma$ and $\delta$ are representations, we say that {\em $\gamma$ is (computably) reducible to $\delta$} or {\em $\gamma$ is (computably) finer than $\delta$} if there is a continuous (computable) function $\Phi$ such that $\gamma=\delta\circ\Phi$.
It is equivalent to saying that the identity map ${\rm id}\colon(X,\gamma)\to(X,\delta)$ is continuous (computable). If $\gamma$ is (computably) reducible to $\delta$ and $\delta$ is (computably) reducible to $\gamma$, we call $\gamma$ and $\delta$ (computably) equivalent. This corresponds to ${\rm id}\colon(X,\gamma)\to(X,\delta)$ being a (computable) isomorphism.

In the topological terminology, $\gamma$ is reducible to $\delta$ if and only if $\tau_\delta\subseteq\tau_\gamma$, where $\tau_\gamma$ and $\tau_\delta$ are the quotient topologies given by $\gamma$ and $\delta$, respectively.
If $X$ is equipped with a topology $\tau$, then $\delta$ is continuous if and only if $\tau\subseteq\tau_\delta$.

\subsubsection{Representation via a countable basis}
\index{represented cb space}
A {\em represented cb space} is a pair $(\xx,\beta)$ of a second-countable space $\xx$ and an enumeration $\beta=(\beta_e)_{e\in\om}$ of a countable open subbasis of $\xx$.
Here, ``cb'' stands for ``countably based''.
If a represented cb space is $T_0$, then it is also called a represented cb$_0$ space.
The enumeration $\beta$ is called a {\em cb representation of $\xx$}.

One of the key observations is that specifying a cb representation $\beta$ of a second-countable $T_0$ space $\xx$ is the same thing as specifying an embedding of $\xx$ into the power set $\mathcal{P}\om$ of $\om$ endowed with the Scott topology (that is, basic open sets are $\{X\subseteq\om:D\subseteq X\}$ where $D$ ranges over finite subsets of $\om$).
Hence, a cb$_0$ representation $\beta$ (and the induced embedding) determines how a point $x\in\xx$ is identified with a subset of the natural numbers.
\index{universal countably-based $T_0$-space}
This observation entails the known fact that the Scott domain $\mathcal{P}\om$ is a universal second-countable $T_0$ space, that is, every second-countable $T_0$ space embeds into $\mathcal{P}\om$.
We describe how an embedding $:\mathcal{X}\hookrightarrow\mathcal{P}\om$ is induced from a representation $\beta$.
One can identify a point $x$ in a represented cb$_0$ space $(\xx,\beta)$ with the coded neighborhood filter
\[\nbaseb{\beta}{x}=\{e\in\om:x\in \beta_e\}.\]
It is not hard to see that $\nbaseem_{\beta}\colon\mathcal{X}\hookrightarrow\mathcal{P}\om$ is a topological embedding. \index{$\nbaseb{\beta}{x}$}
An enumeration of $\nbaseb{\beta}{x}$ is called a {\em $\beta$-name of $x$}, that is, for a $p:\om\to\om$,
\[\mbox{$p$ is a $\beta$-name of $x$} \iff {\rm rng}(p)=\nbaseb{\beta}{x},\]
where one can assume $\beta_0=\xx$ without loss of generality.
If $\beta$ is clear from the context, we also use the symbol $\nbase{x}$ instead of $\nbaseb{\beta}{x}$.

Clearly, a cb-representation $\beta$ always induces a representation $\delta_\beta$ defined by $\delta_\beta(p)=x$ iff $p$ is a $\beta$-name of $x$ (i.e., $p$ enumerates $\nbaseb{\beta}{x}$). This entails that $\nbaseb{\xx}{x}$ is c.e.~ iff $\pt{x}{\xx}$ is computable. In situations where no confusion is expected, we may speak of a cb representation and its induced representation interchangeably. We can also express computability of partial functions between represented cb spaces equivalently as a special case of computability on represented spaces, or in the language of enumeration reducibility: Saying that $f : \subseteq \xx \to \yy$ is computable is equivalent to saying that there is a single enumeration operator $\Psi$ such that
\[(\forall x\in{\rm dom}(f))\;[\nbase{f(x)}\leq_e\nbase{x}\mbox{ via $\Psi$}].\]

\begin{remark}
It is known that the Scott domain $\mathcal{P}\om$ is homeomorphic to the $\om$-power $\mathbb{S}^\om$ of the Sierpi\'nski space, where $\mathbb{S}=\{0,1\}$ which has the three open sets $\emptyset$, $\{1\}$, and $\mathbb{S}$ (see \cite{debrecht6,schroder}).
\end{remark}

\begin{remark}
In Weihrauch-Grubba \cite{WeGr09}, a represented cb$_0$ space is called an {\em effective topological space}.
However we prefer to emphasize second-countability (cb) since the range of computability theory is far larger than second-countable spaces (see e.g.~\cite{schroder,Sch03,selivanov4,ScSe15,dBScSe15,pauly-synthetic-arxiv,paulydebrecht3-csl,pauly-ordinals-mfcs,schroder6}).
We also avoid the use of the terminology ``effective'' since the definition of a represented cb$_0$ space does not involve any effectivity.
\end{remark}

\subsubsection{Changing representations}\label{sec:change-repres}

We have introduced a cb representation as a countable subbasis $(\beta_e)_{e\in\om}$, but without loss of generality, we can always assume that it is actually a countable basis.
To see this, let $\beta=(\beta_e)_{e\in\om}$ be a countable subbasis of a cb$_0$ space $\mathcal{X}$.
Then we get a basis $\beta^+$ of $\mathcal{X}$ by defining $\beta^+_\sigma=\bigcap_{i<|\sigma|}\beta_{\sigma(i)}$ for $\sigma\in\om^{<\om}$.
Note that there is no difference between $\beta$ and $\beta^+$ from the computability-theoretic perspective:
\begin{align*}
e\in\nbaseb{\beta}{x}&\iff\langle e\rangle\in\nbaseb{\beta^+}{x},\\
\sigma\in\nbaseb{\beta^+}{x}&\iff\{\sigma(0),\dots,\sigma(|\sigma|-1)\}\subseteq\nbaseb{\beta}{x}.
\end{align*}

In other words, $\nbaseb{\beta}{x}$ is $e$-equivalent to $\nbaseb{\beta^+}{x}$ in a uniform manner. Indeed, we find
that every subbasis $\beta$ is computably equivalent to the induced basis $\beta^+$ (in the sense of represented spaces).

We observe that translations between cb$_0$ representations have a particular convenient form:
Let $\beta$ and $\gamma$ be cb$_0$ representations of $\xx$.
We see that {\em $\beta$ is computably reducible to $\gamma$} (written $\beta\leq\gamma$) if there is a single enumeration operator witnessing the reduction $\nbaseb{\beta}{x}\leq_e\nbaseb{\gamma}{x}$ for any $x\in\xx$.
It is equivalent to saying that any $\beta$-basic open set is $\gamma$-c.e.\ open in an effective manner, that is, there is a computable function $h$ such that
\[\beta_e=\bigcup\{\gamma^+_{\sigma}:{\sigma\in W_{h(e)}}\},\]
where $\gamma^+$ is a basis for $\xx$ defined as above.

\subsubsection*{Computable topological spaces}
We will consider the following additional effective properties for represented cb spaces $(X,\beta)$.
\begin{enumerate}
\item[(I)] There is a c.e.\ set $S$ such that $\beta_i\cap\beta_j=\bigcup\{\beta_e:(i,j,e)\in S\}$.
\item[(E)] $\{e:\beta_e\not=\emptyset\}$ is c.e.
\end{enumerate}

In Weihrauch-Grubba \cite{WeGr09}, a represented cb$_0$ space with (I) is called a {\em computable topological space}.
In Kurovina-Kudinov \cite{KoKu08}, a represented cb space with (I) and (E) is called a {\em effectively enumerable topological space}.

Moreover, if every positive finite Boolean operation on $\beta$ is computable, then we say that $(X,\beta)$ is {\em decidable}.


\begin{prop}\label{prop:equiv-repre-preserv}
Let $\beta,\gamma$ be representations of $X$ such that $\gamma\equiv\beta$.
If $(X,\beta)$ is computable, so is $(X,\gamma)$.
\end{prop}

\begin{proof}
For computability, since $\gamma\leq\beta$, given $d$ and $e$, $\gamma_d$ and $\gamma_e$ can be written as $\beta$-c.e.\ open set $U_d$ and $U_e$.
Thus, $\gamma_d\cap\gamma_e=U_d\cap U_e$.
One can easily find a $\beta$-index of the $\beta$-c.e.\ open set $U_d\cap U_e$, that is, $U_d\cap U_e=\bigcup_{n}\beta_{f(n,d,e)}$.
Since $\beta\leq\gamma$, we also have a $\gamma$-index of $U_d\cap U_e$, that is, $U_d\cap U_e=\bigcup_{n}\gamma_{g(n,d,e)}$.
Hence, $\gamma_d\cap\gamma_e=\bigcup_{n}\gamma_{g(n,d,e)}$, that is, $\gamma$ is computable.
\end{proof}


\subsubsection{Multi-representations}

As a technical took, we will rarely make use of multi-representations. A {\em (multi-)representation} is a multi-valued partial surjection $\delta:\subseteq\om^\om\rightrightarrows\xx$.
We say that $p\in\om^\om$ is a {\em $\delta$-name of $x$} if $x\in\delta(p)$. Notions such as realizer, computable functions between multi-represented spaces, etc, are all defined analogously to the case of ordinary representations. In the context of computable analysis, multi-representations were introduced by Schr\"oder \cite{Sch02}. They are an instance of assemblies from realizability theory \cite{oosten}.

\subsubsection{Admissible representation}\label{sec:intro-admissible-representation}

For a topological space $\xx=(X,\tau)$, we say (following Schr\"oder \cite{schroder}) that $\delta\colon\om^\om\to\xx$ is {\em admissible} if it is $\leq$-maximal among continuous representations of $X$, that is, it is continuous, and every continuous representation of $\xx$ is reducible to $\delta$.
Equivalently, a representation $\delta:\subseteq\om^\om\to\xx$ is admissible if it is continuous, and for any continuous representation $\gamma:\om^\om\to\xx$, the identity map $(\xx,\gamma)\to(\xx,\delta)$ is continuous.
Note also that admissible representations are the ones which realize the coarsest quotient topology refining $\tau$.

\begin{obs}
The representation $\delta_\beta$ induced from a cb-representation is always admissible.
\end{obs}


In fact, if $\nn=(N_e)_{e\in\om}$ is a countable cs-network (see Section \ref{sec:cs-networks}) for a $T_0$ space $\xx$, Schr\"oder \cite{schroder} showed that the following map $\delta_{\nn}$ always gives an admissible representation of $\xx$:
\[\dnn(p)=x\iff \{N_{p(n)}:n\in\om\}\mbox{ is a strict network at }x\mbox{ (see Definition \ref{def:network-strict-network})}.\]
We call $\delta_{\nn}$ the {\em induced $\om^\om$-representation of $\xx$} (obtained from $\nn$).
We also use the symbol $\nameb{\nn}{x}$ to denote the set of all enumerations of a strict subnetwork of $\nn$ at $x$, that is,
\[\nameb{\nn}{x}=\dnn^{-1}\{x\}=\{p\in\om^\om:\dnn(p)=x\}\]
If $\nn$ is clear from the context, we also use $\name{x}$ instead of $\nameb{\nn}{x}$.
See Section \ref{sec:cs-networks} for more details.

\subsection{Quasi-Polish spaces}\label{sec:quasi-Polish}
\index{Quasi-Polish space}
A {\em quasi-Polish space} is a second-countable space which is Smyth-completely quasi-metrizable \cite{debrecht6}.
Recall that a set in a space is $\mathbf{\Pi}^0_2$ if it is the intersection of countably many constructible sets, where a constructible set is a finite Boolean combination of open sets (see \cite{debrecht6,Seliv06}).
De Brecht \cite[Theorem 24]{debrecht6} showed that a space is quasi-Polish if and only if it is homeomorphic to a $\mathbf{\Pi}^0_2$ subset of $\mathbb{S}^\om$.
Be careful that it is not always the case that $\mathbf{\Pi}^0_2=G_\delta$.
Indeed, $\mathbf{\Pi}^0_2=G_\delta$ holds if and only if the underlying space is a $G_\delta$ space (see Section \ref{sec:G-delta-space}).

Assume that $\xx$ is a represented cb$_0$ space.
Consider the set of all names of points in $\xx$:
\[{\rm Name}(\xx)=\{p\in\om^\om:(\exists x\in\xx)\;{\rm rng}(p)={\rm Nbase}(x)\}.\]

Essentially, ${\rm Name}(\xx)$ is the domain of an admissible representation of the space $\xx$.
For a pointclass $\Gamma$, we say that $\xx$ is {\em $\Gamma$-named} if ${\rm Name}(\xx)$ is $\Gamma$.

\begin{prop}[De Brecht \cite{debrecht6}]\label{prop:deBrecht-named}
A represented cb$_0$ space $\xx$ is quasi-Polish if and only if $\xx$ is $\mathbf{\Pi}^0_2$-named.
\end{prop}

\begin{proof}
Note that $\nbaseem:x\mapsto\nbase{x}$ is an embedding of $\xx$ into $\mathbb{S}^\om$.
If $\xx$ is quasi-Polish, then so is the homeomorphic image $\nbaseem[\xx]$.
By de Brecht \cite[Theorem 21]{debrecht6}, $\nbaseem[\xx]$ is $\mathbf{\Pi}^0_2$ in $\mathbb{S}^\om$.
Note that ${\rm Name}(\xx)={\rm rng}^{-1}[\nbaseem[\xx]]$.
Since ${\rm rng}:\om^\om\to\mathbb{S}^\om$ is clearly continuous, ${\rm Name}(\xx)$ is $\mathbf{\Pi}^0_2$, that is, $\xx$ is $\mathbf{\Pi}^0_2$-named.

Conversely, if $\xx$ is $\mathbf{\Pi}^0_2$-named, then ${\rm Name}(\xx)$ is Polish, and in particular, quasi-Polish.
Note that ${\rm rng}:{\rm Name}(\xx)\to{\rm Nbase}[\xx]$ is an open continuous surjection.
Hence, by de Brecht \cite[Theorem 40]{debrecht6}, ${\rm Nbase}[\xx]$ is quasi-Polish.
Consequently, $\xx$ is quasi-Polish since $\xx$ is homeomorphic to ${\rm Nbase}[\xx]$.
\end{proof}

\begin{remark}
Let $\delta$ be an admissible representation of a space $\xx$.
Then, consider
\[{\rm Eq}(\xx,\delta)=\{(p,q)\in\om^\om:p,q\in{\rm dom}(\delta)\mbox{ and }\delta(p)=\delta(q)\}.\]
Generally, de Brecht et al.\ \cite{dBScSe15} has studied the classification of spaces based on the complexity of ${\rm Eq}(\xx,\delta)$.
\end{remark}

The following fact is useful to show that a space is quasi-Polish.

\begin{fact}[De Brecht {\cite[Theorems 40 and 41]{debrecht6}}]\label{fact:deBrecht-ocs}
A $T_0$ space $\xx$ is quasi-Polish if and only if there is an open continuous surjection from a Polish space onto $\xx$.
\qed
\end{fact}

Conversely, de Brecht has generalized the Hurewicz dichotomy from Polish to quasi-Polish spaces in \cite{debrecht8} to yield the following:

\begin{theorem}[De Brecht {\cite[Theorem 7.2]{debrecht8}}]\label{thm:Hurewicz-dichotomy}
A $\Pi^1_1$-subspace of a quasi-Polish is not quasi-Polish if and only if it contains a homeomorphic copy of one of the following spaces as a $\Pi^0_2$-subspace:
\begin{enumerate}
\item $\mathbb{Q}$ with the subspace topology inherited from $\mathbb{R}$
\item $\om_{\rm cof}$, the integers with the cofinite topology
\item $\omega_<$, the lower integers
\item $S_0$, with underlying set $\om^{<\om}$ equipped with the lower topology, where the basic closed sets are of the form $\uparrow p := \{q \in \om^{<\om} \mid p \preceq q\}$ (here $\preceq$ denotes the prefix relation)
\end{enumerate}
\end{theorem}

Of these, we make use of $\mathbb{Q}$ and $\om_{\rm cof}$ to show that particular spaces are not quasi-Polish. Recently, also a definition of what a \emph{computable quasi-Polish space} should be \cite{paulydebrecht4,stull}. We do not need it for our purposes, but we shall point out that whenever we are arguing that a particular space is quasi-Polish, it will already be computably quasi-Polish.

\section{Enumeration degree zoo}\label{sec:zoo-list}


\subsection{Definitions and overview}\label{sec:deg-examples}

In this section, we focus on the degree structures of second-countable spaces.
Our objective of this section is to see that general topology is surprisingly useful for investigating the enumeration degrees.

\subsubsection{The degree structure of a space}

We now introduce one of the key notions in this article.
To each point $\pt{x}{\xx}$, we assign the {\em degree of difficulty of calling a name of $x$}.

\begin{definition}[see Kihara-Pauly \cite{KP}]
Let $\pt{x}{\xx}$ and $\pt{y}{\yy}$ be points.
Then, we define
\[\pt{x}{\xx}\reduce\pt{y}{\yy}\overset{\rm def}{\iff}\nameb{\xx}{x}\leq_M\nameb{\yy}{y}.\]
\end{definition}

One can see that if $\xx$ and $\yy$ are represented cb$_0$ spaces, then
\[\name{x}\leq_M\name{y}\iff\nbase{x}\leq_e\nbase{y}\]

Therefore, reducibility between points can be defined in the following manner:
\[\pt{x}{\xx}\reduce\pt{y}{\yy} \iff \nbaseb{\xx}{x}\leq_e\nbaseb{\yy}{y}\]


We now describe how we classify the $e$-degrees by using topological notions.
Let $\xx$ be a represented cb$_0$ space.
We say that an enumeration degree $\mathbf{d}$ is an {\em $\xx$-degree} if $\nbase{x}\in\mathbf{d}$ for some $x\in\xx$.
By $\mathcal{D}_\xx$, we denote the set of all $\xx$-degrees.
In other words,
\[\mathcal{D}_\xx=\{\deg_e(\nbaseb{\xx}{x}):x\in\xx\}.\]

A key observation is that every represented cb$_0$ space determines a subset $\mathcal{D}_\xx$ of the $e$-degrees $\mathcal{D}_e$.

\begin{example}
Cantor space $2^\omega$, Baire space $\omega^\omega$, Euclidean $n$-space $\mathbb{R}^n$, and Hilbert cube $[0,1]^\om$ are represented in a standard manner.
\begin{enumerate}
\item If $X\in\{2^\om,\om^\om,\mathbb{R}^n\}$, then $\mathcal{D}_X$ is exactly the total degrees $\mathcal{D}_T$ (for the definition, see \ref{sec:en-deg-zoo}). \index{Degrees!total}
\item $\mathcal{D}_{[0,1]^\om}$ exactly the continuous degrees $\mathcal{D}_r$ (see Miller \cite{miller2}). \index{Degrees!continuous}
\item Let $\mathbb{R}_<$ be the real numbers endowed with the lower topology, and represented by $\beta_e=(q_e,\infty)$, where $q_e$ is the $e$-th rational.
Then, $\mathcal{D}_{\mathbb{R}_<}$ is exactly the semirecursive $e$-degrees (see Kihara-Pauly \cite{KP}). \index{Degrees!semi-recursive}
\end{enumerate}
\end{example}

For (1), Kihara-Pauly \cite{KP} showed that total $e$-degrees are characterized by countable dimensionality, that is, a separable metrizable space $X$ is countable dimensional iff, for any representation $\beta$ of $X$, there is an oracle $C$ such that every $(X,\beta)$-degree is total relative to $C$.
For (2), one can also easily see that metrizability captures continuous degrees by universality of the Hilbert cube.
These are what we indicated by our slogan ``{\em utilizing general topology to classify $e$-degrees}''.


In classical computability theory, Medvedev \cite{medvedev} introduced the notion of quasi-minimality.
An $e$-degree ${\bf a}$ is {\em quasi-minimal} if for every total degree ${\bf b}\leq_e{\bf a}$, we have ${\bf b}={\bf 0}$.
It is equivalent to saying that there is $A\in\mathbf{d}$ such that
\[(\forall x\in 2^\om)\;[\nbaseb{2^\om}{x}\leq_eA\;\Longrightarrow\;\nbaseb{2^\om}{x}\mbox{ is c.e.}]\]
We introduce a topological version of quasi-minimality.

\begin{definition}
Let $\mathcal{T}$ be a collection of represented cb$_0$ spaces. \index{quasi-minimal}
We say that an $e$-degree $\mathbf{d}$ is {\em $\mathcal{T}$-quasi-minimal} if there is $A\in\mathbf{d}$ such that
\[(\forall\xx\in\mathcal{T})(\forall x\in\xx)\;[\nbaseb{\xx}{x}\leq_eA\;\Longrightarrow\;\nbaseb{\xx}{x}\mbox{ is c.e.}]\]
\end{definition}

\subsubsection{Enumeration degree zoo}\label{sec:en-deg-zoo}

Our aim of this section is to investigate $\xx$-degrees for specific represented cb$_0$-spaces $\xx$.
Surprisingly, we will see that for most $\xx$, the $\xx$-degrees have very simple descriptions.

A set $A$ is {\em total} if $\cmp{A}\leq_eA$,\index{Degrees!total} and {\em cototal} \cite{cototal,Jea17} if $A\leq_e\cmp{A}$.\index{Degrees!cototal}
For a total function $g:\om\to\om$ and $b\in\om$, we define the graph ${\rm Graph}(g)$, the {\em cylinder-graph} ${\rm CGraph}(g)$, and the {\em $b$-telograph} ${\rm TGraph}(g)$ of $g$ as follows:
\begin{align*}
{\rm Graph}(g)&=\{\langle n,m\rangle:g(n)=m\},\\
{\rm CGraph}(g)&=\{\sigma\in\om^{<\om}:\sigma\prec g\},\\
{\rm TGraph}_b(g)&=\{\langle n,m\rangle:g(n)=m\mbox{ and }m\geq b\}.
\end{align*}

\begin{definition}\label{def:graph-cototal}
Let $\mathbf{a}$ be an enumeration degree.
\begin{enumerate}
\item We say that $\mathbf{a}$ is {\em graph-cototal} \cite{cototal,solon} if $\mathbf{a}$ contains the complement $\cmp{{\rm Graph}(g)}$ of the graph of a total function $g$. \index{Degrees!graph-cototal}
\item We say that $\mathbf{a}$ is {\em cylinder-cototal} if $\mathbf{a}$ contains the complement $\cmp{{\rm CGraph}(g)}$ of the cylinder graph of a total function $g$. \index{Degrees!cylinder-cototal}
We also say that $\mathbf{a}$ is {\em $n$-cylinder-cototal} if it is the join of $n$ many cylinder-cototal $e$-degrees. \index{Degrees!$n$-cylinder-cototal}
\item We say that $\mathbf{a}$ is {\em telograph-cototal} if $\mathbf{a}$ contains the join $\cmp{{\rm Graph}(g)}\oplus{\rm TGraph}_b(g)$ for some total function $g:\om\to\om$ and $b\in\om$. \index{Degrees!telograph-cototal}
\end{enumerate}
\end{definition}

Recall that a subset of $\om$ is {\em $d$-c.e.}\ if it is the difference of two c.e.~sets, and {\rm co-$d$-c.e.}\ if it is the complement of a d-c.e.\ set, that is, the union $A\cup P$ of a c.e.\ set $P$ and co-c.e.\ set $A$ such that $A$ and $P$ are disjoint.
Note that an enumeration degree contains a co-$d$-c.e.~set if and only if it contains a $3$-c.e.~set.

\begin{definition}\label{def:co-dCEA}
Let $\mathbf{a}$ be an enumeration degree.
\begin{enumerate}
\item We say that $\mathbf{a}$ is {\em co-d-CEA} if $\mathbf{a}$ contains a set of the form $(X\oplus\cmp{X})\oplus(A\cup P)$ for some $X,A,P\subseteq\om$ such that $P$ and $\cmp{A}$ are $X$-c.e., and $A$ and $P$ are disjoint.\index{Degrees!co-d-CEA}
\item Generally, we say that $\mathbf{a}$ is {\em $\Gamma$-above} if $\mathbf{a}$ contains a set of the form $(X\oplus\cmp{X})\oplus Z$ such that $Z$ is $\Gamma$ in $X$.\index{Degrees!$\Gamma$-above}
\item We say that $\mathbf{a}$ is {\em doubled co-d-CEA} if $\mathbf{a}$ contains a set of the form
\[(X\oplus\cmp{X})\oplus(A\cup P)\oplus(B\cup N)\]
for some $X,A,B,P,N\subseteq \om$ such that $P$, $N$, and $\cmp{A\cup B}$ are $X$-c.e., and that $A$, $B$, $P$ and $N$ are pairwise disjoint.\index{Degrees!doubled co-d-CEA}
\end{enumerate}
\end{definition}

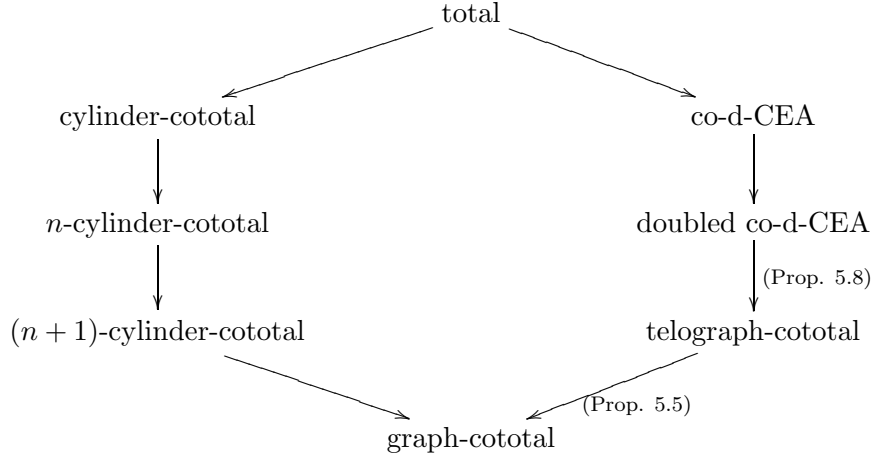
\begin{figure}[t]
{\small
\[
\xymatrix{
& \mbox{total} \ar [dl]  \ar [dr]  &   \\
\mbox{cylinder-cototal} \ar [d] &  & \mbox{co-d-CEA} \ar  [d]  \\
\mbox{$n$-cylinder-cototal} \ar  [d] &   & \mbox{doubled co-d-CEA} \ar  [d] ^{\text{(Prop. \ref{prop:doubleorigin-is-cototal})}} \\
\mbox{$(n+1)$-cylinder-cototal} \ar  [dr]  &  & \mbox{telograph-cototal} \ar  [dl] ^{\text{(Prop. \ref{prop:telophase-is-graph-cototal})}}   \\
& \mbox{graph-cototal} &
}
\]
}
\caption{A zoo of enumeration degrees I}\index{Degrees!total}\index{Degrees!cylinder-cototal}\index{Degrees!co-d-CEA}\index{Degrees!$n$-cylinder-cototal}\index{Degrees!doubled co-d-CEA}\index{Degrees!telograph-cototal}\index{Degrees!graph-cototal}
\label{fig:1}
\end{figure}

In Sections \ref{sec:Arens-square} and \ref{section:Roy-space}, we will introduce further variants of co-$d$-CEA degrees.
We will see that a co-$d$-CEA $e$-degree can be described using a Medvedev degree of separability.
Given $S,A,B\subseteq\om$, consider the following notions:
\begin{align*}
{\rm Enum}(S)&=\{p\in\om^\om:{\rm rng}(p)=S\},\\
{\rm Sep}(A,B)&=\{C\subseteq\om:A\subseteq C\mbox{ and }B\cap C=\emptyset\}.
\end{align*}

Note that an enumeration degree $\mathbf{a}$ is total if and only if $\mathbf{a}$ contains a set $S$ such that ${\rm Enum}(S)$ is Medvedev equivalent to $X\oplus\cmp{X}\oplus{\rm Sep}(A,B)$ for some $X,A,B\subseteq\om$ such that $A$ and $B$ are disjoint and $X$-co-c.e.
\begin{definition}\label{def:separation-above}
An enumeration degree $\mathbf{a}$ is {\em $[\Gamma_0,\Gamma_1;\Gamma_2]$-separating-above} ($[\Gamma_0,\Gamma_1;\Gamma_2]$-SepA) if $\mathbf{a}$ contains a set $S$ such that ${\rm Enum}(S)$ is Medvedev equivalent to $X\oplus\cmp{X}\oplus{\rm Sep}(A,B)$ for some $X,A,B\subseteq\om$ such that $A$ and $B$ are disjoint, $A\in\Gamma_0^{X}$, $B\in\Gamma_1^X$, and $A\cup B\in\Gamma_2^X$.
\end{definition}
In this terminology, an enumeration degree $\mathbf{a}$ is total if and only if $\mathbf{a}$ is $[\Pi^0_1,\Pi^0_1;\Pi^0_1]$-SepA.
We use $\ast$ to denote the pointclass containing all sets.
Then, for instance, we will see that an enumeration degree $\mathbf{a}$ is co-$d$-CEA if and only if $\mathbf{a}$ is $[\ast,\Pi^0_1;\Pi^0_1]$-SepA.

\begin{figure}[t]
{\small
\[
\xymatrix{
\mbox{Arens co-$d$-CEA}
&
\mbox{total} \ar [d] \ar [r]
&
\mbox{continuous} \ar @/^8pc/ [rdddddll]
 \\
\mbox{$[\ast,\Pi^0_1;\Pi^0_1]$-SepA} \ar [d] \ar @{<->} [r] _{\text{\quad\quad (Prop. \ref{prop:irregular-lattice-degree2})}}
& \mbox{co-$d$-CEA} \ar  [d] \ar[ul] \ar[dr]
&  \\
\mbox{$[\ast,\ast;\Pi^0_1]$-SepA} \ar @{<->} [d] _{\text{(Thm. \ref{thm:telophase-degree})}}
& \mbox{Roy halfgraph-above} \ar [d]
& \mbox{$\Delta^0_2$-above} \ar  [d] \\
\mbox{telograph-cototal} \ar [dr] _{\text{(Prop. \ref{prop:telophase-is-graph-cototal})}}
& \mbox{doubled co-$d$-CEA} \ar [l]
& \mbox{$\Sigma^0_2$-above} \ar  [dl] \\
& \mbox{graph-cototal} \ar [d]
&  \\
& \mbox{cototal} &
}
\]
}\index{Degrees!Arens co-d-CEA}\index{Degrees!total}\index{Degrees!continuous}\index{Degrees!co-d-CEA}\index{Degrees!Roy halfgraph-above}\index{Degrees!$\Delta^0_2$-above}\index{Degrees!telograph-cototal}\index{Degrees!doubled co-d-CEA}\index{Degrees!$\Sigma^0_2$-above}\index{Degrees!graph-cototal}\index{Degrees!cototal}
\caption{A zoo of enumeration degrees II}
\label{fig:2}
\end{figure}
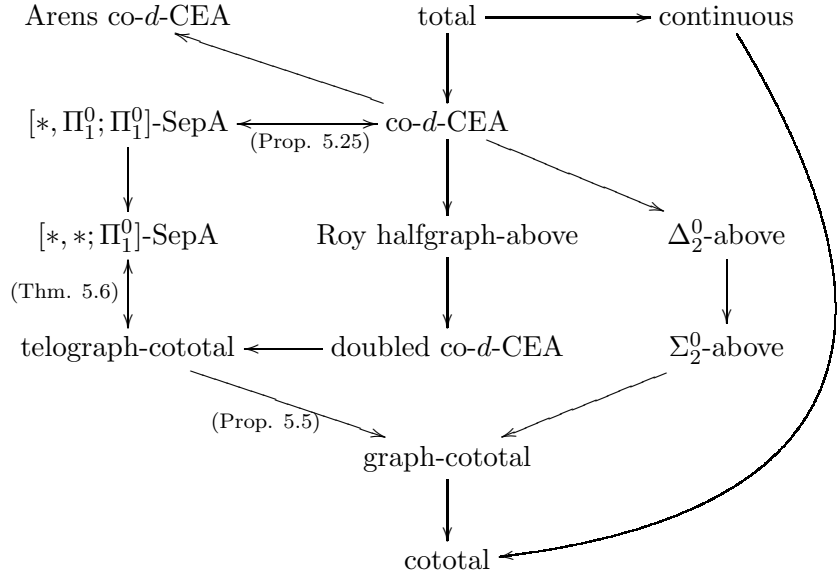

We will identify the above classes of $e$-degrees as the degree structures of certain second-countable $T_0$ spaces.

\begin{table}[h]\footnotesize
  \begin{tabular}{l|l} \hline \hline
    co-d-CEA & the $\om$-power of the irregular lattice space (submetrizable) \\\index{Degrees!co-d-CEA}\index{Space!irregular lattice space}
    Arens co-$d$-CEA & the $\om$-power of the quasi-Polish Arens square ($T_{2.5}$) \\\index{Degrees!Arens co-d-CEA}\index{Space!Arens square}
    Roy halfgraph-above & the $\om$-power of the quasi-Polish Roy space ($T_{2.5}$) \\\index{Degrees!Roy halfgraph-above}\index{Space!Roy}
    doubled co-d-CEA & the $\om$-power of the double origin space $(T_2)$ \\ \index{Degrees!doubled co-d-CEA}\index{Space!double origin}
    telograph-cototal & the $\om$-power of the telophase space $(T_1)$ \\ \index{Degrees!telograph-cototal}\index{Space!telophase}
    $n$-cylinder-cototal & the $n$-power of the cocylinder topology on $\om^\om$ $(T_1)$ \\\index{Degrees!$n$-cylinder-cototal}\index{Space!cocylinder topology}
    graph-cototal & the $\om$-power of the cofinite topology on $\om$ ($T_1$) \index{Degrees!graph-cototal}\index{Space!cofinite topology} \\
    \hline \hline
  \end{tabular}
\end{table}

For non-implications, there is a co-d-CEA $e$-degree which is not cylinder-cototal (Proposition \ref{prop:non-cylinder-cototal}), there is a cylinder-cototal $e$-degree which is quasi-minimal w.r.t.~telograph-cototal $e$-degrees (Theorem \ref{telograph-quasiminimal2}), there is a telograph-cototal $e$-degree which is quasi-minimal w.r.t.~doubled co-d-CEA $e$-degrees (Theorem \ref{thm:T_2-quasiminimal}), there is a quasi-minimal co-d-CEA $e$-degree (see Cooper \cite{CooperE}; Theorem \ref{irrelatti-quasi-minimal}), there is a doubled co-d-CEA $e$-degree which is not co-d-CEA (Proposition \ref{prop:proper-doubled}) and for any $n$, there is an $(n+1)$-cylinder-cototal $e$-degree which is not $n$-cylinder-cototal (Theorem \ref{thm:coBaire-hierarchy}).
We also show that there is a continuous degree which is neither telograph-cototal nor cylinder-cototal (Theorem \ref{thm:continuous-cospec}).

Andrews et al.\ \cite{cototal} showed that every $\Sigma^0_2$ $e$-degree is graph-cototal, while we will see that there is a $\Sigma^0_2$ $e$-degree which is quasi-minimal w.r.t.\ telograph-cototal $e$-degrees (Theorem \ref{telograph-quasiminimal1}).


\subsection{Degrees of points: $T_0$-topology}
\label{sec:deg-point-T0}

Let $\mathbb{R}_<$ be the space of all reals equipped with the lower topology generated by $\rho_<=((q,\infty):q\in\mathbb{Q})$.
Note that $\mathbb{R}_<$ is a $T_0$ space which is not $T_1$.\index{Space!lower reals}
Fixing a bijection $e\mapsto q_e:\om\to\mathbb{Q}$ gives us a representation of $\mathbb{R}_<$ by setting $\beta_e=(q_e,\infty)$.
Then, the coded neighborhood filter of $x$ is given by $\nbase{x}=\{e\in\om:q_e<x\}$.
For notational simplicity, hereafter we fix a standard effective indexing $e\mapsto q_e$, and always assume that every $q\in\mathbb{Q}$ is coded by a natural number without explicitly mentioning $e\mapsto q_e$.
Then, for a point $x\in\mathbb{R}_<$, the coded neighborhood filter of $x$ is just given as follows:
\[\nbaseb{\mathbb{R}_<}{x}=\{q\in\mathbb{Q}:q<x\}.\]

For notational simplicity, for a given $x\in\mathbb{R}$, we assume that $\nbase{x}$ always means $\nbaseb{\mathbb{R}}{x}$, and $\nbaseb{<}{x}$ always means $\nbaseb{\mathbb{R}_<}{x}$.
We say that a real $x\in\mathbb{R}$ is {\em left-c.e.}\ if $\nbaseb{<}{x}$ is c.e.
Similarly, a real $x$ is {\em right-c.e.}\ if $\nbaseb{<}{-x}$ is c.e. Recall that $A \subseteq \om$ is  {\em semirecursive} \cite{Joc68}, if there is a computable function $s : \om \times \om \to \{0,1\}$ such that if $\{n_0,n_1\} \cap A \neq \emptyset$, then $n_{s(n_0,n_1)} \in A$.
As pointed out by Kihara-Pauly \cite{KP},  the $\mathbb{R}_<$-degrees are exactly the semirecursive $e$-degrees .\index{Degrees!semi-recursive}
In other words:
\[\mathcal{D}_{\mathbb{R}_<}=\{\mathbf{d}\in\mathcal{D}_e:\mathbf{d}\mbox{ is semirecursive}\}.\]

In Section \ref{section:t0-nt1-quasiminimality-statements}, we investigate how $\mathbb{R}_<$-degrees behave.
Moreover, for instance, in Section \ref{sec:T-1-quasiminimal}, we will see that almost no $\mathbb{R}_<$-degrees are realized by a point in a $T_1$ space.
Thus, the results on $\mathbb{R}_<$-degrees describe the behavior of typical non-$T_1$-degrees.


\subsection{Degrees of points: $T_D$-topology}

A topological space is $T_D$ if every singleton can be written as the intersection of an open set and a closed set (see \cite{AT61}).\index{Separation axioms!$T_D$}
Note that a space is $T_D$ iff every singleton is $\mathbf{\Delta}^0_2$ in the sense of the non-metrizable Borel hierarchy \cite{debrecht6,Seliv06}:
A set in a space is $\mathbf{\Pi}^0_2$ if it is the union of countably many constructible sets, where a constructible set (in the sense of classical algebraic geometry) is a finite Boolean combination of open sets.
Note that $\mathbf{\Delta}^0_2$ (i.e., $\mathbf{\Sigma}^0_2\cap\mathbf{\Pi}^0_2$) is not necessarily equal to $F_\sigma\cap G_\delta$ (see Section \ref{sec:G-delta-space}).
Recall that a space is $T_1$ if every singleton is closed.

\begin{prop}\label{prop:T_D}
For every represented, $T_D$, cb$_0$ space $\xx$, there is a represented $T_1$, cb$_0$ space $\yy$ whose degree structure is the same as that of $\xx$, i.e., $\mathcal{D}_\xx=\mathcal{D}_\yy$.
\end{prop}

\begin{proof}
Since $\xx=(X,\beta)$ is $T_D$, for any $x\in\xx$, there is an open set $\beta_{e(x)}$ such that $\{x\}=\beta_{e(x)}\cap F_x$ for some closed set $F_x$.
Consider $\yy=\{(e,x)\in\om\times\xx:e(x)=e\}$, whose representation is induced from $\xx$, that is, $\gamma_{d,e}=\yy\cap(\{d\}\times\beta_e)$.
Then $\yy$ is $T_1$ since $\{(e(x),x)\}=\{e(x)\}\times F_x$.
Moreover, $\langle j,k\rangle\in\nbaseb{\yy}{e(x),x}$ iff $e(x)=j$ and $k\in\nbaseb{\xx}{x}$.
Hence, $\nbaseb{\yy}{e(x),x}$ is $e$-equivalent to $\nbaseb{\xx}{x}$.
\end{proof}

One can also consider a uniform version of being $T_D$, that is, having a $\mathbf{\Delta}^0_2$-diagonal.
Again, be careful that such a space does not necessarily have a $G_\delta$-diagonal.
Recall that a space is Hausdorff if it has a closed diagonal.
Following \cite{pauly-synthetic-arxiv}, we say that a represented space $\xx$ is computably Hausdorff if the diagonal $\Delta_\xx$ is $\Pi^0_1$.
For computability on topological separation axioms, see also Weihrauch \cite{Wei10,Wei13}.

\begin{prop}\label{prop:uniform-T_D}
If $\xx$ is a represented cb$_0$ space which has a $\Delta^0_2$-diagonal, then there is a computably Hausdorff cb$_0$ space $\yy$ such that $\mathcal{D}_\xx=\mathcal{D}_\yy$.
\end{prop}

\begin{proof}
Let $\xx$ be an effectively $T_D$ cb$_0$ space.
Then, the diagonal $\Delta_\xx$ is ${\Sigma}^0_2$; therefore, it is written as $\Delta_\xx=\bigcup_nD_n$ where $D_n$ is the intersection of a $\Sigma^0_1$ set $G_n$ and a $\Pi^0_1$ set $F_n$.
Consider $\xx_n=\{x\in\xx:(x,x)\in D_n\}$.
Let $(B_k)_{k\in\omega}$ be a countable open basis of $\xx$.
If $(x,x)$ is contained in an open set $G$ in $\xx^2$, then there is $k$ such that $x\in B_k$ and $B_k^2\subseteq G$.
In particular, for every $x\in\xx_n$ there is $k(x)\in\omega$ such that $(x,x)\in B_{k(x)}^2$ and $\Delta_n\cap B_{k(x)}^2=B_{k(x)}^2\cap F_n$.
Define $\xx_{n,k}=\{x\in\xx_n:k(x)=k\}$.
Then the diagonal on $\xx_{n,k}$ is the restriction of $F_n$ on $\xx_{n,k}$.
This is because, for any $x,y\in\xx_{n,k}$, we have $(x,y)\in B_k^2$.
Therefore $x=y$ if and only if $(x,y)\in F_n$.
Consequently, the diagonal on $\xx_{n,k}$ is $\Pi^0_1$.
Then, define $\yy=\{(n,k,x):x\in\xx_{n,k}\}$.
One can easily see that $\nbaseb{\yy}{n,k,x}$ is $e$-equivalent to $\nbaseb{\xx}{x}$.
\end{proof}

For Propositions \ref{prop:T_D} and \ref{prop:uniform-T_D}, the same observation is independently made by de Brecht.
Roughly speaking, these Propositions show that the $T_D$-degrees are exactly the $T_1$-degrees, and the $\Delta^0_2$-diagonal-degrees are exactly the $T_2$-degrees.
Thus, we do not need to consider the $T_D$-separation axiom (and any notion between $T_D$ and $T_1$ such as $T_{\frac{1}{2}}$), and its uniform version any more.


\subsection{Degrees of points: $T_1$-topology}

A topological space is $T_1$ if every singleton is closed.\index{Separation axioms!$T_1$}
We first consider the cofinite topology, which is one of the most basic constructions obtaining a non-Hausdorff (actually non-sober) $T_1$-topology.
Here, we consider the cofinite topology on $\om$, that is, a basis is given by $(\om\setminus D_e:e\in\om)$, where $D_e$ is the $e$-th finite subset of $\om$.
By $\om_{\rm cof}$, we denote the natural numbers $\om$ endowed with the cofinite topology and the above representation.
As before, we will not mention a fixed canonical indexing $e\mapsto D_e$, and we treat a finite set $D\subset\om$ as if it were a natural number.

The space $\om_{\rm cof}$ itself is countable, and thus, not degree-theoretically interesting. \index{Space!cofinite topology}
Instead, we consider the $\om$-power $(\om_{\rm cof})^\om$, and then for any $x\in(\om_{\rm cof})^\om$, we have
\[\nbase{x}=\{\langle n,D\rangle:x(n)\not\in D\},\]
where $D$ ranges over finite subsets of $\om$.
It is not hard to see that the $(\om_{\rm cof})^\om$-degrees are exactly the graph-cototal $e$-degrees (Definition \ref{def:graph-cototal}), that is,
\[\mathcal{D}_{(\om_{\rm cof})^\om}=\{\mathbf{d}\in\mathcal{D}_e:\mathbf{d}\mbox{ is graph-cototal}\}.\]

For basic properties on graph-cototal $e$-degrees, see Andrews et al.\ \cite{cototal}.\index{Degrees!graph-cototal}
In this section, we will isolate other proper subcollections of graph-cototal $e$-degree as degrees of points of specific non-Hausdorff $T_1$-spaces.

\subsubsection{Cocylinder topology}\label{sec:cocylinder-topo}

For a represented cb$_0$-space $\xx=(X,\beta)$, recall that $\beta$ is an enumeration of a countable open basis of a cb$_0$ space $X$.
We introduce the co-representation $\beta^{\rm co}$ of $X$ by $\beta^{\rm co}_e=X\setminus\beta_e$.
We write $\xx_{\rm co}=(X,\beta^{\rm co})$.

\begin{example}\label{exa:T_0-cofinite-intro}
We define $\lambda_{\langle n,m\rangle}=\{x\in\om^\om:x(n)=m\}$ for any $n,m\in\om$.
Then, $\lambda$ is a representation of the Baire space $\om^\om$.
It is not hard to see that $(\om^\om,\lambda^{\rm co})$ is computably homeomorphic to $(\om_{\rm cof})^\om$.
\end{example}

\begin{example}\label{exa:T_0-cocylinder-intro}
We define $\gamma_{\sigma}=\{x\in\om^\om:\sigma\prec x\}$ for any $\sigma\in\om^{<\om}$.
Then, $\gamma$ is a representation of the Baire space $\om^\om$.
It is easy to see that $\gamma$ is computably equivalent to the representation $\lambda$ in Example \ref{exa:T_0-cofinite-intro}, and therefore $(\om^\om,\lambda)$ and $(\om^\om,\gamma)$ are computably homeomorphic.
However, Proposition \ref{prop:non-cylinder-cototal} shows that $(\om^\om,\lambda^{\rm co})$ and $(\om^\om,\gamma^{\rm co})$ are not homeomorphic!
This indicates that the topology on $X$ induced from $\beta^{\rm co}$ heavily depends on the choice of the representation $\beta$ of $X$.
\end{example}

\begin{remark}
A better-behaved ``co-topology'' is known as the {\em de Groot dual} (after \cite{deGr,deGroot2}).
It only depends on the topology on $X$, but not on its representation.
Unfortunately, the de Groot dual of a cb$_0$ space is not necessarily second-countable, and therefore it exceeds the scope of this section.
However, it is worth mentioning that it does NOT exceed the scope of computability theory (see also Section \ref{sec:cs-networks}).
\end{remark}

Hereafter, by $\om^\om_{\rm co}$ we always mean the cocylinder space $(\om^\om,\gamma^{\rm co})$.
As usual, via a fixed canonical bijection between $\om$ and $\om^{<\om}$, we treat a string $\sigma\in\om^{<\om}$ as if it were a natural number.\index{Space!cocylinder topology}
Then, for any $x\in\om_{\rm co}^\om$,
\[\nbase{x}=\{\sigma\in\om^{<\om}:\sigma\not\prec x\}.\]

By definition, it is clear that the $\om_{\rm co}^\om$-degrees are exactly the cylinder-cototal $e$-degrees\index{Degrees!cylinder-cototal} (Definition \ref{def:graph-cototal}), that is,
\[\mathcal{D}_{\om_{\rm co}^\om}=\{\mathbf{d}\in\mathcal{D}_e:\mathbf{d}\mbox{ is cylinder-cototal}\}.\]

Recall that a join of $n$ cylinder-cototal $e$-degrees is called $n$-cylinder-cototal.
In other words, the $n$-cylinder-cototal $e$-degrees are exactly the $(\om_{\rm co}^\om)^n$-degrees.

\begin{obs}
Every $n$-cylinder-cototal $e$-degree is graph-cototal.
\end{obs}

\begin{proof}
It suffices to show that each product cocylinder space $(\om^\om_{\rm co})^n$ is effectively embedded into $(\om_{\rm cof})^\om$.
To see this, given $x=(x_m)_{m<n}$, consider $\check{x}(n)=\bigoplus_{m<n}x_m\upto n$.
It is not hard to check that $x\mapsto\check{x}$ gives a desired computable embedding.
It is also clear that $(\om^\om_{\rm co})^\om$ is effectively homeomorphic to $(\om_{\rm cof})^\om$.
\end{proof}

We first show that cylinder-cototal $e$-degrees form a proper subclass of graph-cototal $e$-degrees.
As shown by Andrews et al.\ \cite{cototal}, every $\Sigma^0_2$-above degree is graph-cototal.
It is not true for cylinder-cototal $e$-degrees.

\myprop{prop:non-cylinder-cototal}{
There is a co-d-CEA set $A\subseteq\om$ such that $A$ is not cylinder-cototal.}

The above proposition is useful for separating cylinder-cototal degrees and other degrees, since we will see that most collections of degrees obtained from represented cb$_0$ spaces in this article are larger than the co-$d$-CEA $e$-degrees.
We also show that $\om^\om_{\rm co}$ has a nontrivial $\Sigma^0_2$ $e$-degree.
An $e$-degree $\mathbf{d}$ is {\em proper-$\Sigma^0_2$} if $\mathbf{d}$ contains a $\Sigma^0_2$ set, but no $\Delta^0_2$ set.

\myprop{prop:propersigmacylindercototal}{There is a proper-$\Sigma^0_2$ cylinder-cototal degree.}

\subsubsection{Products of cocylinder topology}\label{sec:3-4-2}

In this section, we will see that there is a hierarchy of graph-cototal $e$-degrees.
We write $\mathcal{D}(\xx)$ for the substructure of $\mathcal{D}_e$ consisting of $e$-degrees of neighborhood bases of points in $\xx$.
We will then have the following proper hierarchy of degree structures:
\[\mathcal{D}_T=\mathcal{D}(\om^\om)\subsetneq\mathcal{D}(\om^\om_{\rm co})\subsetneq\dots\subsetneq\mathcal{D}((\om^\om_{\rm co})^n)\subsetneq\mathcal{D}((\om^\om_{\rm co})^{n+1})\subsetneq\dots\subsetneq\mathcal{D}((\om_{\rm cof})^\om)\]

\mythm{thm:coBaire-hierarchy}{
For any $n$, there is an $(n+1)$-cylinder-cototal $e$-degree which is not $n$-cylinder-cototal, that is,
\[\mathcal{D}_{(\om^\om_{\rm co})^{n+1}}\not\subseteq\mathcal{D}_{(\om^\om_{\rm co})^n}.\]
}

\subsubsection{Telophase topology}\label{sec:3-4-3}

Let $\mathcal{L}$ be a linearly ordered set.
The order topology on $\mathcal{L}$ is generated by the subbasis $(\{x:a<x\},\{x:x<a\}:a\in\mathcal{L})$.
Assume that $\mathcal{L}$ has a countable basis $\mathcal{B}$, that is, there is a countable set $\mathcal{B}\subseteq\mathcal{L}$ such that for any $a,b\in\mathcal{L}$, if $a<b$ then there are $c,d\in \mathcal{B}$ such that $a\leq c< d\leq b$.
Then the order topology on $\mathcal{L}$ is separable and metrizable.

We now assume that $\mathcal{L}$ has the greatest element $\mathbf{1}$.
Choose $\mathbf{1}_\star\not\in\mathcal{L}$.
Then $\mathcal{L}\cup\{\mathbf{1}_\star\}$ forms a partial order by adding the relation $a<\mathbf{1}_\star$ for each $a\in\mathcal{L}$ with $a\not=\mathbf{1}$.
Roughly speaking, $\mathcal{L}\cup\{\mathbf{1}_\star\}$ is almost linear ordered except that it has two maximal elements $\mathbf{1}$ and $\mathbf{1}_\star$.
The {\em telophase space $\mathcal{L}_{TP}$} \index{Space!telophase} is defined as the set $\mathcal{L}\cup\{\mathbf{1}_\star\}$ equipped with the Lawson topology, that is, generated by the following subbasis:
\[\{\{x:a\not\leq x\}:a\in\mathcal{L}\cup\{\mathbf{1}_\star\}\}\cup\{\{x:a<x\}:a\in\mathcal{L}\}.\]

If $\mathcal{L}$ has a countable basis $\mathcal{B}$, the following gives us a countable subbasis of $\mathcal{L}_{TP}$:
\[\{\{x:x<a\}:a\in\mathcal{B}\}\cup\{\{x:x\leq a\}:a\in\{\mathbf{1},\mathbf{1}_\star\}\}\cup\{\{x:a<x\}:a\in\mathcal{B}\}\]

\begin{example}[see Steen-Seebach {\cite[II.73]{CTopBook}}]\label{example:teloph-1}
Consider the Cantor space $\mathcal{C}=2^\om$, which is linearly ordered by defining $x\leq_{\rm left}y$ if and only if there is $n\in\om$ such that $x\upto n=y\upto n$ and $x(n)<y(n)$.
Then, the greatest element $\mathbf{1}$ is the sequence $1^\om$ consisting only of $1$'s.
One can easily see that the following gives us a countable subbasis of $\mathcal{C}_{TP}$:
\[\{[\sigma]:\sigma\in 2^{<\om}\}\cup\{[\sigma]_\star:\sigma\prec 1^\om\},\]
where $[\sigma]_\star=([\sigma]\setminus\{\mathbf{1}\})\cup\{\mathbf{1}_\star\}$.
Unfortunately, the degree structure of $\mathcal{C}_{TP}$ is not so interesting since every $\mathcal{C}_{TP}$-degree is total.
This is because, this construction adds only one new point, and thus, it is clear that $\nbaseb{\mathcal{C}_{TP}}{x}\equiv_e\nbaseb{\mathcal{C}}{x}$ for any $x\in\mathcal{C}$, and that $\mathbf{1}_\star$ is computable, that is, $\nbaseb{\mathcal{C}_{TP}}{\mathbf{1}_\star}$ is c.e.

Nevertheless, we will see that the degree structure of the $\om$-power $(\mathcal{C}_{TP})^\om$ is pretty interesting.
Note that for each $x\in(\mathcal{C}_{TP})^\om$, its coded neighborhood filter is given as
\[\nbase{x}=\{\langle n,0,\sigma\rangle:\sigma\prec x(n)\}\cup\{\langle n,1,k\rangle:1^k\prec x(n)\mbox{ and }x(n)\not=1^\om\}.\]
\end{example}

\begin{example}\label{example:teloph-2}
Consider the one-point compactification $\hat{\om}=\om\cup\{\infty\}$ of $\om$, which is naturally linear ordered with the greatest element $\infty$.
The telophase space $\hat{\om}_{TP}$ looks like a ``two-point compactification'' of $\om$.
The topology is generated by
\[\{\{m\},[m,\infty],[m,\infty_\star]:m\in\om\}.\]

Here, $[a,b]$ is the interval $\{c:a\leq c\leq b\}$.
Then, for each $x\in(\hat{\om}_{TP})^\om$, its coded neighborhood filter is given as
\begin{align*}
\nbase{x}=\{\langle n,0,m\rangle:x(n)=m\}&\cup\{\langle n,1,m\rangle:m\leq x(n)\leq\infty\}\\
&\cup\{\langle n,2,m\rangle:m\leq x(n)\leq\infty_\star\},
\end{align*}
where $m$ and $n$ range over $\om$.
\end{example}

It is easy to check that the spaces in Examples \ref{example:teloph-1} and \ref{example:teloph-2} are $T_1$ but not Hausdorff (since $\mathbf{1}$ and $\mathbf{1}_\star$ cannot be separated by open sets).
Later we will see the following:
\begin{align*}
\mathcal{D}_{(\mathcal{C}_{TP})^\om}=\mathcal{D}_{(\hat{\om}_{TP})^\om}&=\{\mathbf{d}\in\mathcal{D}_e:\mathbf{d}\mbox{ is telograph-cototal}\}\\
&=\{\mathbf{d}\in\mathcal{D}_e:[\ast,\ast,\Pi^0_1]\mbox{-separating-above}\}.
\end{align*}

Here, recall from Definitions \ref{def:graph-cototal} and  \ref{def:separation-above} for the above notions.

\myprop{prop:ctpcomputablyembeds}{
$\mathcal{C}_{TP}$ computably embeds into $(\hat{\om}_{TP})^\om$.
Hence, $\mathcal{D}_{(\mathcal{C}_{TP})^\om}=\mathcal{D}_{(\hat{\om}_{TP})^\om}$.}

We will also see that the hierarchy of telograph-cototal $e$-degrees collapses.
For $b\in\om$, we say that an enumeration degree $\mathbf{a}$ is {\em $b$-telograph-cototal} if it contains $\cmp{{\rm Graph}(g)}\oplus{\rm TGraph}_b(g)$ for some total function $g:\om\to\om$.

\myprop{obs:telograph-hierarchy-collapses}{\index{Degrees!telograph-cototal}
The $1$-telograph-cototal $e$-degrees are exactly the total degrees.
For any natural numbers $b,c>1$, the $b$-telograph-cototal $e$-degrees are exactly the $c$-telograph-cototal $e$-degrees.}

We will show that the $(\hat{\om}_{TP})^\om$-degrees are characterized by the telograph-cototal $e$-degrees (Definitions \ref{def:graph-cototal}).

\myprop{prop:telophase-is-graph-cototal}{
The $(\hat{\om}_{TP})^\om$-degrees are exactly the telograph-cototal $e$-degrees.}

We give another characterization of $(\hat{\om}_{TP})^\om$-degrees in terms of separating sets.
For any $A,B\subseteq\om$, recall that ${\rm Sep}(A,B)$ is the collection of sets separating $A$ and $B$:
\[{\rm Sep}(A,B)=\{C\subseteq\om:A\subseteq C\mbox{ and }B\cap C=\emptyset\}.\]

We also recall that ${\rm Enum}(E)$ is the set of all enumerations of $E\subseteq\om$, that is,
\[{\rm Enum}(E)=\{p\in\om^\om:{\rm rng}(p)=E\}.\]
To make our argument simple, we always assume that $E$ is nonempty.
Note that $D\equiv_eE$ if and only if ${\rm Enum}(D)\equiv_M{\rm Enum}(E)$.

\mythm{thm:telophase-degree}{
The $(\hat{\om}_{TP})^\om$-degrees (hence the telograph-cototal $e$-degrees) are exactly the $[\ast,\ast,\Pi^0_1]$-separating-above $e$-degrees.
In other words, a nonempty set $E\subseteq\om$ is $e$-equivalent to $\nbase{x}$ for some $x\in(\hat{\om}_{TP})^\om$ if and only if there are $X,A,B\subseteq\om$ such that $A\cup B$ is $X$-co-c.e., $A\cap B=\emptyset$, and
\[{\rm Enum}(E)\equiv_M\{X\}\times{\rm Sep}(A,B).\]
}

\subsection{Degrees of points: $T_2$-topology}

A topological space $\xx$ is $T_2$ or {\em Hausdorff} if any distinct points $x\not=y\in\xx$ are separated by open sets, that is, there are disjoint open sets $U,V\subseteq\xx$ such that $x\in U$ and $y\in V$.
It is equivalent to saying that the diagonal $\Delta_\xx=\{(x,x):x\in\xx\}$ is closed in $\xx^2$.

\subsubsection{Double Origin Topology}\label{sec:3-5-1}

Let $\mathcal{L}_0$ and $\mathcal{L}_1$ be linearly ordered sets.
Consider the product $\mathcal{L}=\mathcal{L}_0\times\mathcal{L}_1$.
Let $\tau_{\mathcal{L}}$ be the product of the order topologies on $\mathcal{L}_0$ and $\mathcal{L}_1$.
Fix an element $\mathbf{0}=(\mathbf{o}_0,\mathbf{o}_1)\in\mathcal{L}$, and $\mathbf{0}_\star\not\in\mathcal{L}$.
The {\em double origin space}\index{Space!double origin} $\mathcal{L}_{DO}$ is defined as the set $\mathcal{L}\cup\{\mathbf{0}_\star\}$ equipped with the topology generated by the following subbasis:
\begin{align*}
\{U\setminus\{\mathbf{0}\}:U\in\tau_\mathcal{L}\}&\cup\{\{\mathbf{0}\}\cup\{(x,y):a<x<b\mbox{ and } \mathbf{o}_1<y<c\}:a<\mathbf{o}_0<b, c>\mathbf{o}_1\}\\
&\cup\{\{\mathbf{0_\star}\}\cup\{(x,y):a<x<b\mbox{ and }c<y<\mathbf{o}_1\}:a<\mathbf{o}_0<b, c<\mathbf{o}_1\},
\end{align*}
where $a,b,x$ range over $\mathcal{L}_0$ and $y$ ranges over $\mathcal{L}_1$. We observe that $\mathcal{L}_{DO}$ is Hausdorff.

\begin{example}[see Steen-Seebach {\cite[II.74]{CTopBook}}]
For each $i<2$, let $\mathcal{Q}_i$ be the unit open rational interval, that is, $\mathcal{Q}_i=\mathbb{Q}\cap(-1,1)$, equipped with the canonical ordering, and put $\mathbf{o}_0=\mathbf{o}_1=0$.
Then, a countable subbasis of $\mathcal{Q}_{DO}$ is given as follows:
\begin{align*}
&\{((p,q)\times(r,s))\setminus\{\mathbf{0}\}:p,q,r,s\in\mathbb{Q}\cap(-1,1)\}\\
&\cup\{((-k^{-1},k^{-1})\times(0,\ell^{-1}))\cup\{\mathbf{0}\}:k,\ell\in\om\}\\
&\cup\{((-k^{-1},k^{-1})\times(-\ell^{-1},0))\cup\{\mathbf{0}_\star\}:k,\ell\in\om\}.
\end{align*}

Clearly $\mathcal{Q}_{DO}$ is countable, and so its degree structure is not interesting.
Instead, we consider the $\om$-power $(\mathcal{Q}_{DO})^\om$.
We treat each $z\in(\mathcal{Q}_{DO})^\om$ as if it were a pair $(x,y)$.
If $z(n)\not=\mathbf{0}_\star$ for all $n\in\om$, it is actually a pair $(x,y)$ given by $z(n)=\langle x(n),y(n)\rangle$ for any $n\in\om$.
If $z(n)=\mathbf{0}_\star$ for some $n\in\om$, one may put $\mathbf{0}_\star=(0_\star,0_\star)$ by choosing a new symbol $0_\star$, where we assume that $0_\star$ has no relationship with other rationals.
Then, for any point $(x,y)\in(\mathcal{Q}_{DO})^\om$, its coded neighborhood filter is given as follows:
\begin{align*}
\nbase{x,y}=&\{\langle n,0,p,q,r,s\rangle:p<x(n)<q,\;r<y(n)<s,\mbox{ and }(x(n),y(n))\not\in\{\mathbf{0},\mathbf{0}_\star\}\}\\
&\cup\{\langle n,1,k,\ell\rangle:(|x(n)|<k^{-1}\mbox{ and }0<y(n)<\ell^{-1})\mbox{ or }(x(n),y(n))=\mathbf{0}\},\\
&\cup\{\langle n,2,k,\ell\rangle:(|x(n)|<k^{-1}\mbox{ and }-\ell^{-1}<y(n)<0)\mbox{ or }(x(n),y(n))=\mathbf{0}_\star\}.
\end{align*}
\end{example}

\begin{example}\label{exa:double-quasi-Polish}
Define $\mathcal{P}_0=\hat{\om}\simeq \om+1$, $\mathcal{P}_1=\om+1+\om^\ast$, and $\mathcal{P}=\mathcal{P}_0\times\mathcal{P}_1$.
Here, recall that $\hat{\om}=\om\cup\{\infty\}$ is a one-point compactification of $\om$, and $\om^\ast$ is the reverse order of $\om$.
More precisely, $\mathcal{P}_1$ is the set $\{n:n\in\om\}\cup\{\ast\}\cup\{\overline{n}:n\in\om\}$ ordered as follows:
\[0<1<\dots<n<n+1<\dots<\ast<\dots<\overline{n+1}<\overline{n}<\dots<\overline{1}<\overline{0}.\]
Then, define $\mathbf{o}_0=\infty\in\mathcal{P}_0$ and $\mathbf{o}_1=\ast\in\mathcal{P}_1$.
A countable subbasis of $\mathcal{P}_{DO}$ is given as follows:
\begin{align*}
&\{[n,\infty]\times\{m\}:n\in\om,\;m\not=\ast\}\cup\{\{(n,m)\}:n\in\om,\;m\in\mathcal{P}_1\}\cup\{\{n\}\times[n,\overline{m}]:n,m\in\om\}\\
\cup&\,\{([n,\infty]\times(\ast,\overline{n}])\cup\{\mathbf{0}\}:n\in\om\}\cup\{([n,\infty]\times[n,\ast))\cup\{\mathbf{0}_\star\}:n\in\om\}.
\end{align*}
It is clear that $\mathcal{P}_{DO}$ embeds into $\mathcal{Q}_{DO}$.
To see this, consider embeddings $j_0:\mathcal{P}_0\to\mathcal{Q}_0$ and $j_1:\mathcal{P}_1\to\mathcal{Q}_1$ defined by $j_0(n)=2^{-n}$, $j_0(\infty)=0$, $j_1(n)=-2^{-n}$, $j_1(\ast)=0$, and $j_1(\overline{n})=2^{-n}$.
Then $j_0\times j_1$ clearly induces an embedding of $\mathcal{P}_{DO}$ into $\mathcal{Q}_{DO}$.
An advantage of using $\mathcal{P}$ is that $(\mathcal{P}_{DO})^\om$ is quasi-Polish (see Proposition \ref{prop:quasi-Polish-examples}) while $(\mathcal{Q}_{DO})^\om$ is not.
\end{example}

We will show the following characterization of double origin spaces.
\[\mathcal{D}_{(\mathcal{Q}_{DO})^\om}=\mathcal{D}_{(\mathcal{P}_{DO})^\om}=\{\mathbf{d}\in\mathcal{D}_e:\mathbf{d}\mbox{ is doubled co-d-CEA}\}.\]

Here, recall from Definition \ref{def:co-dCEA} for the above notion.

\mythm{thm:double-origin-degree}{
The $(\mathcal{P}_{DO})^\om$ and the $(\mathcal{Q}_{DO})^\om$-degrees each are exactly the doubled co-d-CEA degrees.\index{Space!double origin}\index{Degrees!doubled co-d-CEA}
In other words, an $e$-degree $\mathbf{d}$ is a $(\mathcal{Q}_{DO})^\om$-degree if and only if there are $X,A,B,P,N\subseteq\om$ such that $A$, $B$, $P$ and $N$ are pairwise disjoint, $P$, $N$, and $\cmp{(A\cup B)}$ are $X$-c.e., and
\[X\oplus\cmp{X}\oplus(A\cup P)\oplus(B\cup N)\in\mathbf{d}.\]
}

A set is {\em d-c.e.}~if it is a difference of two c.e.~sets.
A set is {\em co-d-c.e.}~if it is the complement of a d-c.e.~set, that is, the union of a c.e.~set and a co-c.e.~set.
Clearly, double origin degrees include all co-d-c.e.~degrees.
It is known that there is a quasi-minimal co-d-c.e.~degree (see Cooper \cite{CooperE}).
Thus, there is a quasi-minimal double origin degree.


Moreover, one can see that the degree structure of the double origin space is included in that of the telophase space.

\myprop{prop:doubleorigin-is-cototal}{
Every doubled co-$d$-CEA $e$-degree is telograph-cototal.\index{Degrees!doubled co-d-CEA}\index{Degrees!telograph-cototal}
Hence, we have $\mathcal{D}_{(\mathcal{Q}_{DO})^\om}\subseteq\mathcal{D}_{(\hat{\om}_{TP})^\om}$.
}


\subsection{Degrees of points: $T_{2.5}$-topology}

A topological space $\xx$ is $T_{2.5}$ if any distinct points $x\not=y\in\xx$ are separated by closed neighborhoods, that is, there are open sets $U,V\subseteq\xx$ such that $x\in\overline{U}$, $y\in\overline{V}$, and $\overline{U}\cap\overline{V}=\emptyset$.\index{Separation axioms!$T_{2.5}$}
A topological space $\xx$ is {\em completely Hausdorff} or {\em functionally Hausdorff} if any distinct points $x\not=y\in\xx$ are separated by a function, that is, there exists a continuous function $f:\xx\to[0,1]$ with $f(x)=0$ and $f(y)=1$.
Note that every metrizable space is completely Hausdorff, and every completely Hausdorff space is $T_{2.5}$, but none of the converse holds.
In this section, we examine the degree structure of a $T_{2.5}$ space which is not completely Hausdorff (hence, not submetrizable).

\subsubsection{Arens square}\label{sec:3-6-1}

We would like to know a typical degree-theoretic behavior of a space which is second-countable, $T_{2.5}$, but not completely Hausdorff.
As such an example, Steen-Seebach {\cite[II.80]{CTopBook}} mentioned the Arens square; however there it has been found\footnote{The problem was observed by Martin Sleziak on math.stackexchange.com. A direct fix was then proposed by Brian M. Scott (\url{https://math.stackexchange.com/questions/1715435/is-arens-square-a-urysohn-space}). Our modification can be seen as an abstraction of the one proposed by Scott.} that their definition contains an error, that is, the Arens square defined in \cite[II.80]{CTopBook} is not $T_{2.5}$.
We here construct a simple example of a space which fulfills the above required property by modifying the definition of Arens square. Rather than describing the space as a subset of the rational unit square, we observe that the crucial ideas of the construction are all order-theoretic in nature and thus use corresponding language.

\begin{example}\label{exa:Arens-square}\index{Space!Arens square}
Consider a linear ordering $\mathcal{L}$ of type $\om+1+\zeta+1+\om^\ast$, where $\zeta$ is the order type of the integers $\mathbb{Z}$.
More precisely, consider the sets $\om=\{n:n\in\om\}$, $\om^\ast=\{\overline{n}:n\in\om\}$, and $\zeta=\{n_\zeta:n\in\mathbb{Z}\}$, and then, let $\mathcal{L}$ be the linear order consisting of the set $\om\cup\{\infty\}\cup\zeta\cup\{\overline{\infty}\}\cup\om^\ast$
ordered as follows.
\[0<1<\dots<\infty<\dots<(-1)_\zeta<0_\zeta<1_\zeta<\dots<\overline{\infty}<\dots<\overline{1}<\overline{0}.\]

Consider the following subset $I_x$ of the ordinal $\om^3+1$ for each $x\in\mathcal{L}$.
\begin{align*}
I_0&=I_{\overline{0}}=\{\om^3\},\ I_{0_\zeta}=\{\om^2\cdot (j + 1):j\in\om\},\\
I_{\infty}&=\{\om^2\cdot j+\om\cdot (2k+1):j,k\in\om\},\\
I_{\overline{\infty}}&=\{\om^2\cdot j+\om\cdot (2k+2):j,k\in\om\},\\
I_n&=\{\om^2\cdot j+\om\cdot (2k)+2n-1:j,k\in\om\},\\
I_{\overline{n}}&=\{\om^2\cdot j+\om\cdot (2k+1)+2n-1:j,k\in\om\},\\
I_{n_\zeta}&=\{\om^2\cdot j+\om\cdot (2k+1)+2n:j,k\in\om\},\\
I_{(-n)_\zeta}&=\{\om^2\cdot j+\om\cdot (2k)+2n:j,k\in\om\}.
\end{align*}
where $n$ ranges over $\om^+:=\om\setminus\{0\}$.
Note that $(I_x:x\in\mathcal{L}\setminus\{0\})$ is a partition of $(\om^3+1)\setminus\{0\}$.
Moreover, $(I_x:x\in\mathcal{L}^+)$ is a partition of the nonzero successor ordinals $<\om^3$, where $\mathcal{L}^+=\mathcal{L}\setminus\{0,\overline{0},0_\zeta,\infty,\overline{\infty}\}$.
Then define a modified Arens square (which we will call the {\em quasi-Polish Arens space}) $\mathcal{QA}\subseteq\mathcal{L}\times(\om^3+1)$ as follows:
\[\mathcal{QA}=\{(x,y):x\in\mathcal{L}\mbox{ and }y\in I_x\}.\]

The set $\mathcal{QA}$ is topologized by declaring the following as an open basis.
\begin{align*}
&\{\om\times(\alpha,\om^3]:\alpha<\om^3\}\cup\{\om^\ast\times(\alpha,\om^3]:\alpha<\om^3\}\\
\cup\,&\{\zeta\times(\om^2\cdot j+\om\cdot n,\om^2\cdot(j+1)]:n,j\in\om\}\\
\cup\,&\{[n,(-n)_\zeta]\times[\om^2j+\om(2k)+2n-1,\om^2j+\om(2k+1)]:n\in\om^+,\;k,j\in\om\}\\
\cup\,&\{[n_\zeta,\overline{n}]\times[\om^2j+\om(2k+1)+2n-1,\om^2j+\om(2k+2)]:n\in\om^+,\;k,j\in\om\}\\
\cup\,&\{\{(x,y)\}:x\in\mathcal{L}^+\mbox{ and }y\in I_x\}.
\end{align*}
\end{example}

\begin{remark}\label{rem:om31-total}
Note that the second projection $\pi:\mathcal{QA}\to\om^3+1$ given by $\pi(x,y)=y$ is continuous w.r.t.\ the order topology on the ordinal $\om^3+1$.
Hence, given a name of $(x,y)\in\mathcal{QA}$, one can compute a name of $y$ w.r.t.\ a suitable representation of $\om^3+1$.
The computability is ensured by just considering a quotient admissible representation of $\om^3+1$ induced from $\pi$, or equivalently, by considering an embedding $h:\om^3+1\to[0,1]$ defined by
\begin{align*}
&h(\om^3)=0,\ h(\om^3[j])=2^{-j},\ h(\om^3[j+1][k])=2^{-j}(1+2^{-k}),\\
&h(\om^3[j+1][k+1][\ell])=2^{-j}(1+2^{-k}(1+2^{-\ell})),
\end{align*}
where fundamental sequences are given by $\om^{n+1}[j]=\om^n\cdot j$ and $(\alpha+\beta)[k]=\alpha+\beta[k]$.
Thus, one can consider the embedded image of $\om^3+1$ into the unit interval $[0,1]$.
Note also that the ordinal space $\om^3+1$ is a (computably) zero-dimensional compact metrizable space.
Hence, for any $x\in(\om^3+1)^\om$, $\nbase{x}$ has a total $e$-degree.
\end{remark}

\myprop{proparensecondcountable}{
The quasi-Polish Arens space $\mathcal{QA}$ is second-countable, and $T_{2.5}$, but not completely Hausdorff.}

From the descriptive set theoretic perspective, our modified Arens square $\mathcal{QA}$ is better behaved than the original one in a certain sense: The space $\mathcal{QA}$ is quasi-Polish (hence so is the $\om$-power $\mathcal{QA}^\om$) as we will see in Proposition \ref{prop:quasi-Polish-examples}.

We proceed to examine the degree structure of the product Arens space $\mathcal{QA}^\om$.
For $z=(x(n),y(n))_{n\in\om}\in\mathcal{QA}^\om$, the coded neighborhood basis of $z$ is given as follows:
\begin{align*}
&\{\langle 0,n,j\rangle:x(n)\in\om\mbox{ and }y(n)>\om^2j\}\cup\{\langle 1,n,j\rangle:x(n)\in\om^\ast\mbox{ and }y(n)>\om^2j\}\\
\cup\,&\{\langle 2,n,j,k\rangle:x(n)\in\zeta\mbox{ and }\om^2j+\om k<y(n)<\om^2(j+1)\}\\
\cup\,&\{\langle 3,n,j,k,\ell\rangle:j\leq x(n)\leq(-j)_\zeta\mbox{ and }\om^2k+\om(2\ell+1)+2j-1<y(n)\leq\om^2k+\om(2\ell +2)\}\\
\cup\,&\{\langle 4,n,j,k,\ell\rangle:j_\zeta\leq x(n)\leq\overline{j}\mbox{ and }\om^2k+\om(2\ell)+2j<y(n)\leq\om^2k+\om(2\ell +1)\}\\
\cup\,&\{\langle 5,n,x,y\rangle:x(n)=x\in\mathcal{L}^+\mbox{ and }y(n)=y\in I_x\}.
\end{align*}

We will see that the degree structure of the product quasi-Polish Arens space $\mathcal{QA}^\om$ can be considered as a variant of the co-$d$-CEA degrees. To the similarity, we provide a characterization of the co-$d$-CEA degrees similar to what we show below as Definition \ref{def:arenscodcea}.
Let $\mathcal{E}$ be the collection of $e$-degrees $\mathbf{d}$ which contain a set $S\subseteq\om$ of the following form
\[S=X\oplus \cmp{X}\oplus(A\cup P)\oplus (B\cup N)\]
for some $A,B,P,N,X\subseteq\om$ such that $P,N$ and $\cmp{(A\cup B)}$ are c.e.\ in $X$, $A,B,P,N$ are pairwise disjoint, and $P$ and $N$ are $X$-computably separated over $\cmp{(A\cup B)}$.
Here, we say that $P$ and $N$ are {\em $X$-computably separated over $C$} if there are disjoint $X$-c.e.\ sets $H_P,H_N\subseteq\om$ such that $C=H_P\cup H_N$, $P\subseteq H_P$, and $N\subseteq H_N$.

\myobs{obsdineiffdiscodcea}{
For an $e$-degree $\mathbf{d}$, $\mathbf{d}\in\mathcal{E}$ if and only if $\mathbf{d}$ is co-$d$-CEA.\index{Degrees!co-d-CEA}}

\begin{definition}
\label{def:arenscodcea}
We say that an $e$-degree $\mathbf{d}$ is {\em Arens co-$d$-CEA} if $\mathbf{d}$ contains a set $S\subseteq\om$ of the following form
\[S=Y\oplus \cmp{Y}\oplus(L\cup J_L)\oplus (R\cup J_R)\oplus (\cmp{(L\cup R\cup N)}\cup J_M)\]
for some $L,R,N,J_L,J_R,J_M,Y\subseteq\om$ such that $N,J_L,J_R,J_M$ and $\cmp{(L\cup R)}$ are c.e.\ in $Y$, $L,R,N$ are pairwise disjoint, and $J_L,J_R,J_M\subseteq N$ are pairwise disjoint, where $J_L$ and $J_R$ are $Y$-computably separated over $N$, that is, there is a $Y$-c.e.\ partition $\{H_L,H_R\}$ of $N$ such that $J_L\subseteq H_L$ and $J_R\subseteq H_R$.\index{Degrees!Arens co-d-CEA}
\end{definition}

\myobs{codceaisarenscodcea}{
Every co-$d$-CEA $e$-degree is Arens co-$d$-CEA.
}

\mythm{thm:Arens-square}{
The degree structure of the product quasi-Polish Arens space $\mathcal{QA}^\om$ consists exactly of Arens co-$d$-CEA $e$-degrees:
\[\mathcal{D}_{\mathcal{QA}^\om}=\{\mathbf{d}\in\mathcal{D}_e:\mathbf{d}\mbox{ is Arens co-$d$-CEA}\}.\]
}

\subsubsection{Roy's lattice space}\label{sec:3-6-2}

We next introduce another space which has a similar property as the Arens space $\mathcal{QA}$.
The space is a quasi-Polish version of Roy's lattice space (see Steen-Seebach {\cite[II.126]{CTopBook}}).\index{Space!Roy}

To introduce the space, first recall that the ordinal $\om^3$ is order isomorphic to the lexicographical ordering on the set $\om^3$ of strings of length $3$ by identifying an ordinal $\om^2\cdot j+\om\cdot k+n<\om^3$ with the string $\langle j,k,n\rangle\in\om^3$.
Now, consider the Kleene-Brouwer ordering $\leq_{\rm KB}$ on the well-founded tree $\mathcal{O}_{\om^\om}=\{\sigma\in\om^{<\om}:|\sigma|\leq\sigma(0)+1\}$.
Then, $(\mathcal{O}_{\om^\om},\leq_{\rm KB})$ is order isomorphic to the ordinal $(\om^\om+1,\leq)$.
Note that $|\sigma|_{\rm KB}$ is a successor ordinal iff $\sigma$ is a leaf (i.e., a terminal node), and that $|\langle\rangle|_{\rm KB}=\om^\om$.
Hereafter, we use $\mathcal{O}_{\om^\om}^{\rm leaf}$ to denote the set of leaves of $\mathcal{O}_{\om^\om}$.
Given $\sigma,\tau\in \mathcal{O}_{\om^\om}$, we define the interval $[\sigma,\tau]_{\rm KB}=\{\gamma\in \mathcal{O}_{\om^\om}:|\sigma|_{\rm KB}\leq|\gamma|_{\rm KB}\leq|\tau|_{\rm KB}\}$.
We define the half-open interval $(\sigma,\tau]_{\rm KB}$ etc.\ in a similar manner.
One can see that for $\sigma\in \mathcal{O}_{\om^\om}\setminus \mathcal{O}_{\om^\om}^{\rm leaf}$ and $j\in\om$,
\[(\sigma j,\sigma]_{\rm KB}=\{\tau\in \mathcal{O}_{\om^\om}:\tau=\sigma\mbox{ or }(\exists k>j)\;\tau\succeq\sigma k\}.\]

We topologize $\mathcal{O}_{\om^\om}$ by declaring the following as an open basis:
\[\{\{\sigma\}:\sigma\in \mathcal{O}_{\om^\om}^{\rm leaf}\}\cup\{(\sigma j,\sigma]_{\rm KB}:\sigma\not\in \mathcal{O}_{\om^\om}^{\rm leaf}\mbox{ and }j\in\om\}.\]

One can see that $\mathcal{O}_{\om^\om}$ is homeomorphic to the order topology on the ordinal $\om^\om+1$.
As in Remark \ref{rem:om31-total}, one can see that the ordinal space $\om^\om+1$ is zero-dimensional, compact, and metrizable.
For $k>0$, consider the following subsets of the space $\mathcal{O}_{\om^\om}$.
\begin{align*}
I_0&=I_\infty=\{\sigma\in \mathcal{O}_{\om^\om}:|\sigma|=0\}=\{\langle\rangle\},\\
I_{2k}&=\{\sigma\in \mathcal{O}_{\om^\om}:|\sigma|=k\}\\
I_{1}&=\{\sigma\in \mathcal{O}_{\om^\om}^{\rm leaf}:|\sigma|\geq 2\mbox{ and }(\forall j>1)\;\sigma(j)=0\},\\
I_{2k+1}&=\{\sigma\in \mathcal{O}_{\om^\om}^{\rm leaf}:|\sigma|\geq k+2,\;\sigma(k+1)>0\mbox{ and }(\forall j>k+1)\;\sigma(j)=0\}.
\end{align*}

Note that $(I_k:k\in\om)$ is a partition of the ordinal $\om^\om+1$.
Moreover, each set $I_k$ is cofinal in $\om^\om$ for $k\in\om\setminus\{0\}$.
We now introduce the {\em quasi-Polish Roy space}, whose underlying set is given as follows.
\[\mathcal{QR}=\{(x,y):x\in\hat{\om}\mbox{ and }y\in I_x\}.\]

Let $\tau$ be the topology on $\mathcal{O}_{\om^\om}$ defined as above.
The set $\mathcal{QR}$ is topologized by declaring the following as an open basis.
\[\{\{2n+1\}\times U,\ \{0,1\}\times U,\ [2n+1,2n+3]\times U,\ [2n+1,\infty]\times \mathcal{O}_{\om^\om}:n\in\om,\ U\in\tau\}.\]

As in Remark \ref{rem:om31-total}, one can also ensure that the projection $\pi:\mathcal{QR}\to\om^\om+1$ defined by $\pi(x,y)=y$ is (computably) continuous.

\myprop{propquasipolishroy}{
The quasi-Polish Roy space $\mathcal{QR}$ is a second-countable $T_{2.5}$ space which is not completely Hausdorff.
}

We now begin to examine the degree structure of the product chain space $\mathcal{QR}^\om$.
The coded neighborhood basis of $z=(x_n,y_n)\in\mathcal{QR}^\om$ is given as follows.
\begin{align*}
&\{\langle 0,n,k,\sigma\rangle:x_n=2k+1\mbox{ and }y_n=\sigma\}\\
\cup\,&\{\langle 1,n,k,\sigma j\rangle:|x_n-2k|\leq 1\mbox{ and }y_n\in(\sigma j,\sigma]_{\rm KB}\}\\
\cup\,&\{\langle 2,n,k\rangle:x_n>2k\}.
\end{align*}

Fix two symbols $\bot_0,\bot_1\not\in\om$, and consider $\tilde{\om}=\om\cup\{\bot_0,\bot_1\}$.
Then, consider the following {\em halfgraph} of a function $f:\om\to\tilde{\om}$:
\begin{align*}
{\rm HalfGraph}(f)=\,&\{2\langle n,m\rangle:f(n)=2m\}\\
\cup\,&\{2\langle n,m\rangle+1:f(n)\in\om\mbox{ and }f(n)\geq 2m\}.
\end{align*}

We say that a function $f:\om\to\tilde{\om}$ is {\em half-c.e.}\ if it has a c.e.\ halfgraph, that is, ${\rm HalfGraph}(f)$ is c.e.
We also say that a function $f:\om\to\tilde{\om}$ is {\em computably dominated} if there is a partial computable function $g:\subseteq\om\to\om$ such that
\[(\forall n\in\om)\;[f(n)\in\om\;\Longrightarrow\;n\in{\rm dom}(g)\mbox{ and }f(n)<g(n)].\]

Then, we consider the {\em extended halfgraph} of $f:\om\to\tilde{\om}$:
\begin{align*}
{\rm HalfGraph}^+(f)=\,&\{2\langle n,m\rangle:f(n)=\bot_0\mbox{ or }f(n)\leq 2m\}\\
\cup\,&\{2\langle n,m\rangle+1:f(n)=\bot_1\mbox{ or }f(n)\geq 2m\}.
\end{align*}

We say that $\mathbf{d}$ is {\em Roy halfgraph-above} if $\mathbf{d}$ contains a set $S$ of the form
\[S=Y\oplus\cmp{Y}\oplus{\rm HalfGraph}^+(f)\]
for some $Y\subseteq\om$ and $f:\om\to\tilde{\om}$ such that $f$ is half-c.e.\ and computably dominated relative to $Y$.

\myprop{proproyhalfgraph}{\index{Degrees!co-d-CEA}\index{Degrees!Roy halfgraph-above}\index{Degrees!doubled co-d-CEA}
Every co-$d$-CEA $e$-degree is Roy halfgraph-above.
Every Roy halfgraph-above $e$-degree is doubled co-$d$-CEA.
}

\mythm{thm:chained-co-d-CEA}{
The $\mathcal{QR}^\om$-degrees are exactly the Roy-halfgraph-above degrees, that is,
\[\mathcal{D}_{\mathcal{QR}^\om}=\{\mathbf{d}\in\mathcal{D}_e:\mathbf{d}\mbox{ is Roy-halfgraph-above}\}.\]
}

\subsection{Degrees of points: submetrizable topology}

We say that a space is {\em submetrizable} if it admits a continuous metric, that is, either it is metrizable or it has a coarser metrizable topology (see \cite{Gru13}).\index{Separation axioms!submetrizable}
Every submetrizable space is completely Hausdorff, and every completely Hausdorff space is $T_{2.5}$:
\[\mbox{submetrizable}\;\Longrightarrow\;\mbox{completely Hausdorff}\;\Longrightarrow\;T_{2.5}.\]

The Arens square and Roy's lattice space are not completely Hausdorff, and hence, not submetrizable.
Extending the topology of a metrizable space always gives us a submetrizable space which is not necessarily metrizable.

\subsubsection{Extension topology}\label{sec:horrible-extension-topology}

One of the most basic constructions to obtain a non-metrizable, completely Hausdorff topology is extending a metrizable topology by adding new open sets.
Concretely speaking, given a space $\xx$ with a metrizable topology $\tau_\xx$, choose a collection $\mathcal{U}$ of subsets of $\xx$, and consider the topology generated by $\tau_\xx\cup\mathcal{U}$.
We denote the obtained space by $\xx_\mathcal{U}$.\index{Extension topology}
By definition, $\xx_\mathcal{U}$ is submetrizable (hence, completely Hausdorff, and $T_{2.5}$); however $\xx_\mathcal{U}$ is not necessarily metrizable.
In this article, since we are only interested in a second-countable topology, we always assume that $\mathcal{U}$ is countable.

\begin{example}
Let us begin with Cantor space $2^\om$.
Let $\mathcal{U}=(U_e)_{e\in\om}$ be a countable collection of subsets of $2^\om$.
This induces a representation of the extension topology on $2^\om$ induced by $\mathcal{U}$, and for any point $x\in (2^\om)_\mathcal{U}$, its coded neighborhood filter is given as:
\[\nbase{x}=\{\langle 0,\sigma\rangle:\sigma\prec x\}\cup\{\langle 1,e\rangle:x\in U_e\}.\]
\end{example}

Let $\Gamma$ be a countable collection of subsets of $\om\times 2^\om$.
We say that a set $A\subseteq\om$ is {\em $\Gamma$ relative to $X\in 2^\om$} (written $A\in\Gamma^X$) if there is $G\in\Gamma$ such that $A=\{n:\langle n,X\rangle\in G\}$.
We also say that an $e$-degree $\mathbf{d}$ is {\em $\Gamma$-above} if $A\oplus X\oplus\cmp{X}\in\mathbf{d}$ for some $A,X\subseteq\om$ such that $A$ is $\Gamma$ relative to $X$.\index{Degrees!$\Gamma$-above}

Let $\beta$ be a representation of a space $\xx$, and let $\gamma$ be a countable collection of (not necessarily open) subsets of $\xx$.
We say that {\em $\gamma$ computably extends $\beta$} if there is a c.e.\ set $W$ such that $\beta_e=\bigcup\{\gamma_d:\langle d,e\rangle\in W\}$.
Then, $\gamma$ is a representation of $\xx_\gamma$ (the result of adding sets from $\gamma$ to the original space $\xx$ as new open sets).

\begin{obs}\label{obs:comp-extension-rep}
If $\gamma$ computably extends $\beta$, then $\nbaseb{\beta}{x}\leq_e\nbaseb{\gamma}{x}$ for any $x\in\xx$.
\end{obs}

\begin{proof}
This is because $e\in\nbaseb{\beta}{x}$ if and only if there is $d$ such that  $\langle d,e\rangle\in W$ and $d\in\nbaseb{\gamma}{x}$.
\end{proof}

Let $\lambda$ be the canonical representation of Cantor space, that is, $\lambda_e$ is the set of all extensions of the $e$-th binary string.
We characterize $\Gamma$-above $e$-degrees in terms of extension topology.

\myprop{prop:Gamma-above}{
The following are equivalent for a collection $\mathcal{C}$ of $e$-degrees:
\begin{enumerate}
\item There is $\beta$ computably extending $\lambda$ such that $\mathcal{C}=\mathcal{D}_{(2^\om)_\beta}$.
\item There is a countable collection $\Gamma$ of subsets of $2^\om$ such that
\[\mathcal{C}=\{\mathbf{d}\in\mathcal{D}_e:\mathbf{d}\mbox{ is $\Gamma$-above}\}.\]
\end{enumerate}
}

\subsubsection{A horrible behavior of extension topology}\label{sec:newhorrible}

In this section, we examine the notion of a $\mathcal{T}$-quasi-minimal $e$-degree for a collection $\mathcal{T}$ of represented cb$_0$ spaces, and try to construct such a degree.

For instance, on the one hand, in Section \ref{sec:T-1-quasiminimal}, we will see that if $\mathcal{T}$ is a countable collection of $T_1$ spaces, then there is a $\mathcal{T}$-quasi-minimal $e$-degree.
On the other hand, in this section, we will show that if $\mathcal{T}$ is a collection of all decidable $T_1$ spaces, there is no $\mathcal{T}$-quasi-minimal $e$-degree.
At first glance, this looks paradoxical.
--- When we talk about an effective version of some notion $\mathcal{P}$, we often implicitly assume that there are only countably many objects which are effectively-$\mathcal{P}$.
However, it is, strictly speaking, sometimes incorrect.
There is a possibility of the existence of uncountably many effective objects.

Consider partial computable functions on Baire space.
Any restriction $f\upto A$ of a partial computable function $f:\subseteq\om^\om\to\om^\om$ is also partial computable, and therefore, there are uncountably many partial computable functions.
Fortunately, every partial computable function is merely a restriction of a partial computable function with a $\Pi^0_2$ domain, and there are only countably many such functions.
Hence, essentially, we only need to deal with countably many partial computable functions.

Can we say the same thing about computable or decidable $T_i$-spaces?
Is there a countable collection $(\xx_n)_{n\in\om}$ of represented $T_i$-spaces such that every computable cb$_0$ space embeds into $\xx_n$ for some $n\in\om$?
Of course, it is true if $i=0$ since there is a universal decidable cb$_0$ space, $\mathbb{S}^\om$ say, or if $i=3$ since there is a universal decidable metric space, $[0,1]^\om$ say.\index{Universal space}
Contrary to these cases, we will show that the answer is ``no,'' for any $i\in\{1,2,2.5\}$.

In this section, we show that if $\mathcal{T}$ is a countable collection of $T_1$ spaces, then there is a $\mathcal{T}$-quasi-minimal $e$-degree, whereas if $\mathcal{T}$ is a collection of all decidable $T_1$ spaces, there is no $\mathcal{T}$-quasi-minimal $e$-degree.

\subsubsection*{Non-existence of universal spaces}

Recall that every metrizable space is submetrizable, and every submetrizable space is $T_{2.5}$.
By using an argument of extending topology, we show that every $e$-degree is realized as the degree of a point in a decidable submetrizable space.
In particular, every $e$-degree is the degree of a point of a decidable $T_{2.5}$ space.

\mythm{thm:e-degree-realize}{\index{Universal space}
Every $e$-degree is an $\xx$-degree for some decidable, submetrizable, cb$_0$ space $\xx$, that is,
\[\mathcal{D}_e=\bigcup\{\mathcal{D}_\xx:\xx\mbox{ is a decidable, submetrizable, cb$_0$ space}\}.\]
}

We will give a more detailed analysis of Theorem \ref{thm:e-degree-realize}.
Recall from Proposition \ref{prop:deBrecht-named} that a represented space is $\mathbf{\Pi}^0_2$-named if and only if it is quasi-Polish.
We will also use a stronger naming condition.
Consider the following sets:
\begin{align*}
{\rm Sup}(\xx)&=\{p\in\om^\om:(\exists x\in\xx)\;{\rm Nbase}(x)\subseteq{\rm rng}(p)\},\\
{\rm Sub}(\xx)&=\{p\in\om^\om:(\exists x\in\xx)\;{\rm rng}(p)\subseteq{\rm Nbase}(x)\}.
\end{align*}

We always have  ${\rm Name}(\xx)\subseteq{\rm Sup}(\xx)\cap{\rm Sub}(\xx)$.
Moreover,
\[\xx\mbox{ is }T_1\;\Longrightarrow\; {\rm Name}(\xx)={\rm Sup}(\xx)\cap{\rm Sub}(\xx).\]

We say that $\xx$ is {\em strongly $\Gamma$-named}\index{Strongly $\Gamma$-named} if there are $\Gamma$ sets $P,N$ such that
\[{\rm Sub}(\xx)\subseteq N,\;{\rm Sup}(\xx)\subseteq P\mbox{, and }{\rm Name}(\xx)=P\cap N.\]

For instance, one can easily see that Baire space $\om^\om$ is strongly $\Pi^0_2$-named, and the telophase space $(\hat{\om}_{TP})^\om$ is strongly $\Pi^0_3$-named.

\myprop{prop:e-degree-realize-arith}{
Let $n\geq 4$.
If $\mathbf{d}$ is an $e$-degree of a $\Delta^0_n$ set, then there is a decidable, strongly $\Pi^0_n$-named, submetrizable, cb$_0$ space $\xx_\mathbf{d}$ such that $\mathbf{d}$ is an $\xx_\mathbf{d}$-degree.
}

In particular, there is a quasi-minimal $e$-degree realized in a strongly $\Pi^0_4$-named submetrizable space since there is a $3$-c.e.~(hence $\Delta^0_2$) quasi-minimal $e$-degree (see Proposition \ref{irrelatti-quasi-minimal}).
Indeed, Proposition \ref{irrelatti-quasi-minimal} shows that a quasi-minimal $e$-degree can be realized in a strongly $\Pi^0_3$-named submetrizable space.\index{quasi-minimal}

\subsubsection{Gandy-Harrington topology}\label{sec:3-7-3}

We now consider a special kind of extension topology.\index{Space!Gandy-Harrington}
Let $GH_e$ be the $e$-th $\Sigma^1_1$ subset of Baire space $\om^\om$.
The extension topology $\tau_{GH}$ on $\om^\om$ generated by $GH=(GH_e)_{e\in\om}$ is known as the {\em Gandy-Harrington topology}.
This topology is known to have a number of applications in descriptive set theory and related areas (e.g.\ \cite{Lou80,HKL90}).

We always assume that the Gandy-Harrington space $(\om^\om)_{GH}=(\om^\om,\tau_{GH})$ is represented by $GH$, and then, for any point $x\in(\om^\om)_{GH}$, its coded neighborhood filter is given as:
\[\nbase{x}=\{e\in\om:x\in GH_e\}.\]

It is clear that every $(\om^\om)_{GH}$-degree is $\Sigma^1_1$-above, but be careful that the converse is not true.
We now assume that $\lambda$ is the canonical representation of Baire space $\om^\om$.

\myprop{prop:gandy-harrington}{
For every $x\in\om^\om$ and $\alpha<\omega_1^{{\rm CK},x}$ (that is, $\alpha$ is an $x$-computable ordinal), we have the following inequalities:
\[\nbaseb{\lambda}{x^{(\alpha)}}\leq_e\nbaseb{GH}{x}\leq_e\nbaseb{\lambda}{x^{{\rm HJ}}},\]
where $x^{(\alpha)}$ denotes the $\alpha$-th Turing jump of $x$, and $x^{\rm HJ}$ denotes the hyperjump of $x$.
}

Recall that an $e$-degree is said to be {\em continuous} if it is an $\mathcal{H}$-degree, where $\mathcal{H}$ is Hilbert cube $[0,1]^\mathbb{N}$ with the canonical representation.
Concretely speaking, the coded neighborhood filter of $x=(x(n))_{n\in\om}\in\mathcal{H}$ is:
\[\nbaseb{\mathcal{H}}{x}=\{\langle n,s,p\rangle:|x(n)-p|<2^{-s}\},\]
where $n$ and $s$ range over $\om$, and $p$ ranges over $\mathbb{Q}\cap[0,1]$.

\mythm{thm:GH-non-continuous}{
No $(\om^\om)_{GH}$-degree is continuous.\index{Degrees!continuous}\index{Space!Gandy-Harrington}
}

\subsubsection{Irregular Lattice Topology}\label{sec:3-7-4}

As another example of an extension topology we consider the irregular lattice topology. It indeed fails to be metrizable (and hence is irregular) -- this is a consequence of Proposition \ref{irrelatti-quasi-minimal} below.
Let $\hat{\om}=\om\cup\{\infty\}$ be the one-point compactification of $\om$.
For a point $x\in\hat{\om}^\om$, its coded neighborhood filter is given as:
\[\nbaseb{\hat{\om}^\om}{x}=\{\langle 0,n,k\rangle:x(n)=k\}\cup\{\langle 1,n,k\rangle:x(n)\geq k\},\]
where $n$ and $k$ range over $\om$.
Note that $\hat{\om}^\om$ is a zero-dimensional compact metrizable space.
In the computability-theoretic context, this is rephrased as follows:

\begin{obs}\label{obs:one-point-compactification}
The $\hat{\om}^\om$-degrees are exactly the total degrees.
\end{obs}

\begin{proof}
To see this (that is, to show totality of a point in $\hat{\om}^\om$), we identify $x\in\hat{\om}^\om$ with a partial function by interpreting $x(n)=\infty$ as that $x(n)$ is undefined.
Then, it is not hard to check that $\nbase{x}$ is $e$-equivalent to ${\rm Graph}(x)\oplus\cmp{{\rm Graph}(x)}$.
The latter set is clearly total.
\end{proof}

We consider an extension topology on (a subset of) $\hat{\om}^\om\times\hat{\om}^\om$.
Let $\mathcal{L}$ be the space whose underlying set is $L=((\om\times\hat{\om})\cup\{(\infty,\infty)\})^\om$ whose topology is generated by the (subspace) topology on $\hat{\om}^\om\times\hat{\om}^\om$.
Then define $\mathcal{L}_{IL}$ as the space obtained by extending $\mathcal{L}$ by adding new open sets $IL=(V_{a,b})_{a,b\in\om}$, where
\[V_{a,b}=\{(c,d)\in\om^2:c\geq a\mbox{ and }d\geq b\}\cup\{(\infty,\infty)\}.\]
We note that $(n,\infty)\not\in V_{a,b}$ for any $n\in\om$.
This is known as the {\em irregular lattice topology} (see Steen-Seebach \cite[II.79]{CTopBook}).\index{Space!irregular lattice space}

We introduce a countable open subbasis of $(\mathcal{L}_{IL})^\om$ as follows.
First put $P_{a,b}=\{(a,b)\}$, $U_{a,b}=\{(a,d)\in\om\times\hat{\om}:d\geq b\}$, and $V_{a,b}$ for $a,b\in\om$.
Given $Q\subseteq\mathcal{L}_{IL}$ and $n\in\om$, we write $Q^n=\{x\in(\mathcal{L}_{IL})^\om:x(n)\in Q\}$.
Then, $(P^n_{a,b},U^n_{a,b},V^n_{a,b}:a,b\in\om)$ forms a subbasis of $(\mathcal{L}_{IL})^\om$.
For instance, an open neighborhood $\{(c,d)\in L:c\geq a\mbox{ and }d\geq b\}$ of $(\infty,\infty)$ can be written as $V_{a,b}\cup\bigcup_{c\geq a}U_{c,b}$.

Each point $z\in(\mathcal{L}_{IL})^\om$ can be thought of as the unique pair $(x,y)$ satisfying $z(n)=(x(n),y(n))$ for all $n\in\om$.
Then, the coded neighborhood filter of $(x,y)\in(\mathcal{L}_{IL})^\om$ is given as:
\begin{align*}
\nbaseb{(\mathcal{L}_{IL})^\om}{x,y}=&\{\langle 0,n,a,b\rangle:x(n)=a\mbox{ and }y(n)=b\}\\
&\cup\{\langle 1,n,a,b\rangle:x(n)=a\mbox{ and }b\leq y(n)\}\\
&\cup\{\langle 2,n,a,b\rangle:x(n)=y(n)=\infty\mbox{ or }[a\leq x(n)\mbox{ and }b\leq y(n)<\infty]\}.
\end{align*}

We will show that
\begin{align*}
\mathcal{D}_{(\mathcal{L}_{IL})^\om}&=\{\mathbf{d}\in\mathcal{D}_e:\mathbf{d}\mbox{ is co-$d$-CEA}\}\\
&=\{\mathbf{d}\in\mathcal{D}_e:\mathbf{d}\mbox{ is [$\ast$,$\Pi^0_1$,$\Pi^0_1$]-separating-above}\}.
\end{align*}

\myprop{prop:irregular-lattice-degree}{
The $(\mathcal{L}_{IL})^\om$-degrees are exactly the co-$d$-CEA degrees.\index{Space!irregular lattice space}\index{Degrees!co-d-CEA}
}

\myprop{prop:irregular-lattice-degree2}{
The $(\mathcal{L}_{IL})^\om$-degrees (hence the co-$d$-CEA degrees) are exactly the $[\ast,\Pi^0_1,\Pi^0_1]$-separating-above $e$-degrees.
In other words, a nonempty set $E\subseteq\om$ is co-$d$-CEA if and only if there are $X,A,B\subseteq\om$ such that $B$ and $A\cup B$ are $X$-co-c.e., $A\cap B=\emptyset$, and
\[{\rm Enum}(E)\equiv_M\{X\}\times{\rm Sep}(A,B).\]
}

Kalimullin has shown that there is a quasi-minimal co-d-c.e.~$e$-degree.
As a consequence:

\begin{prop}\label{irrelatti-quasi-minimal}
The product irregular lattice space $\mathcal{L}^\om$ contains a $\Delta^0_2$ point of quasi-minimal degree.\index{quasi-minimal}
\end{prop}

We also show that the product irregular lattice space is strictly smaller than the product double-origin space in the sense of degrees.

\myprop{prop:proper-doubled}{
There is a doubled co-$d$-CEA $e$-degree which is not co-$d$-CEA.\index{Degrees!doubled co-d-CEA}\index{Degrees!co-d-CEA}
}


\subsection{Degrees of points: $G_\delta$-topology}\label{sec:network}
The next stop in our investigation are the degrees of points in $G_\delta$-spaces and see that these are just the cototal degrees. We briefly recall the definition of $G_\delta$-spaces:

\begin{definition}
A {\em $G_\delta$-space} or a {\em perfect space} is a topological space in which every closed set is $G_\delta$ (i.e.\ an intersection of countably many open sets).\index{Separation axioms!$G_\delta$}
\end{definition}

In other words, a $G_\delta$-space is a space in which the classical Borel hierarchy is well-behaved.
Recall from Section \ref{sec:quasi-Polish} that a set is $\mathbf{\Pi}^0_2$ if it can be obtained as a countable intersection of constructible sets (in the sense of classical algebraic geometry).
Note that we always have $G_\delta\subseteq\mathbf{\Pi}^0_2$.

\myobs{obs:Gdelta-Borel-hierarchy}{
A space is $G_\delta$ if and only if $G_\delta=\mathbf{\Pi}^0_2$.
}

\subsubsection{Closed networks and $G_\delta$-spaces}\label{sec:G-delta-space}

One of the key notions in this article is a network of a topological space introduced by Arhangel'skii in 1959, which has become one of the most fundamental notions in modern general topology.

\begin{definition}[Arhangel'skii]
A {\em network} $\nn$ for a topological space $\xx$ is a collection of subsets (not necessarily open) of $\xx$ such that, for any open neighborhood $U\subseteq\xx$ of a point $x\in\xx$, there is $N\in\nn$ such that $x\in N\subseteq U$.
\end{definition}

If $\Gamma$ is a collection of subsets of $\xx$, by a {\em $\Gamma$ network}, we mean a network consisting of $\Gamma$ sets.
For instance, an open network for $\xx$ is precisely an open basis of $\xx$.
A closed network is a network consisting of closed sets.
Later, we will see that the notion of a {\em closed network} plays an important role in our work.
We first see the following characterization of $T_1$-spaces.

\myobs{observation:closed-network}{\index{Separation axioms!$T_1$}\index{closed network}
A $T_0$ space $\xx$ is $T_1$ if and only if $\xx$ has a closed network.
}

\myprop{prop:G_delta-equal-closed-network}{
A second-countable space $\xx$ is a $G_\delta$-space if and only if $\xx$ has a countable closed network.}

Note that even if $\xx$ is not second-countable, the proof of Proposition \ref{prop:G_delta-equal-closed-network} shows the following implications:
\[\mbox{$\xx$ has a countable closed network }\Longrightarrow\mbox{ $\xx$ is $G_\delta$ }\Longrightarrow\mbox{ $\xx$ has a closed network}.\]
Thus, the property being a $G_\delta$-space can be thought of as a strengthening of being $T_1$ in the category of $T_0$ spaces.
One can see further implications as follows:

\myobs{obs:twin-second-countable}{
For a second-countable $T_0$ space, we have the following implications:\index{Separation axioms!$T_1$}
\begin{align*}
\mbox{compact and $T_1$ } \rotatebox{-30}{ $\Longrightarrow$}&\\
& \mbox{ $G_\delta$ }\Longrightarrow\mbox{ $T_1$}.\\
\mbox{metrizable }\rotatebox{30}{ $\Longrightarrow$}&
\end{align*}
}

We will also see that none of the above implications can be reversed.

\begin{obs}
There is a non-compact non-metrizable second-countable $G_\delta$-space.
For instance, the double origin space $(\mathcal{Q}_{DO})^\om$ is non-compact, non-metrizable, but $G_\delta$.\index{Space!double origin}
\qed
\end{obs}

\myprop{exa:indiscrete-irrational-extension}{
There exist a second-countable submetrizable space which is not $G_\delta$.
For instance, the indiscrete irrational extension of $\mathbb{R}$ is second-countable, submetrizable, but not $G_\delta$.
}

As seen in Section \ref{sec:horrible-extension-topology}, extending topology has an undesirable degree-theoretic behavior, that is, even if we assume decidability of bases, any $e$-degree can be realized by extending a metrizable topology.
Therefore, extension topologies must avoid any nontrivial degree-theoretic characterizations.
An important observation obtained from Proposition \ref{exa:indiscrete-irrational-extension} is that extending topology destroys the property being $G_\delta$.
We will explain the reason of this phenomenon in terms of degree theory:
In contrast to extension topologies, $G_\delta$-spaces are, degree-theoretically, extremely well-behaved.

\subsubsection{Cototal enumeration degrees}\label{sec:cototalenumeration}

In this section, we describe how the notion of a $G_\delta$-space is useful in computability theory.

Recall that a set $A\subseteq\om$ is {\em cototal} if $A\leq_e\cmp{A}$ holds.\index{Degrees!cototal}
As one of the most important results in this article, we will see that the notion of {\em cototality} is captured by $G_\delta$-spaces.
A represented space is computably $G_\delta$ if given a code of a closed set, one can effectively find its $G_\delta$ code (see Definition \ref{def:comp-G-delta} for the precise definition).


\begin{theorem}\label{thm:cototal-Gdelta-space}{\index{Degrees!cototal}\index{Separation axioms!$G_\delta$}
An $e$-degree is cototal if and only if it is an $\xx$-degree of a computably $G_\delta$, cb$_0$ space $\xx$, that is,
\[\{\mathbf{d}\in\mathcal{D}_e:\mathbf{d}\mbox{ is cototal}\}=\bigcup\{\mathcal{D}_\xx:\xx\mbox{ is a computably $G_\delta$, cb$_0$ space}\}\]
}
\end{theorem}

Moreover, we will construct a decidable $G_\delta$-space $A_{\rm max}^{\rm co}$ which is universal in the sense that $A_{\rm max}^{\rm co}$ captures exactly the cototal $e$-degrees:

\mythm{thm:cototal-Gdelta-space2}{\index{universal $G_\delta$ space}\index{Space!maximal antichain}
There exists a decidable, computably $G_\delta$, cb$_0$ space $\xx = A_{\rm max}^{\rm co}$ such that
\[\mathcal{D}_\xx=\{\mathbf{d}\in\mathcal{D}_e:\mathbf{d}\mbox{ is cototal}\}.\]
}

\subsubsection*{Uniform cototality}

We first introduce notions of cototality for a space.

\begin{definition}
We say that a space $\xx$ is {\em relatively cototal} if there are an oracle $C$ and an enumeration operator $\Psi$ such that
\[(\forall x\in\xx)\;{\rm Nbase}(x)=\Psi({\cmp{{\rm Nbase}(x)}\oplus C\oplus\cmp{C}}).\]
If $C$ can be empty, we say that $\xx$ is {\em uniformly cototal}.
\end{definition}

It is clear that if $x$ is a point in a uniformly cototal space $\xx$, then ${\rm Nbase}_\xx(x)$ is cototal.
Baire space, the Hilbert cube, the double origin space, the telophase space, the cocylinder space, etc.~are all uniformly cototal.

\begin{example}[Jeandel \cite{Jea17}]
In universal algebra, a {\em quasi-variety} is a class of structures satisfying a Horn clause.
Jeandel \cite{Jea17} coded this notion as follows.
Let $S\subseteq\om^{<\om}$ be a set.
We say that $X\subseteq\om$ satisfies $S$ if for any $\sigma\in S$, if $\sigma(n)\in X$ for all $0<n<|\sigma|$, then $\sigma(0)\in X$.
A quasi-variety $V$ defined by a set $S$ is the class of all $X\subseteq\om$ satisfying $S$.
If $S$ is c.e.\ we call $V$ a c.e.\ quasi-variety.
For instance, the set of all (forbidden languages of) subshifts ${\sf Sf}_{\Sigma}$ over a finite alphabet $\Sigma$, and the set of all (words of) groups ${\sf Gr}_n$ with $n$ generators are c.e.\ quasi-varieties.

For each $Y\subseteq\om$, we define $[Y]=\{X\subseteq\om:Y\subseteq X\}$.
A set $Y\subseteq\om$ is a presentation of $X\in V$ if $V\cap[Y]=V\cap[X]$.
A point $X\in V$ is {\em finitely presented} if $X$ has a finite presentation.
A point $X\in V$ is {\em maximal} if $X\not=\om$ and $V\cap[X]=\{X,\om\}$.
Maximal points in ${\sf Sf}_{\Sigma}$ and ${\sf Gr}_n$ are minimal subshifts and simple groups, respectively.

Given a quasi-variety $V$, consider the following space (as a subspace of $\mathcal{P}(\om)\simeq\mathbb{S}^\om$):
\[V_{\rm max}^{\rm co}=\{X\subseteq\om:\cmp{X}\in V\mbox{, and }\cmp{X}\mbox{ is maximal in $V$}\}.\]

Jeandel \cite{Jea17} showed that if $V$ is a c.e.\ quasi-variety such that $\om$ is finitely presented, then $V_{\rm max}^{\rm co}$ is uniformly cototal.
In particular, the spaces $({\sf Sf}_{\Sigma})_{\rm max}^{\rm co}$ and $({\sf Gr}_n)_{\rm max}^{\rm co}$ are uniformly cototal.
\end{example}

\begin{remark}
For a quasi-variety $V$, the space $V_{\rm max}^{\rm co}$ is computably homeomorphic to the space of all maximal elements in $V$ equipped with the basis $([D]_{\rm co})$, where $[D]_{\rm co}=\{X:X\cap D=\emptyset\}$, and $D$ ranges over all finite subsets of $\om$.
\end{remark}

\begin{example}[McCarthy \cite{McC17}]\label{exa:McCarthy}
We topologize $\mathcal{P}(\om^{<\om})$ by putting $[D]=\{X\subseteq\om^{<\om}:D\subseteq X\}$ as a basic open set for any finite set $D\subseteq\om^{<\om}$.
Consider the following subspace of $\mathcal{P}(\om^{<\om})$:\index{Space!maximal antichain}
\[A_{\rm max}^{\rm co}=\{X\subseteq\om^{<\om}:\cmp{X}\mbox{ is a maximal antichain}\}.\]

Then, $A_{\rm max}^{\rm co}$ is uniformly cototal.
Moreover, McCarthy \cite{McC17} showed that the $A_{\rm max}^{\rm co}$-degrees and the $({\sf Sf}_{\Sigma})_{\rm max}^{\rm co}$-degrees are exactly the cototal $e$-degrees:
\[\mathcal{D}_{A_{\rm max}^{\rm co}}=\mathcal{D}_{({\sf Sf}_{\Sigma})_{\rm max}^{\rm co}}=\{\mathbf{d}\in\mathcal{D}_e:\mathbf{d}\mbox{ is cototal}\}.\]
\end{example}

One can check that a represented cb$_0$ space $(X,\beta)$ is {\em decidable} (see Section \ref{sec:change-repres}) if $\{\langle d,e\rangle:\beta^+_d\subseteq\beta^+_e\}$ is computable, where we assume that $\emptyset$ and $X$ appear in $\beta$.

\myobs{obs:Amax-decidable}{
$A_{\rm max}^{\rm co}$ is a decidable cb$_0$ space.
}

Then, what kind of topological property is shared by these spaces?
To answer this question, we effectivize the notion of a $G_\delta$-space.
It is obvious that a space is $G_\delta$ if and only if every open set is $F_\sigma$.
We introduce an effective version of this property.

\begin{definition}\label{def:comp-G-delta}
We say that $\xx$ is {\em computably $G_\delta$} if there is a computable procedure that, given a code of open set, returns its $F_\sigma$-code, that is, there is a computable function $f$ such that
\[(\forall e\in\om)\;\beta_e=\bigcup_{n\in\om}P_{f(e,n)}\mbox{, where }P_{d}=\xx\setminus\bigcup\{\beta_c:c\in W_d\}.\]
\end{definition}

Note that a second-countable $T_0$ space is $G_\delta$ if and only if it is computably $G_\delta$ relative to some oracle, or equivalently, it is computably $G_\delta$ w.r.t.\ some representation $\beta$.

We now show an effective topological characterization of uniform cototality.

\mythm{cototal-equal-twin}{
Let $\xx=(X,\beta)$ be a represented cb$_0$ space.
Then, $\xx$ is computably $G_\delta$ if and only if there is a representation $\gamma\equiv\delta$ of $X$ such that $(X,\gamma)$ is uniformly cototal.
}

Theorem \ref{thm:cototal-Gdelta-space} clearly follows from Theorem \ref{cototal-equal-twin}.
This gives us an (effective) topological explanation of why continuous degrees, double origin degrees, graph cototal degrees, etc.~are cototal.

Theorems \ref{thm:cototal-Gdelta-space} and Theorems \ref{thm:cototal-Gdelta-space2} conclude that the cototal $e$-degrees are characterized by the degrees of difficulty of enumerating (neighborhood bases of) points in (decidable) computably $G_\delta$ spaces.


\subsection{Quasi-Polish topology}\label{sec:3-9}

Contrary to the fact that we always have the completion of a metric, there is no hope of getting the notion of ``quasi-completion'' which preserves $T_i$-separation axioms for $i\not=0,3$ (see Section \ref{sec:newhorrible}).\index{Quasi-Polish space}
Instead of considering the notion of quasi-completion\footnote{Some positive results about the possibility of having a quasi-completion were recently obtained by de Brecht \cite{debrecht9}.}, we will directly show the existence of a quasi-Polish space at every separation level.
Remember the spaces $(\hat{\om}_{TP})^\om$, $(\mathcal{P}_{DO})^\om$, and $\mathcal{QA}^\om$ introduced in Examples \ref{example:teloph-2} and \ref{exa:double-quasi-Polish} and in Section \ref{sec:Arens-square}, respectively.

\myprop{prop:quasi-Polish-examples}{
For any $i\in\{0,1,2,2.5\}$, there is a quasi-Polish $T_i$ space which is not $T_j$ for any $j>i$.
Indeed,\index{Separation axioms}
\begin{enumerate}
\item The telophase space $(\hat{\om}_{TP})^\om$ is a quasi-Polish $T_1$-space which is not $T_2$.\index{Space!telophase}
\item The double origin space $(\mathcal{P}_{DO})^\om$ is a quasi-Polish $T_2$-space which is not $T_{2.5}$.\index{Space!double origin}
\item The Arens space $\mathcal{QA}^\om$ is a quasi-Polish $T_{2.5}$-space which is not submetrizable.\index{Space!Arens square}
\item The irregular lattice space $(\mathcal{L}_{IL})^\om$ is a quasi-Polish submetrizable space which is not metrizable.\index{Space!irregular lattice space}
\item The quasi-completion $\overline{\mathbb{R}_<}=\mathbb{R}_<\cup\{\infty\}$ of the lower real line is a quasi-Polish $T_0$-space which is not $T_1$.\index{Space!lower reals}
\end{enumerate}
}

We also show that several natural spaces are not quasi-Polish.

\myprop{prop:non-quasi-Polish-examples}{
{}~
\begin{enumerate}
\item (De Brecht) The Gandy-Harrington space $(\om^\om)_{GH}$ is not quasi-Polish.\index{Space!Gandy-Harrington}
\item The Golomb space $\mathbb{N}_{\rm rp}$ (see Section \ref{sec:T2-deg-nT25list}) is not quasi-Polish.\index{Space!Golomb}
\item The maximal antichain space $\mathcal{A}_{\rm max}^{\rm co}$ is not quasi-Polish.\index{Space!maximal antichain}
\end{enumerate}
}

For (1), \cite[Theorem 6.1]{mummert2} by Mummert and Stephan implies that the Gandy-Harrington space can be represented as the maximal elements of an $\omega$-algebraic domain. Therefore, the Gandy-Harrington space embeds as a (necessarily strict) co-analytic subset of a quasi-Polish space.



\section{Further degree theoretic results}
\label{sec:further}

In the light of Theorem \ref{thm:e-degree-realize} showing that any enumeration degree can be realized in a decidable, submetrizable, cb$_0$ space, it may seem that separating separation axioms via degree-theoretic properties is not possible.
One may overcome this difficulty by restricting our attention to countably many collections of spaces.
We adopt this approach, and we often show strong separating results by considering the notion of quasi-minimality.

\index{quasi-minimal}\index{quasi-minimal!cover}\index{quasi-minimal!$\mathcal{T}$}
\begin{definition}
Let $\mathcal{T}$ be a collection of represented spaces. We say that $\pt{x}{\xx}$ is {\em $\mathcal{T}$-quasi-minimal in $\pt{y}{\yy}$} if

\[(\forall\zz\in\mathcal{T})(\forall z\in\zz)\; \pt{z}{\zz} \leq_{\mathrm{T}} \pt{x}{\xx} \;\Longrightarrow\; \pt{z}{\zz} \leq_{\mathrm{T}} \pt{y}{\yy} .\]
i.e.~if $x$ cannot compute any more points in spaces from $\mathcal{T}$ than $y$ can. If in addition it holds that $ \pt{y}{\yy} <_\mathrm{T} \pt{x}{\xx}$, we say that $\pt{x}{\xx}$ a {\em strong $\mathcal{T}$-quasi-minimal cover of $\pt{y}{\yy}$}.

We just say that $\pt{x}{\xx}$ is $\mathcal{T}$-quasiminimal, if $\pt{x}{\xx}$ is a strong $\mathcal{T}$-quasi-minimal cover of $\pt{0^\om}{2^\om}$, where $0^\om\in 2^\om$ is the computable infinite sequence consisting only of $0$'s (of course, one may replace $0^\om$ with any computable sequence).
\end{definition}

Recall that an $e$-degree ${\bf a}$ is {\em quasi-minimal} if for every total degree ${\bf b}\leq_e{\bf a}$, we have ${\bf b}={\bf 0}$.
Medvedev \cite{medvedev} first constructed a quasi-minimal $e$-degree by showing that every enumeration $1$-generic is quasi-minimal. Our definition is a generalization, in as far as an enumeration degree is quasi-minimal iff it is $\{2^\om\}$-quasi-minimal.

We use $\mathcal{D}_\mathcal{T}$ to denote the union $\bigcup\{\mathcal{D}_\zz:\zz\in\mathcal{T}\}$.


\subsection{$T_0$-degrees which are not $T_1$}
\label{subsec:t0nott1}

\subsubsection{Quasi-minimality}\label{section:t0-nt1-quasiminimality-statements}

We return to the degree spectrum of $\mathbb{R}_<$ considered in Subsection \ref{sec:deg-point-T0}, and recall that these degrees are just the semirecursive enumeration degrees. Arslanov-Cooper-Kalimullin \cite{ACK} showed that if a point $x\in\mathbb{R}_<$ is neither left- nor right-c.e., then it has a quasi-minimal enumeration degree.
We can relativize this property in the following sense.\index{Space!lower reals}

\mylem{propdichotomy}{(see also Kihara-Pauly \cite{KP}) Let $\xx$ be any represented cb$_0$ space, $x\in 2^\omega$, $y\in\mathbb{R}_<$, and $z\in\xx$.
If $\nbaseb{2^\om}{x}\leq_e\nbaseb{<}{y}\oplus\nbaseb{\xx}{z}$, then either $\nbaseb{2^\om}{x}\leq_e\nbaseb{\xx}{z}$ or $\nbaseb{<}{-y}\leq_e\nbaseb{\xx}{z}$ holds.}


The total degrees in $\mathbb{R}_<$ can be characterized as follows.

\myprop{prop:right-ce}{
Let $\bf{a}$ be an $\mathbb{R}_<$-degree. Then, the following are equivalent.\index{Space!lower reals}\index{Degrees!total}
\begin{enumerate}
\item
$\bf{a}$ is a total $e$-degree.
\item $\bf{a}$ is a  $\Pi^0_1$ $e$-degree.
\item $\bf{a}$ is the $e$-degree of $\nbase{x}$ of a left- or right-c.e.~real $x\in\mathbb{R}_<$.
\item $\bf{a}$ is not quasi-minimal.
\end{enumerate}}

In particular, if $\nbaseb{<}{x}\leq_e\nbaseb{<}{y}$ and $x$ is right-c.e., but not left-c.e., then $y$ is right-c.e.
Note that the structure $\mathcal{D}_{\mathbb{R}_<}$ has a maximal element, i.e., the $e$-degree of $\cmp{K}$.
One can also see that the structure of the c.e.~degrees $\left(\mathcal{R},\leq_T\right)$ is isomorphic to the substructure of $\mathcal{D}_{\mathbb{R}_<}$ consisting of the right-c.e.~reals (which are exactly the $\Pi^0_1$ $e$-degrees by Proposition \ref{prop:right-ce}).

As mentioned above, it is known that each enumeration $1$-generic $e$-degree is quasi-minimal, and the same holds for function $1$-generic \cite[Proposition 2.6]{copestake}.
Therefore, we know at least two ways of constructing a quasi-minimal enumeration degree; choosing a point $x\in\mathbb{R}_<$ which is neither left- nor right-c.e., or choosing an enumeration/function generic point.
Here we show that these two constructions are incomparable in the following sense.

\myprop{prop:lower-unbounding}{
No $\mathbb{R}_<^n$-degree computes a function $2$-generic. I.e.~if $x$ is a real and $G\subseteq\om$ be a function $2$-generic, then $G\not\leq_e\nbaseb{<}{x}$.\index{Degrees!2-generic}}

Note that we can see that every nonzero $\mathbb{R}_<$-degree bounds a quasi-minimal $e$-degree.

\myprop{qminimalbound}{
For every $x\in\mathbb{R}_<$, either $\nbaseb{<}{x}$ is c.e.\ or there is quasi-minimal $S\subseteq\om$ such that $S\leq_e\nbaseb{<}{x}$.\index{Space!lower reals}\index{quasi-minimal}
}



\subsubsection{Degree Structure}\label{sec:4-1-2}

It is immediate that every recursively presented perfect\footnote{Unlike the classical case, here \emph{perfect} cannot be replaced by \emph{uncountable} as shown by Gregoriades \cite{gregoriades3}.} Polish space $\xx$ has a computable homeomorphic copy of Cantor space.
Hence, the jump inversion theorem holds in $\xx$.
However, the lower reals $\mathbb{R}_<$ obviously contain no copy of Cantor space.
Nevertheless, we can have the jump inversion theorem for $\mathbb{R}_<$, though the jump and join operations are not in the same space.
Indeed, we here give a short proof of McEvoy's quasi-minimal jump inversion theorem \cite{mce85} inside $\mathbb{R}_<$.

The enumeration jump operator is introduced by Cooper \cite{Coo84}.\index{Enumeration jump operator}
Given $A\subseteq\om$, define $K^A=\{e\in\om:e\in\Psi_e(A)\}$, where $\Psi_e$ is the $e$-th enumeration operator.
Then, the {\em enumeration jump of $A$} is the set ${\rm EJ}(A)=K^A\oplus\cmp{(K^A)}$.
Gregoriades-Kihara-Ng \cite{GreKih} introduced the jump of a point in a represented cb$_0$ space $\xx=(X,\beta)$.
The {\em jump of $x\in\xx$} is the point $J_\xx(x)\in 2^\om$ defined by $J_\xx(x)=\{e\in\om:x\in U_e^\xx\}$, where $U_e^\xx$ is the $e$-th c.e.\ open set in $\xx$.
One can see that for any $x\in\xx$,
\[\nbaseb{2^\om}{J_\xx(x)}\equiv_e{\rm EJ}(\nbaseb{\xx}{x}).\]

The notation such as $J_\xx(x)\equiv_T C$ also makes sense.
This is equivalent to saying that $\nbaseb{2^\om}{J_\xx(x)}\equiv_eC\oplus\cmp{C}$.

We now show the {\em Semirecursive Jump Inversion Theorem}, which generalizes McEvoy's quasi-minimal jump inversion theorem; nevertheless our proof is far simpler than McEvoy's one:

\myprop{semirecjumpinversion}{
For any $C\geq_T\emptyset'$ there is a semirecursive set $A\subseteq\om$ such that $A$ is quasi-minimal and ${\rm EJ}(A)\equiv_eC\oplus\cmp{C}$.\index{Degrees!semi-recursive}
}

We now focus on the first order degree structure of the lower reals.
One of the most fundamental questions in degree theory is whether a given degree structure forms an upper (lower) semilattice or not.
It is not hard to see that $\mathcal{D}_{\mathbb{R}_<}$ has no supremum operation. A somewhat more involved argument also establishes that $\mathcal{D}_{\mathbb{R}_<}$ has no infimum operation.

\myprop{prop:rlowernotupsemi}{
The structure $\mathcal{D}_{\mathbb{R}_<}$ is not an upper semilattice.

Indeed, if $x$ is not $\Delta^0_2$ (as a point in $\mathbb{R}$), then the pair $\nbaseb{<}{x}$ and $\nbaseb{<}{-x}$ has no common upper bound in $\mathcal{D}_{\mathbb{R}_<}$.}

\myprop{thm:non-lower-semilattice}
{The structure $\mathcal{D}_{\mathbb{R}_<}$ is not a lower semilattice.

Indeed, there are right-c.e.~reals $x,y\in\mathbb{R}$ such that the pair $\nbaseb{<}{x}$ and $\nbaseb{<}{Y}$ has no greatest lower bound in  $\mathcal{D}_{\mathbb{R}_<}$.}


We have already noted that the degree structure of $\mathbb{R}_<$ has a maximal element.
Now, it is natural to ask whether this degree structure has a minimal element.
Recall that the degree structures of $2^\omega$ and $[0,1]^\omega$ each have continuum many minimal elements, whereas Guttridge \cite{Gutt71} showed that there is no minimal degree in $\mathbb{S}^\omega$. We can see that $\mathbb{R}_<$ has no minimal degree. We will prove that every noncomputable point $X\in\mathbb{S}^\omega$ computes a noncomputable point $y\in\mathbb{R}_<$ (i.e.~that every nonzero $e$-degree bounds a nonzero semirecursive $e$-degree; i.e.~that there are no $\mathbb{R}_<$-quasi-minimal $e$-degrees). This allows us to transfer Guttridge's result that there are no minimal $e$-degrees to the $\mathbb{R}_<$-degrees.

\mythm{thm:lower-minimal}{
There is no $\mathbb{R}_<$-quasi-minimal $e$-degree.\index{quasi-minimal!$\mathbb{R}_<$}}

The result of Theorem \ref{thm:lower-minimal} was independently obtained by Downey, Greenberg, Harrison-Trainor, Patey and Turetsky \cite{downey}.

\begin{cor}
There is no minimal semirecursive degree.
\end{cor}
\begin{proof}
If there is some nonzero minimal element $\nbaseb{<}{x}$ in $\mathcal{D}_{\mathbb{R}_{<}}$, then we claim that $\nbaseb{<}{x}$ has minimal $e$-degree (which is a contradiction to Guttridge \cite{Gutt71}). Clearly $\nbaseb{<}{x}$ is not c.e. If there is some non-c.e.~set $Y<_e\nbaseb{<}{x}$ then by Theorem \ref{thm:lower-minimal} we have some $y$ where $\emptyset<_e\nbaseb{<}{y}\leq_e Y<_e \nbaseb{<}{x}$, contradicting minimality of $\nbaseb{<}{x}$ in $\mathcal{D}_{\mathbb{R}_{<}}$.
\end{proof}


\subsubsection{$T_1$-quasi-minimal degrees}\label{sec:T-1-quasiminimal}

We extend our idea of the proof of Lemma \ref{propdichotomy} to show the existence of a $T_1$-quasi-minimal degree.

\mythm{thm:countable-T_1-quasiminimal}{
Let $\mathcal{T}$ be a countable collection of second-countable $T_1$ spaces.
Then, there is a $\mathcal{T}$-quasi-minimal semirecursive $e$-degree.\index{quasi-minimal!$\mathcal{T}$}\index{Degrees!semi-recursive}
}

Indeed, we will show that all but countably many semirecursive $e$-degree satisfy the above property.
In particular, almost no point $x\in\mathbb{R}_<$ computes a nontrivial point in a $T_1$ space.

\subsubsection{Products of lower topology}\label{sec:4-1-4}

We say that an $e$-degree $\mathbf{d}$ is {\em $n$-semirecursive} if there are semirecursive $e$-degrees $\mathbf{c}_1,\dots,\mathbf{c}_n$ such that
\[\mathbf{d}=\mathbf{c}_1\oplus\dots\oplus\mathbf{c}_n.\]

Clearly, the $\mathbb{R}_<^n$-degrees are exactly the $n$-semirecursive $e$-degrees.
These degrees have also been studied by Kihara-Pauly \cite{KP}.
In this section, using the idea developed in Sections \ref{sec:deg-point-T0} and \ref{sec:T-1-quasiminimal}, we provide a more detailed analysis of $n$-semirecursive $e$-degrees.

We first see degree-theoretic differences among $\mathbb{R}$, $\mathbb{R}_<$, $\mathbb{R}\times\mathbb{R}_<$, and $\mathbb{R}_<^2$.
We say that $y\in\mathbb{R}$ is {\em left-c.e.\ in $x\in\xx$} if $\pt{y}{\mathbb{R}}$ is $\Sigma^0_1(x)$, that is, $\nbaseb{<}{y}\leq_e\nbaseb{\xx}{x}$.
Similarly, $y\in\mathbb{R}$ is {\em right-c.e.\ in $x\in\xx$} if $\nbaseb{<}{-y}\leq_e\nbaseb{\xx}{x}$.

The latter result shows that $\mathbb{R}_<$-points are useless to compute a $2^\om$-point.

\myprop{prop:RRT1quasimin1}{
Let $\xx$ be a represented cb$_0$ space.\index{Space!lower reals}
For any $x\in\xx$ and $y\in\mathbb{R}$, if $y$ is neither left- nor right-c.e.~in $x$, then $\pt{(x,y)}{\xx\times\mathbb{R}_<}$ is a strong quasi-minimal cover of $\pt{x}{\xx}$.
In particular, we have the following.
\begin{enumerate}
\item Every $\xx\times\mathbb{R}_<$-degree is either an $(\xx\times\mathbb{R})$-degree or a strong quasi-minimal cover of an $\xx$-degree.
\item For any $\xx$-degree $\mathbf{d}$, there is an $(\xx\times\mathbb{R}_<)$-degree which is a strong quasi-minimal cover of $\mathbf{d}$.\index{quasi-minimal!cover}
\end{enumerate}
}

Indeed, we can show the following:

\myprop{prop:RRT1quasimin2}{
Let $\xx$ be a represented cb$_0$ space, and let $\mathcal{T}$ be a countable collection of second-countable $T_1$-spaces.
For any $x\in\xx$ for all but countably many $y\in\mathbb{R}$ it holds that $\pt{(x,y)}{\xx\times\mathbb{R}_<}$ is a strong $\mathcal{T}$-quasi-minimal cover of $\pt{x}{\xx}$.
}

We now turn our attention to the lower real plane $\mathbb{R}_<^2$.
The space $\mathbb{R}_<^2$ also has a point which is a strong quasi-minimal cover of $\emptyset'$ (since any right-c.e.~real has a total degree).
However, we will show that $\mathbb{R}_<^2$ has no point which is a strong quasi-minimal cover of $\emptyset''$, which implies $\mathcal{D}_{\mathbb{R}\times\mathbb{R}_<}\not\subseteq\mathcal{D}_{\mathbb{R}_<^2}$.

\myprop{proptotalordelta2}{
Every $2$-semirecursive $e$-degree is either total or quasi-minimal in $\emptyset'$.\index{Degrees!2-semirecursive}
}

Next we will discuss degree-theoretic properties of $n$-semirecursive $e$-degrees.
We first construct an $(n+1)$-semirecursive $e$-degree which cannot be obtained from an $n$-semirecursive $e$-degree by joining a $T_1$-degree.

\mythm{thmseparation1}{
Let $\mathcal{T}$ be a countable collection of second-countable $T_1$ spaces.\index{Separation axioms!$T_1$}
Then, there is an $(n+1)$-semirecursive $e$-degree which cannot be written as the join of an $n$-semirecursive $e$-degree and a $\mathcal{T}$-degree.
That is,
\[\mathcal{D}_{\mathbb{R}_<^{n+1}}\not\subseteq\left\{\mathbf{c}\oplus\mathbf{d}:\mathbf{c}\in\mathcal{D}_{\mathbb{R}_<^n}\mbox{ and }(\exists\xx\in\mathcal{T})\;\mathbf{d}\in\mathcal{D}_{\xx}\right\}.\]
}

In particular, we have $\mathcal{D}_{\mathbb{R}_<^{n+1}}\not\subseteq\mathcal{D}_{\mathbb{R}\times\mathbb{R}_<^{n}}$.
We next see a degree-theoretic behavior of the join of a $n$-semirecursive $e$-degree and a total $e$-degree.

\mythm{thm:rrn-semirec}{
There are an $n$-semirecursive $e$-degree $\mathbf{c}\leq\mathbf{0}''$ and a total $e$-degree $\mathbf{d}\leq\mathbf{0}''$ such that the join $\mathbf{c}\oplus\mathbf{d}$ is not $(n+1)$-semirecursive.
In particular,
\[\mathcal{D}_{\mathbb{R}\times\mathbb{R}_<^{n}}\not\subseteq\mathcal{D}_{\mathbb{R}_<^{n+1}}.\]
}

The next theorem reveal a strong quasi-minimal behavior of a finite join of semirecursive $e$-degrees.

\mythm{thm:hprincipal}{
For any $n\in\om$, an $n$-semirecursive $e$-degree is either total or a strong quasi-minimal cover of a total $e$-degree.
}


\subsubsection{Quasi-minimality w.r.t.\ telograph-cototal degrees}\label{sec:4-1-5}
\index{Degrees!telograph-cototal}
In the proof of Theorem \ref{thm:countable-T_1-quasiminimal}, we will see that every semirecursive, non-$\Delta^0_2$ $e$-degree is quasi-minimal w.r.t.\ strongly $\Pi^0_2$-named $T_1$-spaces (see \ref{lem:countable-T_1-quasiminimal}).\index{Strongly $\Gamma$-named}
However, we essentially know just one proven example of a $\Pi^0_2$-named $T_1$ space, namely $\om^\om$ (and its variants). We do conjecture that the spaces $(\om_{\rm cof})^\om$, $(\om^\om)_{\rm co}$ and $(\hat{\om}_{\rm TP})^\om$ are not strongly $\Pi^0_2$-named.
Then we introduce a different technique to prove that, if $x\in\mathbb{R}$ is not $\Delta^0_2$, then $\nbaseb{<}{x}$ is $(\hat{\om}_{\rm TP})^\om$-quasi-minimal.

\mythm{telograph-quasiminimal1}{
Every semirecursive, non-$\Delta^0_2$ $e$-degree is quasi-minimal w.r.t.\ telograph-cototal $e$-degrees.\index{Degrees!$\Delta^0_2$}\index{Degrees!telograph-cototal}
}

We will also show that quasi-minimality w.r.t.\ telograph-cototal $e$-degrees is strictly stronger than quasi-minimality (w.r.t.\ total $e$-degrees).

\mythm{telograph-quasiminimal3}{
There is a semirecursive set $A\subseteq\om$ which is quasi-minimal, but not quasi-minimal w.r.t.\ telograph-cototal $e$-degrees.\index{quasi-minimal}\index{quasi-minimal!telograph-cototal}\index{Degrees!semi-recursive}
}

Theorem \ref{telograph-quasiminimal1} implies that there is a $\Sigma^0_2$ $e$-degree which is quasi-minimal w.r.t.\ telograph-cototal $e$-degrees.
However, the proof of Theorem \ref{telograph-quasiminimal3} indeed implies that a semirecursive, $\Delta^0_2$, $e$-degree is not necessarily quasi-minimal w.r.t.\ telograph-cototal $e$-degrees (see Theorem \ref{thm:lower-minimal}).

\subsubsection{An alternate approach via Kalimullin pairs}
\label{subsubsec:mariya}
Mariya Soskova (personal communication) has pointed out to us that some of the results in Subsection \ref{subsec:t0nott1} can alternatively be obtained via Kalimullin pairs:
\index{$\mathcal{K}$-pair}
\begin{definition}[Kalimullin \cite{kalimullin2}]
A $\mathcal{K}$-pair is a pair $(A,B)$ where $A,B \subseteq \om$ and there is a c.e.~set $W \subseteq \om \times \om$ such that $A \times B \subseteq W$ and $A^c \times B^c \subseteq W^c$. The pair is trivial, if $A$ is c.e., and non-trivial otherwise.
\end{definition}

As observed by Kalimullin, $\mathcal{K}$-pairs generalize the notion of semirecursive set in way that preserves many of their good properties. Namely, whenever $A$ is semirecursive, then $(A,A^c)$ is a $\mathcal{K}$-pair. Conversely, as shown in \cite{cai} if $(A,B)$ is a non-trivial $\mathcal{K}$-pair, then there is a semirecursive $C$ with $A \leq_e C$ and $B \leq_e C^c$. Further relevant properties of $\mathcal{K}$-pairs were provided by Ganchev, Soskova and others in \cite{kalimullin2,cai,cai2,sorbi3,GanSos15}.

Soskova in particular gave alternative proofs of Lemma \ref{propdichotomy}, Proposition \ref{prop:right-ce}, Proposition \ref{prop:lower-unbounding}, Proposition \ref{qminimalbound}, Proposition \ref{semirecjumpinversion} and Proposition \ref{proptotalordelta2} by reasoning about $\mathcal{K}$-pairs. By invoking some highly non-trivial results from \cite{soskova2}, she also provided an alternative proof of Theorem \ref{thm:countable-T_1-quasiminimal}. Given the usefulness of $\mathcal{K}$-pairs and the context of this work, the question is raised whether we can characterise the halves of non-trivial $\mathcal{K}$-pairs as the degrees of points in some natural topological space (see Question \ref{question:kpairs}).

\subsection{$T_1$-degrees which are not $T_2$}\label{sec:4-2}

We consider two cb-representations of $\om^\om$.
When we simply write $\omega^\omega$, we assume that $\omega^\omega$ is endowed with the usual Baire representation.
We again use $(\om^\om)_{\rm co}$ to denote the represented space whose underlying set is $\om^\om$ whose cb-representation is given by $B_\sigma:=\{f\in\omega^\omega:\sigma\not\prec f\}$.\index{Space!cocylinder topology}
In this section, we will show the following:

\mythm{thm:T-1-degree-not-T-2}{
For any represented Hausdorff space $\xx$, there is a cylinder-cototal $e$-degree which is not an $\xx$-degree, that is,\index{Degrees!cylinder-cototal}\index{Separation axioms!$T_2$}
\[\mathcal{D}_{(\om^\om)_{\rm co}}\not\subseteq\mathcal{D}_\xx.\]
}

Our proof of Theorem \ref{thm:T-1-degree-not-T-2} is applicable to show the existence of a quasi-minimal degree w.r.t.\ some non-second-countable space.
We will later introduce the notion of degrees of points of certain non-second-countable (but still separable) spaces.
In particular, we will deal with the degree structure of the Kleene-Kreisel space $\N^{\N^\N}$, which has been studied by Hinman \cite{Hinman73}, Normann \cite{NormannBook}, and others.
Then we will show the following:
\index{quasi-minimal!$\N^{\N^\N}$}
\mythm{thm:T-1-degree-NNN-quasiminimal}{
There is a cylinder-cototal $e$-degree which is $\mathbb{N}^{\mathbb{N}^\mathbb{N}}$-quasi-minimal.\index{Degrees!cylinder-cototal}
}

One can also show the following by using the techniques from previous sections:

\mythm{telograph-quasiminimal2}{
There is a cylinder-cototal $e$-degree which is quasi-minimal w.r.t.\ telograph-cototal $e$-degrees.\index{Degrees!cylinder-cototal}\index{quasi-minimal!telograph-cototal}
}

\subsubsection{$T_1$-degrees which are $T_2$-quasi-minimal}\label{sec:4-2-1}

In this section, we show that there is a $T_1$-degree which is $T_2$-quasi-minimal.
This is one of the most nontrivial theorems in this article.
Indeed, the product telophase space $(\hat{\om}_{TP})^\om$ contains a point of $T_2$-quasi-minimal degree in the following sense.

\mythm{thm:T_2-quasiminimal}{\index{Degrees!telograph-cototal}\index{quasi-minimal!$T_2$}
Given any countable collection $\{S_i\}_{i\in\omega}$ of effective $T_2$ spaces, there is a telograph-cototal $e$-degree which is $S_i$-quasi-minimal for any $i\in\om$.}

\subsubsection{Continuous degrees}\label{sec:4-2-2}

Recall that a continuous degree is a degree of a point in a computable metric space.
It is known that there is no quasi-minimal continuous degree (see Miller \cite{miller2}).
Therefore, there is a co-$d$-c.e.\ $e$-degree which is not continuous (see Proposition \ref{irrelatti-quasi-minimal}).
Conversely, by using the notion of cospectrum, we can show the following:

\myprop{thm:continuous-cospec}{
There is a continuous degree which is neither telograph-cototal nor cylinder-cototal.\index{Degrees!continuous}\index{Degrees!telograph-cototal}\index{Degrees!cylinder-cototal}
}

\subsection{$T_2$-degrees which are not $T_{2.5}$}
\label{sec:T2-deg-nT25list}

Let $\mathbb{Z}_+$ be the set of all positive integers.
The {\em relatively prime integer topology} $\tau_{\rm rp}$ on $\mathbb{Z}_+$ is generated by $\{a+b\mathbb{Z}:\gcd(a,b)=1\}$, where $a+b\mathbb{Z}=\{a+bn\in\mathbb{Z}_+:n\in\mathbb{Z}\}$.
This space is also known as the {\em Golomb space}.\index{Space!Golomb}
We write $\mathbb{N}_{\rm rp}:=(\mathbb{Z}_+,\tau_{\rm rp})$.
It is known that $\mathbb{N}_{\rm rp}$ is second-countable, Hausdorff, but not $T_{2.5}$ (see Steen-Seebach {\cite[II.60]{CTopBook}}).
Its countable product $\mathbb{N}_{\rm rp}^\om$ also has the same properties.
Note that the coded neighborhood basis of $x\in\mathbb{N}_{\rm rp}$ is described as follows.
\[\nbase{x}=\{\langle n,a,b\rangle:x(n)\equiv a\mbox{ mod }b\mbox{, and $a$ and $b$ are relatively prime}\}.\]

\mythm{thm:T-2-degree-not-T-25}{\index{Space!Golomb}\index{Separation axioms!$T_{2.5}$}
For any represented $T_{2.5}$-space $\xx$, there is an $(\mathbb{N}_{\rm rp})^\om$-degree which is not an $\xx$-degree, that is,
\[\mathcal{D}_{(\mathbb{N}_{\rm rp})^\om}\not\subseteq\mathcal{D}_\xx.\]
}

As before, our proof of Theorem \ref{thm:T-2-degree-not-T-25} is applicable to show the existence of a quasi-minimal degree w.r.t.\ some non-second-countable space.
\index{quasi-minimal!$\N^{\N^\N}$}
\mythm{thm:T-2-degree-NNN-quasiminimal}{\index{Space!Golomb}\index{Space!Kleene-Kreisel}
There is an $(\mathbb{N}_{\rm rp})^\om$-degree which is $\mathbb{N}^{\mathbb{N}^\mathbb{N}}$-quasi-minimal.
}

\subsection{$T_{2.5}$-degrees which are not $T_3$}\label{sec:4-4}

Recall that $(\om^\om)_{GH}$ is the set $\om^\om$ endowed with the Gandy-Harrington topology.
Recall from Theorem \ref{thm:GH-non-continuous} that no $(\om^\om)_{GH}$-degree is continuous.
In this section, we will prove much stronger results.

\mythm{thm:Gandy-Harrington-closed-neighborhood}{\index{Space!Gandy-Harrington}\index{Separation axioms!regular}
Let $\xx=(X,\nn)$ be a regular Hausdorff space with a countable cs-network.
Then there is an $(\om^\om)_{GH}$-degree which is not an $\xx$-degree, that is,
\[\mathcal{D}_{(\om^\om)_{GH}}\not\subseteq\mathcal{D}_\xx.\]
}


\mythm{thm:GH-not-NNN}{
The Gandy-Harrington space has no point of $\mathbb{N}^{\mathbb{N}^\mathbb{N}}$-degree, that is,\index{Space!Gandy-Harrington}\index{Space!Kleene-Kreisel}
\[\mathcal{D}_{(\om^\om)_{GH}}\cap\mathcal{D}_{\mathbb{N}^{\mathbb{N}^\mathbb{N}}}=\emptyset.\]
}

For an $\om$-parametrized pointclass $\Gamma$, the {\em $\Gamma$-Gandy-Harrington topology} is the topology $\tau_\Gamma$ on $\om^\om$ generated by the subbasis consisting of all $\Gamma$ subsets of $\om^\om$.
By $(\om^\om)_{GH(n)}$, we denote $\om^\om$ endowed with the $\Sigma^1_n$-Gandy-Harrington topology.
We show that there is a hierarchy of degree structures of Gandy-Harrington topologies.

\mythm{thm:Gandy-Harrington-hierarchy}{
For any distinct numbers $n,m\in\om$, there is no $e$-degree which is both an $(\om^\om)_{GH(n)}$-degree and an $(\om^\om)_{GH(m)}$-degree, that is,
\[n\not=m\;\Longrightarrow\;\mathcal{D}_{(\om^\om)_{GH(n)}}\cap\mathcal{D}_{(\om^\om)_{GH(m)}}=\emptyset.\]
}

\section{Proofs for Section \ref{sec:zoo-list}}
\label{sec:zoo-proofs}

\subsection{Degrees of points: $T_1$-topology}

\subsubsection{Cocylinder topology}\label{section:cocylinder}

We first show several basic results on cylinder-cototal degrees (Section \ref{sec:cocylinder-topo}).
The following shows that the standard construction of a proper-$\Sigma^0_2$ $e$-degree is doable in the cocylinder space.

\propproof{prop:propersigmacylindercototal}{[Proof (Sketch)]
In this proof, we assume that the reader is familiar with a priority argument.
We will construct a point $x$ in the cocylinder space fulfilling the following requirements:
\[
\mathcal{P}_{D,\Phi,\Psi}:\Psi({\rm Nbase}(x))\not=D\mbox{ or }\Phi(D)\not={\rm Nbase}(x)\mbox{ or }D\mbox{ is not }\Delta^0_2.
\]
where $D$ ranges over all $\Sigma^0_2$ subsets of $\om$, and $\Phi$ and $\Psi$ range over all enumeration operators.
The following describe the action of a $\mathcal{P}_{D,\Phi,\Psi}$-strategy $\xi$:
\begin{enumerate}
\item Choose $\sigma_\xi$, and enumerate $\sigma_\xi$ into $A={\rm Nbase}(x)$.
\item Wait for $\Phi(D)(\sigma_\xi)=A(\sigma_\xi)=1$ by $\Phi(D)$ enumerating $\sigma_\xi$ with some use $F\subseteq D$.
\item Wait for $F\subseteq \Psi(A)$ with some use $G\subseteq A$.
\item Then the strategy $\xi$ declares that we decided to enumerate $G\setminus\{\sigma_\xi\}$ into $A$ forever.
\item Remove $\sigma_\xi$ from $A$.
\item Wait for $\Phi(D)(\sigma_\xi)=A(\sigma_\xi)=0$ being recovered by $\Phi(D)$ removing $\sigma_\xi$.
This forces $F\not\subseteq D$.
\item Then enumerate $\sigma_\xi$ into $A$.
This recovers $G\subseteq A$, and therefore forces $F\subseteq\Psi(A)$ and thus $\Psi(A)\upto F\not=D\upto F$.
\item Wait for $F\subseteq D$ being recovered.
This may recover $\Psi(A)\upto F=D\upto F$, but this forces $\Phi(D)(\sigma_\xi)=1$.
Then go back to Step 5.
\end{enumerate}

For each stage reaching at Step 5, the strategy $\xi$ returns the infinitary outcome $\infty$.
Otherwise, the strategy $\xi$ returns the finitary outcome ${\tt f}$.
It should be careful about the choice of $\sigma_\xi$.
Let $\zeta$ be the maximal string such that $\zeta\fr\infty\preceq\xi$.
\begin{enumerate}
\item Let $\sigma_\xi$ be the lexicographically least immediate successor of $\sigma_\zeta$ which is neither chosen by any strategy nor declared to be determined.
\item Enumerate all strings incomparable with $\sigma_\zeta$ into $A$.
Remove all initial segments of $\sigma_\xi$ from $A$.\qedhere
\end{enumerate}
}

\subsubsection{Products of cocylinder topology}

We next prove the following result mentioned in Section \ref{sec:3-4-2}.


\thmproof{thm:coBaire-hierarchy}{
First note that a basic open set in $(\om^\om_{\rm co})^n$ is of the form $\prod_{k<n}\cmp{[D_k]}$ for some collection $(D_k)_{k<n}$ of finite sets of strings.
We code the set $\prod_{k<n}\cmp{[D_k]}$ by $(D_k)_{k<n}$.
Therefore, an enumeration operator from $(\om^\om_{\rm co})^{m}$ to $(\om^\om_{\rm co})^{n}$ is a c.e.~set $\Psi$ of collection of $(n+m)$-tuples of the form $((\sigma_k)_{k<n},(D_\ell)_{\ell<m})$.
For each such enumeration operator $\Psi$, we write $\Psi^{[k]}$ for its $k$-th section, that is, a collection of tuples of the form $(\sigma_k,(D_\ell)_{\ell<m})$.

We will construct a tuple $\bvec{x}=(x_k)_{k\leq n}\in(\om^\om_{\rm co})^{n+1}$ fulfilling the following requirements:
\[R_{\langle d,e\rangle}: [(\exists \bvec{y}\in(\om^\om_{\rm co})^n\;\Phi_d^{{\rm Nbase}(\bvec{x})}={\rm Nbase}(\bvec{y})]\;\Longrightarrow\;\Psi_e^{{\rm Nbase}(\bvec{y})}\not={\rm Nbase}(\bvec{x}).\]

Let $s=\langle d,e\rangle$.
Inductively we assume that $x_k\upto s$ is constructed for every $k\leq n$.
We also assume that we have constructed a collection $(E_{k,s})_{k\leq n}$ of finite sets of strings such that $\sigma\not\in E_{k,s}$ for any $\sigma\preceq x_k\upto s$, where $E_{k,0}=\emptyset$.
At stage $s$, proceed the following strategy:

\begin{enumerate}
\item Choose an $(n+1)$-tuple $(a_k)_{k\leq n}$ of large numbers which are not mentioned by $E_{k,s}$.
\item Ask whether there exists a finite collection $\bvec{D}=(D_\ell)_{\ell<n}$ of finite sets of strings such that $\langle ((x_k\upto s)\fr a_k)_{k\leq n},(D_\ell)_{\ell<n}\rangle\in\Psi_e$.  \label{coBaire-question}
\item If there is no such $\bvec{D}$, for each $k\leq n$, choose a large number $a_k'\not=a_k$ not mentioned in $E_{k,s}$, and define $x_k(s)=a_k'$, and $E_{k,s+1}=E_{k,s}$.\label{coBaire-item1}
Go to stage $s+1$.
\item If such a $\bvec{D}$ exists, we say that a finite set $G$ of strings is {\em $k$-avoidable} if $G$ has no initial segment of $(x_k\upto s)\fr a_k$.\label{coBaire-yes}
Then we say that a $(n+1)$-tuple $(G_k)_{k\leq n}$ of finite sets of strings is {\em avoidable except for $j$} if $G_k$ is $k$-avoidable for any $k\not=j$, and that $(G_k)_{k\leq n}$ is {\em all-but-one avoidable} if it is avoidable except for at most one $j$.
\item Ask whether for any $\ell<n$ there is an all-but-one avoidable tuple $(G^\ell_k)_{k\leq n}$ such that for any $\sigma\in D_\ell$, $\langle\sigma,(G^\ell_k)_{k\leq n}\rangle\in\Phi_d^{[\ell]}$.\label{coBaire-item5}
\item If yes, define $E_{k,s+1}=E_{k,s}\cup\bigcup_{\ell<n}G^\ell_{k}$.\label{coBaire-item2}
Note that there is $m\leq n$ such that $G^\ell_{m}$ is $m$-avoidable for any $\ell<n$ since $((G^\ell_{k})_{k\leq n})_{\ell<n}$ is an $n$-tuple of all-but-one avoidable $(n+1)$-tuples.
For such $m$, we put $a_m'=a_m$, and for each $k\not=m$, choose a large $a_k'\not=a_k$ not mentioned in $E_{k,s+1}$.
Then define $x_k(s)=a_k'$.
Go to stage $s+1$.
\item If no with $\ell<n$, note that if $G_0$ and $G_1$ are $k$-avoidable, so is $G_0\cup G_1$.
Therefore, for any $j\leq n$, there is $\sigma_j\in D_\ell$ such that if $(G_k)_{k\leq n}$ is avoidable except for $j$, we have $\langle\sigma_j,(G_k)_{k\leq n}\rangle\not\in\Phi_d^{[\ell]}$.
\item If $\sigma_i$ and $\sigma_j$ are incomparable for some $i\not=j$, define $x_k(s)=a_k$ and $E_{s+1}=E_s$.
Go to stage $s+1$.\label{coBaire-item6}
\item If $(\sigma_j)_{j\leq n}$ is pairwise comparable, then let $\sigma_{\bvec{D}}$ be the shortest one. \label{coBaire-item3}
\end{enumerate}

\noindent
{\bf Case 1.}
We reach Step (\ref{coBaire-item1}).
In this case, note that for any $\bvec{y}\in(\om^\om_{\rm co})^n$, $\Psi_e^{{\rm Nbase}(\bvec{y})}$ does not enumerate $((x_k\upto s)\fr a_k)_{k\leq n}$, that is, if $\Psi_e^{{\rm Nbase}(\bvec{y})}$ is a neighborhood basis of a point $(z_k)_{k\leq n}\in(\om^\om_{\rm co})^{n+1}$, then $z_k$ must extend $(x_k\upto s)\fr a_k$ for some $k\leq n$.
Since our action at Step (\ref{coBaire-item1}) ensures that $x_k\upto s+1$ is incomparable with $(x_k\upto s)\fr a_k$ for every $k\leq n$, the requirement $R_{\langle d,e\rangle}$ is fulfilled.

\medskip
\noindent
{\bf Case 2.}
Otherwise, for $\bvec{a}=(a_k)_{k\leq n}$, let $\mathcal{D}_{\bvec{a}}$ be the set of all $\bvec{D}$'s witnessing that the question in Step (\ref{coBaire-question}) is true.
Consider the case that we reach Step (\ref{coBaire-item2}) or (\ref{coBaire-item6}) with some $\bvec{D}\in\mathcal{D}_{\bvec{a}}$.

Assume that we reach Step (\ref{coBaire-item2}) with a collection $((G^\ell_k)_{k\leq n})_{\ell<n}$ of avoidable tuples.
Fix $m\leq n$ such that $G_m^\ell$ is $m$-avoidable.
Then $\bigcup_\ell G^\ell_{m}$ is also $m$-avoidable, and moreover, $E_{m,s}$ is $m$-avoidable by our choice of $a_m$.
Therefore, $E_{m,s+1}$ is also $m$-avoidable, and therefore, $E_{m,s+1}$ has no initial segment of $x_m\upto s+1=(x_m\upto s)\fr a_m$.
Moreover, for each $k\not=m$, by our choice of $a_k'$, $E_{k,s+1}$ has no initial segment of $x_k\upto s+1=(x_k\upto s)\fr a_k'$.
Since $G^\ell_k\subseteq E_{k,s+1}$ for any $\ell<n$ and $k\leq n$, given $\bvec{x}=(x_k)_{k\leq n}$, if $x_k$ is an extensions of $x_k\upto s+1$, then $(G^\ell_k)_{k\leq n}\subseteq{\rm Nbase}(\bvec{x})$, and therefore, the $\ell$-th section of $\Phi_d^{{\rm Nbase}(\bvec{x})}$ enumerates all strings in $D_\ell$ for any $\ell<n$.
However, since we have $\langle ((x_k\upto s)\fr a_k)_{k\leq n},(D_\ell)_{\ell<n}\rangle\in\Psi_e$, if $\Psi_e\Phi_d^{{\rm Nbase}(\bvec{x})}$ enumerates a neighborhood basis of a point $(z_k)_{k\leq n}\in(\om_{\rm co}^\om)^{n+1}$, $z_k$ cannot extend $(x_k\upto s)\fr a_k$ for any $k\leq n$.
Since $x_m\upto s+1=(x_m\upto s)\fr a_m$, we must have $\Psi_e\Phi_d^{{\rm Nbase}(\bvec{x})}\not={\rm Nbase}(\bvec{x})$.
Thus, the requirement $R_{\langle d,e\rangle}$ is fulfilled.

Assume that we reach Step (\ref{coBaire-item6}) with $\ell$ and incomparable $\sigma_i$ and $\sigma_j$.
Later we will also use the symbol $\ell_{\bvec{D}}$ to specify this $\ell$.
In this case, our action ensures that $x_k\upto s+1=(x_k\upto s)\fr a_k$, and therefore, if $(G_k)_{k\leq n}\subseteq{\rm Nbase}(\bvec{x})$, then $G_k$ must be $k$-avoidable for any $k\leq n$.
In particular, for any $k\leq n$, $(G_k)_{k\leq n}$ is avoidable except for $k$.
Thus, we have $\langle\sigma_k,(G_k)_{k\leq n}\rangle\not\in\Phi_d^{[\ell]}$.
Since $(\sigma_k)_{k\leq n}$ contains an incomparable pair, the $\ell$-th section $\Phi_d^{{\rm Nbase}(\bvec{x})}$ does not define a  point in $\om^\om_{\rm co}$.
Thus, the requirement $R_{\langle d,e\rangle}$ is fulfilled.

\medskip

\noindent
{\bf Case 3.}
Assume that we reach Step (\ref{coBaire-item3}) for any $\bvec{D}\in\mathcal{D}_{\bvec{a}}$.
We say that $\bvec{z}=(z_k)_{k\leq n}\in(\om^\om_{\rm co})^{n+1}$ is {\em all-but-one good} (for $\bvec{a}=(a_k)_{k\leq n}$) if $z_k$ extends $(x_k\upto s)\fr a_k$ for all but one $k\leq n$.
If $\bvec{z}$ is all-but-one good, any $\bvec{G}\subseteq{\rm Nbase}(\bvec{z})$ is all-but-one avoidable.
Therefore, if $\Phi_d^{{\rm Nbase}(\bvec{z})}$ enumerates a neighborhood basis of a point $(y_\ell)_{\ell<n}$, $y_\ell$ must extend $\sigma_{\bvec{D}}$, where $\ell=\ell_{\bvec{D}}$.

If there are $\bvec{D},\bvec{D}'\in\mathcal{D}_{\bvec{a}}$ such that $\ell_{\bvec{D}}=\ell_{\bvec{D}'}$ and that $\sigma_{\bvec{D}}$ is incomparable with $\sigma_{\bvec{D}'}$, then this means that $\Phi_d^{{\rm Nbase}(\bvec{z})}$ is not a neighborhood basis of a point for any all-but-one good tuple $\bvec{z}$.
In this case, by putting $x_k(s)=a_k$, and $E_{k,s+1}=E_{k,s}$ for each $k\leq n$, the requirement $R_{\langle d,e\rangle}$ is fulfilled.
Then go to stage $s+1$.


\medskip

\noindent
{\bf Case 4.}
Otherwise, for any $\bvec{D},\bvec{D}'\in\mathcal{D}_{\bvec{a}}$, if $\ell_{\bvec{D}}=\ell_{\bvec{D}'}$, then $\sigma_{\bvec{D}}$ and $\sigma_{\bvec{D}'}$ are comparable.
Then we define $\sigma^{\bvec{a}}_\ell$ for any $\ell<n$ and large tuple $\bvec{a}$ not mentioned in $E_s$ as follows:
\[\sigma^{\bvec{a}}_\ell=\bigcup\{\sigma_{\bvec{D}}:\bvec{D}\in\mathcal{D}_{\bvec{a}}\mbox{ and }\ell_{\bvec{D}}=\ell\}.\]

If there is a tuple $(a_k)_{k\leq n}$ such that, whenever $x_k$ extends $(x_k\upto s)\fr a_k$ for any $k\leq n$, $\Phi_d^{{\rm Nbase}(\bvec{x})}$ does not enumerate a neighborhood basis of a point in $(\om_{\rm co}^\om)^n$, then we just put $x_k(s)=a_k$ and $E_{k,s+1}=E_{k,s}$ for each $k\leq n$, and then go to stage $s+1$.

Therefore, we can assume that for any $(a_k)_{k\leq n}$, there are $\bvec{x}=(x_k)_{k\leq n}$ such that $x_k$ extends $(x_k\upto s)\fr a_k$ and that $\Phi_d^{{\rm Nbase}(\bvec{a})}$ enumerates a neighborhood basis of a point in $(\om_{\rm co}^\om)^n$.
Under this assumption, we show the following claim.

\begin{claim}
For any $(n+1)$-tuples $\bvec{a},\bvec{b}\in\om^{n+1}$ of large numbers and $\ell<n$, $\sigma^{\bvec{a}}_\ell$ and $\sigma^{\bvec{b}}_\ell$ are comparable.
\end{claim}

\begin{proof}
Given large $\bvec{a}$, assume that $\bvec{z}$ is an all-but-one good tuple (for $\bvec{a}$), and that $\Phi_d^{{\rm Nbase}(\bvec{z})}$ defines a neighborhood basis of a point $\bvec{y}\in(\om_{\rm co}^\om)^n$.
For any $t$, there is $\bvec{D}\in\mathcal{D}_{\bvec{a}}$ such that $\ell_{\bvec{D}}=\ell$ and $\sigma^{\bvec{a}}_\ell\upto t\preceq\sigma_{\bvec{D}}$.
As mentioned in Case 3, for any such $\bvec{D}$, $y_\ell$ extends $\sigma_{\bvec{D}}$.
In particular, we have $\sigma^{\bvec{a}}_{\ell}\upto t\preceq y_\ell$.
Since $t$ is arbitrary, we get  $\sigma^{\bvec{a}}_{\ell}\preceq y_\ell$.
Note that $\bvec{x}\upto s$ followed by $\bvec{a}=(a_0,a_1,a_2,\dots,a_n)$ is all-but-one good for $\bvec{a}[b_0/a_0]:=(b_0,a_1,a_2\dots, a_n)$, $\bvec{x}\upto s$ followed by $\bvec{a}[b_0/a_0]$ is all-but-one good for $\bvec{a}[b_0/a_0,b_1/a_1]:=(b_0,b_1,a_2\dots, a_n)$, and so on.
Therefore, for any $\bvec{x}=(x_k)_{k\leq n}$, if $x_k$ extends $(x_k\upto s)\fr a_k$ for every $k$, and if $\Phi_d^{{\rm Nbase}(\bvec{x})}$ defines a neighborhood basis of a point $(y_\ell)_{\ell<n}\in (\om_{\rm co}^\om)^n$, then $y_\ell$ extends $\sigma_\ell^{\bvec{a}}$, $\sigma^{(b_0,a_1,a_2,\dots,a_n)}_{\ell}$, $\sigma^{(b_0,b_1,a_2,\dots,a_n)}_{\ell}$, and so on.
This implies that all of these strings are comparable.
By our assumption, for any large $\bvec{a}$, there is such $\bvec{x}$, and therefore, for any large $\bvec{a}$ and $\bvec{b}$, and any $\ell<n$, $\sigma^{\bvec{a}}_\ell$ and $\sigma^{\bvec{b}}_\ell$ are comparable.
Thus, our claim is verified.
\end{proof}

Now, for any $\bvec{a}\in(\om^\om_{\rm co})^{n+1}$, we define $r(\bvec{a})\in(\om+1)^n$ as follows:
\[r(\bvec{a})(\ell)=|\sigma^{\bvec{a}}_\ell|.\]

Consider any infinite sequence $(\bvec{a}^i)_{i\in\om}$ of large $(n+1)$-tuples $\bvec{a}^i=(a^i_k)_{k\leq n}$ such that if $i\not=j$ then ${a}^i_k\not={a}^j_k$ for any $k\leq n$.
Let $\leq^n$ be the product order on $(\om+1)^n$.
Since $\om+1$ is a well quasi order, so is $((\om+1)^n,\leq^n)$ by Dickson's lemma.
Thus, there are $j<i$ such that $r(\bvec{a}^j)\leq^nr(\bvec{a}^i)$.
Put $\bvec{a}=\bvec{a}^i$ and $\bvec{b}=\bvec{a}^j$, and define $x_k(s)=a_k$ for each $k\leq s$.

Assume that $x_k$ extends $x_k\upto s+1$ for any $k\leq n$.
Then, $(x_k\upto s)\fr b_k)_{k\leq n}\in{\rm Nbase}(\bvec{x})$ since $x_k(s)=a_k\not=b_k$ for any $k\leq n$.
If $\Phi_d^{{\rm Nbase}(\bvec{x})}$ defines a neighborhood basis of a point $\bvec{y}=(y_\ell)_{\ell<n}\in(\om_{\rm co}^\om)^n$ then we must have $y_\ell\succeq\sigma^{\bvec{a}}_\ell$ for every $\ell<n$.
Since $r(\bvec{b})\leq^nr(\bvec{a})$, we get $y_\ell\succeq\sigma^{\bvec{b}}_\ell$ for every $\ell<n$.
However, by our choice of $\sigma^{\bvec{b}}_{\ell}$, if $\langle ((x_k\upto s)\fr b_k)_{k\leq n},\bvec{D}\rangle\in\Psi_e$, then for $\ell=\ell_{\bvec{D}}$, there is $\sigma\in D_\ell$ such that $\sigma\preceq\sigma^{\bvec{b}}_\ell$.
Therefore, $D_\ell$ contains an initial segment of $y_\ell$, and thus $D\not\subseteq{\rm Nbase}(\bvec{y})$.
Hence, we have $\Psi_e^{{\rm Nbase}(\bvec{y})}\not={\rm Nbase}(\bvec{x})$, that is, $\Psi_e\Phi_d^{\nbase{\bvec{x}}}\not={\rm Nbase}(\bvec{x})$.
Put $E_{k,s+1}=E_{k,s}$ for each $k\leq n$, and then go to stage $s+1$.
Then, the requirement $R_{\langle d,e\rangle}$ is fulfilled.
}


\subsubsection{Telophase topology}\label{sec:telophasetopo}

Next, we show several basic results on telophase spaces and telograph-cototal degrees.
For definitions, see Section \ref{sec:3-4-3}.

\propproof{prop:ctpcomputablyembeds}{
We define a function $c:\mathcal{C}_{TP}\to 2^\om$ by $c(x)=1^\om$ if $x=\mathbf{1}_\star$; otherwise $c(x)=x$.
Given $x\in\mathcal{C}_{TP}$, we define $h(x)(n+1)=c(x)(n)$.
We define $h(x)(0)=n$ if we find that $1^n0\prec x$.
Otherwise, $x\in\{\mathbf{1},\mathbf{1}_\star\}$.
If $x=\mathbf{1}$, then define $h(x)(0)=\infty$, and if $x=\mathbf{1}_\star$, then define $h(x)(0)=\infty_\star$.
It is not hard to check that $h\colon\mathcal{C}_{TP}\to(\hat{\om}_{TP})^\om$ is an embedding, that is, there are enumeration operators $\Phi,\Psi$ witnessing that $\nbase{x}\equiv_e\nbase{h(x)}$ for any $x\in\mathcal{C}_{TP}$.

For the latter assertion, it is clear that $\xx^{\om\times\om}$ is computably homeomorphic to $\xx^{\om}$.
Thus, $(\mathcal{C}_{TP})^\om$ computably embeds into $(\hat{\om}_{TP})^\om$, and therefore $\mathcal{D}_{(\mathcal{C}_{TP})^\om}=\mathcal{D}_{(\hat{\om}_{TP})^\om}$.
}

\propproof{obs:telograph-hierarchy-collapses}{
For the first assertion, given $g:\om\to\om$, consider
\[G=\{\langle n,m,0\rangle:g(n)\not=m+1\}\cup \{\langle n,m,1\rangle:g(n)=m+1\}.\]
Clearly, $G$ is total, and one can check that $G\equiv_e\cmp{{\rm Graph}(g)}\oplus{\rm TGraph}_1(g)$.

One can easily check that every $b$-telograph-cototal $e$-degree is $(b+1)$-telograph-cototal by considering $\tilde{g}(n)=g(n)+1$.
To see that every $b$-telograph-cototal $e$-degree is $2$-telograph-cototal, given $g:\om\to\om$ and $i<b$, consider the following:
\[
g_i(n)=
\begin{cases}
0&\mbox{ if }g(n)=i,\\
1&\mbox{ if }g(n)<b\mbox{ and }g(n)\not=i,\\
g(n)-b+2&\mbox{ if }g(n)\geq b.
\end{cases}
\]

Define $\tilde{g}$ by $\tilde{g}(bn+i)=g_i(n)$.
Then, we claim that $\cmp{{\rm Graph}(\tilde{g})}\oplus{\rm TGraph}_2(\tilde{g})\equiv_e\cmp{{\rm Graph}(g)}\oplus{\rm TGraph}_b(g)$.
It is straightforward to see the following for $n,m\in\om$ and $i<b$,
\begin{align*}
\langle bn+i,0\rangle\in\cmp{{\rm Graph}(\tilde g)}&\iff\langle n,i\rangle\in\cmp{{\rm Graph}(g)},\\
\langle bn+i,1\rangle\in\cmp{{\rm Graph}(\tilde g)}&\iff(\forall j<b)\;j=i\mbox{ or }\langle n,j\rangle\in\cmp{{\rm Graph}(g)},\\
\langle bn+i,m+2\rangle\in\cmp{{\rm Graph}(\tilde g)}&\iff\langle n,m+b\rangle\in\cmp{{\rm Graph}(g)},\\
\langle bn+i,m+2\rangle\in{\rm TGraph}_2(\tilde g)&\iff\langle n,m+b\rangle\in{\rm TGraph}_b(g).
\end{align*}
The reduction $\leq_e$ clearly follows from the above equivalences.
For the reduction $\geq_e$, see the first, third, and forth equivalences.
}

\propproof{prop:telophase-is-graph-cototal}{
By Proposition \ref{obs:telograph-hierarchy-collapses}, it suffices to show that the $(\hat{\om}_{TP})^\om$-degrees are exactly the $2$-telograph-cototal $e$-degrees.
Given a point $x\in(\hat{\om}_{TP})^\om$, consider the following $g_x$:
\[
g_x(n)=
\begin{cases}
0&\mbox{ if }x(n)=\infty,\\
1&\mbox{ if }x(n)=\infty_\star,\\
x(n)+2&\mbox{ if }x(n)\in\om.
\end{cases}
\]

We claim that $\nbase{x}\equiv_e\cmp{{\rm Graph}(g_x)}\oplus{\rm TGraph}_2(g_x)$.
Recall from Example \ref{example:teloph-2} the definition of $\nbase{x}$ in the telophase space $(\hat{\om}_{TP})^\om$.
To verify the reduction $\leq_e$, we claim that
\begin{align*}
\langle n,0,m\rangle\in\nbase{x} & \iff\langle n,m+2\rangle\in{\rm TGraph}_2(g_x),\\
\langle n,1,m\rangle\in\nbase{x} & \iff\langle n,1\rangle,\langle n,2\rangle,\dots,\langle n,m+1\rangle\in\cmp{{\rm Graph}(g_x)},\\
\langle n,2,m\rangle\in\nbase{x} & \iff\langle n,0\rangle,\langle n,2\rangle,\dots,\langle n,m+1\rangle\in\cmp{{\rm Graph}(g_x)}.
\end{align*}
This is because, for the first equivalence, if $x(n)=m\in\om$, then $g_x(n)=m+2\geq 2$, and thus $\langle n,m+2\rangle$ is enumerated into ${\rm TGraph}_2(g_x)$.
For the second equivalence, $m\leq x(n)\leq\infty$ if and only if $x(n)\not=\infty_\star$ and $x(n)\not\in\{0,\dots,m-1\}$.
This means that $\langle n,1\rangle\not\in{\rm Graph}(g_x)$ and $\langle n,2\rangle,\dots,\langle n,m-1+2\rangle\not\in{\rm Graph}(g_x)$.
The last equivalence holds by a similar reason.
The above three equivalences give us a reduction witnessing $\nbase{x}\leq_e\cmp{{\rm Graph}(g_x)}\oplus{\rm TGraph}_2(g_x)$.

For the reduction $\geq_e$, first note that $g_x(n)\not=\infty$ if and only if $m\leq x(n)\leq\infty_\star$ for some $m\in\om$.
Similarly, $g_x(n)\not=\infty_\star$ if and only if $m\leq x\leq\infty$ for some $m\in\om$.
Therefore, for $i<2$,
\[\langle n,i\rangle\in\cmp{{\rm Graph}(g_x)}\iff\langle n,2-i,m\rangle\in\nbase{x}.\]

For $m\in\om$, $g_x(n)\not=m+2$ if and only if either $x(n)=k$ for some $k<m$ or $k\leq x$ for some $k>m$.
Therefore, for $m\in\om$,
\begin{align*}\langle n,m+2\rangle\in\cmp{{\rm Graph}(g_x)}\iff&(\exists k<m)\;\langle n,0,k\rangle\in\nbase{x}\\
&\mbox{ or }(\exists k>m)(\exists i<2)\;\langle n,i+1,k\rangle\in\nbase{x}.
\end{align*}

Finally, it is clear that for any $n,m\in\om$,
\[\langle n,m\rangle\in{\rm TGraph}_2(g_x)\iff m\geq 2\mbox{ and }\langle n,0,m-2\rangle\in\nbase{x}.\]

The above equivalences give us a reduction witnessing $\cmp{{\rm Graph}(g_x)}\oplus{\rm TGraph}_2(g_x)\leq_e\nbase{x}$.
This concludes the proof.
}


\thmproof{thm:telophase-degree}{
Given a point $x\in(\hat{\om}_{TP})^\om$, we define
\[X=\{2\langle n,m\rangle:x(n)=m\}\cup\{2\langle n,m\rangle+1:x(n)\not=m\},\]
where $m$ ranges over $\om$.
Clearly, $X$ is total.
We define $A=\{n\in\om:x(n)=\infty\}$ and $B=\{n\in\om:x(n)=\infty_\star\}$.
It is clear that $A\cap B=\emptyset$.
Note that $A\cup B$ is co-c.e.~relative to $X$, since $n\in A\cup B$ if and only if $2\langle n,m\rangle+1\in X$ for any $m>0$.

It is clear that $X\oplus\cmp{X}\leq_e{\rm Nbase}(x)$.
To see ${\rm Sep}(A,B)\leq_M{\rm Nbase}(x)$, given $n\in\om$, wait for the first triple $\langle n,i,m\rangle$ to be enumerated into $\nbase{x}$.
If $i=1$ (then, $x(n)\not=\infty_\ast$), enumerate $n$ into $C$.
If $i=2$ (then, $x(n)\not=\infty$), enumerate $n$ into $\cmp{C}$.
If $i=0$ (then, $x(n)\in\om$), enumerate $n$ into $C$.
The constructed set $C$ satisfies that $C\in{\rm Sep}(A,B)$.

Conversely, assume that $C\in{\rm Sep}(A,B)$ is given.
For each $n$, if $n\in C$, then $x(n)\not=\infty_\star$.
Thus, enumerate $\langle n,1,m\rangle$ into $\nbase{x}$ (which indicates that $m\leq x(n)\leq\infty$) if $2\langle n,k\rangle+1$ is enumerated into $X$.
If $n\not\in C$, then $x(n)\not=\infty$.
Thus, enumerate $\langle n,2,m\rangle$ into $\nbase{x}$ (which indicates that $m\leq x(n)\leq\infty_\star$) if $2\langle n,k\rangle+1$ is enumerated into $X$.
Moreover, if $2\langle n,m\rangle$ is enumerated into $X$, then enumerate $\langle n,0,m\rangle$ and $\langle n,i+1,k\rangle$ into $\nbase{x}$ for each $i<2$ and $k\leq m$.
It is not hard to check that this procedure eventually enumerates ${\rm Nbase}(x)$.

Now let $A,B$ be a pair of disjoint sets such that $A\cup B$ is $X$-co-c.e.
We construct a point $x\in(\hat{\om}_{TP})^\om$ such that $\{X\}\times{\rm Sep}(A,B)$ is equivalent to ${\rm Nbase}(x)$.
We define $x(2n)$ to be $X(n)$.
Fix an $X$-computable enumeration of the complement of $A\cup B$.
We define $x(2n+1)$ as follows:
\[
x(2n+1)=
\begin{cases}
\infty\mbox{ if }n\in A,\\
\infty_\star\mbox{ if }n\in B,\\
s\mbox{ if we see $n\not\in A\cup B$ at stage $s$}.
\end{cases}
\]
As in the above argument, it is not hard to see that $\{X\}\times{\rm Sep}(A,B)$ is Medvedev-equivalent to ${\rm Nbase}(x)$.
}


\subsection{Degrees of points: $T_2$-topology}

\subsubsection{Double Origin Topology}\label{sec:doubleorigintopology}

We show several basic results on double origin spaces and doubled co-$d$-CEA degrees.
For definitions, see Section \ref{sec:3-5-1}.

\thmproof{thm:double-origin-degree}{
We first show the ``only if'' part, where we only need to consider $(\mathcal{Q}_{DO})^\om$.
To simplify the argument, we treat $\mathbf{0}_\star$ as a pair $(0_\star,0_\star)$, where $0_\star\not\in\mathcal{Q}_i=\mathbb{Q}\cap(-1,1)$.
Define $c\colon\mathcal{Q}_i\cup\{0_\star\}\to\mathcal{Q}_i$ by $c(x)=x$ if $x\not=0_\star$; otherwise $c(x)=0$.
Given $(x,y)=(x_n,y_n)_{n\in\om}\in(\mathcal{Q}_{DO})^\om$, define $c_{(x,y)}(2n)=c(x_n)$ and $c_{(x,y)}(2n+1)=c(y_n)$. We observe that $c_{\cdot} : \mathcal{Q}_{DO})^\om \to \mathbb{Q}^\omega$ is the computable quotient map that merges the two origins. 

Note that $\mathbb{Q}$ embeds into $\mathbf{2}^\omega$ (as the subspace of eventually constant sequences), which is inherited by $\mathbb{Q}^\omega$. Thus, $c_{(x,y)}$ has total degree, and we pick a suitable total representative $X$.
Moreover, we define $A,B,P,N$ as follows:
\begin{align*}
A&=\{n\in\om:(x_n,y_n)=\mathbf{0}\},\\
B&=\{n\in\om:(x_n,y_n)=\mathbf{0}_\star\},\\
P&=\{n\in\om:y_n>0\},\\
N&=\{n\in\om:y_n<0\}.
\end{align*}

Clearly, $A,B,P,N$ are pairwise disjoint, $A\cup B$ is $X$-co-c.e., and $P,N$ are $X$-c.e.
It is easy to see that $X\oplus\cmp{X}\leq_e{\rm Nbase}(x,y)$.
Note that $\langle n,1,1,1\rangle$ is enumerated into ${\rm Nbase}(x,y)$ if and only if $n$ is enumerated into $A\cup P$.
Similarly, $\langle n,2,1,1\rangle$ is enumerated into ${\rm Nbase}(x,y)$ if and only if $n$ is enumerated into $B\cup N$.
Therefore, we get $X\oplus\cmp{X}\oplus(A\cup P)\oplus(B\cup N)\leq_e{\rm Nbase}(x,y)$.

Conversely, assume that an enumeration of $X\oplus\cmp{X}\oplus(A\cup P)\oplus(B\cup N)$ is given.
Then we proceed the following algorithm:
\begin{enumerate}
\item[(I)] If we see $n\not\in A\cup B$ by using $X\oplus\cmp{X}$, we start to enumerate all tuples of the form $\langle n,0,p,q,r,s\rangle$ such that $p<c(x_n)<q$ and $r<c(y_n)<s$.
\item[(II)] if we see $n\in A\cup P$, by using $X\oplus\cmp{X}$, we start to enumerate all tuples of the form $\langle n,1,k,\ell\rangle$ such that $|c(x_n)|<k^{-1}$ and $c(y_n)<\ell^{-1}$.
\item[(III)] If we see $n\in B\cup N$, by using $X\oplus\cmp{X}$, we start to enumerate all tuples of the form $\langle n,2,k,\ell\rangle$ such that $|c(x_n)|<k^{-1}$ and $-\ell^{-1}<c(y_n)$.
\end{enumerate}

Here, for (I), recall that $A\cup B$ is co-c.e.\ relative to $X$.
We claim that the above procedure gives us an enumeration of $\nbase{x,y}$.
To show this claim, let $\nbaseb{n}{x,y}$ be the $n$-th section of $\nbase{x,y}$, that is, $\nbaseb{n}{x,y}=\{\alpha:\langle n,\alpha\rangle\in\nbase{x,y}\}$.

If $n\in A$, clearly, $\nbaseb{n}{x,y}=\{\langle 1,k,\ell\rangle:k,\ell\in\om\}$.
Since $n\in A$, our algorithm only proceeds (II), and since $c(x_n)=c(y_n)=0$, we have $|c(x_n)|<k^{-1}$ and $c(y_n)<\ell^{-1}$ for all $k,\ell\in\om$.
Thus, the $n$-th section enumerated by the above algorithm is exactly $\nbaseb{n}{x,y}$.
If $n\in B$, a similar argument holds.
If $n\in P$, then $\nbaseb{n}{x,y}$ is the union of the set of all $\langle 0,p,q,r,s\rangle$ such that $p<x_n<q$ and $r<y_n<s$, and that of all $\langle n,1,k,\ell\rangle$ such that $|x_n|<k^{-1}$ and $y_n<\ell^{-1}$.
Since $n\in P$, our algorithm proceeds (I) and (II), and clearly, $c(x_n)=x_n$ and $c(y_n)=y_n$.
Thus, the $n$-th section enumerated by the above algorithm is exactly $\nbaseb{n}{x,y}$.
If $n\in N$, a similar argument holds.
Finally, assume that $n\not\in A\cup B\cup P\cup N$.
In this case, $\nbaseb{n}{x,y}$ is the set of all $\langle 0,p,q,r,s\rangle$ such that $p<x_n<q$ and $r<y_n<s$.
Since $n\not\in A\cup B\cup P\cup N$, our algorithm only proceeds (I), and clearly, $c(x_n)=x_n$ and $c(y_n)=y_n$.
Thus, the $n$-th section enumerated by the above algorithm is exactly $\nbaseb{n}{x,y}$.
This verifies the claim.

To show the ``if'' part, let $X,A,B,P,N$ be such that $A$, $B$, $P$, and $N$ are disjoint, and $P$, $N$, and $\cmp{(A\cup B)}$ are $X$-c.e.
We will construct a point $(x,y)=(x_n,y_n)_{n\in\om}\in(\mathcal{Q}_{DO})^\om$ such that $\nbase{x,y}\equiv_eX\oplus\cmp{X}\oplus(A\cup P)\oplus (B\cup N)$; where $(x,y)$ will be in the image of the computale embedding of $(\mathcal{P}_{DO})^\om$ into $(\mathcal{Q}_{DO})^\om$ constructed in Example \ref{exa:double-quasi-Polish}. Hence, $(x,y)$ has even a $(\mathcal{P}_{DO})^\om$-degree.

Fix $X$-computable enumerations of $\cmp{A\cup B}$, $P$, and $N$.
First, $(x_{2n},y_{2n})$ is used to code $X\oplus\cmp{X}$.
For instance, put $(x_{2n},y_{2n})=(X(n)/2,1/2)$.
We first define $c(x_{2n+1})$ and $c(y_{2n+1})$ as follows:
\begin{align*}
c(x_{2n+1})&=
\begin{cases}
2^{-s}&\mbox{ if we see $n\not\in A\cup B$ at stage $s$},\\
0&\mbox{ if $n\in A\cup B$},
\end{cases}\\
c(y_{2n+1})&=
\begin{cases}
2^{-s}&\mbox{ if we see $n\in P$ at stage $s$},\\
-2^{-s}&\mbox{ if we see $n\in N$ at stage $s$},\\
0&\mbox{ if $n\not\in P\cup N$},
\end{cases}
\end{align*}

Note that $n\in A\cup B$ implies $(c(x_{2n+1}),c(y_{2n+1}))=(0,0)$ since $A,B,P,N$ are disjoint.
If $(c(x_{2n+1}),c(y_{2n+1}))=(0,0)$, then define $(x_{2n+1},y_{2n+1})=\mathbf{0}$ if $x\in A$, and define $(x_{2n+1},y_{2n+1})=\mathbf{0}_\star$ if $x\in B$.
We can then decode $A,B,P,N$ as before, and we get $X\oplus\cmp{X}\oplus(A\cup P)\oplus(B\cup N)\equiv_e{\rm Nbase}(x,y)$ as in the above argument.
}

\propproof{prop:doubleorigin-is-cototal}{
Every doubled co-$d$-CEA $e$-degree is of the form $Y=(C\oplus\cmp{C})\oplus(A\cup P)\oplus (B\cup N)$ such that $P,N,\cmp{(A\cup B)}$ are $C$-c.e.
Put $Z=(A\cup P)\oplus (B\cup N)$.
It suffices to construct a total function $g$ such that $C\oplus\cmp{C}\oplus Z\equiv_eC\oplus\cmp{C}\oplus\cmp{{\rm Graph}(g)}\oplus{\rm TGraph}_2(g)$ (since we can remove $C\oplus \cmp{C}$ from the right-hand set by replacing $g$ with $\tilde{g}$ such that $\tilde{g}(2n)=C(n)+2$ and $\tilde{g}(2n+1)=g(n)$).
Fix $C$-computable enumerations of $P$, $N$, and $\cmp{(A\cup B)}$.
For each $n$, let $t^0_n$ ($t^1_n$, resp.)~be the first stage such that $n$ is enumerated into $P$ ($N$, resp.)~if $n\in P\cup N$.
Then we define $g$ as follows.
\begin{align*}
g(2n)&=
\begin{cases}
0,&\mbox{ if }n\not\in A\cup P,\\
1,&\mbox{ if }n\in A,\\
t^0_n+2,&\mbox{ if }n\in P,
\end{cases}
\\
g(2n+1)&=
\begin{cases}
0,&\mbox{ if }n\not\in B\cup N,\\
1,&\mbox{ if }n\in B,\\
t^1_n+2,&\mbox{ if }n\in N,
\end{cases}
\end{align*}

Clearly, $g$ is total.
As $A,B,P,N$ are disjoint, trivially we have
\[A=(A\cup B)\setminus (B\cup N),\mbox{ and }B=(A\cup B)\setminus (A\cup P).\]

Thus, $n\in A$ if and only if $n\in A\cup B$ and $2n+1\in\cmp{Z}$ (i.e., $n\in B\cup N$).
Similarly, $n\in B$ if and only if $n\in A\cup B$ and $2n\in\cmp{Z}$ (i.e., $n\in A\cup P$).
Thus, we note that
\begin{align*}
g(2n)\not=0&\iff 2n\in Z,\\
g(2n)\not=1&\iff n\in\cmp{(A\cup B)}\mbox{ or }2n+1\in Z,\\
g(2n)=m+2&\iff m=t^0_n.
\end{align*}

Here, note that the equality $m=t^0_n$ is $C$-computable.
Similarly,
\begin{align*}
g(2n+1)\not=0&\iff 2n+1\in Z,\\
g(2n+1)\not=1&\iff n\in\cmp{(A\cup B)}\mbox{ or }2n\in Z,\\
g(2n+1)=m+2&\iff m=t^1_n.
\end{align*}

The above equivalences clearly witness that $C\oplus\cmp{C}\oplus Z$ is $e$-equivalent to $C\oplus\cmp{C}\oplus\cmp{{\rm Graph}(g)}\oplus{\rm TGraph}_2(g)$.
Consequently, $Y$ has a telograph-cototal $e$-degree.
}


\subsection{Degrees of points: $T_{2.5}$-topology}

\subsubsection{Arens square}\label{sec:Arens-square}

We show several basic results on the Arens space and Arens co-$d$-CEA degrees.
For definitions, see Section \ref{sec:3-6-1}.

\propproof{proparensecondcountable}{
It is clear that $\mathcal{QA}$ is second-countable.
To see that $\mathcal{QA}$ is $T_{2.5}$, choose two distinct points $(x_0,y_0)\not= (x_1,y_1)$ in $\mathcal{QA}$.
If $y_0\not=y_1$, then there are disjoint open sets $U,V\subseteq\om^3+1$ such that $y_0\in\overline{U}$ and $y_1\in\overline{V}$  (w.r.t.\ the order topology on $\om^3+1$).
Then, it is easy to see that $\mathcal{L}\times{U}$ and $\mathcal{L}\times{V}$ separate $(x_0,y_0)$ and $(x_1,y_1)$.
If $y_0=y_1$, then we must have $\{x_0,x_1\}=\{0,\overline{0}\}$ and $y_0=y_1=\om^3$.
Assume that $x_0=0$ and $x_1=\overline{0}$.
Then, it is easy to see that $\om\times(\om^3+1)$ and $\om^\ast\times (\om^3+1)$ separate $(x_0,y_0)$ and $(x_1,y_1)$.

Assume for the sake of a contradiction that there is a  continuous function $f:\mathcal{QA}\to\mathbb{R}$ such that $f(0,\om^3)=0$ and $f(\overline{0},\om^3)=1$.
As $\{\om\times(\alpha,\om^3]:\alpha<\om^3\}$ (respectively $\{\om^\ast\times(\alpha,\om^3]:\alpha<\om^3\}$) is a neighbourhood basis of $(0,\om^3)$ (respectively of $(\overline{0},\om^3)$), there is $\alpha<\om^3$ such that $f(x,y)<1/4$ if $x\in\om$ and $y>\alpha$, and that $f(x,y)>3/4$ if $x\in\om^\ast$ and $y>\alpha$.
Note that if $\beta\in I_\infty$ and $\beta>\alpha$, then $(\infty,\beta)\in\overline{\om\times(\alpha,\om^3]}$.
This is because $\beta$ is a limit ordinal, and hence $\beta$ is an accumulation point of $I_n$ for any $n\in\om$.
Hence, $f(\infty,\beta)\leq 1/4$.
Similarly, $(\overline{\infty},\beta)\in\overline{\om^\ast\times(\alpha,\om^3]}$, and hence $f(\overline{\infty},\beta)\leq 3/4$.
Choose $\gamma\in I_{0_\zeta}$ such that $\gamma>\alpha$, and consider the value $f(0_\zeta,\gamma)$.
There is $\delta$ such that $\alpha<\delta<\gamma$ and $|f(x,y)-f(0_\zeta,\gamma)|<1/8$ for any $x\in\zeta$ and $\delta<y\leq \gamma$.
Since $\gamma$ is of the form $\om^2\cdot j$, $\gamma$ is an accumulation point of both $I_{\infty}$ and $I_{\overline{\infty}}$.
Hence, there are $\beta\in I_{\infty}$ and $\overline{\beta}\in I_{\overline{\infty}}$ such that $\alpha<\delta<\beta,\overline{\beta}<\gamma$.
Then, we also have $(\infty,\beta),(\overline{\infty},\overline{\beta})\in\overline{\zeta\times(\delta,\gamma]}$.
Hence,
\[\frac{3}{8}\geq f(\infty,\beta)+\frac{1}{8}\geq f(0_\zeta,\gamma)\geq f(\overline{\infty},\overline{\beta})-\frac{1}{8}\geq \frac{5}{8},\]
which is clearly a contradiction.
}

\obsproof{obsdineiffdiscodcea}{
It is clear that if $\mathbf{d}$ is co-$d$-CEA, then $\mathbf{d}\in\mathcal{E}$.
Conversely, assume that $S=X\oplus \cmp{X}\oplus(A\cup P)\oplus (B\cup N)$ for some $A,B,P,N,X\subseteq\om$ satisfying the above mentioned conditions.
Define $Z=(A\cup P)\cup H_P$, and then $Z\leq_e S$ since $H_P$ is $X$-c.e.
One can see that $\cmp{Z}=(B\cup N)\cup H_N$, and thus $\cmp{Z}\leq_eS$.
Then, we have $A\cup P=Z\cap((A\cup B)\cap P)$ and $B\cup N=\cmp{Z}\cap((A\cup B)\cap N)$.
Since $P,N,\cmp{A\cup B}$ are $X$-c.e., the sets $A\cup P$ and $B\cup N$ are co-$d$-c.e.\ relative to $X\oplus Z$.
Hence, $S\oplus Z\oplus\cmp{Z}\equiv_eS$ is co-$d$-CEA.
}

\obsproof{codceaisarenscodcea}{
If an $e$-degree $\mathbf{d}$ is co-$d$-CEA, then there are $X,A,P\subseteq\om$ such that $P$ and $\cmp{A}$ are $X$-c.e., $A$ and $P$ are disjoint, and $(X\oplus\cmp{X})\oplus(A\cup P)$.
Put $Y=X$, $L=A$, $J_L=P$, $N=\cmp{L}$, and $R=J_R=J_M=\emptyset$.
This witnesses that $\mathbf{d}$ is Arens co-$d$-CEA.
}

\thmproof{thm:Arens-square}{
Assume that $z=(x_n,y_n)_{n\in\om}\in\mathcal{QA}^\om$ is given.
Consider the following sets:
\begin{align*}
L&=\{n\in\om:(x_n,y_n)=(0,\om^3)\},\\
R&=\{n\in\om:(x_n,y_n)=(\overline{0},\om^3)\},\\
M&=\{n\in\om:x_n=0_\zeta\}=\{n\in\om:(\exists j\in\om)\;y_n=\om^2\cdot j\},\\
J_L&=\{n\in\om:x_n\in\om\setminus\{0\}\}=\{n\in\om:(\exists k\in\om\setminus\{0\})\;y_n\in I_k\},\\
J_R&=\{n\in\om:x_n\in\om^\ast\setminus\{\overline{0}\}\}=\{n\in\om:(\exists k\in\om\setminus\{0\})\;y_n\in I_{\overline{k}}\},\\
J_M&=\{n\in\om:x_n\in\zeta\setminus\{0_\zeta\}\}=\{n\in\om:(\exists k\in\zeta\setminus\{0\})\;y_n\in I_{k_\zeta}\}.
\end{align*}

Let $Y$ code the information on the second coordinate $(y_n)_{n\in\om}$.
Then $Y$ is total as seen in Remark \ref{rem:om31-total}.
The sets $L\cup R$, $M$, $J_L$, $J_R$, and $J_M$ are characterized only by using $(y_n)_{n\in\om}$.
For instance, $n\in L\cup R$ iff $y_n=\om^3$, which is co-c.e.\ condition relative to $Y$.
One can also see that $n\in J_L$ iff $y_n\in I_k$ for some $k\in\om$, which is a c.e.\ condition relative to $Y$, since $I_k$ consists of successor ordinals, which are isolated in the space $\om^3+1$.
Similarly, $J_R$, and $J_M$ are c.e.\ in $Y$.
Define $N=\cmp{(L\cup R\cup M)}$.
Then, $N=\{n:y(n)\not\in\{\om^3,\om^2\cdot j:j\in\om\}\}$, which is a c.e.\ condition relative to $Y$.
Define $H_L=\{n:0<x_n<0_\zeta\}$, and $H_R=\{n:0_\zeta<x_n<\overline{0}\}$.
Clearly, $\{H_L,H_R\}$ is a partition of $N$.
Note that $n\in H_L$ iff $y_n\in I_\infty\cup\bigcup_{k\in\om^+}(I_k\cup I_{(-k)_\zeta})$, which is c.e.\ condition relative to $Y$.
Similarly, $H_R$ is c.e.\ in $Y$.
Hence, $A=Y\oplus \cmp{Y}\oplus(L\cup J_L)\oplus (R\cup J_R)\oplus (\cmp{(L\cup R\cup N)}\cup J_M)$ is of an Arens co-$d$-c.e.\ $e$-degree.

We first check that $A\leq_e\nbase{z}$.
It is easy to see that
\begin{align*}
n\in L\cup J_L&\iff(\exists j)\;\langle 0,n,j\rangle\in\nbase{z},\\
n\in R\cup J_R&\iff(\exists j)\;\langle 1,n,j\rangle\in\nbase{z}.
\end{align*}
Moreover, one can see that $\cmp{(L\cup R\cup N)}=M$, and therefore,
\[n\in\cmp{(L\cup R\cup N)}\cup J_M=M\cup J_M\iff(\exists j,k)\;\langle 2,n,j,k\rangle\in\nbase{z}.\]

This verifies that $A\leq_e\nbase{z}$.
Conversely, we first recover $y(n)$ from an enumeration of $Y\oplus\cmp{Y}$.
Then consider the following.
\begin{enumerate}
\item If $y(n)$ is a successor ordinal, then one finds it at finite stage.
One can then compute $x(n)=m_{y(n)}$, which determines $z(n)=(x(n),y(n))$.
Enumerate all neighborhoods of $z(n)$.
\item Even if $y(n)$ is a limit ordinal, if $n\not\in I_0\cup I_1\cup I_{0_\zeta}$, we see $\om^2k+\om u+2j-1<y(n)\leq\om^2k+\om (u+1)$ at some finite stage.
If $u$ is even, $u=2\ell$ say, it is ensured that $j\leq x(n)\leq(-j)_\zeta$, and hence we can enumerate $\langle 3,n,j,k,\ell\rangle$ into $\nbase{z}$.
If $u$ is odd, $u=2\ell+1$ say, it is ensured that $j_\zeta\leq x(n)\leq\overline{j}$, and hence we can enumerate $\langle 4,n,j,k,\ell\rangle$ into $\nbase{z}$.
\item Simultaneously, wait until $n$ is enumerated into $(L\cup J_L)\oplus(R\cup J_R)\oplus(M\cup J_M)$.
\begin{itemize}
\item If we see $n\in L\cup J_L$, enumerate $\langle 0,n,j\rangle$ for any $j$ such that $y(n)>\om^2j$.
\item If we see $n\in R\cup J_R$, enumerate $\langle 1,n,j\rangle$ for any $j$ such that $y(n)>\om^2j$.
\item If we see $n\in M\cup J_M$, enumerate $\langle 2,n,j,k\rangle$ for any $j,k$ such that $\om^2j+\om k<y(n)<\om^2(j+1)$.
\end{itemize}
\end{enumerate}

One can check that the above procedure witnesses that $\nbase{z}\leq_eA$.
Next, given such $L,R,N,J_L,J_R,J_M,Y\subseteq\om$, we define a point $z\in\mathcal{QA}^\om$.
We first define $z(2n+1)=(1,1)$ if $n\in Y$; otherwise $z(2n+1)=(2,3)$.
This clearly ensures that $Y\oplus\overline{Y}\leq_e\nbase{z}$.

Fix $Y$-computable enumerations of $N,J_L,J_R,J_M$, and $\cmp{(L\cup R)}$.
If $n\in L$, define $z(2n)=(0,\om^3)$.
If $n\in R$, define $z(2n)=(\overline{0},\om^3)$.
Put $M=\cmp{(L\cup R\cup N)}$.
If $n\not\in(L\cup R)$ happens, then let $s$ be the first stage when we confirm that (w.r.t.\ a $Y$-computable enumeration of $\cmp{L\cup R}$).
If $n\in M$, we define $z(2n)=(0_\zeta,\om^2\cdot(s+1))$.
If $n\not\in(L\cup R\cup M)$ happens, that is, if $n\in N$, then we see either $n\in H_L$ or $n\in H_R$ at some stage $t$.
Put $j=1$ if $n\in H_L$, and $j=2$ if $n\in H_R$.
If $n\in N\setminus(J_L\cup J_R\cup J_M)$, define $z(2n)=(\infty,\om^2\cdot s+\om(2t+j))$.
If $n\in J_L\cup J_R\cup J_M$ happens, let $u$ be the first stage when we confirm that.
If $n\in J_L\cup J_R$, define the second coordinate of $z(2n)$ to be $\om^2\cdot s+\om(2t+j-1)+2u+1$; otherwise define it to be $\om^2\cdot s+\om(2t+j-1)+2u$.
The second coordinate uniquely determines the first coordinate.

Now, it is not hard to verify that the coded neighborhood basis of the second coordinate of $z$ is $e$-reducible to $Y\oplus\cmp{Y}$.
Moreover, the sets $L,R,N,J_L,J_R,J_M$ satisfy the equations mentioned in the first paragraph in this proof (where $z(2n)=(x_n,y_n)$).
Hence, the above argument shows that $A\equiv_e\nbase{z}$ as desired.
}

\subsubsection{Roy's lattice space}\label{section:Roy-space}

We show several basic results on Roy's lattice space and corresponding degrees.
For definitions, see Section \ref{sec:3-6-2}.

\propproof{propquasipolishroy}{
It is clear that the space $\mathcal{QR}$ is second-countable.
To see that $\mathcal{QR}$ is $T_{2.5}$, let $(x_0,y_0),(x_1,y_1)\in\mathcal{QR}$ be given two distinct points.
If $y_0\not=y_1$, since the ordinal space $\om^\om+1$ is metrizable; hence $T_{2.5}$, choose open sets $U,V\subseteq \mathcal{O}_{\om^\om}$ such that $y_0\in U$, $y_1\in V$, and $\overline{U}\cap\overline{V}=\emptyset$.
Then $\hat{\om}\times U$ and $\hat{\om}\times V$ separate $(x_0,y_0)$ and $(x_1,y_1)$.
If $y_0=y_1$, then we must have $\{x_0,x_1\}=\{0,\infty\}$ and $y_0=y_1$ are the empty string $\langle\rangle$, since $(I_x:x\in\om\setminus\{0\})$ is a partition of the ordinal $\om^\om$.
Then define $U=\{0,1\}\times\mathcal{O}_{\om^\om}$ and $V=[5,\infty]\times\mathcal{O}_{\om^\om}$.
Clearly, we have $(0,\langle\rangle)\in U$, $(\infty,\langle\rangle)\in V$, $\overline{U}\subseteq[0,2]\times \mathcal{O}_{\om^\om}$ and $\overline{V}\subseteq[4,\infty]\times \mathcal{O}_{\om^\om}$.
Hence, $U$ and $V$ separates $(x_0,y_0)$ and $(x_1,y_1)$.

To see that $\mathcal{QR}$ is not completely Hausdorff, suppose for the sake of contradiction that there is a continuous function $f:\mathcal{QR}\to[0,1]$ such that $f(0,\langle\rangle)=0$ and $f(\infty,\langle\rangle)=1$.
Since every open neighborhood of $\langle\infty,\langle\rangle\rangle$ is of the form $[2n+1,\infty]\times\mathcal{O}_{\om^\om}$, there is $k\in\om$ such that $f(n,\sigma)>3/4$ for any $n>2k$ and $\sigma\in\mathcal{O}_{\om^\om}$.
Similarly, since the open sets of the form $\{0,1\}\times(\langle m\rangle,\langle\rangle]_{\rm KB}$ form a local basis at $\langle 0,\langle\rangle\rangle$, there is $m$ such that $\{1\}\times(\langle m\rangle,\langle\rangle]_{\rm KB}\subseteq f^{-1}[0,\ep)$.
Put $\ep=4^{-1}$, and $\ell=\max\{m,k\}+1$.
Then, in particular, we have $\{1\}\times(\langle\ell-1\rangle,\langle\ell\rangle]_{\rm KB}\subseteq f^{-1}[0,4^{-1})$.

Note that $(2,\langle\ell\rangle)\in\mathcal{QR}$ since the length of $\langle\ell\rangle$ is $1$, and thus $\langle\ell\rangle\in I_2$.
The closure of $\{1\}\times (\langle\ell-1\rangle,\langle\ell\rangle]_{\rm KB}$ contains the point $(2,\ell)$ for some $m$ since every open neighborhood of $(2,\langle\ell\rangle)$ contains $\{1\}\times(\langle\ell,t\rangle,\langle\ell\rangle]_{\rm KB}$ for almost all $t\in\om$.
In particular, the closure of $f^{-1}[0,4^{-1})$ contains $(2,\langle\ell\rangle)$.
This shows that $f(2,\langle\ell\rangle)\leq 4^{-1}$.
Since $f^{-1}[0,4^{-1}+4^{-2})$ is an open set containing the point $(2,\langle\ell\rangle)$, it also includes $\{3\}\times(\langle\ell,t\rangle,\langle\ell\rangle]_{\rm KB}$ for almost all $t\in\om$.
Then there is $t_1$ such that $\{3\}\times(\langle\ell,t_1-1\rangle,\langle\ell,t_0\rangle]_{\rm KB}\subseteq f^{-1}[0,4^{-1}+4^{-2})$.
By the same argument as above, one can see that $(4,\langle\ell,t_1\rangle)$ is contained in the closure of the above set, and thus $f(4,\langle\ell,t_1\rangle)\leq 4^{-1}+4^{-2}$.
Continue this procedure.
We eventually get a string $\sigma=\langle\ell,t_1,\dots,t_{\ell-1}\rangle$ such that $f(2\ell,\sigma)<2^{-1}$.
Since $2\ell>2k$, by our choice of $k$, we get $3/4<f(2\ell,\sigma)<1/2$, a contradiction.
}

\propproof{proproyhalfgraph}{
Let $\mathbf{d}$ is a co-$d$-CEA degree.
Then, $X\oplus\cmp{X}\oplus (A\cup P)\in\mathbf{d}$ for some $X,A,P\subseteq\om$ such that $\cmp{A}$ and $P$ are c.e.\ in $X$, and $A$ and $P$ are disjoint.
Define $f(n)=\bot_0$ if $n\in A$; $f(n)=0$ if $n\in P$; otherwise put $f(n)=1$.
One can check that $f$ is half-c.e.\ since $\cmp{A}$ and $P$ are c.e.\ in $X$.
Clearly, $f$ is $X$-computably dominated.
We claim that $X\oplus\cmp{X}\oplus(A\cup P)\equiv_eX\oplus\cmp{X}\oplus{\rm HalfGraph}^+(f)$.
For the direction $\leq_e$, one can see that $n\in A\cup P$ iff $2\langle n,0\rangle\in{\rm HalfGraph}^+(f)$.
For the converse direction, $2\langle n,0\rangle+1\in{\rm HalfGraph}^+(f)$ iff $n\not\in A$, which is a $X$-c.e.\ condition.
Moreover, whenever $m>0$, we always have $2\langle n,m\rangle\in{\rm HalfGraph}^+(f)$ and $2\langle n,m\rangle+1\not\in{\rm HalfGraph}^+(f)$.
Hence, every co-$d$-CEA degree is Roy halfgraph-above.

For the second assertion, let $\mathbf{d}$ be a Roy halfgraph-above degree.
Then, $\mathbf{d}$ contains a set of the form $Y\oplus\cmp{Y}\oplus{\rm HalfGraph}^+(f)$, where $f$ is $Y$-half-c.e.\ and $Y$-computably dominated.
We define $A=\{n:f(n)=\bot_0\}$, $B=\{n:f(n)=\bot_1\}$, and given $I\subseteq\om$, we also define $C_I=\{n:f(n)\in I\}$.
We claim that $S=Y\oplus\cmp{Y}\oplus{\rm HalfGraph}^+(f)$ is $e$-equivalent to the following set $Q$.
\[Q=Y\oplus\cmp{Y}\oplus(A\cup C_{\{0\}})\oplus\bigoplus_{k\in\om}(B\cup C_{[2k,\infty)})\oplus\bigoplus_{k\in\om}C_{[2k,2k+2]}.\]

It is obvious that $n\in A\cup C_{\{0\}}$ iff $2\langle n,0\rangle\in{\rm HalfGraph}^+(f)$, and that $n\in B\cup C_{[2k,\infty)}$ iff $2\langle n,m\rangle\in{\rm HalfGraph}^+(f)$.
To see $Q\leq_eS$, note that $n\in C_{[2k,2k+2]}$ iff $2\langle n,k\rangle+1,2\langle n,k+1\rangle\in{\rm HalfGraph}^+(f)$.
Similarly, to see $S\leq_eQ$, one can see that $2\langle n,m\rangle\in{\rm HalfGraph}^+(f)$ iff $n\in A\cup C_{\{0\}}$ or $n\in C_{[2k,2k+2]}$ for some $k<m$.
This verifies the claim.

We define $A_{2k}=A$, $B_{2k}=B$, $P_{2k}=C_{\{0\}}=\{n:f(n)=0\}$, and $N_{2k}=C_{[2k,\infty)}=\{n:f(n)\geq 2k\}$.
Note that $n\not\in A_{2k}\cup B_{2k}$ iff there is $m$ such that $2\langle n,m\rangle+1\in{\rm HalfGraph}(f)$.
It is also easy to see that $n\in P_{2k}$ iff $2\langle n,m\rangle\in{\rm HalfGraph}(f)$, and that $n\in N_{2k}$ iff $2\langle n,m\rangle\in{\rm HalfGraph}(f)$.
Since $f$ is $Y$-half-c.e., the above shows that $P_{2k}$, $N_{2k}$, and $\cmp{(A_{2k}\cup B_{2k})}$ are c.e.\ in $Y$.
It is clear that $A_{2k}$, $B_{2k}$, $P_{2k}$, and $N_{2k}$ are disjoint.
Hence, $Z_{2k}=Y\oplus\cmp{Y}\oplus (A_{2k}\cup P_{2k})\oplus (B_{2k}\cup N_{2k})$ is doubled-co-$d$-c.e.

We now start to code $C_{[2k,2k+2]}$ in a set of a doubled co-$d$-CEA degree.
Note that $C_{[2k,2k+2]}$ is $3$-c.e.\ in $Y$, and hence it is $e$-equivalent to a set which is co-$d$-c.e.\ in $Y$ (see Cooper \cite{CooperE}).
For the sake of completeness (and to check uniformity of the proof) we explicitly write the coding procedure.
Since $f$ is $Y$-half-c.e., there is a uniform $Y$-computable enumeration of $(C_{[2k,\infty)})_{k\in\om}$.
We use the symbol $C_{[2k,\infty)}[s]$ to denote the set of elements enumerated into $C_{[2k,\infty)}$ by stage $s$.
Note that $C_{[2k,\infty)}[s]$ is computable uniformly in $k$ and $s$.
Then, we define $A_{2k+1}$ and $P_{2k+1}$ as follows.
\begin{align*}
A_{2k+1}&=\{\langle n,s\rangle:\mbox{either }n\in C_{[2k,\infty)}[s]\mbox{ or }n\in A\cup B\cup C_{[0,2k+2)}\},\\
P_{2k+1}&=\{\langle n,s\rangle:n\in C_{[2k,\infty)}[s]\mbox{ and }n\in C_{\{2k+2\}}\}.
\end{align*}

We define $B_{2k+1}=N_{2k+1}=\emptyset$.
Note that $n\in A_{2k+1}$ iff $n\in C_{[2k,\infty)}[s]$ or $2\langle n,k\rangle+1\not\in{\rm HalfGraph}(f)$.
Hence, $A_{2k+1}$ is co-c.e.\ in $Y$.
It is also easy to see that $n\in P_{2k+1}$ iff $n\in C_{[2k,\infty)}[s]$ and $2\langle n,k+1\rangle\in{\rm HalfGraph}(f)$, which is a $Y$-c.e.\ condition.
Therefore, $Z_{2k+1}=Y\oplus\cmp{Y}\oplus (A_{2k+1}\cup P_{2k+1})$ is co-$d$-CEA.

We claim that $Z_{2k+1}$ is $e$-equivalent to $Y\oplus\cmp{Y}\oplus C_{[2k,2k+2]}$.
Note that $n\in C_{[2k,\infty)}[s]$ ensures that $n\not\in A\cup B$, and hence if $n\in C_{[2k,\infty)}[s]$, then the condition $n\in A_{2k+1}\cup P_{2k+1}$ is equivalent to that $n\in C_{[2k,2k+2]}$.
Hence, $n\in C_{[2k,2k+2]}$ if and only if there is $s$ such that $n\in C_{[2k,\infty)}[s]$ and $\langle n,s\rangle\in A_{2k+1}\cup P_{2k+1}$.
For the converse direction, one can see that $\langle n,s\rangle\in A_{2k+1}\cup P_{2k+1}$ iff either $n\not\in C_{[2k,\infty)}[s]$ or $n\in C_{[2k,2k+2]}$.
Thus, we conclude that $\bigoplus_iZ_i$ is $e$-equivalent to $Q$.


We claim that $\bigoplus_iZ_i$ is also doubled co-$d$-CEA.
To see this, consider $Z=Y\oplus\cmp{Y}\oplus(\bigoplus_iA_i\cup\bigoplus_iP_i)\oplus(\bigoplus_iB_i\cup\bigoplus_iN_i)$.
Note that $(\bigoplus_iA_i)\cup(\bigoplus_iB_i)=\om\times(A_i\cup B_i)$ is co-c.e.\ in $Y$, and $\bigoplus_iA_i$, $\bigoplus_iB_i$, $\bigoplus_iP_i$, and $\bigoplus_iN_i$ are disjoint.
Hence, $Z$ is doubled co-$d$-CEA.
It is easy to check that $Z$ is $e$-equivalent to $\bigoplus_iZ_i$, which is also $e$-equivalent to $S$.
Consequently, every Roy halfgraph-above degree is doubled-co-$d$-CEA.
}

\thmproof{thm:chained-co-d-CEA}{
Assume that $z=(x_n,y_n)_{n\in\om}\in\mathcal{QR}^\om$ is given.
To see that $\nbase{z}$ has a Roy halfgraph-above degree, we define $f:\om\to\tilde{\om}$ as follows.
\begin{align*}
f(n)=\bot_0&\iff x_n=0,\\
f(n)=\bot_1&\iff x_n=\infty,\\
f(n)=k&\iff x_n=k+1\iff y_n\in I_{k+1}.
\end{align*}

Let $Y$ be a set coding the second coordinate $(y_n)_{n\in\om}$, which has a total degree.
By definition, we have $f(n)=2m$ iff $y_n\in I_{2m+1}$.
This condition is c.e.\ in $Y$ since $I_{2m+1}$ is a computable set of isolated points (successor ordinals).
We also note that $f(n)\in\om$ and $f(n)\geq 2m$ iff $2m+1\leq x_n<\om$ iff, $|y_n|\geq m$, and, whenever $y_n$ is a leaf, there is $\ell\geq m$ such that $y_n(\ell)>0$.
It is equivalent to saying that we see that, for some $\sigma\in\mathcal{O}_{\om^\om}$ with $|\sigma|\geq m$, the open set $(\sigma 0,\sigma]_{\rm KB}$ is a neighborhood of $y_n$.
This shows that $f$ is half-c.e.
To see that $f$ is $Y$-computably dominated, we define $g:\subseteq\om\to\om$ as follows.
\[g(n)=2y_n(0)+1.\]

Note that $y_n(0)=k$ iff $y_n\in(\langle k-1\rangle,\langle k\rangle]_{\rm KB}$.
Hence, one can recover $y_n(0)$ from $Y$ by a partial computable way.
Obviously, $f(n)\in\om$ if and only if $y_n(0)$ is defined.
This verifies that $g$ is $Y$-computable, and ${\rm dom}(g)=\{n:f(n)\in\om\}$.
One can see that $g(n)\leq 2k$ iff $y_n$ extends $\langle j\rangle$ for some $j\leq k-1$.
This implies that $|y_n|\leq k$.
Note that $\sigma\in I_{\ell}$ for some $\ell>2k$ implies that $|\sigma|>k$.
Hence, $|y_n|\leq k$ implies $x_n\leq 2k$, that is, $f(n)\leq 2k-1$.
Therefore, we have $f(n)<g(n)$ whenever $f(n)\in\om$.
Consequently, $f$ is $Y$-computably dominated.

For ${\rm HalfGraph}^+(f)\leq_e\nbase{z}$, it is easy to see the following.
\begin{align*}
2\langle n,k\rangle\in{\rm HalfGraph}^+(f)&\iff x_n\leq 2k+1\\
&\iff(\exists i\leq 1)(\exists\ell\leq k)(\exists\sigma)\;\langle i,n,\ell,\sigma\rangle\in{\rm Nbase}(z),\\
2\langle n,k\rangle+1\in{\rm HalfGraph}^+(f)&\iff x_n\geq 2k+1\iff\langle 2,n,k\rangle\in{\rm Nbase}(z).
\end{align*}

To see that $Y\oplus\overline{Y}\leq_e\nbase{z}$, if $x_n<\infty$ (that is, $\langle i,n,\ell,\sigma\rangle$ is enumerated into $\nbase{z}$ for some $i,n,\ell,\sigma$), then just follow the approximation of the second coordinate $y_n$.
If we see $\langle 2,n,k\rangle\nbase{z}$, then we know that $x_n>2k$.
As seen above, $x_n>2k$ implies $|y_n|>k$, and thus $y_n\in (\langle k\rangle,\langle\rangle]_{\rm KB}$ is ensured.
Hence, in any cases, we can recover a code $Y\oplus\cmp{Y}$ of $(y_n)_{n\in\om}$ from $\nbase{z}$ in a uniform manner.

For $\nbase{z}\leq_eS:=Y\oplus\cmp{Y}\oplus{\rm HalfGraph}^+(f)$, by monitoring an enumeration of $Y$, we can recover the second coordinate $(y_n)_{n\in\om}$.
Then, one can see the following.
\begin{align*}
\langle 0,n,k,\sigma\rangle\in\nbase{z}\iff\,& y_n=\sigma\mbox{ and }(\forall i<2)\;2\langle n,k\rangle+i\in{\rm HalfGraph}^+(f),\\
\langle 1,n,k,\sigma j\rangle\in\nbase{z}\iff\,& y_n\in(\sigma j,\sigma]_{\rm KB}\\
&\mbox{and }(\forall i<2)\;2\langle n,k\rangle+i\in{\rm HalfGraph}^+(f),\\
\langle 2,n,k\rangle\in\nbase{z}\iff\,& 2\langle n,k\rangle+1\in{\rm HalfGraph}^+(f).
\end{align*}

Hence, we obtain $S\equiv_e\nbase{z}$, and conclude that every $\mathcal{QR}^\om$-degree is chained co-$d$-CEA.

Now, given such $S=Y\oplus\cmp{Y}\oplus{\rm HalfGraph}^+(f)$, we define a point $z\in\mathcal{QR}^\om$ such that $\nbase{z}\equiv_eS$.
We first define $z_{2n+1}=(1,1)$ if $n\in Y$; otherwise $z_{2n+1}=(1,0)$.
It is clear that $Y\oplus\cmp{Y}$ is $e$-equivalent to the coded neighborhood basis of $(z_{2n+1})_{n\in\om}$ in the Roy space $\mathcal{QR}^\om$.

We now describe how to define $z_{2n}=(x_n,y_n)$.
Given $n\in\om$, we begin with $y_n[0]=\langle\rangle$.
We wait until we see $g(n)$ is defined, say $g(n)=2k-1$.
In this case, we have $f(n)<2k-1$.
Then, we declare that $y_n\in(\langle k-1\rangle,\langle k\rangle]_{\rm KB}$, that is, $y_n(0)=k$.
If we see $2\langle n,1\rangle+1\in{\rm HalfGraph}(f)$, that is, $f(n)\geq 2$, at stage $s_1$, then we declare that $y_n\in(\langle k,s_1-1\rangle,\langle k,s_1\rangle]_{\rm KB}$, that is, $y_n(1)=s_1$.
Continue this procedure.
Assume that we have already seen $f(n)\geq 2m$, and thus $y_n[s_m]=\sigma=\langle k,s_1,\dots,s_m\rangle$.
If we see $f(n)\geq 2m+2$ at stage $s_{m+1}$, then we declare $y_n\in(\sigma\fr(s_{m+1}-1),\sigma\fr s_{m+1}]_{\rm KB}$, that is, $y_n(m+1)=s_{m+1}$.
If we see $f(n)=2m$ at some stage $t_m$, then we declare $y_n=\sigma\fr t_m 0\dots 0\in \mathcal{O}_{\om^\om}^{\rm leaf}$, where we can assume that $t_m>0$.
If $f(n)\not=2m+1$ for any $m$, then, since $f(n)<2k-1$, our construction ensures that $|y_n|\leq k$, and hence $y_n\in \mathcal{O}_{\om^\om}$.
Note that this procedure gives us an $\mathcal{O}_{\om^\om}$-name of $y_n$ in a $Y$-computable manner.
If $y_n$ is nonempty, let $x_n$ be the unique $x$ such that $y_n\in I_x$.
If $y_n$ is empty and $f(n)=\bot_0$, define $x_n=0$; otherwise put $x_n=\infty$.

We claim that if $f(n)\in\om$, then $x_n=f(n)+1$.
If $f(n)=2m$, then we see $f(n)=2m$ at some stage, and our construction ensures that $y_n$ is of the form $\sigma\fr t_m\fr 0\dots 0\in\mathcal{O}_{\om^\om}^{\rm leaf}$, where $|\sigma|=m+1$ and $t_m>0$.
Hence, $y_n(m+1)>0$ and $y_n(k)=0$ for any $k>m+1$.
This means that $y_n\in I_{2m+1}$, and hence $x_n=2m+1=f(n)+1$.
If $f(n)=2m+1$, then we see $f(n)\geq 2m$ at some stage, and neither $f(n)\geq 2m+2$ nor $f(n)=2m$ happens.
Hence, our construction ensures that $y_n$ is a string of length $m+1$.
This means that $y_n\in I_{2m+2}$, and hence $x_n=2m+2=f(n)+1$.

We define $z_{2n}=(x_n,y_n)$ for each $n\in\om$.
As mentioned above, the coded neighborhood basis of the second coordinate of $(y_n)_{n\in\om}$ in the product ordinal space $(\mathcal{O}_{\om^\om})^\om$ (that is, the $Y$ constructed from $(x_n,y_n)_{n\in\om}$ as in the second paragraph of this proof) is $e$-reducible to $Y\oplus\cmp{Y}$.
Moreover, the function $f$ satisfy the equations mentioned in the first paragraph in this proof (where $z_{2n}=(x_{n},y_{n})$).
Hence, the above argument shows that $S\equiv_e\nbase{z}$ as desired.
}

\subsection{Degrees of points: submetrizable topology}

\subsubsection{Extension topology}

For definitions, see Section \ref{sec:horrible-extension-topology}.

\propproof{prop:Gamma-above}{
Assume that $\beta$ computably extends $\gamma$.
Then, define $G=\{\langle e,x\rangle:x\in \beta_e\}$ and $\Gamma=\{G\}$.
Define $A=\{n:\langle n,x\rangle\in G\}$, which is clearly $\Gamma$ relative to $x$.
Note that
\[e\in A\iff \langle e,x\rangle\in G\iff e\in\nbaseb{\beta}{x}.\]
Thus, by the above equivalence and Observation \ref{obs:comp-extension-rep}, clearly $x\oplus\cmp{x}\oplus A\equiv_e\nbaseb{\beta}{x}$, and thus, $\nbaseb{\beta}{x}$ is $\Gamma$-above.
Conversely, if $\mathbf{d}$ is $\Gamma$-above, then there are $x\in 2^\om$ and $A\in\Gamma^x$ such that $x\oplus\cmp{x}\oplus A\in\mathbf{d}$.
Since $\Gamma=\{G\}$, we must have $A=\{n:\langle n,x\rangle\in G\}$.
As in the previous argument, one can see that $x\oplus\cmp{x}\oplus A\equiv_e\nbaseb{\beta}{x}$.

To show the converse direction, fix a countable collection $\Gamma=(G_e)_{e\in\om}$.
Define $\beta$ by $\beta_{\langle 0,e\rangle}=\lambda_e$ and $\beta_{\langle 1,e,n\rangle}=\{\langle e,y\rangle:\langle n,y\rangle\in G_e\}$.
Clearly, $\beta$ computably extends $\lambda$.
For each $x\in(2^\om)_{\beta}$, let $e$ be the first entry of $x$, that is, $x=\langle e,y\rangle$.
Define $A_e=\{n:\langle n,y\rangle\in G_e\}$.
Then, $A_e$ is $\Gamma$ relative to $y$, and therefore $y\oplus\cmp{y}\oplus A_e$ is $\Gamma$-above.
Note that
\[n\in A_e\iff \langle n,y\rangle\in G_e\iff \langle 1,e,n\rangle\in\nbaseb{\beta}{\langle e,y\rangle}=\nbaseb{\beta}{x}.\]
By the above equivalence and Observation \ref{obs:comp-extension-rep}, we have that $x\oplus\cmp{x}\oplus A_e\leq_e\nbaseb{\beta}{x}$.
Moreover, for any $d\not=e$ and $n\in\om$, $\langle 1,d,n\rangle\not\in\nbaseb{\beta}{x}$, and the $0$-th section of $\nbaseb{\beta}{x}$ is obviously $e$-equivalent to $x\oplus\cmp{x}$.
Hence, $x\oplus\cmp{x}\oplus A_e\equiv_e\nbaseb{\beta}{x}$.
Since $x\oplus\cmp{x}\oplus A_e$ is clearly $e$-equivalent to $y\oplus\cmp{y}\oplus A_e$, this shows that $\nbaseb{\beta}{x}$ is $\Gamma$-above.

Conversely, if $\mathbf{d}$ is $\Gamma$-above, then there are $y\in 2^\om$ and $A\in\Gamma^y$ such that $y\oplus\cmp{y}\oplus A\in\mathbf{d}$.
Then there is $e$ such that $A=\{n:\langle n,y\rangle\in G_e\}$.
Consider $x=\langle e,y\rangle$.
As in the previous argument, one can see that $y\oplus\cmp{y}\oplus A\equiv_e\nbaseb{\beta}{x}$.
}

\thmproof{thm:e-degree-realize}{
Note that by Theorem \ref{thm:countable-T_1-quasiminimal}, the enumeration degrees are not covered by countably many $T_1$-spaces.
Therefore, the above equivalence must be realized by an {\em uncountable} union.
As a corollary, there are uncountably many decidable submetrizable spaces in an essential sense, that is, for any countably many submetrizable (indeed $T_1$) spaces $(\xx_i)_{i\in\om}$, there is a decidable submetrizable space $\yy$ that cannot be embedded into $\xx_i$ for any $i\in\om$.

To prove Theorem \ref{thm:e-degree-realize}, given an enumeration degree $\mathbf{d}$, we will construct an decidable submetrizable space $\xx_\mathbf{d}$ and a point $x\in\xx$ such that the enumeration degree of ${\rm Nbase}(x)$ is exactly $\mathbf{d}$.
Moreover, if $\mathbf{d}$ is $\Delta^0_n$ with $n\geq 4$, $\xx_\mathbf{d}$ can be strongly $\Pi^0_n$-named.

\subsubsection*{Construction}

Given a topological space $(\xx,\tau_X)$ and a set $D\subseteq\xx$ let $\xx_D$ be the {\em extension topology} of $\xx$ plus $D$, that is, the topological space with the underlying set $\xx$ and the topology generated by $\tau_X\cup\{D\}$.
If $\xx$ is submetrizable, then so is the $D$-extension $\xx_D$.
However, the $D$-extension $\xx_D$ is not necessarily metrizable even if $\xx$ is.

Consider the $\om$-power $(\mathcal{B}_D)^\om$ of the $D$-extension of Baire space $\mathcal{B}:=\om^\om$.
Note that an open subbasis of $(\mathcal{B}_D)^\om$ is given by $B_{0,n,\sigma}=\{x\in(\mathcal{B}_D)^\om:x(n)\succ\sigma\}$ and $B_{1,n}=\{x:x(n)\in D\}$.
Therefore, the coded neighborhood filter of $x\in(\mathcal{B}_D)^\om$ is given as follows:
\[\nbase{x}=\{\langle 0,n,\sigma\rangle:\sigma\prec x(n)\}\cup\{\langle 1,n\rangle:x(n)\in D\}.\]

\begin{obs}\label{obs:dense-codense}
If $D$ is dense and co-dense, then $(\mathcal{B}_D)^\om$ is a decidable, submetrizable, cb$_0$ space.
\end{obs}

\begin{proof}
If $D$ is dense and co-dense, given $\sigma,\tau\in \om^{<\om}$, we have that $[\sigma]\not\subseteq D\cap [\tau]$, and that $[\sigma]\cap(D\cap [\tau])=\emptyset$ if and only if $\sigma\bot\tau$.
This gives a decidable basis of $(\mathcal{B}_D)^\om$.
\end{proof}

We now describe how we extend a metric topology to code a given $e$-degree.
Define $\mathbb{Q}^-=\mathbb{Q}\setminus\{0^\om\}$.
Given $A\subseteq\om$, define
\[D_A=\{n\fr x\in\om^\om:[n\in A\mbox{ and }x\not\in\mathbb{Q}^-]\mbox{ or }[n\not\in A\mbox{ and }x\not\in\mathbb{Q}]\}.\]
It is clear that $D_A$ is dense and co-dense.
Note that $\mathbb{Q}\cap D_A=\{n\fr 0^\om:n\in A\}$.
We show that every $e$-degree is realized in the space of the form $(\mathcal{B}_{D_A})^\om$.

Given $D\subseteq\om^\om$ and $x\in(\mathcal{B}_D)^\om$, define $X=\{\langle k,m\rangle:x(k)=m\}$.
Then, it is not hard to see the following.
\[{\rm Nbase}(x)\equiv_eX\oplus\cmp{X}\oplus\{n\in\om:x(n)\in D\}.\]

Given an $e$-degree $\mathbf{d}$, choose $A\in\mathbf{d}$.
By Observation \ref{obs:dense-codense}, $(\mathcal{B}_{D_A})^\om$ is a decidable, submetrizable, cb$_0$ space since $D_A$ is dense and co-dense.
Define $x(n)=n\fr 0^\om$.
Then clearly, $x(n)\in D_A$ if and only if $n\in A$.
Moreover, since $X$ and $\cmp{X}$ are c.e., we have the following.
\[{\rm Nbase}(x)\equiv_eX\oplus\cmp{X}\oplus\{n\in\om:x(n)\in D_A\}\equiv_eA.\]

Thus, by putting $\xx_\mathbf{d}=(\mathcal{B}_{D_A})^\om$, this verifies our claim.
}

We now claim that, if $\mathbf{d}$ is $\Delta^0_n$ with $n\geq 4$, $\xx_\mathbf{d}$ can be strongly $\Pi^0_n$-named.

\propproof{prop:e-degree-realize-arith}{
We introduce ad-hoc technical notions.
A countable set $E\subseteq \om^\om$ is {\em $(\Lambda,\Gamma)$-enumerable} if there is a sequence $(r_e)_{e\in\om}$ of reals such that
\[E\subseteq\{r_e:e\in\om\},\;\{(e,\sigma):r_e\succ\sigma\}\in\Lambda\mbox{, and }\{e:r_e\in E\}\in\Gamma.\]
We also says that a countable set $E\subseteq \om^\om$ is {\em strongly $(\Delta^0_m,\Delta^0_{n})$-enumerable} if it is $(\Delta^0_m,\Delta^0_{n})$-enumerable, and moreover it satisfies the following condition:
\[(\forall S\subseteq\om)\;[S\in\Pi^0_2\;\Longrightarrow\;\{n\in\om:(\exists e)\;[r_e\in E\mbox{ and }\langle e,n\rangle\in S]\}\in\Delta^0_{n}].\]

Let $\mathbb{Q}\subseteq\om^\om$ be the set of all infinite binary strings $x$ such that $x(n)=0$ for almost all $n$.
For instance, $\mathbb{Q}$ is $(\Delta^0_1,\Sigma^0_2)$-enumerable, and strongly $(\Delta^0_1,\Delta^0_4)$-enumerable.

\begin{lemma}\label{lem:strong-co-enumerable}
For any $k\geq 4$, if $A\in\Delta^0_k$ then $\om^\om\setminus D_A$ is strongly $(\Delta^0_1,\Delta^0_k)$-enumerable.
\end{lemma}

\begin{proof}
Define $r_{n,e}=n\fr \sigma_e\fr 0^\om$, where $\sigma_e$ is the $e$-th finite string.
We note that
\[\om^\om\setminus D_A=\mathbb{Q}\setminus\{n\fr 0^\om:n\in A\},\]
and therefore, we have $\om^\om\setminus D_A\subseteq\mathbb{Q}=\{r_{n,e}:n,e\in\om\}$.
Moreover, we have that
\[r_{n,e}\not\in D_A\iff n\not\in A\mbox{ or }(\exists s<|\sigma_e|)\;\sigma_e(s)\not=0.\]

This condition is $\Delta^0_k$, and thus, $\om^\om\setminus D_A$ is co-$(\Delta^0_1,\Delta^0_k)$-enumerable.
Moreover, given a set $S\in\Pi^0_2$,
\begin{align*}
(\exists e)\;[r_{n,e}\not\in D_A\mbox{ and }\langle e,n\rangle\in S]\iff
& [n\in A\mbox{ and }(\exists e)[r_{n,e}\in\mathbb{Q}^-\mbox{ and }\langle e,n\rangle\in S]\\
&\mbox{and }[n\not\in A\mbox{ and }(\exists e)[r_{n,e}\in\mathbb{Q}\mbox{ and }\langle e,n\rangle\in S]].
\end{align*}

Clearly, this condition is $\Delta^0_k$ since $A\in\Delta^0_k$ and $k\geq 4$.
\end{proof}

\begin{lemma}\label{lem:strong-co-enumerable2}
For any $n\geq 4$ and $m\leq n-2$, if $\cmp{D}$ is strongly $(\Delta^0_m,\Delta^0_n)$-enumerable, then $(\mathcal{B}_D)^\om$ is strongly $\Pi^0_{n}$-named.
\end{lemma}

\begin{proof}
Let $(r_e)_{e\in\om}$ witness that $\cmp{D}$ is strongly $(\Delta^0_m,\Delta^0_n)$-enumerable.
We define the predicate $p\in P$ as follows:
\begin{align*}
&(\forall n)(\forall\ell)(\exists\sigma\in\om^\ell)\;\langle 0,n,\sigma\rangle\in{\rm rng}(p),\\
\mbox{and }&(\forall n)\;[\langle 1,n\rangle\not\in{\rm rng}(p)\;\Longrightarrow\;(\exists e)\;r_e\not\in D\mbox{ and }(\forall\sigma\prec r_e)\;\langle 0,n,\sigma\rangle\in{\rm rng}(p)].
\end{align*}
The first line says that $p$ extends the Baire name of a point $x\in\mathcal{B}^\om$, and the second line says that if $p$ does not enumerate $\langle 1,n\rangle$, then some of such $x$ satisfies that $x(n)\not\in D$.
To see that ${\rm Sup}((\mathcal{B}_D)^\om)\subseteq P$, fix $p\in{\rm Sup}((\mathcal{B}_D)^\om)$, that is, $p$ extends a $(\mathcal{B}_D)^\om$-name $q$.
We show that $p$ satisfies the contrapositive of the second line in the definition of $p\in P$.
Assume that any point $x$ whose Baire name is extended by $p$ satisfies $x(n)\in D$.
Then the unique point $x$ coded by $q$ must satisfy $x(n)\in D$ since $p$ extends $q$ and therefore extends the Baire name of $x$.
Then $q$ must enumerate $\langle 1,n\rangle$, and so does $p$.

We next define the predicate $p\in N$ as follows:
\begin{align*}
(\forall n)&(\forall\sigma,\tau)\;[\sigma\bot\tau\;\Longrightarrow\;\langle 0,n,\sigma\rangle\not\in{\rm rng}(p)\mbox{ or }\langle 0,n,\tau\rangle\not\in{\rm rng}(p)],\\
\mbox{ and }(\forall n)&\;[(\langle 1,n\rangle\in{\rm rng}(p)\mbox{ and }(\exists^\infty\sigma)\;\langle 0,n,\sigma\rangle\in{\rm rng}(p))\\
&\;\Longrightarrow\;(\forall e)[r_e\not\in D\;\rightarrow\;(\exists\sigma\prec r_e)(\exists\tau)\;\tau\bot\sigma\mbox{ and }\langle 0,n,\tau\rangle\in{\rm rng}(p)].
\end{align*}
The first line says that $p$ does not enumerate two incomparable strings for each coordinate.
Note that, in this case, $p$ generates a sequence $x^p=(x^p(n))_{n\in\om}\in \om^{\leq\om}$.
The second and third lines say that if $p$ enumerates $\langle 1,n\rangle$ and such $x^p(n)$ is an infinite string, then $x^p(n)\in D$.
Note that every $p\in{\rm Sub}((\mathcal{B}_D)^\om)$ satisfies this condition.
Otherwise, $p$ enumerates $\langle 1,n\rangle$, $x^p(n)\in \om^\om$ is determined but $x^p(n)\not\in D$.
Thus, if $q$ is a $(\mathcal{B}_D)^\om$-name extending $p$, we must have $x^q(n)\not\in D$, and then $q$ never enumerates $\langle 1,n\rangle$, which is impossible.
Hence, we get that ${\rm Sub}((\mathcal{B}_D)^\om)\subseteq N$.

It is not hard to check that $P\in\Pi^0_{n}$ since $\cmp{D}$ is strongly $(\Delta^0_m,\Delta^0_n)$-enumerable.
It is also straightforward to see that $N\in\Pi^0_n$ since $\cmp{D}$ is $(\Delta^0_m,\Delta^0_n)$-enumerable.
Finally, we claim that $P\cap N\subseteq{\rm Name}((\mathcal{B}_D)^\om)$.
The first lines in the definitions of $P$ and $N$ say that $p$ determines $x^p=(x^p(n))_{n\in\om}\in(\om^\om)^\om$.
Then the second line of $P$ ensures that if $p$ does not enumerate $\langle 1,n\rangle$ then $x^p(n)\not\in D$.
Conversely, the second and third lines of $N$ ensures that if $p$ enumerates $\langle 1,n\rangle$ then $x^p(n)\in D$.
This verifies our claim.
\end{proof}



Given $D\subseteq\om^\om$ and $x\in(\mathcal{B}_D)^\om$, define $X=\{\langle k,m\rangle:x(k)=m\}$.
Then, it is not hard to see the following.
\[{\rm Nbase}(x)\equiv_eX\oplus\cmp{X}\oplus\{n\in\om:x(n)\in D\}.\]
Let $\mathbf{d}$ be a $\Delta^0_n$-enumeration degree for $n\geq 4$, and choose $A\in\mathbf{d}$.
By Lemma \ref{lem:strong-co-enumerable}, $D_A$ is strongly co-$(\Delta^0_1,\Delta^0_n)$-enumerable.
Therefore, by Lemma \ref{lem:strong-co-enumerable2}, $(\mathcal{B}_{D_A})^\om$ is strongly $\Pi^0_{n}$-named.
Define $x_n=n\fr 0^\om$.
Then clearly, $x_n\in D_A$ if and only if $n\in A$.
Moreover, since $X$ and $\cmp{X}$ are c.e., we have the following.
\[{\rm Nbase}(x)\equiv_eX\oplus\cmp{X}\oplus\{n\in\om:x(n)\in D_A\}\equiv_eA.\]

Thus, by putting $\xx_\mathbf{d}=(\mathcal{B}_{D_A})^\om$, this verifies our claim.
}


\subsubsection{Gandy-Harrington topology}\label{sec:GHtopology}

For definitions, see Section \ref{sec:3-7-3}

\propproof{prop:gandy-harrington}{
First it is easy to see $\nbaseb{GH}{x}\leq_e\nbaseb{\lambda}{x^{{\rm HJ}}}$ since the hyperjump of $x$ determines whether the $e$-th $\Sigma^1_1$ set contains $x$ or not.
To see $\nbaseb{\lambda}{x^{(\alpha)}}\leq_e\nbaseb{GH}{x}$, given $e$, one can effectively find $\Sigma^1_1$ indices $p_e$ and $n_e$ of a $\Sigma^0_{1+\alpha}$ set $\{y\in\omega^\omega:y^{(\alpha)}(e)=1\}$ and a $\Pi^0_{1+\alpha}$ set $\{y\in\omega^\omega:y^{(\alpha)}(e)=0\}$.
Then for each $e$ either $p_e$ or $n_e$ is enumerated into $\{e:x\in GH_e\}$.
By waiting for either one to occur, one can make an enumeration procedure witnessing $\nbaseb{\lambda}{x^{(\alpha)}}\leq_e\nbaseb{GH}{x}$.
}

\thmproof{thm:GH-non-continuous}{
Suppose that $\nbaseb{\mathcal{H}}{z}\leq_e\nbaseb{GH}{x}$ for $z\in\mathcal{H}$.
Then there is a c.e.~set $\Psi$ such that $\langle n,s,p\rangle\in\nbaseb{\mathcal{H}}{z}$ if and only if $(n,s,p,D)\in\Psi$ for some finite set $D\subseteq\nbaseb{GH}{x}$.
Let $L_s$ be the set of all $\langle n,t,p\rangle$ such that $t>s$ and
\[(\forall (m,u,q,D)\in\Psi)\;[(m=n\mbox{ and }D\subseteq\nbaseb{GH}{x})\;\rightarrow\;|q-p|<2^{-t}+2^{-u}].\]

In other words, the diameter of the ball $B_{t,p}=\{y\in[0,1]:|y-p|<2^{-t}\}$ determined by $\langle n,t,p\rangle$ is less than $2^{-s}$, and the ball $B_{t,p}$ must intersect with any ball enumerated by $\Psi^{\nbaseb{GH}{x}}$ at the $n$-th coordinate.

Note that $\nbaseb{GH}{x}$ is a $\Sigma^1_1(x)$ subset of $\omega$.
Therefore, $L_s$ is a $\Pi^1_1(x)$ subset of $\omega$ uniformly in $s$, and clearly nonempty.
We claim that $z(n)\in\overline{B_{t,p}}$ for any $\langle n,t,p\rangle\in L_s$.
To see this, let $V$ be an arbitrary open neighborhood of $z(n)$.
Then, there is $v>u$ and $q$ such that $z(n)\in B_{v,q}\subseteq V$.
Since $\langle n,v,q\rangle\in\nbaseb{\mathcal{H}}{z}$, $\Psi^{\nbaseb{GH}{x}}$ enumerates $\langle n,v,q\rangle$, and then, as mentioned above, $B_{t,p}$ intersects with such $B_{v,q}$.
Therefore, $B_{t,p}$ intersects with any open neighborhood of $z(n)$, that is, $z(n)\in\overline{B_{t,p}}$; hence $z(n)\in B_{t-1,p}$.

Since $L_s$ is $\Pi^1_1(x)$ uniformly in $s$, by uniformization (see \cite[Theorem II.2.3]{SacksBook}), there is a $\Pi^1_1(x)$ total function $\langle n,s\rangle\mapsto h(n,s)$ such that $\langle n,s,h(n,s)\rangle\in L_s$.
By totality, $h$ is $\Delta^1_1(x)$.
Thus, we obtain a $\Delta^1_1(x)$-sequence $(B_{s-1,h(n,s)})_{s\in\omega}$ of open balls such that $z(n)\in \bigcap_sB_{s-1,h(n,s)}$ for all $s$.
Indeed, we have $\{z(n)\}=\bigcap_sB_{s-1,h(n,s)}$ (that is, $(h(n,s))_{s\in\om}$ is a Cauchy sequence rapidly converging to $z(n)$) since the diameter of $B_{s-1,h(s)}$ is at most $2^{-s+2}$.
Hence, one can enumerate $\nbaseb{\mathcal{H}}{z}$ using $h$.
Since $h$ is $\Delta^1_1(x)$, this shows that $\nbaseb{\mathcal{H}}{z}\leq_e\nbaseb{\lambda}{x^{(\alpha)}}$ for some $\alpha<\omega_1^{{\rm CK},x}$.
However, this implies that $\nbaseb{GH}{x}\not\leq_e\nbaseb{\mathcal{H}}{z}$ by Proposition \ref{prop:gandy-harrington}.
}

\subsubsection{Irregular Lattice Topology}

For definitions, see Section \ref{sec:3-7-4}.

\propproof{prop:irregular-lattice-degree}{
Given $(x,y)\in(\mathcal{L}_{IL})^\om$, define $X$ as the coded neighborhood filter of $(x,y)$ in $\mathcal{L}^\om$ (which is equivalent to $\nbaseb{\hat{\om}^\om}{x}\oplus\nbaseb{\hat{\om}^\om}{y}$).
Note that $X$ is total as mentioned above.
Then, define $A$ and $P$ as follows:
\begin{align*}
A&=A(x):=\{n\in\om:x(n)=(\infty,\infty)\},\\
P&=P(y):=\{n\in\om:y(n)\in\om\}.
\end{align*}

Clearly, $\cmp{A}$ and $P$ are c.e.\ relative to $X$.
Since $(\mathcal{L}_{IL})^\om$ is finer than $\mathcal{L}^\om$, we can recover the coded $\mathcal{L}^\om$-neighborhood filter $X$ of $x$ from an enumeration of $\nbaseb{(\mathcal{L}_{IL})^\om}{x}$.
One can see that
\[n\in A\cup P\iff (\exists a,b\in\om)(\exists i\in\{0,2\})\;\langle n,i,a,b\rangle\in\nbase{x,y}.\]
Thus, $X\oplus\cmp{X}\oplus(A\cup P)$ is $e$-reducible to $\nbaseb{(\mathcal{L}_{IL})^\om}{x,y}$.

Conversely, from $X$, we first decode $\nbaseb{\hat{\om}^\om}{x}$ and $\nbaseb{\hat{\om}^\om}{y}$.
Then, it is easy to see that for any $i<2$ and $n,a,b\in\om$,
\[\langle n,i,a,b\rangle\in\nbaseb{(\mathcal{L}_{IL})^\om}{x,y}\iff\langle n,0,a\rangle\in\nbaseb{\hat{\om}^\om}{x}\mbox{ and }\langle n,i,b\rangle\in\nbaseb{\hat{\om}^\om}{y}.
\]
Moreover, one can see that
\begin{align*}
\langle n,2,a,b\rangle\in\nbaseb{(\mathcal{L}_{IL})^\om}{x,y}\iff\langle n,1,a\rangle\in\nbaseb{\hat{\om}^\om}{x},\;\langle n,1,b\rangle\in\nbaseb{\hat{\om}^\om}{y}\\
\mbox{ and }n\in A\cup P.
\end{align*}

The above two equality given us an $e$-reduction from $\nbaseb{(\mathcal{L}_{IL})^\om}{x,y}$ to $X\oplus\cmp{X}\oplus(A\cup P)$.
Consequently, every $(\mathcal{L}_{IL})^\om$-degree is co-$d$-CEA.

Next, assume that a co-$d$-CEA set is given, i.e.\ sets $X,A,P$ such that $\cmp{A}$ and $P$ are $X$-c.e.\ are given, and consider $A\cup P$.
Without loss of generality, we can assume that $A\cap P=\emptyset$ since replacing $P$ with the new $X$-c.e.\ set $P\setminus A$ does not affect on the set $A\cup P$.
Then, we construct $(x,y)\in(\mathcal{L}_{IL})^\om$ as follows.
Fix $X$-computable enumerations of $\cmp{A}$ and $P$.
First we use $(x(2n),y(2n))_{n\in\om}$ to code $X$.
Then define $x(2n+1)$ and $y(2n+1)$ as follows:
\begin{align*}
x(2n+1)&=
\begin{cases}
\infty&\mbox{ if }n\in A,\\
s&\mbox{ if we see $n\in \cmp{A}$ at stage $s$,}
\end{cases}
\\
y(2n+1)&=
\begin{cases}
\infty&\mbox{ if }n\not\in P,\\
t&\mbox{ if we see $n\in P$ at stage $t$,}
\end{cases}
\end{align*}
Then, $A$ and $P$ are recovered from $(x,y)$ as above, i.e., $A=A(x)$ and $P=P(y)$.
The above argument shows that for any $(x,y)\in(\mathcal{L}_{IL})^\om$, $\nbaseb{(\mathcal{L}_{IL})^\om}{x,y}$ is $e$-equivalent to $X\oplus\cmp{X}\oplus(A(x)\cup P(y))$.
This concludes the proof.
}

\propproof{prop:irregular-lattice-degree2}{
Fix $X,A,B\subseteq\om$ such that $B$ and $A\cup B$ are $X$-co-c.e., and $A$ and $B$ are disjoint.
Note that $\cmp{A}=B\cup\cmp{(A\cup B)}$, that is, it is the union of an $X$-co-c.e.\ set and an $X$-c.e.\ set, and thus, $\cmp{A}$ is co-$d$-c.e.\ relative to $X$.
We claim that ${\rm Enum}(X\oplus\cmp{X}\oplus\cmp{A})$ is Medvedev equivalent to $\{X\}\times{\rm Sep}(A,B)$.

We first show that there is a $X$-computable function that, given enumeration of $\cmp{A}$, returns a set $C$ separating $A$ from $B$.
Fix an enumeration of $B$ relative to $X$.
Then, given an enumeration of $\cmp{A}$, wait until we see either $n\in\cmp{A}$ or $n\in\cmp{B}$ (by using an enumeration relative to $X$).
Since $A$ and $B$ are disjoint, this happens at some stage.
If we see $n\in\cmp{A}$ (before seeing $n\in\cmp{B}$), we enumerate $n$ into $\cmp{C}$.
If we see $n\in\cmp{B}$ (before seeing $n\in\cmp{A}$), we enumerate $n$ into $C$.
Clearly, $C$ separates $A$ from $B$.

Conversely, assume that a set $C$ separating $A$ from $B$ is given.
Then, note that $\cmp{A}=\cmp{C}\cup\cmp{(A\cup B)}$.
Thus, wait until we see either $n\in\cmp{C}$ or $n\in\cmp{(A\cup B)}$ (by using an enumeration relative to $X$).
If we see this, enumerate $n$ into $\cmp{A}$.
This procedure gives us a correct enumeration of $\cmp{A}$.

Next, assume that a co-$d$-CEA set is given, that is, disjoint sets $B,P\subseteq\om$ such that $B$ is $X$-co-c.e.~and $P$ is $X$-c.e.\ are given.
Define $A=\cmp{(B\cup P)}$.
Note that $A$ and $B$ are disjoint, and that $B$ and $A\cup B=\cmp{P}$ are $X$-co-c.e.
Since $B\cup P=\cmp{A}$, by the same argument as above, we can show that ${\rm Enum}(X\oplus\cmp{X}\oplus(B\cup P))$ is Medvedev equivalent to $\{X\}\times{\rm Sep}(A,B)$.
}

\propproof{prop:proper-doubled}{
We construct $Z=(A\cup P)\oplus(B\cup N)$.
Let $E_e$ be the $e$-th co-$d$-c.e.\ set.
Begin with $n\in B$.
Wait until $2n+1$ is enumerated into $\Psi\circ\Phi(Z_s)$ with $\Phi$-use $\varphi_s(n)$.
If we see this, remove $n$ from $B$, and enumerate $n$ into $A$.
That is, define $Z_{s+1}=(Z_s\setminus\{2n+1\})\cup\{2n\}$.
Restrain $Z_{s+1}\upto \varphi_s(n)$.
Wait until $2n$ is enumerated into $\Psi\circ\Phi(Z_t)$ with $\Phi$-use $\varphi_t(n)$.
Restrain $Z_{t}\upto\varphi_t(n)$.
Given $S$, we write $S^0=(S\setminus\{2n\})\cup\{2n+1\}$ and $S^1=(S\setminus\{2n+1\})\cup\{2n\}$.
For any stage $u$ after $t$, either both $2n$ and $2n+1$ are enumerated into $\Psi\circ\Phi(Z_u^i)$ for some $i<2$ or there is $m<\max\{\varphi_s(n),\varphi_t(n)\}$ such that $m\in\Phi(Z_u^i)$, but $m\not\in\Phi(Z_u^{1-i})$ by monotonicity of an enumeration operator.
In the former case, put $Z_{u+1}=Z_u^i$, and restrain the $\Phi$-use.
In the latter case, search for such $m$, and choose $i$ such that the current guess of $\Phi(Z^i_u;m)$ is unequal to the current approximation of $E_e(m)$.
Then put $Z_{u+1}=Z^i_u$.
Note that, at some later stage $v>u$, we may see that $\Phi(Z^i_v;m)=E(m)$.
In this case, we search for new $m$ and $i$, and continue the similar procedure.
This procedure converges at some stage, and therefore, this is finite injury.
Combine the quasi-minimal strategy with this.
}

\subsection{Degrees of points: $G_\delta$-topology}

\subsubsection{Closed networks and $G_\delta$-spaces}

We show basic properties of $G_\delta$-spaces.
For definitions, see Section \ref{sec:G-delta-space}.

\obsproof{observation:closed-network}{
If $\xx$ is $T_1$, every point is closed.
Thus, $\nn=\{\{x\}:x\in\xx\}$ forms a closed network.
We show the converse direction.
Fix $x\not=y$.
Since $\xx$ is $T_0$, there is an open set $U$ such that either $x\in U\not\ni y$ or $x\not\in U\ni y$.
Without loss of generality, we may assume that $x\in U$ and $y\not\in U$.
Since $\xx$ has a closed network, there is a closed set $F$ such that $x\in F\subseteq U$.
Then, $V=\xx\setminus F$ is open, and we have $x\not\in V$ and $y\in V$.
This shows that $\xx$ is $T_1$.
}

\obsproof{obs:Gdelta-Borel-hierarchy}{
Clearly, if $G_\delta=\mathbf{\Pi}^0_2$ in a space $\xx$, then $\xx$ is a $G_\delta$ space since every closed set is constructible, and hence $\mathbf{\Pi}^0_2$.
To see the converse, note that the class of $G_\delta$ sets is closed under finite union and countable intersection.
If $\xx$ is a $G_\delta$ space, then any open or closed set is $G_\delta$.
Hence, every constructible set is $G_\delta$, and therefore, any $\mathbf{\Pi}^0_2$ set is $G_\delta$.
}

\propproof{prop:G_delta-equal-closed-network}{
Let $\xx$ be a $T_0$ space with a countable basis $(\beta_e)_{e\in\om}$.
If $\xx$ is a $G_\delta$-space, then every open set is $F_\sigma$, and therefore, for any $e\in\om$, there is a countable collection $(F^e_n)_{n\in\om}$ of closed sets such that $\beta_e=\bigcup_{n\in\om}F^e_n$.
Since $(\beta_e)_{e\in\om}$ is a basis, $(F^e_n)_{e,n\in\om}$ forms a countable closed network for $\xx$.

Conversely, if $\nn$ is a countable closed network for $\xx$, for any open set $U$, consider the $F_\sigma$ set $N(U)=\bigcup\{N\in\nn:N\subseteq U\}$.
We claim that $U=N(U)$.
The inclusion $N(U)\subseteq U$ is clear.
For the inclusion $U\subseteq N(U)$, given $x\in U$, since $\nn$ is a network, there is $N\in\nn$ such that $x\in N\subseteq U$.
This means that $x\in N(U)$, and therefore, $U=N(U)$, that is, $U$ is $F_\sigma$.
This concludes that $\xx$ is a $G_\delta$-space.
}

\obsproof{obs:twin-second-countable}{Let $(B_i)_{i\in\om}$ be a countable basis for $\xx$.
Assume that $\xx$ is compact and $T_1$.
For each finite set $D$, let $N_D$ be the complement of $\bigcup_{i\in D}B_i$.
We claim that $(N_D)_{D\subseteq_{\rm fin}\om}$ forms a countable closed network for $\xx$.
To see this, fix a point $x\in \xx$ and open neighborhood $U$ of $x$.
Since $\xx$ is $T_1$, the complement $\{x\}$ is open, and thus it is written as $\bigcup_{i\in I}B_i$.
Since $\xx\setminus\{x\}$ covers the closed subset $\xx\setminus U$ of the compact space $\xx$, there is a finite set $D\subseteq I$ such that $\xx\setminus U$ is covered by $\bigcup_{i\in D}B_i$, which means that $N_D\subseteq U$.
Moreover, we have that $\bigcup_{i\in D}B_i\subseteq \xx\setminus\{x\}$, and therefore $x\in N_D$.
This shows that $(N_D)_{D\subseteq_{\rm fin}\om}$ forms a countable closed network for $\xx$.
Thus, $\xx$ is $G_\delta$ by Proposition \ref{prop:G_delta-equal-closed-network}.

Next, if $\xx$ is metrizable, it is easy to see that every open ball is a countable union of closed balls.
Hence, every metrizable space is $G_\delta$.
The implication from being $G_\delta$ to being $T_1$ follows from Observation \ref{observation:closed-network} and Proposition \ref{prop:G_delta-equal-closed-network}.
}

\propproof{exa:indiscrete-irrational-extension}{
To simplify our argument, we consider the indiscrete irrational extension of $\mathcal{C}=2^\om$ rather than $\mathbb{R}$.
More formally, let $J$ be the set of all infinite binary sequences containing infinitely many $0$'s, and then consider the $J$-extension $\mathcal{C}_J$ of the Cantor topology $\tau_\mathcal{C}$, i.e., the topology generated by $\tau_\mathcal{C}\cup\{J\}$.

Suppose for the sake of contradiction that $\mathcal{C}_J$ is $G_\delta$.
Let $(B_i)_{i\in\om}$ and $(N_i)_{i\in\om}$ be a countable basis and a countable closed network for $\mathcal{C}_J$ (by Proposition \ref{prop:G_delta-equal-closed-network}).
Note that every $N_i$ is closed in $\mathcal{C}_J$, and therefore $N_i$ can be written as the complement of $V_i\cap J$ or $V_i$ for some $\tau_\mathcal{C}$-open set $V_i$.
Then, there is an oracle $Z$ such that $U_i$ and $V_i$ are $Z$-c.e.~open for any $i\in\om$.
Let $x$ be a $1$-generic real relative to $Z$.
Clearly, $x\in J$.
Therefore, there is $i$ such that $x\in B_i\subseteq J$.
Then, since $(N_i)_{i\in\om}$ is a closed network for $\mathcal{C}_J$, there is $j$ such that $x\in N_j\subseteq B_i$.
Since $x\in J$, we have $x\not\in V_j$.
Note that the complement of $B_i$ is dense, and thus $V_j$ is dense.
However, since $x$ is $1$-generic relative to $Z$, and $V_j$ is a dense $Z$-c.e.~open set, we must have $x\in V_j$, a contradiction.
}

\subsubsection{Cototal enumeration degrees}

In this section, we show one of the important results in this article connecting the notion of a cototal $e$-degree and the notion of a $G_\delta$-space.
For definitions, see Section \ref{sec:cototalenumeration}



\obsproof{obs:Amax-decidable}{
For finite sets $D,E\subseteq\om^{<\om}$, we claim that $D\subseteq E$ if and only if $A_{\rm max}^{\rm co}\cap[E]\subseteq A_{\rm max}^{\rm co}\cap[D]$.
It suffices to show that $D\not\subseteq E$ implies $A_{\rm max}^{\rm co}\cap[E]\not\subseteq A_{\rm max}^{\rm co}\cap[D]$.
Choose $\sigma\in D\setminus E$.
Then, it is easy to construct a maximal antichain $X\subseteq\om^{<\om}$ such that $X\cap E=\emptyset$ and $\sigma\in X$.
For instance, consider $X=\{\sigma\}\cup\{\tau\in\om^\ell:\sigma\not\preceq\tau\}$ for a sufficiently large $\ell$.
Then, $E\subseteq \cmp{X}$, but $D\not\subseteq\cmp{X}$.
This shows that $\cmp{X}\in[E]\setminus[D]$; therefore $A_{\rm max}^{\rm co}\cap[E]\not\subseteq A_{\rm max}^{\rm co}\cap[D]$.}

%

%
%

\thmproof{cototal-equal-twin}{
To prove Theorem \ref{cototal-equal-twin}, we will see that one can assume that, in a computably $G_\delta$ space, every open set can be written as a computable union of finitary closed sets, where a set is finitary closed (w.r.t.\ $\beta$) if it is the complement of finitely many open sets in the basis generated by $\beta$.
Be careful that the notion of being finitary closed depends on the choice of the representation $\beta$; hence it is not a topological property.

Recall from Section \ref{sec:change-repres} the notion of reducibility of representations; for instance, by $\gamma\equiv\delta$ we mean that $\gamma$ is bi-reducible to $\delta$.

\begin{obs}\label{obs:comp-cl-network2}
Let $\xx=(X,\beta)$ be a represented cb$_0$ space which is computably $G_\delta$.
Then, there is a representation $\gamma\equiv\beta$ of $X$ such that, given $e\in\om$, one can effectively find a computable sequence $(Q^e_n)_{n\in\om}$ of $\gamma$-finitary closed sets with $\gamma_e=\bigcup_nQ^e_n$.
\end{obs}

\begin{proof}
Let $f$ be a computable function witnessing that $\xx$ is computably $G_\delta$.
Define $\gamma_{2e}=\beta_e$ and $\gamma_{2\langle e,n\rangle+1}=X\setminus P_{f(e,n)}=\bigcup\{\beta_d:d\in W_{f(e,n)}\}$.
Then, $\gamma_{2e}=\beta_e=\bigcup_nP_{f(e,n)}=\bigcup_n(X\setminus\gamma_{2\langle e,n\rangle+1})$, which is a computable union of $\gamma$-finitary closed sets.
Moreover, $\gamma_{2\langle e,n\rangle+1}$ is a computable union of sets of the form $\beta_d$, where $\beta_d$ can be written as a computable union of $\gamma$-finitary closed sets.
This concludes the proof.
\end{proof}

\begin{obs}\label{obs:robust-network-G-delta}
Let $(X,\beta)$ be a represented cb$_0$ space, and let $\gamma$ be a representation of $X$ such that $\beta\equiv\gamma$.
Then, if $(X,\beta)$ is a computably $G_\delta$ space, so is $(X,\gamma)$.
\end{obs}

\begin{proof}
Note that $\beta\leq\gamma$ iff, given a $\beta$-code of an open set in $\xx$, one can effectively find its $\gamma$-code.
Given a $\gamma$-basic open set $U$, one can find its $\beta$-code since $\gamma\leq\beta$.
Since $(X,\beta)$ is computably $G_\delta$, one can find a $\beta$-computable sequence of closed sets whose union is $U$.
Since $\beta\leq\gamma$, it is also computable w.r.t.\ $\gamma$.
Hence, $(X,\gamma)$ is computably $G_\delta$.
\end{proof}

We now assume that $\xx=(X,\beta)$ is computably $G_\delta$.
By Observation \ref{obs:comp-cl-network2}, there is $\gamma\equiv\delta$ such that every basic open set can be written as a computable union of finitary closed sets w.r.t.\ $\gamma$ in an effective manner.
We will construct an enumeration operator $\Psi$.
Let $W$ be a c.e.\ set such that $\gamma_e=\bigcup_{\langle e,D\rangle\in W}N_D$, where $N_D=\xx\setminus\bigcup_{j\in D}\gamma_j$.
Then we claim that $W$ witnesses uniform cototality of $(X,\gamma)$, that is,
\[e\in\nbaseb{\gamma}{x}\iff(\exists D)\;[D\subseteq\cmp{\nbaseb{\gamma}{x}}\mbox{ and }\langle e,D\rangle\in W].\]

To see the implication ``$\Leftarrow$'', we first note that $D\subseteq\cmp{\nbaseb{\gamma}{x}}$ if and only if $x\in N_D$ by definition of $N_D$.
Moreover, if $\langle e,D\rangle\in W$, then $N_D\subseteq \gamma_e$, and therefore, the right formula implies $e\in\nbase{x}$.
For the implication ``$\Rightarrow$'', if $e\in\nbase{x}$, since $\gamma_e=\bigcup_{\langle e,D\rangle \in W}N_D$, there is a finite set $D$ such that $\langle e,D\rangle\in W$ and $x\in N_D\subseteq \gamma_e$.
Then we have $D\subseteq\cmp{\nbase{x}}$, and $\langle e,D\rangle\in W$ as desired.
Consequently, $(X,\gamma)$ is uniformly cototal.

Conversely, we assume that $(X,\gamma)$ is relatively cototal via an enumeration operator $\Psi$.
Then for any finite set $D$, define $N_D=\xx\setminus\bigcup_{n\in D}\gamma_n$.
We claim that $\gamma_n=\bigcup\{N_D:\langle n,D\rangle\in\Psi\}$.
For the inclusion ``$\subseteq$,'' if $x\in\gamma_n$, then since $\Psi(\cmp{\nbase{x}})=\nbase{x}$, there is a finite set $D\subseteq\cmp{\nbase{x}}$ (i.e., $x\not\in\bigcup_{j\in D}\gamma_j$, and therefore $x\in N_D$) such that $\langle n,D\rangle\in\Psi$.
For the inclusion ``$\supseteq$,'' we show that if $\langle n,D\rangle\in\Psi$, then $\gamma_n\cup\bigcup_{j\in D}\gamma_j=\xx$ (i.e., $N_D\subseteq \gamma_n$).
Otherwise, there is $y\in\xx$ such that $y\not \in \gamma_n\cup\bigcup_{j\in D}\gamma_j$.
However, we then have $D\subseteq\cmp{\nbase{y}}$, which implies $n\in\Psi(\cmp{\nbase{y}})$, while $n\not\in\nbase{y}$.
Then, we get $\Psi(\cmp{\nbase{y}})\not=\nbase{y}$, which contradicts our choice of $\Psi$.
This shows that given $n$, one can effectively find a computable sequence $(N_D:\langle n,D\rangle\in\Psi)$ of finitary closed sets whose union is $\gamma_n$, that is, $(X,\gamma)$ is computably $G_\delta$.
Hence, $(X,\beta)$ is also computably $G_\delta$ by Observation \ref{obs:robust-network-G-delta}.
}


\thmproof{thm:cototal-Gdelta-space2}{
As mentioned in Example \ref{exa:McCarthy}, McCarthy \cite{McC17} showed that the space $A_{\rm max}^{\rm co}$ is uniformly cototal, and moreover, the $A_{\rm max}^{\rm co}$-degrees are exactly the cototal $e$-degrees.
Hence, by Theorem \ref{cototal-equal-twin}, $A_{\rm max}^{\rm co}$ is computably $G_\delta$.
Moreover, as seen in Observation \ref{obs:Amax-decidable}, $A_{\rm max}^{\rm co}$ is a decidable cb$_0$ space.
Consequently, $A_{\rm max}^{\rm co}$ is a decidable $G_\delta$-space which captures the cototal $e$-degrees.
}

\subsection{Quasi-Polish topology}

We give a proof of some results mentioned in Section \ref{sec:3-9}

\propproof{prop:quasi-Polish-examples}{
It is easy to see that if $\xx$ is $\mathbf{\Pi}^0_2$-named, then so is $\xx^\om$.
For (1), by Fact \ref{fact:deBrecht-ocs}, it suffices to show that $\hat{\om}_{TP}$ is an open continuous image of a Polish space.
Define a function $\delta$ by $\delta(j0^n10^\om)=n$ for each $j<2$, $\delta(0^\om)=\infty$, and $\delta(10^\om)=\infty_\star$.
It is clear that the domain of $\delta$ is a closed subset of $2^\om$; hence Polish.
For continuity, the preimages of basic open sets $\{n\}$, $[n,\infty]$, and $[n,\infty_\star]$ are the clopen sets $[00^n1]\cup[10^n1]$, $[00^n]$, and $[10^n]$.
For openness, the images of basic open sets $[j0^n1\tau]$, $[00^n]$, and $[10^n]$ are the open sets $\{n\}$, $[n,\infty]$, and $[n,\infty_\star]$.
Hence, $\delta$ is open and continuous.

For (2), define a partial surjection $\delta:\subseteq\om^\om\to\mathcal{P}_{DO}$ as follows.
\begin{align*}
&\delta(0n0^m10^\om)=(n,m),\ \delta(0n0^m20^\om)=(n,\overline{m}),\ \delta(0n0^\om)=(n,\ast),\\
&\delta(10^\om)=\mathbf{0}_\star,\;\delta(20^\om)=\mathbf{0},\\
&\delta(10^s1m0^n10^\om)=(n+s,m+s),\  \delta(10^s1m0^\om)=(\infty,m+s),  \\
&\delta(20^s1m0^n10^\om)=(n+s,\overline{m+s}),\ \delta(20^s1m0^\om)=(\infty,\overline{m+s}).
\end{align*}

It is clear that ${\rm dom}(\delta)$ is a closed subset of $\om^\om$; hence Polish.
For continuity, the preimages of some basic open sets are:
\begin{align*}
\delta^{-1}[(n,\overline{m})]&=[0n0^m2]\cup\bigcup_{s\leq m}[20^s1(m-s)0^{n-s}],\\
\delta^{-1}[[n,\infty]\times\{m\}]&=\bigcup_{s\leq m}[10^s1(m-s)0^{n-s}]\cup\bigcup_{k\geq n}\delta^{-1}[(k,m)],\\
\delta^{-1}[\{n\}\times[m,\overline{m}]]&=[0n0^m]\cup\bigcup_{k\in\om\cup\om^\ast}\delta^{-1}[(n,k)],\\
\delta^{-1}[([n,\infty]\times(\ast,\overline{n}])\cup\{\mathbf{0}\}]&=[20^n]\cup\bigcup_{m\geq n}\delta^{-1}[[n,\infty]\times\{\overline{m}\}].
\end{align*}

These sets are open, and therefore, $\delta$ is continuous.
For openness, the images of some basic open sets $[0n0^m]$, $[0n0^m2\tau]$, $[10^s]$, $[20^s1m0^n]$, $[10^s1m0^n1\tau]$ are the open sets $\{n\}\times[m,\overline{m}]$, $\{(n,\overline{m})\}$, $([s,\infty]\times[s,\ast))\cup\{\mathbf{0}_\star\}$, $[n+s,\infty]\times\{\overline{m+s}\}$, and $\{(n+s,m+s)\}$.
Consequently, $\delta$ is open and continuous.

For (3), we define a partial surjection $\delta:\subseteq\om^\om\to\mathcal{QA}$ as follows.
\begin{align*}
\delta(0^\om)&=(0,\om^3),\;\delta(10^\om)=(\overline{0},\om^3),\\
\delta(00^j1k\ell\tau)&=(\ell+1,\om^2\cdot j+\om\cdot(2k)+2\ell+1),\\
\delta(10^j1k\ell\tau)&=(\overline{\ell+1},\om^2\cdot j+\om\cdot(2k+1)+2\ell+1),\\
\delta(2k0^\om)&=(0_\zeta,\om^2\cdot k+1),\\
\delta(2k0^{2j}1\ell\tau)&=((-\ell-1)_\zeta,\om^2\cdot k+\om\cdot(2j)+2\ell+2),\\
\delta(2k0^{2j+1}1\ell\tau)&=((\ell+1)_\zeta,\om^2\cdot k+\om\cdot(2j+1)+2\ell+2),\\
\delta(3k\ell 0^\om)&=(\infty,\om^2\cdot k+\om\cdot (2\ell+1)),\\
\delta(3k\ell 0^{2j}1\tau)&=(j+1,\om^2\cdot k+\om\cdot (2\ell)+2j+1),\\
\delta(3k\ell 0^{2j+1}1\tau)&=((-j-1)_\zeta,\om^2\cdot k+\om\cdot (2\ell)+2j+2),\\
\delta(4k\ell 0^\om)&=(\overline{\infty},\om^2\cdot k+\om\cdot (2\ell+2)),\\
\delta(4k\ell 0^{2j}1\tau)&=(\overline{j+1},\om^2\cdot k+\om\cdot (2\ell+1)+2j+1),\\
\delta(4k\ell 0^{2j+1}1\tau)&=((j+1)_\zeta,\om^2\cdot k+\om\cdot (2\ell+1)+2j+2).
\end{align*}
where $j,k,\ell\in\om$, and $\tau$ is an arbitrary finite string.
It is not hard to check that $\delta$ is open continuous.

For (4), we define a partial surjection $\delta:\subseteq\om^\om\to\mathcal{C}$ as follows.
\[
\begin{array}{ll}
\delta(0^\om)=(\infty,\infty),& \delta(00^jab\tau)=(j+a,j+b),\\
\delta(1n0^\om)=(n,\infty),&\delta(1n0^j1\tau)=(n,j).
\end{array}
\]
It is clear that the domain of $\delta$ is a closed subset of $\om^\om$.
It is also easy to check that $\delta$ is open continuous.
}

\propproof{prop:non-quasi-Polish-examples}{
(1)
Let $A$ be a (light-face) strictly co-analytic subset of Baire space, and let $A'$ be the same subset inside $(\om^\om)_{GH}$. Then $A'$ is closed in $(\om^\om)_{GH}$, so if $(\om^\om)_{GH}$ were quasi-Polish, then $A'$ as a subspace of $(\om^\om)_{GH}$ would be quasi-Polish, too. Thus, there would exist a continuous function $g : (\om^\om) \to (\om^\om)_{GH}$ such that $g(\om^\om) = A'$. Note that ${\rm id} : (\om^\om)_{GH} \to (\om^\om)$ is trivially continuous, and that $A = ({\rm id} \circ g)(\om^\om)$, i.e.~ $A$ is a continuous image of Baire space, hence analytic, which contradicts the choice of $A$ as a strictly co-analytic set.

(2)
We use a theorem by de Brecht \cite{debrecht8} showing that a $\Pi^1_1$-subspace of a quasi-Polish space is either quasi-Polish or contains one of four canonical counterexamples as $\Pi^0_2$-subspace (see Theorem \ref{thm:Hurewicz-dichotomy}). From its definition, it is easy to see that the canonic embedding of $\mathbb{N}_{\rm rp}$ into $\om^\om$ is as $\Sigma^0_3$-subspace. Moreover, $\mathbb{N}_{\rm rp}$ is Hausdorff, and the only Hausdorff space amongst de Brecht's four counterexamples is $\mathbb{Q}$. We thus arrive at: Either $\mathbb{N}_{\rm rp}$ is quasi-Polish or $\mathbb{Q}$ embeds as $\Pi^0_2$-subspace into $\mathbb{N}_{\rm rp}$.

Thus, it suffices to show that $\mathbb{Q}$ embeds into $\mathbb{N}_{\rm rp}$ as a $\Pi^0_2$-subspace.
Inductively choose $n(s)$ as a number satisfying $1 + \sum_{i<s} \prod_{k\leq n(i)} p_k < p_{n(s)}$.
Then, given $b = b_0b_1b_2\dots$ define
\[h(b) = 1 + \sum_{i} b_i \prod_{k\leq n(i)} p_k.\]

\begin{claim}
$h$ is a computable embedding of the dyadic rationals into $\mathbb{N}_{\rm rp}$.
\end{claim}

For computability of $h$, to check whether $h(b) \equiv u$ mod $v$, let $p_k$ be the largest prime factor of $v$. Then there is $s$ such that $k\leq n(s)$.
   Then, to compute the value of $h(b)$ mod $v$, we only need to check the first $s$ terms of $h(b)$.

For computability of $h^{-1}$, given $h(b)$, inductively assume that we have already computed $b_0,b_1,\dots, b_{s-1}$.
 To compute $b_s$, let $r_i = \prod_{k\leq n(i)} p_k$.
 We have already computed $h(b)[s] := 1 + \sum_{i<s} b_i r_i$.
By our choice of $n(s+1)$ we have $h(b)[s] + r(s) < p_{n(s+1)}$. Hence, since $p_{n(s+1)}$ is prime, we have
        $\gcd ( h(b)[s] , p_{n(s+1)} ) = 1$,  and    $\gcd ( h(b)[s] + r(s) , p_{n(s+1)} ) = 1$.
Thus, to the relatively prime integer topology, we can ask whether
        $h(b) \equiv h(b)[s]$ mod $p_{n(s+1)}$,  or   $h(b) \equiv h(b)[s] + r(s)$ mod $p_{n(s+1)}$.
If the former holds, then $b_s=0$, and if the latter holds, then $b_s=1$.

(3)
By Theorem \ref{thm:Hurewicz-dichotomy}, it suffices to show that $\om_{\rm cof}$ embeds into the maximal antichain space $\mathcal{A}_{\rm max}^{\rm co}$ as a $\Pi^0_2$-subspace.
For that we map $n\in\mathbb{N}$ to the complement of the set of all strings of length $n$.
It is not hard to check that this gives indeed an embedding. To see that the range is $\Pi^0_2$, note that we can define it as \emph{for any pair of words of distinct length, at least one of them belongs to the complement of the antichain}.
}

\section{Cs-networks and non-second-countability}\label{sec:cs-networks}

In this section, we develop techniques which will be used in the next section.
As we mentioned repeatedly, one can develop computability theory on some non-second-countable spaces (without using notions from higher computability theory such as $\alpha$-recursion, $E$-recursion, infinite time Turing machines, etc.)
To explain this idea, we consider the following basic notion in general topology (originally introduced by Arhangel'skii in 1959).

\begin{definition}\label{def:network-strict-network}
Let $\xx$ be a topological space, and $\nn$ be a collection of subsets of $\xx$.
We say that $\nn$ is a {\em network at a point $x\in\xx$} if for any open neighborhood $U$ of $x$, there is $N\in\nn$ such that $x\in N\subseteq U$.
Moreover, if $\nn$ is a network at $x$, and if $x\in N$ holds for all $N\in\nn$, then we also say that $\nn$ is a {\em strict network at $x$}.
\end{definition}

We now consider a space $\xx$ which has no countable basis, but has a countable network $\nn=(N_e)_{e\in\om}$.
Recall that by the ``{\em degree of a point $x$}'' in a (represented) cb$_0$ space, we meant the {\em degree of difficulty of enumerating a neighborhood basis of $x$}.
However, if a space is non-second-countable, there may be no $\om$-step enumeration of a neighborhood basis of a point.
Instead, we consider the {\em degree of difficulty of enumerating a strict subnetwork of $\nn$ at }$x\in\xx$.
That is, we consider the following representation:
\[p\mbox{ is a name of }x\iff\{N_{p(n)}:n\in\om\}\mbox{ is a strict network at }x.\]

Now, the induced computability theory on $\xx$ heavily depends on the choice of a network $\nn$.
Of course, the same was true for a basis representation.
But the situation regarding a network is worse than the case of a basis.
On the one hand, one can always recover the topology on $\xx$ from a basis, and thus, any representations yield the same computability natures relative to some oracle.
On the other hand, a network does not memorize information on topology, and thus, the computability structure induced from a network can be almost arbitrary.
In summary, the notion of a network is too weak, and therefore, we need a more restrictive notion.

A number of variants of a network have been extensively studied in general topology (see \cite{Gr84,MNbook,LY16}).
Schr\"oder \cite{Sch03,schroder} clarified that the following variant captures the territory of computability theory.

\begin{definition}[Guthrie \cite{Guth}]
A {\em cs-network} $\nn$ for a topological space $\xx$ is a collection of subsets of $\xx$ such that, for any open set $U\subseteq\xx$, if a sequence $(x_n)_{n\in\omega}$ converges to $x\in U$, then there are $N\in\nn$ and $n_0\in\omega$ such that
\[\{x\}\cup\{x_n:n\geq n_0\}\subseteq N\subseteq U.\]
\end{definition}

Here, ``cs'' stands for ``convergent sequence''.
The following implications are clear:
\[\mbox{basis }\Longrightarrow\mbox{ cs-network }\Longrightarrow\mbox{ network.}\]

For notational simplicity, in this article, we assume that a {\em cs-network $\nn$ always contains the whole space}, that is, $\xx\in\nn$.
Schr\"oder \cite{Sch03,schroder} showed that a topological space $\xx$ has an admissible representation if and only if $\xx$ is a $T_0$ space with a countable cs-network.
Since then, the notion of a countable cs-network have become a key notion in the context of {\em a convenient category of domains} \cite{ELS04,BSS06,BSS07}\footnote{In \cite{Sch03,schroder,ELS04,BSS07}, a cs-network is called a {\em pseudobase} or a {\em sequential pseudobase}.}.

More explicitly, if $\nn=(N_e)_{e\in\om}$ is a countable cs-network for a $T_0$ space $\xx$, recall from Section \ref{sec:intro-admissible-representation} that the induced $\om^\om$-representation of $\xx$ from $\nn$ is given as follows:
\[\dnn(p)=x\iff \{N_{p(n)}:n\in\om\}\mbox{ is a strict network at }x.\]

This map $\delta_{\nn}$ always gives an admissible representation of $\xx$.
Note that the convention $\xx\in\nn$ makes it possible for $p$ to output no information at each stage.
Recall from Section \ref{sec:intro-admissible-representation} the definition of reducibility $\reduce$.
By $\pt{y}{\yy}\reduce\pt{x}{\xx}$ we mean that there is a partial computable function which, given an $\xx$-name of $x$, returns a $\yy$-name of $y$.

\begin{obs}\label{obs:network-e-basis}
Let $\xx=(X,\beta)$ be a represented cb$_0$ space, and $\yy=(Y,\nn)$ be a topological space with a countable cs-network.
Then $\pt{y}{\yy}\reduce\pt{x}{\xx}$ if and only if there is $J\leq_e\nbaseb{\xx}{x}$ such that $\{N_e:e\in J\}$ is a strict network at $y$.
\end{obs}

\begin{proof}
The ``if'' direction is obvious.
Assume that $\pt{y}{\yy}\reduce\pt{x}{\xx}$.
Then, there is a computable function $\Phi$ such that if $p$ enumerates $\nbaseb{\xx}{x}$, then $\Phi(p)$ enumerates a strict subnetwork of $\nn$ at $y$.
We define $\Psi$ as follows:
\[\langle e,{\rm rng}(\tau)\rangle\in\Psi\iff(\exists n)\;\Phi(\tau)(n)\downarrow=e.\]
Clearly, $\Psi$ is c.e.
Then, we define $J=\Psi(\nbaseb{\xx}{x})$.
We claim that $\{N_e:e\in J\}$ is a strict network at $y$.
For any open neighborhood $U$ of $y$, if $p$ enumerates $\nbaseb{\xx}{x}$, then $\Phi(p\upto s)(n)$ must output $e$ such that $y\in N_e\subseteq U$ for some $e,n,s\in\om$.
Since ${\rm rng}(p\upto s)\subseteq\nbaseb{\xx}{x}$, we have $e\in J$.
Therefore, $\{N_e:e\in J\}$ is a network at $y$.
For strictness, suppose that $y\not\in N_e$ for some $e\in J$.
Then, there are $\tau,n$ such that ${\rm rng}(\tau)\subseteq\nbaseb{\xx}{x}$ and $\Phi(\tau)(n)\downarrow=e$.
Clearly, $\tau$ can be extended to an $\xx$-name $p$ of $x$; however, we have $y\not\in N_{\Phi(p)(n)}$, and thus $\Phi(p)$ is not an enumeration of a strict network at $y$, which is a contradiction.
\end{proof}




%
%
%
%

\subsection{Regular-like networks and closure representation}

For a topological space $\xx$ with a countable network $\nn$, we introduce a new represented space with the underlying space $\xx$ names of whose points are given by a sequence of closures of network elements whose intersection captures the point.
Formally, we define that $p$ is a $\overline{\dnn}$-name of $x$ if and only if
\[\{N_{p(n)}:n\in\om\}\mbox{ is a network at $x$, and }x\in\overline{N_{p(n)}}\mbox{ for all $n\in\om$.}\]

Here, we do not require $\{N_{p(n)}:n\in\om\}$ to be strict, that is, $x\not\in N_{p(n)}$ can happen, while we always have $x\in\overline{N_{p(n)}}$.
At first glance this definition may look very strange; however, we will later see that this is a very useful technical notion.
In this section, we investigate how $(\xx,\dnn)$ and $(\xx,\overline{\dnn})$ are related, and we will show the following.

\begin{theorem}
Assume that $\xx$ is a topological space with a countable cs-network.
\begin{enumerate}
\item If $\xx$ is regular and Hausdorff, then $\xx$ has a countable cs-network $\nn$ such that $(\xx,\dnn)$ is isomorphic to $(\xx,\overline{\dnn})$.
\item There is a non-regular Hausdorff space $\xx$ which has a countable cs-network $\nn$ such that $(\xx,\dnn)$ is isomorphic to $(\xx,\overline{\dnn})$.
\end{enumerate}
\end{theorem}

Before proving this theorem, we have to warn the reader that $\overline{\dnn}$ may be a multi-representation in general (as studied e.g.~by Weihrauch in \cite{weihrauchg}), that is, a single $p$ can be a name of many points.
For instance, if $\nn$ is an open network (i.e.\ basis) of $(\om^\om)_{\rm co}$, then ${\rm id}:\om\to\om$ is an $\overline{\dnn}$-name of any point $x\in(\om^\om)_{\rm co}$.
Then, when does $p$ determine a single point $x$?
It is only if $\{\bigcap_n\overline{N_{p(n)}}\}$ is a strict network at $x$.
We first check that it is always true if $\xx$ is Hausdorff.

\begin{obs}\label{obs:closure-representation-1}
If a Hausdorff space $\xx$ has a countable network $\nn$, then $\overline{\dnn}$ is a (single-valued) representation of $\xx$.
\end{obs}

\begin{proof}
Note that $\xx$ is Hausdorff if and only if every point in $\xx$ can be written as the intersection of all its closed neighborhoods.
For a collection $\mm$ of subsets of $\xx$, we write $\overline{\mm}=\{\overline{M}:M\in\mm\}$.
Thus, $\{x\}=\bigcap\overline{\mathcal{O}_x}$, where $\mathcal{O}_x$ is the set of all open neighborhoods of $x$.
Now, assume that $p$ is a $\overline{\dnn}$-name of $x_0,x_1\in\xx$.
Then, for each $i<2$, $\nn_p=\{N_{p(n)}:n\in\om\}$ forms a network at $x_i$, and therefore $x_i\in\bigcap\overline{\nn_p}\subseteq\bigcap\overline{\mathcal{O}_{x_i}}=\{x_i\}$.
Hence, $\bigcap\overline{\nn_p}=\{x_0\}=\{x_1\}$, which implies $x_0=x_1$.
\end{proof}

We should be careful that, even if $\xx$ is Hausdorff, $\{\overline{N_{p(n)}}:n\in\om\}$ is not necessarily a network at $x$, that is, there may exist an open neighborhood $U$ of $x$ such that $\overline{N_{p(n)}}\not\subseteq U$ for all $n\in\om$, while we eventually have $\bigcap_n\overline{N_{p(n)}}\subseteq U$.

Now, it is clear that the identity map ${\rm id}:(\xx,\dnn)\to(\xx,\overline{\dnn})$ is computable (in the sense of Section \ref{intro:rep-sp-00}), that is, there is a partial computable function $F$ such that $\overline{\dnn}(F(p))=\dnn(p)$ for any $\dnn$-name $p$ of a point in $\xx$.
Under a certain assumption on a network, we also have computability of its inverse.

\begin{obs}
If a topological space $\xx$ has a countable closed network $\nn$, then the identity map ${\rm id}:(\xx,\overline{\dnn})\to(\xx,\dnn)$ is computable.
\end{obs}

\begin{proof}
Assume that $\xx$ has a countable closed cs-network $\nn$.
Let $p$ be a $\overline{\dnn}$-name of a point $x\in\xx$.
Since $\nn$ is closed, $\overline{\nn_{p(n)}}=\nn_{p(n)}$, and therefore, $\{{\nn_{p(n)}}:n\in\om\}$ forms a strict network at $x$.
Thus, $p$ is also a $\dnn$-name of $x$.
Hence, the identity map ${\rm id}:(\xx,\overline{\dnn})\to(\xx,\dnn)$ is computable.
%
%
\end{proof}

Recall that a topological space $\xx$ is {\em regular} if for any open neighborhood $U$ of a point $x\in\xx$, there is an open neighborhood $V$ of $x$ such that $x\in V\subseteq\overline{V}\subseteq U$.
It is equivalent to saying that if $\mm$ is a neighborhood basis at $x$, then so is $\overline{\mm}=\{\overline{N}:N\in\mm\}$.
We say that a network $\nn$ of $\xx$ is {\em regular-like} if for any $\mm\subseteq\nn$ and $x\in\xx$, if $\mm$ is a network at $x$, then so is $\overline{\mm}=\{\overline{N}:N\in\mm\}$.

It is clear that every closed network is regular-like.
The converse is not always true, but it is easy to see that a space has a closed network of cardinality $\kappa$ iff it has a regular-like network of cardinality $\kappa$.
In particular, if a space has a countable regular-like network, it is a $G_\delta$-space by Proposition \ref{prop:G_delta-equal-closed-network}.
However, it is unclear if every space with a regular-like cs-network always has a closed cs-network.

\begin{obs}\label{obs:regular-like-network}
Every network of a regular space is regular-like.
\end{obs}

\begin{proof}
Let $\nn$ be a network for a regular space $\xx$, and let $\mm\subseteq\nn$ be a network at $x\in\xx$.
Given an open neighborhood $U$ of $x$, by regularity of $\xx$, there is an open neighborhood $V$ of $x$ such that $x\in V$ and $\overline{V}\subseteq U$.
Since $\mm$ is a network for $\xx$, there is $N\in\mm$ such that $x\in N\subseteq V$.
Therefore, $x\in\overline{N}\subseteq\overline{V}\subseteq U$.
This shows that $\overline{M}=\{\overline{N}:N\in\mm\}$ is a network at $x$ as well.
Consequently, $\nn$ is regular-like.
\end{proof}

A regular space which has a countable network is known as a {\em cosmic space}, and a regular space which has a countable cs-network is known as an {\em $\aleph_0$-space} (see \cite{Mic66,Guth}).

\begin{theorem}\label{thm:regular-like-network}
A topological space $\xx$ has a countable regular-like cs-network if and only if $\xx$ has a countable cs-network $\nn$ such that the identity map ${\rm id}:(\xx,\overline{\dnn})\to(\xx,\dnn)$ is continuous.
\end{theorem}

\begin{proof}
Assume that $\xx$ has a countable regular-like cs-network $\nn$.
Then, $\mm=\nn\cup\overline{\nn}$ is also a countable regular-like cs-network, and enumerate $\mm=(M_e)_{e\in\om}$.
Given a $\overline{\dmm}$-name $p$, let $h(p)$ be an enumeration of all $e\in\om$ such that $\overline{M_{p\upto s}}\subseteq M_e$ for some $s$.
Clearly, $h$ is continuous.
We claim that $h$ realizes ${\rm id}:(\xx,\overline{\dmm})\to(\xx,\dmm)$, that is, if $p$ is a $\overline{\dmm}$-name of a point $x\in\xx$, then $h(p)$ is a $\dmm$-name of the same point $x$.
For any $n$, it is clear that $x\in M_{h(p)(n)}$ since there is $s$ such that $\overline{M_{p\upto s}}\subseteq M_{h(p)(n)}$, and $x\in\overline{M_{p\upto s}}$ for any $s$.
It remains to show that $\{M_{h(p)(n)}:n\in\om\}$ is a network at $x$.
Let $U$ be an open neighborhood of $x$.
Since $p$ is a $\overline{\dmm}$-name of $x$, $\{M_{p(n)}:n\in\om\}\subseteq\mm$ is a network at $x$.
Since $\mm$ is regular-like, $\{\overline{M_{p(n)}}:n\in\om\}$ is also a network at $x$.
Therefore, there is $n\in\om$ such that $x\in \overline{M_{p(n)}}\subseteq U$.
Clearly, $\overline{M_{p(n)}}\in\mm$, and in particular, $\overline{M_{p(n)}}\subseteq M_e\subseteq U$ for some $e$.
This means that $h(p)$ enumerates such $e$ at some stage, that is, $M_{h(p)(t)}\subseteq U$ for some $t$.
Hence, $h(p)$ is a $\dmm$-name of  $x$.
This verifies the claim.

Conversely, assume that $\xx$ has a countable cs-network $\nn$ such that the identity map ${\rm id}:(\xx,\overline{\dnn})\to(\xx,\dnn)$ is continuous.
Assume that $\mm\subseteq\nn$ is a strict network at $x\in\xx$.
Then, any enumeration $p$ of $\mm$ is a $\overline{\dnn}$-name of a point $x\in\xx$.
Suppose for the sake of contradiction that $\overline{\mm}=\{\overline{N}_{p\upto n}:n\in\om\}$ is not a network at $x$.
Then, there is an open neighborhood $U$ of $x$ such that $\overline{N}_{p\upto n}\not\subseteq U$ for any $n\in\om$.
Let $h$ be a realizer of ${\rm id}:(\xx,\overline{\dnn})\to(\xx,\dnn)$.
Then, $h(p)$ is a $\dnn$-name of $x$.
Therefore, there is $s\in\om$ such that $N_{h(p\upto s)}\subseteq U$.
Choose any $\overline{\dnn}$-name $q$ of $y\in \overline{N}_{p\upto s}\setminus U$.
Then, $(p\upto s)\fr q$ is also a $\overline{\dnn}$-name of $y$.
However, since $y\not\in U\supseteq N_{h(p\upto s)}$, $h((p\upto s)\fr q)$ cannot be a $\dnn$-name of $y$, which is a contradiction.
Hence, $\nn$ is regular-like.
\end{proof}

In particular, if $\nn$ is a countable cs-network of an $\aleph_0$-space $\xx$, then the identity map ${\rm id}:(\xx,\overline{\dnn})\to(\xx,\dnn)$ is continuous.

\begin{remark}
The proof of Theorem \ref{thm:regular-like-network} actually shows that $\nn$ is regular-like if and only if the identity map ${\rm id}:(X,\overline{\dnn})\to(X,\dnnp)$ is continuous.
Here, $\nnp=(M_e)_{e\in\om}$ is defined by $M_{2e}=N_e$ and $M_{2e+1}=\overline{N_e}$.
\end{remark}

We now turn look to second-countable spaces.
Recall that a basis can be thought of as a cs-network, and thus every countable basis $\beta$ also induces the representation $\overline{\beta}:=\overline{\delta_\beta}$.
Recall that the identity map ${\rm id}:(\xx,\beta)\to(\xx,\overline{\beta})$ is always computable.
The continuity of the inverse requires regularity (and thus, metrizability if the space is $T_0$, since a second-countable $T_0$ space is regular if and only if it is metrizable).

\begin{prop}\label{prop:reg-closurecontinuous}
Let $(\xx,\beta)$ be a second-countable space.
Then $\xx$ is regular if and only if the identity map ${\rm id}\colon(\xx,\overline{\beta})\to(\xx,\beta)$ is continuous.
\end{prop}

\begin{proof}
Assume that $\xx$ is regular.
We construct $h$ witnessing that ${\rm id}\colon(\xx,\overline{\beta})\to(\xx,\beta)$ is computable.
Let $p$ be a $\overline{\beta}$-name of a point $x\in\xx$.
If $p$ enumerates $e$, then $h(p)$ enumerates all $d$ such that $\overline{\beta_e}\subseteq\beta_d$.
We claim that $h(p)$ is a $\beta$-name of $x$.
Note that $x\in \beta_{h(p)(n)}$ for any $n$.
This is because $h(p)(n)=d$ only if there are $s\in\om$ such that $\overline{\beta_{p(s)}}\subseteq\beta_e$, and moreover $x\in\overline{\beta_{p(s)}}$ since $p$ is a $\overline{\beta}$-name of $x$.
Now, to see that $\{\beta_{h(p)(n)}:n\in\om\}$ is equal to $\nbase{x}$, assume that $x\in\beta_e$.
By regularity of $\xx$, there is $d$ such that $x\in\beta_{d}\subseteq\overline{\beta_d}\subseteq\beta_e$.
Since $p$ is a $\overline{\beta}$-name of $x$, there is $n$ such that $x\in\beta_{p(n)}\subseteq \beta_d$.
Since $\overline{\beta_{p(n)}}\subseteq\beta_d$, by definition, $h(p)$ enumerates $e$.
This verifies the claim.


Conversely, assume that the identity map ${\rm id}\colon(\xx,\overline{\beta})\to(\xx,\beta)$ is continuous.
Then, the proof of Theorem \ref{thm:regular-like-network} shows that $\beta$ is regular-like.
However, one can easily check that if a space has a regular-like basis, it is actually regular.
\end{proof}


It is worth noting that there is a {\em non-regular} Hausdorff space which has a countable regular-like cs-network.

\begin{example}\label{exa:Kleene-Kreisel}
The Kleene-Kreisel space $\mathbb{N}^{\mathbb{N}^\mathbb{N}}$ is a non-regular Hausdorff space which has a countable closed cs-network (hence, has a countable regular-like cs-network).
Here, the Kleene-Kreisel space consists of (total) continuous functions from Baire space $\om^\om$ to the natural numbers $\om$ endowed with the quotient topology given by the following map:
\[\delta(e\fr z)=f\iff \Phi_e^z=f,\]
where $\Phi_e^z$ is the $e$-th partial $z$-computable function from $\om^\om$ to $\om$, and the domain of $\delta$ is the set of all $e\fr z$ such that $\Phi_e^z$ is total, i.e., ${\rm dom}(\Phi_e^z)=\om^\om$.

Define $N_{\sigma,n}=\{f:(\forall x\succ\sigma)\;f(x)=n\}$ for each $\sigma\in\omega^{<\omega}$ and $n\in\omega$, and then $\nn=(N_{\sigma,n})$ forms a countable cs-network for the Kleene-Kreisel space.
We claim that $N_{\sigma,n}$ is closed.
One can see that $\delta(e\fr z)\not\in N_{\sigma,n}$ if and only if there are $s,k$, and $\tau$ such that $\Phi_e^{z\upto s}(\tau)\downarrow=k\not=n$, and $\tau$ is comparable with $\sigma$.
Therefore, $\delta^{-1}[N_{\sigma,n}]$ is closed in $\omega^\omega$.
Since the topology on $\mathbb{N}^{\mathbb{N}^\mathbb{N}}$ is given by the quotient map $\delta$, we conclude that $N_{\sigma,n}$ is closed in $\mathcal{K}$.
That is, $\nn$ is a closed network for $\mathcal{K}$.
Finally, Schr\"oder \cite{sch09} has shown that the Kleene-Kreisel space $\mathbb{N}^{\mathbb{N}^\mathbb{N}}$ is not regular.

Hence, by Theorem \ref{thm:regular-like-network}, ${\rm id}:(\mathbb{N}^{\mathbb{N}^\mathbb{N}},\overline{\dnn})\to(\mathbb{N}^{\mathbb{N}^\mathbb{N}},\dnn)$ is continuous.
Indeed, one can easily see that it is actually computable.
\end{example}

\subsection{Near quasi-minimality}

We now start to study computability w.r.t.\ closure representations.

\begin{definition}
Let $\xx=(X,\nn)$ be a topological space $X$ with a countable cs-network $\nn$.
We say that a point $x\in\xx$ is {\em nearly computable} if $\pt{x}{\overline{\dnn}}$ is computable.
\end{definition}

Clearly, every computable point is nearly computable.
If $\xx=[0,1]^\mathbb{N}$ or $\xx=\mathbb{N}^{\mathbb{N}^\mathbb{N}}$, then the converse is also true.
By Observation \ref{obs:closure-representation-1}, if $\xx$ is Hausdorff, there are only countably many nearly computable points.
If ${\rm id}:(\xx,\overline{\delta_\nn})\to(\xx,\delta_\nn)$ is computable, then near computability and computability coincide.
For instance, $f\in\N^{\N^\N}$ is nearly computable if and only if $f$ is computable.

%

\begin{definition}
Let $\xx=(X,\nn)$ and $\yy=(Y,\mm)$ be topological spaces with countable cs-networks.
Then, we say that a point $x\in\xx$ is {\em nearly $\yy$-quasi-minimal} if
\[(\forall y\in\yy)\;[\pt{y}{\yy}\reduce \pt{x}{\xx}\;\Longrightarrow\;y\mbox{ is nearly computable}].\]
\end{definition}

If computability and near computability are equivalent in a space $\yy$, then  so are $\yy$-quasi-minimality and near $\yy$-quasi-minimality.
For instance, near $[0,1]^\mathbb{N}$-quasi-minimality is equivalent to quasi-minimality, and near $\mathbb{N}^{\mathbb{N}^\mathbb{N}}$-quasi-minimality is equivalent to $\mathbb{N}^{\mathbb{N}^\mathbb{N}}$-quasi-minimality.

This notion and its variant will play important roles in Sections \ref{sec:T1-not-T2-qmin} and \ref{sec:T2-not-T25-qmin}.
But before that, we will show a few results on the closure representation.

\subsection{Borel extension topology}

Recall from Section \ref{sec:GHtopology} that the Gandy-Harrington topology is generated by $\Sigma^1_1$ sets.
Then, it is natural to study topologies generated by lightface Borel pointclasses.
Indeed, the topology generated by $\Sigma^1_1\cap\mathbf{\Pi}^0_\xi$ sets has played a prominent role in the proof of Louveau's separation theorem \cite{Lou80}.
The topology generated by $\Pi^0_1$ sets has also been used, for instance, by Miller \cite{MillerPhDthesis} and Monin \cite[Section 3.1]{Monin16}.

In this section, we discuss topologies generated by collections of $\Sigma^0_\alpha$ sets.
Our first motivation was, for instance, to understand the degree-theoretic behavior of $2^\om$ equipped with the standard Cantor topology plus Martin-L\"of conull sets.
Here, we say that a set $A\subseteq 2^\om$ is Martin-L\"of null if there is a computable sequence $(U_n)_{n\in\om}$ of c.e.\ open sets such that $A\subseteq U_n$ and $\mu(U_n)\leq 2^{-n}$ for any $n\in\om$, and a set is Martin-L\"of conull if its complement is Martin-L\"of null.

However, in contrast to the Gandy-Harrington topology, we will see that such a space is uninteresting from the perspective of enumeration degrees because the degree structure is exactly the same as the total degrees relative to some oracle.
By Kihara-Pauly \cite{KP}, the latter property is equivalent to saying that such a space is {\em $\sigma$-metrizable}, that is, it is written as the union of countably many metrizable subspaces (see \cite{Gru13}).
Note that the Gandy-Harrington space $(\om^\om)_{GH}$ is not $\sigma$-metrizable by relativizing Theorem \ref{thm:GH-non-continuous} (see also Theorem \ref{thm:Gandy-Harrington-closed-neighborhood}) and by Kihara-Pauly \cite{KP}.

By Proposition \ref{prop:reg-closurecontinuous}, we know that a represented cb$_0$ space $(\xx,\beta)$ is metrizable if and only if the identity map is an isomorphism between $(\xx,\beta)$ and $(\xx,\overline{\beta})$.
However, for nonmetrizable $(\xx,\beta)$, this proposition does not ensure that $(\xx,\beta)$ and $(\xx,\overline{\beta})$ are not isomorphic.
By using a Borel extension topology, we will see the following.

\begin{prop}\label{prop:filter-topology}
There is a represented, submetrizable, $\sigma$-metrizable, cb$_0$ space $(\xx,\beta)$ such that $(\xx,\overline{\beta})$ is not isomorphic to $(\xx,\beta)$.
\end{prop}

A {\em filter $\mathcal{F}$ on $\xx$} is a nonempty collection of nonempty subsets of $\xx$ such that $A\in\mathcal{F}$ and $A\subseteq B$ implies $B\in\mathcal{F}$ and that $A,B\in\mathcal{F}$ implies $A\cap B\in\mathcal{F}$.
A {\em generator} of $\mathcal{F}$ is a subcollection $G$ of $\mathcal{F}$ such that $\xx\in G$ and for every $A\in\mathcal{F}$ there is $B\in G$ such that $B\subseteq A$.
If $(\xx,\tau)$ is a topological space, we say that a filter $\mathcal{F}$ on $\xx$ is {\em $\tau$-consistent} if every $A\in\mathcal{F}$ is $\tau$-dense, and $\mathcal{F}$ is {\em $\tau$-nontrivial} if there is $A\in\mathcal{F}$ whose complement is $\tau$-dense in some $\tau$-open set.
For instance, the collection of all Martin-L\"of conull subsets of $2^\omega$ forms a non-trivial filter, and the complements of Martin-L\"of tests form a countable generator of this filter.

Let $G$ be a countable generator of a filter $\mathcal{F}$ on a second-countable space $(\xx,\tau)$.
Then consider the new topology $\tau_G$ on $\xx$ generated by the following set:
\[\{A\cap U:A\in G\mbox{ and }U\in\tau\}.\]

Note that the new topology $\tau_G$ is second-countable since $G$ is countable, and that $\tau_G$ is finer than $\tau$ since $\xx\in G$.
The latter implies that $(\xx,\tau_G)$ is submetrizable whenever $(\xx,\tau)$ is submetrizable.
If $(\xx,\tau)$ is represented by $\beta$ and if $G$ is enumerated as $(G_n)_{n\in\om}$, then $(\xx,\tau_G)$ is represented by $\beta^G_n=G_n\cap\beta_n$.

\begin{example}
All of the following examples are countable generators of a $\tau$-consistent $\tau$-nontrivial filter on a separable metrizable space.
\begin{enumerate}
\item The generator $G=\{\mathbb{Q},\mathbb{R}\}$ of a principal filter on the Euclidean line $\mathbb{R}$ yields the so-called indiscrete rational extension of $\mathbb{R}$ (see \cite[II.66]{CTopBook}).
The generator $G=\{\mathbb{R}\setminus\mathbb{Q},\mathbb{R}\}$ of a principal filter on the Euclidean line $\mathbb{R}$ yields the so-called indiscrete irrational extension of $\mathbb{R}$ (see \cite[II.67]{CTopBook}).
\item Let ${\sf MLR}$ be the complements of (the union of) Martin-L\"of tests on $2^\omega$.
Then, ${\sf MLR}$ is a countable generator of the filter consisting of all measure $1$ subsets of $2^\omega$.
This yields the space ${\xx}_{\sf MLR}:=(2^\omega,\tau_{\sf MLR})$ where $\tau$ is the usual Cantor topology.
\end{enumerate}
\end{example}

\begin{lemma}\label{lem:filter-topology1}
Let $G$ be a countable generator of a $\tau$-consistent filter on $\xx$.
Then, the identity map yields a homeomorphism between $(\xx,\overline{\beta^G})$ and $(\xx,\overline{\beta})$.
\end{lemma}

\begin{proof}
Let $A\cap U$ and $B\cap V$ be $\tau_G$-open sets.
The assumption $A,B\in G$ implies $A\cap B\in\mathcal{F}$, and therefore $A\cap B$ is $\tau$-dense since $\mathcal{F}$ is $\tau$-nontrivial.
Therefore, if $A\cap U$ and $B\cap V$ are disjoint, then $U$ and $V$ must be disjoint.
Consequently, the $\tau_G$-closure of $A\cap U$ is exactly the $\tau$-closure of $U$.
\end{proof}

\begin{lemma}\label{lem:filter-topology2}
Let $G$ be a countable generator of a $\tau$-nontrivial filter on $\xx$.
Then, $(\xx,\tau_G)$ is not metrizable.
\end{lemma}

\begin{proof}
It suffices to show that $(\xx,\tau_G)$ is not regular.
Since $\mathcal{F}$ is $\tau$-nontrivial, there are $D\in\mathcal{F}$ and $T\in\tau$ such that $T\setminus D$ is $\tau$-dense in $T$.
Then, there exists $x\in D\cap T$ since $D$ is $\tau$-dense.
We claim that there is no disjoint pair of open sets separating $x$ and $\cmp{D}$.
Suppose that $x\in A\cap U$ and $\cmp{D}\subseteq B\cap V$.
As in the proof of Lemma \ref{lem:filter-topology1}, if $A\cap U$ and $B\cap V$ are disjoint, so are $U$ and $V$.
Since $T\setminus D$ is $\tau$-dense in $T$, $V$ must include $T$.
However, $x\in T$ and therefore $U\cap V\not=\emptyset$.
\end{proof}

For a computable ordinal $\alpha$, we say that a generator $G$ is {\em $\Sigma^0_\alpha$-generated} if every element of $G$ is $\Sigma^0_\alpha$ in $\tau$, and the union of a $\tau_G$-open set and a $\Sigma^0_\alpha$-in-$\tau$ set is $\tau_G$-open.
For instance, the complements of Martin-L\"of tests is $\Sigma^0_2$-generated.

If $(\xx,\tau)$ is Polish, the topology $\tau_\alpha$ generated by $\Sigma^0_{\alpha}$ sets in $\tau$ yields a zero-dimensional metrizable topology whenever $\alpha>0$.
This is because the collection of all $\Pi^0_\beta$-sets in $\tau$ for $\beta<\alpha$ forms a basis of the topology $\tau_\alpha$, and each $\Pi^0_\beta$ set is $\tau_\alpha$-clopen.
Therefore, $\tau_\alpha$ has a basis consisting of clopen sets $(C_n)_{n\in\omega}$ and then the metric $d_\alpha$ on $(\xx,\tau_\alpha)$ is given by $d_\alpha(x,y)=2^{-n}$ where $n$ is the least index $n$ such that $C_n$ separates $x$ and $y$.

\begin{lemma}\label{lem:filter-topology}
Let $(\xx,\tau)$ be a Polish space.
If $G$ is $\Sigma^0_\alpha$-generated, then $(\xx,\tau_G)$ is $\sigma$-metrizable.
\end{lemma}

\begin{proof}
We say that $x$ is {\em $G$-quasi-generic} if $x\in A$ for any $A\in G$.
If $x$ is $G$-quasi-generic, $x\in A\cap U$ if and only if $x\in U$.
Therefore, $(\xx\upto\mathcal{R}_G,\tau_G)$ is homeomorphic to $(\xx\upto\mathcal{R}_G,\tau)$ where $\mathcal{R}_G$ is the set of all $G$-quasi-generic points.

If $x$ is not $G$-quasi-generic, there is $A\in G$ such that $x\not\in A$.
Let $\mathcal{S}_e$ be the set of all $x\in\xx$ such that $x\not\in A_e$ where $A_e$ is the $e$th element of $G$.
We claim that $(\xx\upto\mathcal{S}_e,\tau_G)$ is homeomorphic to $(\xx\upto\mathcal{S}_e,\tau_\alpha)$ for any $e$.
It is clear that the identity map ${\rm id}:(\xx,\tau_\alpha)\to(\xx,\tau_G)$ is continuous, since $G$ is $\Sigma^0_\alpha$-generated, which implies that $\tau_\alpha$ is finer than $\tau_G$.
For any $x\in\mathcal{S}_e$, given $\Sigma^0_\alpha$ set $S$, $x\in A_e$ if and only if $x\in A_e\cup S$.
Since $G$ is $\Sigma^0_\alpha$-generated, we always have $A_e\cup S\in G$ and hence $A_e\cup S$ is $\tau_G$-open.
Therefore, the identity map $(\xx\upto\mathcal{S}_e,\tau_G)\to(\xx\upto\mathcal{S}_e,\tau_\alpha)$ is also continuous.
\end{proof}

\begin{proof}[Proof of Proposition \ref{prop:filter-topology}]
Let $(\xx,\tau)$ be a separable metrizable space and $G$ be a countable $\Sigma^0_\alpha$-generated generator of a $\tau$-consistent $\tau$-nontrivial filter on $\xx$.
For instance, let $(\xx,\tau)$ be Cantor space, and $G$ be the collection of complements of Martin-L\"of tests.
Then, $(\xx,\tau_G)$ is submetrizable, and $\sigma$-metrizable by Lemma \ref{lem:filter-topology}.
Given a representation of $\xx$, consider $\xx_G=(\xx,\beta^G)$.
Then $(\xx,\overline{\beta^G})$ is isomorphic to $(\xx,\overline{\beta})$ by Lemma \ref{lem:filter-topology1}, and this is isomorphic to $(\xx,\beta)$ since $\beta$ is metrizable.
Then, since $\xx_G$ is non-metrizable by Lemma \ref{lem:filter-topology2}, we know that $(\xx,\beta^G)$ is not isomorphic to $(\xx,\overline{\beta^G})$.
\end{proof}

\section{Proofs for Section \ref{sec:further}}\label{sec:quasiminimaltec}
\label{sec:further-proofs}

\subsection{$T_0$-degrees which are not $T_1$} In this section we examine properties of the lower topology and $T_1$-quasi-minimality in this space.

\subsubsection{Quasi-minimality}\label{sec:T_0-deg:quasi-minimal}

We first show basic degree-theoretic properties of the lower reals $\mathbb{R}_<$.
For definitions, see Section \ref{section:t0-nt1-quasiminimality-statements}.

\lemproof{propdichotomy}{
Assume that $\nbaseb{2^\om}{x}\leq_e\nbaseb{<}{y}\oplus\nbaseb{\xx}{z}$.
Then, there is a c.e.\ set $\Phi\subseteq\omega\times 2\times\mathbb{Q}\times\omega$ such that for any $n\in\om$ and $i<2$,
\[x(n)=i\iff(\exists p\in\mathbb{Q})(\exists e\in\om)\;[p<y,\;e\in\nbase{z}\mbox{, and }\langle n,i,p,e\rangle\in\Phi].\]

Suppose that there is $\varepsilon>0$ such that for any $p<y+\varepsilon$ and $e\in\nbase{z}$, $\langle p,e,n,i\rangle\in\Phi$ implies $x(n)=i$.
In this case, fix a rational $q\in\mathbb{Q}$ such that $y<q<y+\varepsilon$.
Then,
\[x(n)=i\iff(\exists p<q)(\exists e\in\om)\;[e\in\nbase{z}\mbox{, and }\langle n,i,q,e\rangle\in\Phi].\]
This shows $\nbaseb{2^\om}{x}\leq_e\nbaseb{\xx}{z}$.
Hereafter we assume that $y$ is irrational; otherwise $y$ is computable and thus $\nbaseb{<}{-y}\leq_e\nbaseb{\xx}{z}$ trivially holds.

Otherwise, for all $\varepsilon>0$, there are $p<y+\varepsilon$ and $e\in\nbase{z}$ such that $\langle n,i,p,e\rangle\in\Phi$ but $x(n)\not=i$.
We claim that $y<p$ if and only if there are $n,i,p,d,e$ such that
\[q<p\ \ \&\ \ d,e\in\nbase{z}\ \ \&\ \ \langle n,1-i,q,d\rangle\in\Phi\ \ \&\ \ \langle n,i,p,e\rangle\in\Phi.\]

If $y<p$ is not true, we have $p<y$ since $y$ is irrational, and then $q<p$ implies $q<p<y$.
Since $\nbaseb{2^\om}{x}\leq_e\nbaseb{<}{y}\oplus\nbaseb{\xx}{z}$, for any $d,e\in\nbase{z}$, whenever $\langle n,i,q,d\rangle$ and $\langle n,j,p,e\rangle$ are enumerated into $\Phi$, we must have $i=j$.
If $y<p$, then by our assumption, there are $\hat{p}<p$ and $e\in\nbase{z}$ such that $\langle n,i,\hat{p},e\rangle\in\Phi$ but $x(n)\not=i$.
By monotonicity, one can assume that $\hat{p}=p$.
Since $\nbaseb{2^\om}{x}\leq_e\nbaseb{<}{y}\oplus\nbaseb{\xx}{z}$, there also exist $q<y$ and $d\in\nbase{z}$ such that $\langle n,j,q,d\rangle\in\Phi$ and $x(n)=j$, i.e., $j=1-i$.
This verifies the claim.

Now note that $y<p$ if and only if $-p\in\nbaseb{<}{-y}$.
Thus, by the above claim, we conclude $\nbaseb{<}{-y}\leq_e\nbase{z}$.
}

\propproof{prop:right-ce}{Fix $x\in\mathbb{R}_<$.
If $x$ is rational, then $x$ satisfies all conditions (1)--(4).
We now assume that $x$ is irrational.
Clearly, the conditions (2) and (3) are equivalent.
We show the equivalence of (1) and (3).
If $x$ is right-c.e., $\nbaseb{\mathbb{R}_<}{-x}$ is c.e., and thus $\cmp{\nbaseb{<}{x}}$ is c.e.
Hence, $\nbaseb{<}{x}\oplus\cmp{\nbaseb{<}{x}}\equiv_e \nbaseb{<}{x}$.
Thus, $\nbaseb{<}{x}$ is total.
Conversely, if $A\oplus\cmp{A}\equiv_e\nbaseb{<}{x}$ for a set $A\subseteq\mathbb{N}$, by Lemma \ref{propdichotomy}, we have $A\oplus\cmp{A}$ is c.e.\ (thus $\nbaseb{<}{x}$ is c.e.)\ or $\nbaseb{<}{-x}$ is c.e.
In other words, $x$ is either left- or right-c.e.
The equivalence of (3) and (4) follows from Lemma \ref{propdichotomy}.}

Here, we review the definition of function $n$-genericity as introduced by Copestake \cite{copestake}.
Fix a new symbol $\bot\not\in\omega$, and assume that $\omega\cup\{\bot\}$ is endowed with the discrete topology, that is, fix an effective bijection between $\om$ and $\om\cup\{\bot\}$.
Then, $(\omega\cup\{\bot\})^\omega$ is effectively homeomorphic to Baire space $\omega^\omega$.
Thus, the $(\omega\cup\{\bot\})^\omega$-degrees are the total degrees.
Indeed, for any $g\in(\omega\cup\{\bot\})^\omega$,
\[{\rm Graph}(g)\equiv_e\{\langle n,m+1\rangle:g(n)=m\}\cup\{\langle n,0\rangle:g(n)=\bot\}.\]

From $g\in(\omega\cup\{\bot\})^\omega$ we get a partial function $\hat{g}:\subseteq\om\to\om$ by interpreting $\bot$ as ``undefined''.
Then,
\[{\rm Graph}(\hat{g})=\{\langle n,m\rangle\in\om^2:g(n)\not=\bot\mbox{, and }g(n)=m\}.\]


\begin{definition}\label{def:enumeration-n-generic}
We say that $G\subseteq\om$ is {\em function $n$-generic} if it is of the form ${\rm Graph}(\hat{g})$ for some $n$-generic point $f$ in the Baire space $(\omega\cup\{\bot\})^\omega$.
\end{definition}

As shown by Copestake, every function $n$-generic computes an enumeration $n$-generic, but the notions are distinct.
Note that ${\rm Graph}(\hat{g})\leq_e{\rm Graph}(g)$ is always true, but ${\rm Graph}(g)\leq_e{\rm Graph}(\hat{g})$ is not necessarily true.

\propproof{prop:lower-unbounding}{Suppose that ${\rm Graph}(g)\leq_e\nbaseb{<}{x}$ for a partial function $\hat{g}:\subseteq\om\to\om$.
Then, there is a c.e.~set $\Phi$ such that $\hat{g}(n)\downarrow$ if and only if $p<x$ for some $\langle n,p\rangle\in\Phi$.
For each $n$ we let $\theta(n)=\inf \{q\in\mathbb{Q}:(n,q)\in \Phi\}$.
There is an increasing sequence $\{n_k\}_{k\in\omega}$ such that $\{\theta(n_k)\}_{k\in\omega}$ is a monotonic sequence of reals.
Note that the relation $\theta(n)\leq\theta(m)$ is $\Pi^0_2$.
Therefore, such a sequence $\{n_k\}_{k\in\omega}$ can be found computably in $\emptyset''$.
Now it is easy to check that $\hat{g}(n_k)$ is either defined for finitely many or for co-finitely many $k\in\omega$.
Consider $S^0_\ell=\{f\in(\omega\cup\{\bot\})^\omega:(\exists k>\ell)\;f(n_k)=\bot\}$ and $S^1_\ell=\{f\in(\omega\cup\{\bot\})^\omega:(\exists k>\ell)\;f(n_k)\in\omega\}$.
Note that $g$ is not contained in $S^0_\ell$ or $S^1_\ell$ for any sufficiently large $\ell$.
However, $S^0_\ell$ and $S^1_\ell$ are dense $\emptyset''$-open sets with respect to the Baire topology on $(\omega\cup\{\bot\})^\omega$.
Consequently, $\hat{g}$ is not function $2$-generic.}

\propproof{qminimalbound}{
To prove Proposition \ref{qminimalbound}, we need the following lemma:

\begin{lemma}\label{lem:rightrenotleftre}
Given non-computable c.e.\ sets $A,B\subseteq\om$, there are left-c.e.~reals $z\leq_TA$ and $y\leq_TB$ such that $y-z$ is neither left- nor right-c.e.
\end{lemma}

\begin{proof}
Given non-computable c.e.\ sets $A_0,A_1\subseteq\om$, we construct a c.e.~reals $z^i=\Gamma_i^{A_i}$ with $A_i$-use $\gamma_i(n)=n$.
The $(e,i)$-th strategy $\mathcal{R}_{e,i}$ tries to diagonalize $W_e={\rm Left}(z^{1-i}-z^i)$, where ${\rm Left}(z)=\{q\in\mathbb{Q}:q\leq z\}$.
We describe the action of $\mathcal{R}_{e,i}$ at stage $s$.
If this is the first action of $\mathcal{R}_{e,i}$ after its initialization, choose a large number $n_{e,i}$ which is bigger than all numbers $+2$ mentioned in previous stages $<s$.
Then, put $z^{1-i}_{s+1}=z^{1-i}_s+2^{-n_{e,i}}$.
Wait for ${\rm Left}(z^{1-i}-z^i)\upharpoonright n_{e,i}+1\subseteq W_e$.
Here, ${\rm Left}(z^{1-i}-z^i)\upharpoonright n$ is defined as ${\rm Left}(z)$ for a dyadic rational $z=0.(\sigma\upharpoonright n)000\dots$, where $\sigma$ is a unique binary string satisfying $z^{1-i}-z^i=0.\sigma 000\dots$ (note that such $\sigma$ exists since our strategy ensures that $z^i$ and $z^{1-i}$ are dyadic rationals).
If it happens, wait for the change of $A_i\upharpoonright n_{e,i}+1$, choose a fresh large number $n_{e,i}'$, injure all lower priority strategies, and go back to the first step with $n_{e,i}'$.
If we see the change of $A_i\upharpoonright n_{e,i}+1$ at stage $t\geq s$, the strategy $\mathcal{R}_{e,i}$ acts by putting $z^{i}_{t+1}=z^{i}_t+2^{-n_{e,i}}$, and stop the action of $\mathcal{R}_{e,i}$.

If ${\rm Left}(z^{1-i}-z^i)\upharpoonright n_{e,i}+1\subseteq W_e$ does not happen, then the requirement is clearly fulfilled.
If it happens with an infinite increasing sequence $(n_{e,i}^k)_{k\in\omega}$ since $A_i\upharpoonright n_{e,i}^k+1$ never changes for all $k\in\omega$, then $A_i$ is computable.
It contradicts the choice of $A_i$.
Hence, $\mathcal{R}_{e,i}$ acts with $n_{e,i}^k$ for some $k$.
If it is never injured after the action, then the requirement is fulfilled sine the sum of weights added to $z^{1-i}$ by all lower priority strategies is less than $2^{-n_{e,i}}$.
Consequently, $z^{1-i}-z_i$ is neither left- nor right-c.e.
\end{proof}

It suffices to show that if $r\in\mathbb{R}_<$ is right-c.e., but not left-c.e., then $r$ bounds a quasi-minimal $e$-degree.
As mentioned in the paragraph below Proposition \ref{prop:right-ce}, the $\mathbb{R}_<$-degrees of right-c.e.\ reals are exactly the c.e.\ Turing degrees.
Therefore, by Lemma \ref{lem:rightrenotleftre}, for any right-c.e.~real $r\in\mathbb{R}$ if $r$ is not left-c.e., we have right-c.e.~reals $y,z\in\mathbb{R}$ such that $y-z$ is neither left- nor right-c.e., and $\nbase{y},\nbase{z}\leq_e\nbaseb{<}{r}$.
Put $x=y-z$.
Then, $\nbase{x}\leq_e\nbaseb{<}{r}$.
By Lemma \ref{propdichotomy}, $\nbaseb{<}{x}$ is quasi-minimal.}

\subsubsection{Degree Structure}

We examine the degree structure of the lower reals $\mathbb{R}_<$.
For definitions, see Section \ref{sec:4-1-2}.

\propproof{semirecjumpinversion}{
It suffices to find a real $x\in\mathbb{R}_<$ such that $\nbaseb{<}{x}$ is quasi-minimal and $J_\xx(x)\equiv_T C$.
First note that every c.e.~open set in $\mathbb{R}_<$ is of the form $\bigcup_{q\in W_e}\{y:y>q\}$ for a c.e.~set $W_e\subseteq\mathbb{Q}$.
Clearly, $r_e=\inf W_e$ is right-c.e.
Then, $J_{\mathbb{R}_<}(x)=\{e\in\omega:r_e<x\}$.
Assume that $x$ is not right-c.e.
Then, either $x<r_e$ or $x>r_e$ holds.
Since $r_e\leq_T\emptyset'$ in $2^\omega$, we have $J_{\mathbb{R}_<}(x)\equiv_Tx\oplus\emptyset'$ (i.e., $\nbaseb{2^\om}{J_{\mathbb{R}_<}(x)}\equiv_e\nbaseb{\mathbb{R}}{x}\oplus K\oplus\cmp{K}$).
Now, by the Friedberg jump inversion theorem in $2^\omega$, there is a $1$-generic real $z\in\mathbb{R}$ such that $z'\equiv_Tz\oplus\emptyset'\equiv_TC$.
By $1$-genericity, $z$ is neither left-c.e.~nor right-c.e.
By Lemma \ref{propdichotomy}, $\nbaseb{<}{z}$ is semirecursive and quasi-minimal.
Moreover, we have $J_{\mathbb{R}_<}(z)\equiv_Tz\oplus\emptyset'\equiv_TC$ since $z$ is not right-c.e.
Note that $J_{\mathbb{R}_<}(z)$ is $e$-equivalent to ${\rm EJ}(\nbaseb{<}{z})$ as mentioned above.
Consequently, $A=\nbaseb{<}{z}$ satisfies the desired property.
}

\propproof{prop:rlowernotupsemi}{
Assume that $x$ is not $\Delta^0_2$, and that $\nbaseb{<}{x}\oplus\nbaseb{<}{-x}\leq_e\nbaseb{<}{y}$. Then on the one hand, $\nbaseb{<}{y}$ is not quasi-minimal. On the other hand, $y$ is not $\Delta^0_2$ and hence neither left- nor right-c.e. But this contradicts Lemma \ref{propdichotomy}.}

\begin{lemma}\label{prop:minimalpairrelativized}
Let $A,B\subseteq\om$ be c.e.~sets, and $y\in\mathbb{R}_<$.
Suppose that $\nbaseb{<}{y}\leq_e A\oplus\cmp{A}$ and $\nbaseb{<}{y}\leq_e B\oplus\cmp{B}$.
Then there exists a total function $h\in\omega^\omega$ such that $h\leq_T A,B$ and $y$ is left-c.e.~relative to $h$.
\end{lemma}

\begin{proof}
Suppose that $\nbaseb{<}{y}\leq_e A\oplus\cmp{A}$ and $\nbaseb{<}{y}\leq_e B\oplus\cmp{B}$.
Then, there are computable functions $W,V:2^\omega\to\mathbb{Q}^\omega$ such that $y=\sup W^A=\sup V^B$.
Let $\alpha_s=\sup W^A[s]$ and $\beta_s=\sup V^B[s]$, where $W^A[s]=(W(A[s])(n))_{n<s}$ and $V^B[s]=(V(B[s])(n))_{n<s}$.
Define $s_{0}$ to be the least stage such that for every $t\geq s_0$ we have $\alpha_t\geq q_0$ or $\beta_t\geq q_0$, where $q_0=\min\{\alpha_{s_0},\beta_{s_0}\}$.
Similarly define $s_{n+1}>s_n$ and $q_{n+1}>q_n$. Let $h(n)=s_n$.
Clearly $h$ is total.
We claim that $h$ is computable from both $A$ and $B$.
For each stage $u>s_n$, compute $q[u]=\min\{\alpha_{u},\beta_{u}\}$, and use $A$ to compute stage $v\geq u$ such that $\alpha_t\geq q[u]$ for all $t\geq v$.
Then, check whether $\beta_t\geq q[u]$ for each $t$ such that $u\leq t\leq v$.
Clearly, for the least such $u>s_n$ with $q[u]>q_n$, we have $u=s_{n+1}$.
By the same argument, $B$ can also compute $h$.
Now given $h$ we can recover the increasing sequence $q_n$ with limit $y=\sup W^A=\sup V^B$.
\end{proof}

\propproof{thm:non-lower-semilattice}{
Lachlan and Yates (see \cite[Corollary IX.3.3]{SoareBook}) proved the existence of c.e.~sets $A,B$ such that for any (not necessarily c.e.) set $H\leq_T A,B$ there exists a c.e.~set $C$ such that $H<_T C\leq_T A,B$.
We claim that for $x=\sum_{n\in \cmp{A}} 2^{-n}$ and $y=\sum_{n\in \cmp{B}} 2^{-n}$, $\nbaseb{<}{x}$ and $\nbaseb{<}{y}$ have no greatest lower bound in $\mathcal{D}_{\mathbb{R}_<}$.
Let $z\in\mathbb{R}_<$ be any point such that $\nbaseb{<}{z}\leq_e\nbaseb{<}{x},\nbaseb{<}{y}$.
Since $A$ and $B$ are c.e., $\nbaseb{<}{x}\equiv_e A\oplus\cmp{A}$ and $\nbaseb{<}{y}\equiv_e B\oplus\cmp{B}$.
Thus, we can apply Lemma \ref{prop:minimalpairrelativized} to obtain some total function $h\leq_T A,B$ such that $\nbaseb{<}{z}\leq_e\nbaseb{\om^\om}{h}$.
Since $A$ and $B$ have no greatest lower-bound in the c.e.\ Turing degrees $\mathcal{R}_T$ and $h$ is total, there is a c.e.\ set $C$ such that $h<_TC<_TA,B$.
Therefore, if we set $w=\sum_{n\in\cmp{C}}2^{-n}$, then $\nbaseb{<}{w}\equiv_eC\oplus\cmp{C}$; hence,
\[\nbaseb{<}{z}<_e\nbaseb{<}{w}<_e\nbaseb{<}{x},\nbaseb{<}{y}.\]

Consequently, there is no $z$ such that $\nbaseb{<}{z}$ can be a greatest lower bound of $\nbaseb{<}{x}$ and $\nbaseb{<}{y}$.}

\thmproof{thm:lower-minimal}{
We show that given any $X\subseteq \omega$ not c.e., there is some $y\in\mathbb{R}_<$ such that $y$ is not left-c.e., and $\nbaseb{<}{y}\leq_e X$.

If $0.X=\sum _{n\in X} 2^{-n}$ is not left-c.e., consider the real $y=0.X$.
Otherwise assume that $0.X$ is left-c.e. In this case, $X$ is $\Delta^0_2$. Furthermore we may fix an approximation $X_s$ of $X$ such that if $n$ leaves $X_s$ then there is some $m<n$ which is enumerated in $X_s$ at the same time.

Let $(\alpha_e)_{e\in\omega}$ be an effective enumeration of all left-c.e.~reals.
We wish to construct a some left c.e.~real $y$ relative to $X\in\mathbb{S}^\omega$ such that $y\neq \alpha_e$ for all $e\in\omega$.
In other words, we will construct a c.e.~set $\Phi$ of pairs of rationals and finite subsets of $\omega$ such that $y=\sup\{q:(\exists D\subseteq X)\;(q,D)\in\Phi\}$.
At stage $s$, the $e$-th strategy for $e\leq s$ is eligible to act with parameters $y_{e,s}$, $m_{e,s}$, $s^-_{e,s}$, $s_{e,s}$ and $p_{e,s}$.
Here, $y_{-1,s}=m_{-1,s}=t_{-1,s}=0$ for all $s$.

At substage $e\leq s$, the $e$-th strategy acts as follows:
\begin{enumerate}
\item If this is the first action of the $e$-th strategy after its initialization, set
\begin{align*}
m_{e,s+1}&=m_{e-1,s+1}+2^{p_{e-1,s+1}}+1,\\
y_{e,s+1}&=y_{e-1,s+1}+2^{-m_{e,s}-1}.
\end{align*}
Put $p_{e,s+1}=p_{e-1,s+1}$, and $s_{e,s+1}=s+1$.
Initialize all lower priority strategies, and go to stage $s+1$.
Otherwise, go to step (2).
\item If the $e$-th strategy is active, go to step (3).
If the $e$-th strategy is inactive because of the previous action in (3b), go to step (5).
\item Check whether $X_{s_{e,s}}\cap[p_{e-1,s},p_{e,s})\subseteq X_s$.
\begin{enumerate}
\item If yes, and go to step (4).
\item If no with $X_{s_{e,s}^-}\cap[p_{e-1,s},p_{e,s}-1)\not\subseteq X_s$, define $y_{e,s+1}$ as the current value of $\Phi(X_s)$ at stage $s$, and put $s^-_{e,s+1}=s^-_{e,s}$, $s_{e,s+1}=s_{e,s}$ and $p_{e,s+1}=p_{e,s}$.
This strategy is shifted into the inactive state.
Initialize all lower priority strategies, and go to stage $s+1$.
\item Otherwise, define
\[y_{e,s+1}=y_{e,s}+2^{-m_{e,s}-1},\]
and maintain the computation by enumerating $(y_{e,s+1},X_{s_{e,s+1}}\upharpoonright p_{e,s})$ into $\Phi$.
Put $s^-_{e,s+1}=s^-_{e,s}$, $s_{e,s+1}=s$ and $p_{e,s+1}=p_{e,s}$.
Initialize all lower priority strategies, and go to stage $s+1$.
\end{enumerate}
\item Check whether $\alpha_{e,s}\geq y_{e,s}-2^{-m_{e,s}-2}$.
\begin{enumerate}
\item If yes, put $p_{e,s+1}=p_{e,s}+1$, and
\begin{align*}
m_{e,s+1}&=m_{e-1,s+1}+2^{p_{e,s+1}}+1,\\
y_{e,s+1}&=y_{e,s}+2^{-m_{e,s+1}-1}.
\end{align*}
Moreover, put $s^-_{e,s+1}=s_{e,s}$, $s_{e,s+1}=s$ and $p_{e,s+1}=p_{e,s}+1$.
Then, enumerate $(y_{e,s+1},X_{s_{e,s+1}}\upharpoonright p_{e,s+1})$ into $\Phi$.
Initialize all lower priority strategies, and go to stage $s+1$.
\item If no, go to substrategy $e+1$ unless $e=s$.
If $e=s$, then go to stage $s+1$.
\end{enumerate}
\item Check whether $X_{s-1}\cap[p_{e-1,s},p_{e,s})\subseteq X_s$.
\begin{enumerate}
\item If yes, do the action described in (4b).
\item If no, check whether $\Phi(X_s)<y_s$.
If yes, the $e$-th strategy keeps on being in the inactive state, and do the action described in (3b).
If no, we must have stage $t<s$ such that $X_t\subseteq X_s$ and the $e$-th strategy is active at stage $t$.
We recover the parameters for the last such stage $t<s$, and the $e$-th strategy is shifted into the active state.
\end{enumerate}
\end{enumerate}

Suppose that the $e$-th strategy is never injured after stage $t$.
In this case, any parameter for $e'<e$ never changes after stage $t$.
If (4a) happens infinitely often, then (3b) never happens after $t$.
Hence, we have a computable $\subseteq$-increasing sequence $(X_{s^-_{e,s}}\cap[p_{e-1,s},p_{e,s}-1))_{s\geq t}$ converging to $X\cap\{n:n\geq p_{e-1,s}\}$ since $p_{e-1,s}=p_{e-1,t}$ for all $s\geq t$, and moreover, $s_{e,s}^-$ and $p_{e,s}$ are nondecreasing, and tend to infinity.
Therefore, $X$ is c.e.; however, it is impossible by our assumption.
Thus, (4a) never happens after some stage $t'\geq t$.
Then, $p_{e,s}=p_{e,t'}$, $m_{e,s}=m_{e,t'}$, $s^-_{e,s}=s^-_{e,t'}$ and $\alpha_{e,s}\leq {y}_{e,s}-2^{-m_{e,s}-2}$ for all $s\geq t'$.
Suppose that (3b) never happens.
First assume that $X_{s_{e,s}}\cap[p_{e-1,s},p_{e,s})\subseteq X_s$.
If $X_{s_{e,s}}\upharpoonright p_{e-1,s}\not\subseteq X_s$, then a higher priority strategy acts with (5b), which is impossible by our assumption.
Thus, $X_{s_{e,s}}\upharpoonright p_{e,s}\subseteq X_s$.
Hence, our computation is maintained, that is, $\Phi(X_{s_{e,s}}\upharpoonright p_{e,s})\geq {y}_{e,s}$.
If $X_{s_{e,s}}\cap[p_{e-1,s},p_{e,s})\not\subseteq X_s$, the $e$-th strategy maintains the computation at (3c) since (3b) never happens.
By monotonicity of our approximation of $X$, (3c) can happen only finitely often after stage $t'$.
Fix stage $t''\geq t'$ such that (3c) never happens after stage $t''$.
Then, $y_{e,s}=y_{e,t''}$ for all $s\geq t''$.
Consequently, $\Phi(X)\geq \Phi(X_{s_{e,t''}}\upharpoonright p_{e,t''})\geq y_{e,t''}>{y}_{e,t''}-2^{-m_{e,t''}-2}\geq\alpha_e$.

If (3b) happens, $X_{s^-_{e,s}}\cap [p_{e-1,s},p_{e,s}-1)\not\subseteq X_s$ at some stage $s\geq t$.
Let $u\geq t$ be the least such stage, and let $u'<u$ be the last stage when (4a) happened.
Then, $\alpha_e\geq y_{e,u'}-2^{-m_{e,u'}-2}$.
Note that $s^-_{e,u}<u'$ is the last stage when either (3c) or (4a) happens before $u'$.
Thus, $X_{s^-_{e,u}}\cap[p_{e-1,s},p_{e,s})\subseteq X_s$ for $s^-_{e,u}<s\leq u'$ since neither (3c) nor (3c) happens between $s^-_{e,u}$ and $u'$, and ${s_{e,s}}={s^-_{e,u}}$ for such $s$.
Moreover, since $s^-_{e,s}=s^-_{e,u}$ for all $s$ with $u'<s<u$, we have $X_{s^-_{e,u}}\cap [p_{e-1,s},p_{e,s}-1)\subseteq X_s$ for such $s$ by our choice of $u$.
Consequently, $X_s\cap [p_{e-1,s},p_{e,u'})\not\subseteq X_u$ for all $s$ with $s^{-}_{e,u}<s<u$.
In particular, if $s^{-}_{e,u}<s<u$ and $e'\geq e$, then $X_s\upharpoonright p_{e',s+1}\not\subseteq X_u$ since $p_{e,s}\leq p_{e',s}\leq p_{e',s+1}$.
Now, note that every computation enumerated into $\Phi$ at stage $s\geq t$ is of the form $(y_{e',s+1},X_{s}\upharpoonright p_{e',s+1})$ for some $e'\geq e$.
Therefore, every computation enumerated into $\Phi$ at stage $s$ with $s^{-}_{e,u}<s<u$ is destroyed at stage $u$.
We also note that $(y_{e,s})$ is monotone in the sense that $e\leq e'$ implies $y_{e,s}\leq y_{e',s}$ and that $t\leq s\leq s'\leq u$ implies $y_{e,s}\leq y_{e,s'}$.
Moreover, $y_{e,u''}+2^{-m_{e,u''+1}-1}=y_{e,u''+1}\leq y_{e,u'}$, where $u''=s^-_{e,u}$.
Therefore, $y_{e,u''}+2^{-m_{e,u''+1}-2}\leq \alpha_u$.
By the previous argument, the $e$-th strategy only enumerate the value $y_{e,u''}$ under the oracle $X_u$, and therefore, the value enumerated by lower priority is at most   $y_{e,u''}+2^{-m_{e,u''+1}-2}$.
Consequently, we have $\Phi(X)<\alpha_e$.
}

\subsubsection{$T_1$-quasi-minimal degrees}\label{sec:T-1-quasiminimalp}

We show the existence of $T_1$-quasi-minimal $\mathbb{R}_{<}$-degrees mentioned in Section \ref{sec:T-1-quasiminimal}.

\thmproof{thm:countable-T_1-quasiminimal}{
A {\em lightface pointclass} $\Gamma$ is a countable collection of subsets of $\om^\om$ which is closed under computable substitution, that is, if $S\in\Gamma$, then $\Phi^{-1}[S]\in\Gamma$ for any computable function $\Phi:\om^\om\to\om^\om$.
By $\check{\Gamma}$ we denote the dual of $\Gamma$, that is, $\check{\Gamma}=\{A:\cmp{A}\in\Gamma\}$.

We say that a set $A\subseteq\om$ is $\Gamma$ if there is a $\Gamma$ set $S$ such that for any $n\in\om$, $n\in A$ if and only if $n0^\om\in S$.
We say that a real $x\in\mathbb{R}$ is {\em left-$\Gamma$} if $\{n\in\om:p_n<x\}$ is in $\Gamma$, where $p_n$ is the $n$-th rational.
Similarly, we say that a real $x\in\mathbb{R}$ is {\em right-$\Gamma$} if $\{n\in\om:x<p_n\}$ is in $\Gamma$.
Then we say that $\pt{x}{\mathbb{R}}$ is $\Delta$ if it is both left-$\Gamma$ and right-$\Gamma$.

\begin{lemma}\label{lem:countable-T_1-quasiminimal}
Let $\Gamma$ be a lightface pointclass.
For any $x\in\mathbb{R}$, if $\pt{x}{\mathbb{R}}$ is not $\Delta$, then $\pt{x}{\mathbb{R}_<}$ is quasi-minimal w.r.t.~strongly $\check{\Gamma}$-named $T_1$ spaces.
\end{lemma}

\begin{proof}
We show a relativized version, which improves Lemma \ref{propdichotomy}.
For any $r\in\om^\om$, we say that $A\subseteq\om^\om$ is {\em $\Gamma$ relative to $r$}, or simply $\Gamma(r)$ if there is a $\Gamma$ set $G\subseteq\om^\om$ such that $A=\{y:\langle r,y\rangle\in G\}$.
Let $\xx$ be any represented cb$_0$ space.
Given $o\in \xx$, we say that $A\subseteq\om^\om$ is {\em $\Gamma$ relative to $o$}, or simply $\Gamma(o)$, if $A$ is $\Gamma$ relative to any $\xx$-name of $o$ in a uniform manner, that is, there is a $\Gamma$ set $G\subseteq\om^\om$ such that $A=\{y:\langle r,y\rangle\in G\}$ for any $\xx$-name $r$ of $o$.
Then, we define the notion of a $\Delta(o)$ real in a straightforward manner.

Let $\xx$ be a represented cb$_0$ space, and $\zz$ be a strongly $\check{\Gamma}$-named $T_1$ space.
We will show that for any $x\in\mathbb{R}$, $o\in\xx$, and $z\in\zz$,
\[\nbaseb{\zz}{z}\leq_e\nbaseb{<}{x}\oplus\nbaseb{\xx}{o}\;\Longrightarrow\;\nbaseb{\zz}{z}\leq_e\nbaseb{\xx}{o}\mbox{ or }\pt{x}{\mathbb{R}}\mbox{ is $\Delta(o)$}.\]

Since the specialization order on $\mathbb{R}_<$ forms a chain while the specialization order on a $T_1$ space forms an antichain, for any continuous function $\Phi$ from $\xx\times\mathbb{R}_<$ to a $T_1$ space $\zz$ and any $o\in\xx$, for $\Phi_o:y\mapsto\Phi(o,y)$, the cardinality of ${\rm rng}(\Phi_o)$ is at most one.
In the context of an enumeration operator, this is because, for any enumeration operator $\Phi$, any point $o\in\xx$, and any reals $x,y\in\mathbb{R}$,
\[x<y\;\Longrightarrow\Phi_o(\nbaseb{<}{x})\subseteq\Phi_o(\nbaseb{<}{y}),\]
where we define $\Phi_o(A)=\Phi(\nbaseb{\xx}{o}\oplus A)$, and recall that if $\zz$ is $T_1$, for any $z_0,z_1\in\zz$, $z_0\not=z_1$ implies $\nbaseb{\zz}{z_i}\not\subseteq\nbaseb{\zz}{z_{1-i}}$ for any $i<2$.
Here, by symbols ${\rm dom}(\Phi_o)$ and ${\rm rng}(\Phi_o)$ we mean the domain and the range of the function $\tilde{\Phi}_o$ from $\mathbb{R}_<$ to a $T_1$ space $\zz$ induced from the operator $\Phi_o$, that is,
\begin{align*}
{\rm dom}(\Phi_o)&=\{x\in\mathbb{R}_<:(\exists z\in\zz)\;\Phi_o(\nbaseb{<}{x})=\nbaseb{\zz}{z}\},\\
{\rm rng}(\Phi_o)&=\{z\in\zz:(\exists x\in\mathbb{R}_<)\;\Phi_o(\nbaseb{<}{x})=\nbaseb{\zz}{z}\}.
\end{align*}

Assume that ${\rm rng}(\Phi_o)$ is nonempty in a $T_1$ space $\zz$, that is, there are $x\in\mathbb{R}$ and $z\in\zz$ such that $\Phi_o(\nbaseb{<}{x})=\nbaseb{\zz}{z}$.
If ${\rm dom}(\Phi_o)$ is not a singleton, $x,y\in{\rm dom}(\Phi_o)$ with $x<y$ say, it contains a non-degenerate interval $[x,y]$, and therefore contains a rational $q\in[x,y]\subseteq{\rm dom}(\Phi_o)$.
Thus, $\Phi_o(\nbaseb{<}{q})$ gives us a unique element $z$ of ${\rm rng}(\Phi_o)$, and then we must have $\nbaseb{\zz}{z}\leq_e\nbaseb{\xx}{o}$ since $q$ is computable.

Therefore, if ${\rm rng}(\Phi_o)$ contains a point $z$ such that $\nbaseb{\zz}{z}\not\leq_e\nbaseb{\xx}{o}$, then ${\rm dom}(\Phi_o)$ has to be a singleton $\{x\}$.
Let $P$ and $N$ witness that $\zz$ is strongly $\check{\Gamma}$-named.
Let $F$ be a computable realizer of $\tilde{\Phi}$, that is, given $\alpha,\beta\in\om^\om$, if $s$ is the first stage such that we see $D\subseteq{\rm rng}(\alpha)$ and $E\subseteq{\rm rng}(\alpha\upto s)$ and $\langle k,D,E\rangle\in\Phi$ by stage $s$, then put $F(\alpha,\beta)(s)=k$.
One can assume that $F$ is a total computable function on $\om^\om\times\om^\om$ since it is generated from an enumeration operator.
If we fix an $\xx$-name $\alpha$ of $o\in\xx$, then $F(\alpha,\beta)\in{\rm Sub}(\zz)\cup{\rm Sup}(\zz)$ for any $\beta\in{\rm Name}(\mathcal{R}_<)$.
Thus, we have the following:
\begin{align*}
F(\alpha,\beta)\not\in P&\iff F(\alpha,\beta)\in{\rm Sub}(\zz)\setminus{\rm Name}(\zz),\\
F(\alpha,\beta)\not\in N&\iff F(\alpha,\beta)\in{\rm Sup}(\zz)\setminus{\rm Name}(\zz).
\end{align*}

By our assumption, both $L=F^{-1}[\cmp{P}]$ and $R=F^{-1}[\cmp{N}]$ are $\Gamma$ subsets of $\om^\om$.
If $\alpha$ be an $\xx$-name of $o$, and $\beta$ be an $\mathbb{R}_<$-name of $z$, one can see that $(\alpha,\beta)\in L$ implies $z<x$, and similarly, $(\alpha,\beta)\in R$ implies $z>x$.
For each rational $p_n$, choose an $\mathbb{R}_<$-name $\beta_n$ of $p_n$ in an effective manner.
One can ensure that $g:(\alpha,n0^\om)\mapsto(\alpha,\beta_n)$ is computable.
Therefore, $g^{-1}[L]$ and $g^{-1}[R]$ are in $\Gamma$ since $\Gamma$ is closed under computable substitution.
If $\alpha$ is an $\xx$-name of $o$, we get that
\begin{align*}
p_n<x&\iff (\alpha,\beta_n)\in L\iff (\alpha,n0^\om)\in g^{-1}[L],\\
x<p_n&\iff (\alpha,\beta_n)\in R\iff (\alpha,n0^\om)\in g^{-1}[R].
\end{align*}
Thus, these give left- and right-$\Gamma$ approximations of $x$ uniformly relative to any name of $o$.
This concludes that $x$ is $\Delta(o)$, which verifies our claim.
If $o=\emptyset$ in $2^\om$, this means that $x$ is quasi-minimal w.r.t.~strongly $\check{\Gamma}$-named $T_1$ spaces.
\end{proof}

For any countable collection $\mathcal{T}$ of $T_1$-spaces, clearly there is a lightface pointclass $\Gamma$ such that $\xx$ is strongly $\check{\Gamma}$-named for any $\xx\in\mathcal{T}$.
Thus, by Lemma \ref{lem:countable-T_1-quasiminimal}, if $\pt{x}{\mathbb{R}}$ is not $\Delta$, then $\pt{x}{\mathbb{R}_<}$ is $\mathcal{T}$-quasi-minimal.
Recall that every ${\mathbb{R}_<}$-degree is semirecursive.
This concludes the proof.
}


\subsubsection{Products of lower topology}

For definitions, see Section \ref{sec:4-1-4}.

\propproof{prop:RRT1quasimin1}{Clearly $\nbaseb{\xx}{x}<_e\nbaseb{\xx\times\mathbb{R}_<}{x,y}\equiv_e\nbaseb{\xx}{x}\oplus\nbaseb{<}{y}$ since $y$ is not left-c.e.~in $x$.
Moreover, if $\nbaseb{2^\om}{z}\leq_e\nbaseb{\xx}{x}\oplus\nbaseb{<}{y}$ for some $z\in 2^\omega$, by Lemma \ref{propdichotomy}, we have $\nbaseb{2^\om}{z}\leq_e\nbaseb{\xx}{x}$ since $y$ is not right-c.e.\ in $x$, that is, $\nbaseb{<}{-y}\not\leq_e\nbaseb{\xx}{x}$.
Therefore, $\pt{(x,y)}{\xx\times\mathbb{R}_<}$ is a strong quasi-minimal cover of $\pt{x}{\xx}$.
For (1), if $y$ is left-c.e.\ in $x$, then $\nbase{y}\leq_e\nbase{x}$, and therefore, $\nbaseb{\xx}{x}\equiv_e\nbaseb{\xx}{x}\oplus\nbaseb{<}{y}$.
This means that $\nbaseb{\xx\times\mathbb{R}_<}{x,y}$ has an $\xx$-degree, and in particular, has an $\xx\times\mathbb{R}$-degree.
If $y$ is right-c.e.\ in $x$, then $\nbaseb{<}{-y}\leq_e\nbaseb{\xx}{x}$, and therefore, $\nbaseb{\xx}{x}\oplus\nbase{y}\equiv_e\nbaseb{\xx}{x}\oplus\nbaseb{<}{y}$.
This means that $\nbaseb{\xx\times\mathbb{R}_<}{x,y}$ has an $\xx\times\mathbb{R}$-degree.
If $y$ is neither left- nor right-c.e.\ in $x$, then $\nbaseb{\xx\times\mathbb{R}_<}{x,y}$ is a strong quasi-minimal cover of $x$ as shown above.
For (2), let $\mathbf{d}$ be an $\xx$-degree.
Then given a point $x\in\xx$ of degree $\mathbf{d}$, choose a real $y$ which is neither left- nor right-c.e.\ in $x$.
Such a $y$ must exist.
Then $\nbaseb{\xx\times\mathbb{R}_<}{x,y}$ is a strong quasi-minimal cover of $x$ as shown above.}

\propproof{prop:RRT1quasimin2}{
Let $\Gamma$ be some countably pointclass such that every space from $\mathcal{T}$ is strongly $\Gamma$-named. We adapt the proof of Proposition \ref{prop:RRT1quasimin1} by using Lemma \ref{lem:countable-T_1-quasiminimal} instead of Lemma \ref{propdichotomy}. We conclude that
for any $x\in\xx$ and $y\in\mathbb{R}$, if $\pt{y}{\mathbb{R}}$ is not $\Delta(x)$, then $\pt{(x,y)}{\xx\times\mathbb{R}_<}$ is a strong $T_1[\Gamma]$-quasi-minimal cover of $\pt{x}{\xx}$, where $T_1[\Gamma]$ is the collection of strongly $\Gamma$-named, second-countable, $T_1$ spaces.
}


\propproof{proptotalordelta2}{
Let $x\in\mathbb{R}$, and $y,z\in\mathbb{R}_<$.
By iterating Lemma \ref{propdichotomy}, if $\nbase{x}\leq_e\nbaseb{<}{y}\oplus\nbaseb{<}{z}$, then either (1) $\nbase{x}\leq_e\emptyset$ (i.e., $x\leq_T\emptyset$), (2) $\nbase{x}\leq_e\nbaseb{<}{z}$ and $\nbaseb{<}{-z}\leq_e\emptyset$, or (3) $\nbaseb{<}{-y}\leq_e\nbaseb{<}{z}$.
If (1) or (2) holds, then $x\leq_T\emptyset'$.
Thus, assume that (3) holds.
We claim that either $\nbaseb{<}{-y}\leq_e\emptyset$ or $\nbaseb{<}{-z}\leq_e\nbaseb{<}{y}$ holds.
Let $\Phi$ be a witness of our assumption (3), that is, for any rational $p\in\mathbb{Q}$, $y<p$ if and only if there is $q<z$ such that $\langle p,q\rangle\in\Phi$.
If there is $\varepsilon>0$ such that $q<z+\varepsilon$ and $\langle p,q\rangle\in\Phi$ implies $y\leq p$, then $y$ is right-c.e.
Otherwise, for all $\varepsilon>0$, there are $p,q\in\mathbb{Q}$ such that $q<z+\varepsilon$, $\langle p,q\rangle\in\Phi$, and $p<y$.
Now, by using a left-approximation of $y$, search $q$ such that $\langle p,q\rangle\in\Phi$ and $p<y$.
Such a $q$ must satisfy $q\leq z$ by our choice of $\Phi$, and for any $\varepsilon>0$ there is such $q<z+\varepsilon$ by our assumption.
Therefore, by enumerating all such $q$'s, we obtain a right-approximation of $z$.
This shows that $\nbaseb{<}{-z}\leq_e\nbaseb{<}{y}$.
Consequently, $\nbaseb{<}{-y}\oplus\nbaseb{<}{-z}\leq_e\nbaseb{<}{y}\oplus\nbaseb{<}{z}$, which implies that $\pt{(y,z)}{\mathbb{R}_<^2}$ is total.
}

\thmproof{thmseparation1}{
One can assume that $\mathcal{T}=T_1[\Gamma]$, i.e.\ the strongly $\Gamma$-named, second countable, $T_1$-spaces, for some lightface pointclass $\Gamma$.
To prove Theorem \ref{thmseparation1}, we need the notion of an independent point.
We say that a point $(x_i)_{i\leq n}\in\mathbb{R}_<^{n+1}$ is {\em independent} if neither $\nbaseb{<}{x_i}$ nor $\nbaseb{<}{-x_i}$ is $e$-reducible to $\nbase{(x_j)_{j\not=i}}$ for any $i\leq n$.
It is not hard to see that if $(x_i)_{i\leq n}$ is $1$-generic as a point in $\mathbb{R}^{n+1}$, then $(x_i)_{i\leq n}$ is independent.
One can also see that if $(x_i)_{i\leq n}$ is a Martin-L\"of random point in $\mathbb{R}^{n+1}$ (w.r.t.~the product measure), then $(x_i)_{i\leq n}$ is independent.
Given a lightface pointclass $\Gamma$, we say that a point $(x_i)_{i\leq n}\in\mathbb{R}_<^{n+1}$ is {\em $\Delta$-independent} if it is independent, and if for any $i\leq n$, neither $\pt{x_i}{\mathbb{R}_<}$ nor $\pt{-x_i}{\mathbb{R}_<}$ is $\Delta$ in  $\pt{(x_j)_{j\not=i}}{\mathbb{R}_<^n}$.
If $(x_i)_{i\leq n}$ is sufficiently generic or random, then $(x_i)_{i\leq n}$ is $\Delta$-independent.
Therefore, the set of $\Delta$-independent points is comeager and conull in $\mathbb{R}^{n+1}$.

\begin{lemma}\label{lemseparation1}
If $\bvec{x}\in\mathbb{R}_<^{n+1}$ is independent, then $\nbase{\bvec{x}}$ does not have an $\mathbb{R}\times\mathbb{R}^n_<$-degree.
If $\bvec{x}\in\mathbb{R}_<^{n+1}$ is $\Delta$-independent, then $\nbase{\bvec{x}}$ does not have an $\xx\times\mathbb{R}^n_<$-degree for any strongly $\Gamma$-named $T_1$-space $\xx$.
\end{lemma}

\begin{proof}
Suppose not.
Then, there are an independent point $\bvec{x}=(x_i)_{i\leq n}\in\mathbb{R}_<^{n+1}$ and points $z\in\mathcal{X}$ and $\bvec{y}=(y_j)_{j<n}\in\mathbb{R}_<^n$ such that $\nbase{z,\bvec{y}}\equiv_e\nbase{\bvec{x}}$.
In particular, $\nbase{z}\leq_e\nbase{\bvec{x}}$.
If $\bvec{x}$ is independent and $\xx=2^\om$, then $\bvec{x}$ is quasi-minimal by iterating Lemma \ref{propdichotomy}.
If $\bvec{x}$ is $\Delta$-independent, then $\bvec{x}$ is $T_1[\Gamma]$-quasi-minimal by iterating (the general claim in the proof of) Lemma \ref{lem:countable-T_1-quasiminimal}.
Therefore, $\nbase{z}\leq_e\nbase{\bvec{x}}$ implies $\nbase{z}\leq_e\emptyset$.
Thus, we now have $\nbase{\bvec{y}}\equiv_e\nbase{\bvec{x}}$.
Let $\Phi$ and $\Psi$ be enumeration operators witnessing this $e$-equivalence, that is,
\[\Phi(\nbase{\bvec{x}})=\nbase{\bvec{y}}\mbox{, and }\Psi(\nbase{\bvec{y}})=\nbase{\bvec{x}}.\]
Here, the coded neighborhood basis is given as $\nbase{(x_i)_{i\leq n}}=\{\langle i,p\rangle:p<x_i\}$.
First consider the case that there are $k<n$ and $\varepsilon>0$ such that for any $p_k<y_k+\varepsilon$,
\[(\forall i\leq n)(\forall q\in\mathbb{Q})\;[\langle i,q\rangle\in\Psi(\nbase{y_0,\dots,y_{k-1},p_k,y_{k+1},\dots,y_{n-1}})\;\Longrightarrow\;q<x_i].\]
Then, we get that $\nbase{(x_i)_{i\leq n}}\leq_e\nbase{(y_j)_{j\not=k}}$ as in the proof of Lemma \ref{propdichotomy}.
By induction, we can ensure that it is impossible.

Thus, for any $k<n$ and $\varepsilon>0$, there has to be $q_k<y_k+\varepsilon$ such that
\[(\exists i_k\leq n)(\exists q\in\mathbb{Q})\;[\langle i_k,q\rangle\in\Psi(\nbase{y_0,\dots,y_{k-1},q_k,y_{k+1},\dots,y_{n-1}})\mbox{ and }x_{i_k}<q].\]
By pigeon hole principle, one can assume that $i_k$ only depends on $k$, and does not depend on $\varepsilon$.
Without loss of generality, we may assume that $n\not\in\{i_k:k<n\}$.
Now, let $g_\Phi:\mathbb{R}_<^{n+1}\to\mathbb{R}_<^n$ be the computable function induced from the enumeration operator $\Phi$, that is, $g_\Phi(\bvec{r})=\bvec{s}$ if and only if $\Phi(\nbase{\bvec{r}})=\nbase{\bvec{s}}$.
Then, given $p\in\mathbb{Q}$, we define $\bvec{z}^p=(z^p_j)_{j<n}$ as follows:
\[\bvec{z}^p=g_\Phi(x_0,\dots,x_{n-1},p),\]

It is easy to see that for any $j<n$, if $p\leq x_n$ then $z^p_j\leq y_j$, and if $p\geq x_n$ then $y_j\leq z^p_j$.
If there is $p>x_n$ such that $z^p_j=y_j$ for any $j<n$, then clearly,
\[\nbase{\bvec{y}}=\nbase{\bvec{z}^p}\leq_e\nbase{(x_i)_{i<n}}<_e\nbase{(x_i)_{i\leq n}}=\nbase{\bvec{x}},\]
which contradicts our assumption.
Therefore, for any $p>x_n$, there is $k<n$ such that $y_k<z^p_k$.
By pigeon hole principle, one can fix such a $k$.
Let $q_k$ be a rational such that $y_k<q_k<z^p_k$.
Then,
\[\nbase{y_0,\dots,y_{k-1},q_k,y_{k+1},\dots,y_{n-1}}\subseteq\nbase{\bvec{z}^p}.\]
Hence, there is $q\in\mathbb{Q}$ such that $\langle i_k,q\rangle\in\Psi(\nbase{\bvec{z}^p})$ and $x_{i_k}<q$.
By combining these observations, we have
\[x_n<p\iff(\exists q\in\mathbb{Q})\;[\langle i_k,q\rangle\in\Psi(\nbase{\bvec{z}^p})\mbox{ and }x_{i_k}<q].\]

This gives us a right-approximation of $x_n$ from a left-approximation of $\bvec{z}^p$ and a right-approximation of $x_{i_k}$.
This concludes that
\[\nbase{-x_n}\leq_e\nbase{\bvec{z}^p}\oplus\nbase{-x_{i_k}}\leq_e\nbase{(x_j)_{j<n}}\oplus\nbase{-x_{i_k}}.\]

This contradicts our assumption on independence of $\bvec{x}$.
Consequently, we obtain that $\nbase{z,\bvec{y}}\not\equiv_e\nbase{\bvec{x}}$.
\end{proof}

By Lemma \ref{lemseparation1}, if $\mathbf{a}$ is an $e$-degree of an independent point in $\mathbb{R}_{<}^{n+1}$, then $\mathbf{a}$ is not the join of an $\mathbb{R}_{<}^{n}$-degree $\mathbf{c}$ and a $T_1[\Gamma]$-degree $\mathbf{d}$.
}

\thmproof{thm:rrn-semirec}{
For $z\in\mathbb{R}$, we say that a point $(x_i)_{i<n}\in\mathbb{R}_<^{n}$ is {\em independent relative to $z$} if neither $\nbaseb{<}{x_i}$ nor $\nbaseb{<}{-x_i}$ is $e$-reducible to $\nbase{z,(x_j)_{j\not=i}}$ for any $i\leq n$.
One can easily see that $(z,\bvec{x})$ is $1$-generic or Martin-L\"of random as a point in $\mathbb{R}^{n+1}$, then $\bvec{x}$ is independent relative to $z$.

\begin{lemma}\label{lem:rrn-semirec}
If $\bvec{x}\in\mathbb{R}_<^{n}$ is independent relative to $z\in\mathbb{R}$, and $z\not\leq_T\emptyset'$, then $\nbase{z,\bvec{x}}$ is not $(n+1)$-semirecursive.
\end{lemma}

\begin{proof}
Let $z\in\mathbb{R}$ and $\bvec{x}\in\mathbb{R}_<^{n}$ be such that $(z,\bvec{x})$ is independent, and that $\nbase{z}\not\leq_e K\oplus\cmp{K}$.
Suppose that there is $\bvec{y}=(y_j)_{j\leq n}\in\mathbb{R}_<^{n+1}$ such that $\nbase{\bvec{y}}\equiv_e\nbase{z,\bvec{x}}$.
Consider any permutation $\sigma$ on $n+1$.
First assume that $\nbase{z}\leq_e\nbase{y_{\sigma(0)}}$.
Then, by Lemma \ref{propdichotomy}, we have $\nbase{z}\leq_e\emptyset$ or $\nbase{-y_{\sigma(0)}}\leq_e\emptyset$.
The latter inequality implies that $\nbase{z}\leq_e\nbase{y_{\sigma(0)}}\leq_e K\oplus\cmp{K}$, which is impossible by our assumption.

Assume that for some $k\leq n$, there is a permutation $\sigma$ on $n+1$ such that $\nbase{z}\leq_e\nbase{(y_{\sigma(j)})_{j\leq k}}$, but there is no permutation $\tau$ on $k+1$ such that $\nbase{z}\leq_e\nbase{(y_{\sigma\circ\tau(j)})_{j<k}}$.
By Lemma \ref{propdichotomy}, the inequality $\nbase{z}\leq_e\nbase{(y_{\sigma(j)})_{j\leq k}}$ implies that for any $\ell\leq k$, either $\nbase{z}\leq_e\nbase{(y_{\sigma(j)})_{j\not=\ell,j\leq k}}$ or $\nbase{-y_{\sigma(\ell)}}\leq_e\nbase{(y_{\sigma(j)})_{j\not=\ell,j\leq k}}$.
The former is impossible by our choice of $k$.
Hence, we get that $\nbase{(-y_{\sigma(j)})_{j\leq k}}\leq_e\nbase{(y_{\sigma(j)})_{j\leq k}}$.
Note that by induction, there must exist such a $k$, and we have $k\geq 1$.

Eventually, we must have $\nbase{\bvec{y},-y_s,-y_t}\equiv_e\nbase{\bvec{y}}$ for some $s\not=t\leq n$.
Hence, one can identify $\bvec{y}$ with an element of $\mathbb{R}\times\mathbb{R}_<^{n-1}$.
However, by Lemma \ref{lemseparation1} relative to $z$ and by independence of $\bvec{x}$ relative to $z$, it is impossible to have $\nbase{\bvec{y}}\equiv_e\nbase{z,\bvec{x}}$.
\end{proof}

By choosing a $2$-generic or a $2$-random, it is easy to construct an independent sequence $(z,\bvec{x})\in\mathbb{R}\times\mathbb{R}_<^{n}$ of $\Delta^0_3$ reals such that $z\not\leq_T\emptyset'$.
Let $\mathbf{c}$ be the $e$-degree of $\nbase{\bvec{x}}$, and $\mathbf{d}$ be that of $\nbase{z}$.
Since these are $\Delta^0_3$, we have $\mathbf{c},\mathbf{d}\leq\mathbf{0}''$.
By Lemma \ref{lem:rrn-semirec}, the join $\mathbf{c}\oplus\mathbf{d}$ is not $(n+1)$-semirecursive.
}

\thmproof{thm:hprincipal}{
Fix $\bvec{x}=(x_i)_{i\leq n}\in\mathbb{R}_<^n$.
Suppose that $\bvec{x}$ bounds a total degree $z\in 2^\omega$.
As in the proof of Lemma \ref{lem:rrn-semirec}, by iterating Lemma \ref{propdichotomy}, for any permutation $\sigma$ on $n+1$, either $\nbase{z}\leq_e\emptyset$ or $\nbase{-x_{\sigma(k)}}\leq_e\nbase{(x_{\sigma(i)})_{i<k}}$ holds for some $k\leq n$.
We assume $\nbase{z}\not\leq_e\emptyset$.
Therefore, by induction, we can find a permutation $\sigma$ on $n+1$ and a number $k\leq n$ such that $\nbase{-x_{\sigma(\ell)}}\not\leq_e\nbase{(x_{\sigma(i)})_{i<\ell}}$ for any $k<\ell\leq n$ and $\nbase{-x_{\sigma\circ\tau(k)}}\leq_e\nbase{(x_{\sigma\circ\tau(i)})_{i<k}}$ for any permutation $\tau$ on $k+1$.
The latter condition implies that $\nbase{(x_{\sigma(i)},-x_{\sigma(i)})_{i\leq k}}\equiv_e\nbase{(x_{\sigma(i)})_{i\leq k}}$.
This shows that ${(x_{\sigma(i)})_{i\leq k}}$ has a total degree.
We claim that ${(x_{\sigma(i)})_{i\leq k}}$ has the greatest total degree below $\nbase{\bvec{x}}$.
Suppose that $Y$ is a subset of $\om$ such that $Y\oplus\cmp{Y}\leq_e\nbase{\bvec{x}}$.
Note that $\nbase{\bvec{x}}\equiv_e\nbase{(x_{\sigma(i)})_{i\leq n}}$.
Therefore, by iterating Lemma \ref{propdichotomy}, either $Y\oplus\cmp{Y}\leq_e\nbase{(x_{\sigma(i)})_{i\leq k}}$ or $\nbase{-x_{\sigma(\ell)}}\leq_e\nbase{(x_{\sigma(i)})_{i<\ell}}$ holds for some $k<\ell\leq n$.
However, the latter condition cannot hold by our choice of $\sigma$ and $k$.
Consequently, $\bvec{x}$ is either total or a strong quasi-minimal cover of $(x_{\sigma(i)})_{i\leq k}$.
}

\subsubsection{Left- and right-totality}\label{sec:quasiminimalspaces}


In Section \ref{sec:T-1-quasiminimalp}, we developed techniques constructing $T_1$-quasi-minimal degrees.
In contrast to non-$T_1$ cb$_0$ spaces, there are a number of interesting $T_1$-space, and therefore $T_1$-degrees, e.g.\ graph-cototal degrees, cylinder-cototal degrees, telograph-cototal degrees, etc.
We examine the behavior of $T_1$-degrees using quasi-minimality arguments.
In particular, we give proofs of some results mentioned in Section \ref{sec:4-1-5}.

\subsubsection*{Total-like properties}

To prove Theorems \ref{telograph-quasiminimal1} and \ref{telograph-quasiminimal2}, we introduce some variants of totality, and show that every telograph-cototal $e$-degree is close to being total in this sense.

The power set of $\om$ can be topologized as the countable product $\mathbb{S}^\om$ of the Sierpinsk\'i space $\mathbb{S}=\{0,1\}$, where open sets in $\mathbb{S}$ are $\emptyset$, $\{1\}$, and $\mathbb{S}$.
Every element $S\in\mathbb{S}^\om$ is identified with $S^{-1}\{1\}\subseteq\om$.
We consider the effective Borel hierarchy on $\mathbb{S}^\om$.
Here, we note that Selivanov and his collaborators (see e.g.\ \cite{debrecht6,Seliv06}) have studied a modified version of the Borel hierarchy.
As seen in Observation \ref{obs:Gdelta-Borel-hierarchy}, the underlying space is a $G_\delta$-space if and only if the classical Borel hierarchy and the modified Borel hierarchy coincide.
In particular, a $\mathbf{\Pi}^0_2$ set is not necessarily $G_\delta$ in the space $\mathbb{S}^\om$.

We say that $\mathcal{P}\subseteq\mathbb{S}^\om$ is {\em $C$-computably $e$-closed} if there is a $C$-computable sequence $(D_{n})_{n\in\om}$ of finite subsets of $\om$ such that
\[A\in\mathcal{P}\iff(\forall n)\;D_{n}\not\subseteq A.\]
We also say that $\mathcal{P}\subseteq\mathbb{S}^\om$ is {\em $C$-computably $e$-$G_\delta$} if there is a $C$-computable sequence $(D_{n,i})_{n,i\in\om}$ of finite subsets of $\om$ such that
\[A\in\mathcal{P}\iff(\forall n)(\exists i)\;D_{n,i}\subseteq A.\]

As mentioned above, a computably $e$-closed set is not necessarily computably $e$-$G_\delta$.
Therefore we avoid use of the terminology such as ``$e$-$\Pi^0_2$''.
Note that every computably $e$-closed set is downward closed, and every computably $e$-$G_\delta$ set is upward closed.

\begin{definition}
A set $A\subseteq\om$ is said to be {\em $G_\delta$-left-total} if there are a set $C\subseteq\om$, and an computably $e$-$G_\delta$ set $\mathcal{P}$ such that
\begin{align*}
C\oplus\overline{C}\leq_e A\mbox{, and }&A\in\mathcal{P},\\
(\forall X)\;[X\in\mathcal{P}\mbox{ and }X\subseteq A&\;\Longrightarrow\;A\leq_eX\oplus C\oplus\overline{C}].
\end{align*}

A set $A\subseteq\om$ is said to be {\em jump-right-total} if there are a set $C\subseteq\om$, and a $C'$-computably $e$-closed set $\nn$  such that
\begin{align*}
C\oplus\overline{C}\leq_e A\mbox{, and }&A\in\nn,\\
(\forall X)\;[X\in\nn\mbox{ and }A\subseteq X&\;\Longrightarrow\;A\leq_eX\oplus C\oplus\overline{C}].
\end{align*}
\end{definition}

\begin{example}\label{exa:telcototal-totallike}
Every telograph-cototal $e$-degree is $G_\delta$-left-total and jump-right-total.
\end{example}

\begin{proof}
We first show that every telograph-cototal $e$-degree is $G_\delta$-left-total.
Let $A$ be ${\rm Nbase}(x)$ for a point $x$ in the product telophase space $(\hat{\om}_{TP})^\om$.
Let $C$ be the total part of $x$, that is, $C=\{2\langle n,m\rangle:x(n)=m\}\cup\{2\langle n,m\rangle+1:x(n)=m\}$, and define $R=\{n\in\om:x(n)=\infty\}$ and $S=\{n\in\om:x(n)=\infty_\star\}$.
The proof of Theorem \ref{thm:telophase-degree} shows that $A={\rm Nbase}(x)$ is Medvedev equivalent to $\{C\}\times{\rm Sep}(R,S)$.
In particular $C\oplus\cmp{C}\leq_eA$.
Define $D_{n,i}=\{\langle n,j,m\rangle\}$, where $i=\langle j,m\rangle$.
This sequence generates a computably $e$-$G_\delta$-set:
\[\mathcal{P}=\{X:(\forall n)(\exists i)\;D_{n,i}\subseteq X\}.\]

Recall the definition of $\nbase{x}$ in Example \ref{example:teloph-2}.
Clearly, for any $n$, there is $i=\langle j,\sigma\rangle$ such that $D_{n,i}=\{\langle n,j,\sigma\rangle\}\subseteq A={\rm Nbase}(x)$.
Fix $X\in\mathcal{P}$ such that $X\subseteq A$.
Then, for any $n$, $D_{n,i_n}\subseteq X$ for some $i_n$.
Given an enumeration of $X$, one can find a sequence $(i_n)_{n\in\om}$ of such witnesses.
The sequence $(i_n)_{n\in\om}$ may depend on how to enumerate $X$.
As in the proof of Theorem \ref{thm:telophase-degree}, we can construct $B\subseteq\om$ according to $(i_n)_{n\in\om}$ as follows.
Define $n\not\in B$ iff $i_n$ is of the form $\langle 2,m\rangle$, that is, $i_n$ is an index of the interval $[m,\infty_\star]$.
We claim that $B\in{\rm Sep}(R,S)$.
If $n\in R$, then $x(n)=\infty$, and therefore, $x(n)\not\in[m,\infty]$ for any $m$, and thus $\langle n,2,m\rangle\not\in A$.
Since $X\subseteq A$, we also have $\langle n,2,m\rangle\not\in A$.
Therefore, $n\in B$, that is, $R\subseteq B$.
If $n\in S$, then $x(n)=\infty_\star$, and thus, $\langle n,j,m\rangle\in A$ implies $j=2$.
Since $A\subseteq X$, $i_n$ must be of the form $\langle 2,m\rangle$.
This implies that $n\not\in B$, that is, $B\cap S=\emptyset$.
Consequently, we have $B\in{\rm Sep}(R,S)$.
This procedure gives us an enumeration operator $\Psi$ (independent of $X$) such that $\Psi({X\oplus C\oplus\cmp{C}})=A$.
Hence, $A$ is $G_\delta$-left-total.

We next show that every telograph-cototal $e$-degree is jump-right-total.
Let $A$ be ${\rm Nbase}(x)$ for a point $x$ in the product telophase space $(\hat{\om}_{TP})^\om$.
We define $C$, $R$, and $S$ as above.
Consider the following $C'$-computably $e$-closed set:
\[\nn=\{X:(\forall n\in R\cup S)(\forall i,j,s,t)\;[\langle n,i,s\rangle,\langle n,j,t\rangle\in X\;\Longrightarrow\;i=j\not=0]\}.\]

It is clear that $A\in\nn$.
Note that $R\cup S$ is $C$-co-c.e., and in particular, $C'$-computable.
To see that $\nn$ is $C'$-computably $e$-closed, let $I$ be the set of all $\langle n,i,j,s,t\rangle$ such that $n\in R\cup S$ and either $i\not=j$ or $i=0$.
Define $D_k=\{\langle n_k,i_k,s_k\rangle,\langle n_k,j_k,t_k\}$, where $\langle n_k,i_k,j_k,s_k,t_k\rangle$ is the $k$-th element of $I$.
Clearly, $(D_k)_{k\in\om}$ is a $C'$-computable sequence of finite sets.
It is easy to see that $X\in\nn$ iff $D_k\not\subseteq X$ for all $k\in\om$.

Fix $X\in\nn$ such that $A\subseteq X$.
Consider the procedure $\Psi$ which enumerates all $\langle n,i,m\rangle\in X$ which agree with $C$, that is, $\Psi$ enumerates $\langle n,0,m\rangle$ if $\langle n,0,m\rangle\in X$ and $2\langle n,m\rangle\in C$, and for $i\not=0$, $\Psi$ enumerates $\langle n,i,m\rangle$ if $\langle n,i,m\rangle$ and $2\langle n,s\rangle+1\in C$ for any $s<m$.

We claim that this procedure $\Psi(X\oplus C\oplus \cmp{C})$ precisely enumerates $A$.
To see this, for $n,m\in\om$, first note that $x(n)=m$ (that is, $\langle n,0,m\rangle\in A$) if and only if $2\langle n,m\rangle\in C$.
Thus, $x(n)=m$ if and only if $\Psi$ enumerates since $\langle n,0,m\rangle\in A\subseteq X$.
Next, $x(n)\geq m$ if and only if $2\langle n,s\rangle+1\in C$ for any $s<m$.
If $x(n)\in\om$, then for each $i\not=0$ and $m\leq x(n)$, we have $\langle n,i,m\rangle\in A\subseteq X$, and therefore $\Psi$ enumerates $\langle n,i,m\rangle$.
If $x(n)\not\in\om$, then this means that $n\in R\cup S$.
Assume that $x(n)=\infty$.
Then $\langle n,1,m\rangle\in A$ for any $m$.
Since $A\subseteq X\in\mathcal{N}$, $\langle n,i,m\rangle\in X$ if and only if $i=1$.
Thus, $\Psi$ enumerates $\langle n,i,m\rangle\in X$ if and only if $i=1$.
Similarly, when $x(n)=\infty_\star$, $\Psi$ enumerates $\langle n,i,m\rangle\in X$ if and only if $i=2$.
Consequently, we have $\Psi(X\oplus C\oplus \cmp{C})=A$.
Hence, $A$ is jump-right-total.
\end{proof}

\subsubsection{Quasi-minimality w.r.t.\ left-/right-totality}

We now describe how to use these total-like properties introduced in Section \ref{sec:quasiminimalspaces} to show some results on quasi-minimality mentioned in Section \ref{sec:4-1-5}.

\begin{lemma}\label{lem:semirec-avoids-totallike}
Let $x$ be a real.
If $x$ is not left-c.e.\ in $\emptyset'$, then $\nbaseb{<}{x}$ is quasi-minimal w.r.t.~$G_\delta$-left-total degrees.
If $x$ is not right-c.e.\ in $\emptyset'$, then $\nbaseb{<}{x}$ is quasi-minimal w.r.t.~jump-right-total degrees.
\end{lemma}

\begin{proof}
Let $x$ be a real which is not left-c.e.\ in $\emptyset'$.
For the first assertion, let $A$ be a $G_\delta$-left-total set, witnessed by $(D_{n,i})_{n,i\in\om}$, and $C$.
That is, let $\mathcal{P}$ be the computable $e$-$G_\delta$ set defined by $X\in\mathcal{P}$ iff for each $n$, there is $i$ such that $D_{n,i}\subseteq X$.
Suppose that $\Phi(\nbaseb{<}{x})=A\in\mathcal{P}$.
By Lemma \ref{propdichotomy}, if $x$ is neither left- nor right-c.e., then $\nbaseb{<}{x}$ is quasi-minimal, and therefore $C$ is computable since $C\oplus\cmp{C}\leq_eA\leq_e\nbaseb{<}{x}$.
Hence, $(D_{n,i})_{n,i\in\om}$ is a computable sequence.

As mentioned above, $\mathcal{P}$ is upward closed, and so is $\Phi^{-1}[\mathcal{P}]$ by monotonicity of the enumeration operator $\Phi$.
We define $\Phi^\ast[\mathcal{P}]=\{z\in\mathbb{R}:\Phi(\nbaseb{<}{z})\in\mathcal{P}\}$, which is upward closed w.r.t.\ the standard ordering on $\mathbb{R}$.
Consider $q=\inf\Phi^\ast[\mathcal{P}]$.
Then, $q\leq x$ since $x\in\Phi^\ast[\mathcal{P}]$.
Note that $q\leq p$ iff $p+\varepsilon\in\Phi^\ast[\mathcal{P}]$ for any $\varepsilon>0$, and therefore, for any rational $p$,
\[p< q\iff\neg(\forall\varepsilon>0)(\forall n)(\exists i)\;D_{n,i}\subseteq\Phi(\nbaseb{<}{p+\varepsilon}).\]

This shows that $\nbaseb{<}{q}$ is $\Sigma^0_2$, and thus, $q$ is left-c.e.\ in $\emptyset'$.
Since $x$ is not $\emptyset'$-left-c.e., we have $q<x$.
Therefore, there is a rational $p$ such that $q<p<x$.
Since $\nbaseb{<}{p}\subseteq\nbaseb{<}{x}$, by monotonicity of an enumeration operator, we have $\Phi(\nbaseb{<}{p})\subseteq\Phi(\nbaseb{<}{x})=A$.
Moreover, we have $p\in\Phi^\ast[\mathcal{P}]$, and therefore, $\Phi(\nbase{p})\in\mathcal{P}$.
Since $p$ is rational, $C$ is computable, and $A$ is $G_\delta$-left-total (via $\mathcal{P}$), we have $A\leq_e\Phi(\nbaseb{<}{p})\leq_e\emptyset$.
Consequently, $\nbaseb{<}{x}$ is quasi-minimal w.r.t.~$G_\delta$-left-total degrees.

For the second assertion, let $A$ be a jump-right-total set, witnessed by $(D_{n})_{n\in\om}$, and $C$.
That is, $(D_n)_{n\in\om}$ is a $C'$-computable sequence, and let $\mathcal{N}$ be the $C'$-computable closed set defined by $X\in\nn$ iff for all $n\in\om$, $D_n\not\subseteq X$.
Suppose that $\Phi(\nbaseb{<}{x})=A\in\mathcal{N}$.
By the same argument as before, $C$ has to be computable, and thus $(D_{n})_{n\in\om}$ is a $\emptyset'$-computable sequence.

As mentioned above, $\mathcal{N}$ is downward closed, and so is $\Phi^{-1}[\mathcal{N}]$ by monotonicity of the enumeration operator $\Phi$.
We define $\Phi^\ast[\mathcal{N}]=\{z\in\mathbb{R}:\Phi(\nbaseb{<}{z})\in\mathcal{N}\}$, which is downward closed w.r.t.\ the standard ordering on $\mathbb{R}$.
Consider $q=\sup\Phi^\ast[\mathcal{N}]$.
Then, $x\leq q$ since $x\in\Phi^\ast[\mathcal{N}]$.
Moreover, by a similar argument as above, for any rational $p$,
\[q<p\iff\neg(\forall\varepsilon>0)(\forall n)\;D_n\not\subseteq\Phi(\nbaseb{<}{p+\varepsilon}).\]

This shows that $\nbaseb{<}{-q}$ is $\Sigma^0_2$, and thus, $q$ is right-c.e.\ in $\emptyset'$.
Therefore, by our assumption on $x$, we have $x<q$.
Thus, there is a rational $p$ such that $x<p<q$.
Since $\nbaseb{<}{x}\subseteq\nbaseb{<}{p}$, by monotonicity of an enumeration operator, we have $A=\Phi(\nbaseb{<}{x})\subseteq\Phi(\nbaseb{<}{p})$.
Moreover, we have $p\in\Phi^\ast[\mathcal{N}]$, and therefore, $\Phi(\nbase{p})\in\mathcal{N}$.
Since $p$ is rational, $C$ is computable, and $A$ is jump-right-total (via $\mathcal{N}$), we have $A\leq_e\Phi(\nbaseb{<}{p})\leq_e\emptyset$.
Consequently, $\nbaseb{<}{x}$ is quasi-minimal w.r.t.~jump-right-total degrees.
\end{proof}

\thmproof{telograph-quasiminimal1}{
Let $\mathbf{d}$ be a semirecursive, non-$\Delta^0_2$ $e$-degree.
As mentioned in Section \ref{sec:deg-point-T0}, the semirecursive degrees are characterized as the $\mathbb{R}_<$-degrees.
Hence, there is a real $x$ such that $\nbaseb{<}{x}\in\mathbf{d}$ and $\nbaseb{<}{x}$ is not $\Delta^0_2$.
In particular, $x$ is not $\emptyset'$-left-c.e.\ or $x$ is not $\emptyset'$-right-c.e.
As seen in Example \ref{exa:telcototal-totallike}, every telograph-cototal $e$-degree is $G_\delta$-left-total and jump-right-total.
Therefore, by Lemma \ref{lem:semirec-avoids-totallike}, $\nbaseb{<}{x}$ is quasi-minimal w.r.t.\ telograph-cototal $e$-degree.
}

\begin{lemma}\label{lem:telog-qmin-3}
There is a $\Delta^0_2$ real $x$ such that $x$ is neither left- nor right-c.e., but $\nbaseb{<}{x}$ is not quasi-minimal w.r.t.~telograph-cototal $e$-degrees.
\end{lemma}

\begin{proof}
We construct a $\Delta^0_2$ real $x$ and a computably inseparable pair $(A,B)$ such that $x$ is not right-c.e., $A\cup B$ is co-c.e., and that any enumeration of $\nbaseb{<}{x}$ computes a separator of $(A,B)$.
Let $r_e$ be the $e$-th right-c.e.~real.
Consider the following requirements:
\begin{align*}
P_e\colon & \Phi_e\mbox{ is total }\Longrightarrow\;\Phi_e\not\in{\rm Sep}(A,B),\\
N_e\colon & \nbaseb{<}{x}\not=r_e,\\
G\colon & (\exists\Gamma)(\forall p)\;{\rm rng}(p)=\nbaseb{<}{x}\;\Longrightarrow\;\Gamma^{p}\in{\rm Sep}(A,B).
\end{align*}

Begin with $x_0=0$, $A_0=\emptyset$, $B_0=\om$, and $\Gamma_0=\emptyset$.
The global strategy $G$ constructs $\Gamma_s$ as follows:
The operator $\Gamma_s$ is a collection of tuples of the form $\langle n,i,p\rangle$, which indicates that $\Gamma(\nbaseb{<}{x})(n)\downarrow=i$ for any $x>p$.
For any $s$, we ensure that only finitely many tuples of the form $\langle n,i,p\rangle$ is enumerated into $\Gamma_s$.
At the beginning of each stage $s$, the global strategy tries to recover the destroyed computations as follows.
Given $n<s$, the $G$-strategy check if there is a rational $p<x_s$ such that $\langle n,i,p\rangle$ is enumerated into $\Gamma_s$ for some $i<2$.
If there is no such $p$, enumerate $\langle n,i,x_s-2^{-2n}\rangle$ into $\Gamma_{s+1}$, where $i=1$ iff $n\in A_s$.

A priority ordering is given by $P_e<N_e<P_{e+1}$, where $S<T$ means that $S$ is a higher priority strategy than $T$.
The $P_e$-action may decrease the value of $x_s$, and the $N_e$-action may increase the value of $x_s$.
At stage $s$, a $P_e$-strategy acts as follows:
\begin{enumerate}
\item[(P1)] Choose a large $n_e\in B_s$.
\item[(P2)] Wait for $\Phi_{e,s}(n_e)\downarrow=0$.
\item[(P3)] Move $n_e$ from $B$ to $A$, that is, define $B_{s+1}=B_s\setminus\{n_e\}$ and $A_{s+1}=A_s\cup\{n_e\}$.
\item[(P4)] Try to destroy the computation of $\Gamma^x(n_e)$ by putting $x_{s+1}=x_{s}-2^{-2n_e+1}$.
\item[(P5)] Injure all lower priority strategies by resetting all parameters.
\end{enumerate}

An $N_e$-strategy acts as follows:
\begin{enumerate}
\item[(N1)] Choose a large $m_e$.
\item[(N2)] Wait for $x_s<r_{e,s}<x_s+2^{-2m_e}$.
\item[(N3)] Put $x_{s+1}=x_s+2^{-2m_e}$, and for a sufficiently large $u$ such that $r_{e,s}<x_{s+1}-2^{-2u}$, we hereafter require that a large number should be bigger than $u$.
\item[(N4)] The above action may cause inconsistent computations on $\Gamma^x(n)$, that is, it is possible to have $p<q<x_{s+1}$ such that both $\langle n,i,p\rangle$ and $\langle n,1-i,q\rangle$ is contained in $\Gamma_{s+1}$.
For all such $n$, remove $n$ from $A\cup B$.
\item[(N5)] Injure all lower priority strategies by resetting all parameters.
\end{enumerate}

If a strategy reaches (P5) or (N5), then the strategy never acts unless it is later initialized.
Thus, every strategy acts once with the same $n_e$ or $m_e$.
The $P$-requirements can only be injured by (N3), and the $N$-requirements can only be injured by (P4).

\begin{claim}
The action (P4) always forces $\Gamma_{s+1}(x_{s+1};n_e)$ to be undefined, that is, there is no $p<x_{s+1}$ and $i$ such that $\langle n_e,i,p\rangle\in\Gamma_{s+1}$.
\end{claim}

\begin{proof}
An action of a higher priority strategy injures $P_e$ at some stage $t<s$, which forces $P_e$ to redefine $n_e$ as a large number $>t$.
At some stage $v$, $\langle n_e,i,x_v-2^{2n_e}\rangle$ is enumerated into $\Gamma_v$.
The $G$-strategy ensures $v>n_e>t$.
After $n_e$ is settled after stage $t$, only a lower priority strategy can act.
Only $N$-strategies can increase the value of $x$.
Therefore, for $m=\sum_{d\geq e}2^{-2m_d}$, and we have $x_s<x_v+m$.
One can see that $m< 2^{-2n_e-1}$, and hence
\[x_{s+1}=x_s-2^{-2n_e+1}<x_v+m-2^{-2n_e+1}<x_v-3\cdot 2^{-2n_e-1}<x_v-2^{2n_e}.\]
This forces $\Gamma_{s+1}(\nbaseb{<}{x_{s+1}};n_e)$ to be undefined.
\end{proof}

Thus, the $G$-strategy can recover the correct value for a separating set.
We would like to make sure that $\Gamma^x$ is total.
Of course, enumerating inconsistent computations makes $\Gamma^x$ be partial.
Hence, instead of dealing with $\Gamma^x$, we consider $\Gamma^p$ for any enumeration $p$ of $\nbaseb{<}{x}$.
The computation $\Gamma^p$ is given as follows.
Let $p$ be an enumeration of $\nbaseb{<}{x}$.
For each $n$, wait for $s$ such that $q<p(s)$ and $\langle\langle n,i\rangle,q\rangle\in\Gamma_s$ for some $i<2$, and then for the first such, we define $\Gamma^p(n)=i$.
By the action of the $G$-strategy, $\Gamma^p$ becomes total.

\begin{claim}
The $G$-requirement is fulfilled.
\end{claim}

\begin{proof}
Straightforward.
\end{proof}

\begin{claim}
The action (N3) only cause inconsistent $\Gamma^x(n)$ for $n\geq m_e$.
Hence, if the $P_e$-strategy acts, and is never injured, then $P_e$ is fulfilled.
\end{claim}

\begin{proof}
Otherwise, $\langle n,1,p\rangle$ is already enumerated.
In this case, we must have $n_d=n$ for some $d$.
Assume that $P$ acts with $n_d<m_e$ at some stage $t<s$.
As in the above argument, at some stage $v$, $\langle n_d,i,x_v-2^{2n_d}\rangle$ is enumerated into $\Gamma_v$.
As seen above, $x_{t+1}<x_v-3\cdot 2^{-2n_d-1}$.
Thus, to make $\Gamma^x(n_d)$ inconsistent, we need to increase $x$ by $2^{-2n_d-1}$.
As before, we have $m<2^{-2n_d-1}$.
\end{proof}

\begin{claim}
If the $N_e$-strategy acts, and is never injured, then the property $r_{e}<x$ is preserved forever.
\end{claim}

\begin{proof}
To see this, note that the $P_e$-action (P3) at stage $s$ only injure lower priority strategies.
This is because, for $d<e$, if $N_d$ has already acted at some stage $t<s$, and not injured until $s$, then $n_e$ must be bigger than $u$ chosen by $N_d$.
Therefore, $r_e\leq r_{e,t}<x_{t+1}-2^{-u}\leq x_s$.
\end{proof}

These claims conclude the proof.
\end{proof}

\thmproof{telograph-quasiminimal3}{
By Lemma \ref{lem:telog-qmin-3}.
}

\subsection{$T_1$-degrees which are not $T_2$}\label{sec:T1-not-T2-qmin}

We give proofs of some important separation results mentioned in Section \ref{sec:4-2}.
We say that a function $f:\omega\to\omega$ is {\em computably dominated} if there is a computable function $h:\omega\to\omega$ such that $f(n)\leq h(n)$ for almost all $n\in\omega$.
We also say that a function $f:\omega\to\omega$ is {\em computably dominating} if for any computable function $h:\omega\to\omega$, $h(n)\leq f(n)$ for almost all $n\in\omega$.

\begin{obs}\label{obs:basic-cylinder-cototal}~
\begin{enumerate}
\item For any $f:\om\to\om$, $\nbaseb{(\om^\om)_{\rm co}}{f}\leq_e\nbaseb{\om^\om}{f}$.
\item If $f:\omega\to\omega$ is computably dominated, then $\nbaseb{(\om^\om)_{\rm co}}{f}\equiv_e\nbaseb{\om^\om}{f}$.
Hence, $\nbaseb{(\om^\om)_{\rm co}}{f}$ has a total degree.
\item There is a computably dominating function $f$ such that $\nbaseb{(\om^\om)_{\rm co}}{f}$ is c.e.
\end{enumerate}
\end{obs}

\begin{proof}
For (1), if we see $\sigma\in\nbaseb{\om^\om}{f}$ (that is, $\sigma\preceq f$), we know that $\tau\not\preceq f$ whenever $\tau$ is incomparable with $\sigma$, and therefore enumerate all such $\tau$ into $\nbaseb{(\om^\om)_{\rm co}}{f}$.
For (2), it suffices to show that $\nbaseb{\om^\om}{f}\leq_e\nbaseb{(\om^\om)_{\rm co}}{f}$.
Assume that $f$ is bounded by a computable function $g$.
Then, for any $\ell$, there are only finitely many strings $\sigma$ of length $\ell$ such that $\sigma(n)<g(n)$ for any $n<\ell$.
Let $I(g,\ell)$ be the set of all such strings.
Then, all but one string in $I(g,\ell)$ is enumerated into $\nbaseb{(\om^\om)_{\rm co}}{f}$.
By finiteness of $I(g,\ell)$, all such strings are enumerated by some finite stage, and then one can know what the unique string $\sigma\in I(g,\ell)\setminus \nbaseb{(\om^\om)_{\rm co}}{f}$ is.
Then, enumerate $\sigma$ into $\nbaseb{\om^\om}{f}$.
This procedure clearly witnesses that $\nbaseb{\om^\om}{f}\leq_e\nbaseb{(\om^\om)_{\rm co}}{f}$.
For (3), let $f(n)$ be the least stage $s$ such that for each $e<n$, if $\Phi_e(e)$ halts then $\Phi_e(e)$ halts by stage $s$.
Clearly $f$ is computably dominating.
Moreover, $\sigma\prec f$ iff for every $n<|\sigma|$ and $e<n$, $\Phi_e(e)$ does not halt or $\Phi_e(e)$ halts by stage $\sigma(n)$.
Clearly, this condition is $\Pi^0_1$.
Therefore, $\nbaseb{(\om^\om)_{\rm co}}{f}$ is c.e.
\end{proof}

\begin{lemma}\label{lem:cocylinder-cetree}
Suppose that $f:\omega\to\omega$ is a function which is not $\emptyset'$-computably dominated.
For every computable sequence $(D_s)_{s\in\om}$ of finite sets of finite strings, if $f\in\bigcap_s[D_s]$ then there is $\ell\in\om$ such that $[f\upto\ell]\subseteq\bigcap_s[D_s]$, i.e., for any $s$ there is $\sigma\preceq f\upto\ell$ such that $\sigma\in D_s$.
\end{lemma}

\begin{proof}
For a given computable sequence $(D_s)_{s\in\om}$ we will define a finite-branching c.e.\ subtree $T$ of $\om^{<\om}$.
Let $T_0$ be the tree only having the root, that is, $T_0=\{\langle\rangle\}$.
Given $T_s$, if $\sigma$ is a leaf of $T_s$ and there is no $\tau\preceq\sigma$ such that $\tau\in D_s$, then enumerate all strings $\alpha$ such that $\sigma\prec\alpha\preceq\tau$ for some $\tau\in D_s$ into $T_{s+1}$.
Define $T=\bigcup_sT_s$.
Note that $T$ is finite-branching.
This is because we only enumerate finitely many elements into $T_{s+1}$ extending a leaf $\sigma\in T_s$ since $D_s$ is finite, and then $\sigma$ never become a leaf of $T_t$ for any $t>s$.
Since $T$ is a finite-branching c.e.\ tree, there is a $\emptyset'$-computable function dominating all infinite paths through $T$.

If $f\in\bigcap_s[D_s]$ but $[f\upto\ell]\not\subseteq\bigcap_s[D_s]$ for all $\ell\in\om$, then we claim that $f$ is an infinite path through $T$.
Assume that $f\upto n$ is a leaf of $T_s$.
Then there is $t\geq s$ such that $f\upto m\not\in D_t$ for any $m\leq n$ since $[f\upto n]\not\subseteq\bigcap_s[D_s]$.
Let $t$ be the first such stage.
Since $f\in\bigcap_s[D_s]$, there is $k>n$ such that $f\upto k\in D_t$.
By our construction, all initial segments of $f\upto k$ are enumerated into $T_{t+1}$.
Consequently, $f$ is an infinite path through $T$; however this implies that $f$ is $\emptyset'$-computably dominated.
\end{proof}

By using the above lemma, we will show that if $f$ is not $\emptyset'$-computably dominated then $\nbaseb{(\om^\om)_{\rm co}}{f}$ is quasi-minimal.
Indeed, the following abstract lemma implies more concrete results.
Let $\nn=(N_e)_{e\in\om}$ be a network.
Then, define we define the {\em disjointness diagram} of $\nn$ as ${\rm Disj}_\nn=\{\langle d,e\rangle:N_d\cap N_e=\emptyset\}$.

\begin{lemma}\label{lem:cocylinder-quasiminimal}
Let $f\colon\omega\to\omega$ be a function which is not $C'$-computably dominated.
For any cs-second-countable space $\mathcal{Z}=(Z,\nn)$ with a $C$-c.e.\ disjointness diagram, $\pt{f}{(\om^\om)_{\rm co}}$ is nearly $\mathcal{Z}$-quasi-minimal.
\end{lemma}

\begin{proof}
Now suppose that $\pt{z}{\mathcal{Z}}\reduce \pt{y}{(\om^\om)_{\rm co}}$ holds where $\mathcal{Z}$ is a given space with a countable cs-network $\nn=(N_e)_{e\in\omega}$.
By Observation \ref{obs:network-e-basis}, there is $J\leq_e\nbaseb{(\om^\om)_{\rm co}}{f}$ such that $\{N_e:e\in J\}$ forms a strict network at $z$.
Let $\Psi$ witness that $J\leq_e\nbaseb{(\om^\om)_{\rm co}}{f}$.
We first note that if there are $d,e$ such that $(d,D)$ and $(e,E)$ are enumerated in $\Psi$ while $N_d$ and $N_e$ are disjoint, then we must have $D\cup E\not\subseteq\nbaseb{(\om^\om)_{\rm co}}{f}$, that is, there is $\sigma\in D\cup E$ such that $\sigma\prec f$.
Enumerate all such tuples $(d_s,e_s,D_s,E_s)_{s\in\omega}$.
Such a $C$-computable enumeration exists since the disjointness diagram of $\nn$ is $C$-c.e.

Then, either this gives a finite sequence, or else $(D_s\cup E_s)_{s\in\om}$ is a $C$-computable sequence such that $f\in\bigcap_s[D_s\cup E_s]$.
In any case, by relativizing Lemma \ref{lem:cocylinder-cetree}, there is $p$ such that $[f\upto p]\subseteq\bigcap_s[D_s\cup E_s]$, that is, for every $s$ there is $\sigma\in D_s\cup E_s$ such that $\sigma\preceq f\upto p$.
We then consider
\[L=\{e\in\omega:(\exists E)\;[(e,E)\in\Psi\mbox{ and }(\forall \sigma\in E)\;\sigma\not\preceq f\upto p]\}.\]

Clearly, $J\subseteq L$, and therefore, $\{N_e:e\in L\}$ forms a network at $z$.
We claim that $z\in\overline{N_e}$ for any $e\in L$.
By our choice of $p$, if $e\in L$, then there is no $(d,D)\in\Psi$ such that $d\in L$ and $N_d\cap N_e=\emptyset$.
If $z\not\in\overline{N_e}$, there is an open neighborhood $U$ of $z$ such that $U\cap N_e=\emptyset$.
Then, there is $d\in L$ such that $z\in N_d\subseteq U$.
In particular, $N_d\cap N_e=\emptyset$, which implies a contradiction.
Consequently, any enumeration of $L$ gives a $\overline{\dnn}$-name of $z$.
Since $L$ is c.e., we conclude that $z$ is $\overline{\dnn}$-computable.
\end{proof}

\thmproof{thm:T-1-degree-not-T-2}{
Let $\xx$ be a topological space with a countable cs-network $\nn$.
Let $C$ be an oracle such that the disjointness diagram of $\nn$ is $C$-c.e.
By relativizing Lemma \ref{lem:cocylinder-quasiminimal}, if $f$ is not $C'$-computably dominated, then $\pt{f}{(\om^\om)_{\rm co}}$ is nearly $\xx$-quasi-minimal, that is, if $\pt{x}{\xx}\reduce \pt{f}{(\om^\om)_{\rm co}}$ then $\pt{x}{\xx}$ is nearly $C$-computable.
If $\xx$ is a Hausdorff space, then only countably many points in $\xx$ can be nearly $C$-computable by Observation \ref{obs:closure-representation-1}.
However, there are uncountably many functions which is not $C'$-computably dominated.
Thus, one can choose such a function which is not $\mathbf{T}$-equivalent to any nearly computable points in $\xx$.
}

\thmproof{thm:T-1-degree-NNN-quasiminimal}{
The canonical network $\nn$ of $\mathbb{N}^{\mathbb{N}^\mathbb{N}}$ has a computable disjointness diagram.
Moreover, ${\rm id}\colon(\mathbb{N}^{\mathbb{N}^\mathbb{N}},\overline{\dnn})\to(\mathbb{N}^{\mathbb{N}^\mathbb{N}},{\dnn})$ is computable as seen in Example \ref{exa:Kleene-Kreisel}.
Therefore, near $\mathbb{N}^{\mathbb{N}^\mathbb{N}}$-quasi-minimality is equivalent to $\mathbb{N}^{\mathbb{N}^\mathbb{N}}$-quasi-minimality by definition.
Now, by Lemma \ref{lem:cocylinder-quasiminimal}, if $f$ is not $\emptyset'$-computably dominated, then $\pt{f}{(\om^\om)_{\rm co}}$ is $\mathbb{N}^{\mathbb{N}^\mathbb{N}}$-quasi-minimal.
}

Applying the above lemmas, we eventually show the following:

\propproof{prop:non-cylinder-cototal}{
Let $A$ be a co-$d$-c.e.\ set relative to $K=\emptyset'$ such that $K\oplus K^{\sf c}\leq_TA$ and the $e$-degree of $A$ is non-total.
Suppose that $A\equiv_e\nbaseb{(\om^\om)_{\rm co}}{f}$ for any $f$.
If $f$ is $\emptyset'$-computably dominated, then by Observation \ref{obs:basic-cylinder-cototal} relative to $\emptyset'$, we have that $\nbaseb{(\om^\om)_{\rm co}}{f}\oplus K\oplus K^{\sf c}$ is total.
Since $K\oplus K^{\sf c}\leq_e\nbaseb{(\om^\om)_{\rm co}}{f}$ by our assumption, $\nbaseb{(\om^\om)_{\rm co}}{f}\equiv_e\nbaseb{(\om^\om)_{\rm co}}{f}\oplus K\oplus K^{\sf c}$ is total, which is impossible since $A$ is nontotal.
If $f$ is not $\emptyset'$-computably dominated, then by Lemma \ref{lem:cocylinder-quasiminimal}, $\nbaseb{(\om^\om)_{\rm co}}{f}$ is quasi-minimal, which is impossible since $A$ is not quasi-minimal.
}

\subsubsection{Cocylinder topology and left-/right-totality}

We next see that the property of jump-right-totality is not shared by (strongly arithmetically named) decidable $T_1$ cb$_0$ spaces.
Indeed, the cocylinder space $(\om^\om)_{\rm co}$ (see Section \ref{section:cocylinder}), one of the simplest decidable $T_1$ cb$_0$ space, is not jump-right-total.
By using a similar idea as in Section \ref{sec:quasiminimalspaces}, we will show the following:

\thmproof{telograph-quasiminimal2}{
For an oracle $C$, we say that $x\in\om^\om$ is {\em $C$-computably dominated} if there is a $C$-computable $y\in\om^\om$ such that $x(n)<y(n)$ for all $n\in\om$.

\begin{lemma}\label{lem:dominated-jumprighttotal}
If $x\in\om^\om$ is not $\emptyset''$-computably dominated, then $\nbaseb{\rm co}{x}$ is quasi-minimal w.r.t.~jump-right-total $e$-degrees.
\end{lemma}

\begin{proof}
Assume that $x\in\om^\om$ is not $\emptyset''$-computably dominated.
By Lemma \ref{lem:cocylinder-quasiminimal}, such $x$ is quasi-minimal.
Let $A$ be a jump-right-total set, witnessed by $C$ and $(D_n)_{n\in\om}$ (generating $\nn$), and assume that $A\leq_e\nbaseb{\rm co}{x}$ via an enumeration operator $\Psi$.
Note that, by quasi-minimality of $x$, the total part $C$ has to be computable.
Thus, $(D_n)_{n\in\om}$ is a $\emptyset'$-computable sequence.

Let $(E_s)$ be a $\emptyset'$-computable enumeration of finite sets $E$ such that $D_n\subseteq\Psi(E)$ for some $n$.
Then, $A\in\nn$ implies that $\Psi(E_s)\not\subseteq A$ for any $s$.
Since $\Psi(\nbaseb{\rm co}{x})=A$, we have $E_s\not\subseteq\nbaseb{\rm co}{x}$, that is, there is $\sigma\prec x$ such that $\sigma\in E_s$.
Since $x$ is not $\emptyset''$-computably dominated, by Lemma \ref{lem:cocylinder-cetree} (relative to $\emptyset'$), there is $\ell$ such that for any $s$, there is $\tau\preceq x\upto\ell$ such that $\tau\in E_s$.
Consider the following c.e.~set:
\[L=\{n:(\exists H)\;[\langle n,H\rangle\in\Psi\mbox{ and }(\forall \sigma\prec x\upto\ell)\;\sigma\not\in H]\}.\]

It is easy to see that $A\subseteq L$ since $\Psi(\nbaseb{\rm co}{x})\subseteq L$.
We claim that $L\in\nn$.
Otherwise, there is $n$ such that $D_n\subseteq L$.
For each $m\in D_n$, there is $H_m$ such that $\langle m,H_m\rangle\in\Psi$ and $\sigma\not\in H_m$ for all $\sigma\prec x\upto\ell$.
Thus, for $E=\bigcup_{m\in D_n}H_m$ we have $D_n\subseteq\Psi(E)$, and hence, $E=E_s$ for some $s\in\om$.
By our choice of $\ell$, there is $\sigma\prec x\upto\ell$ such that $\sigma\in E=E_s$, that is, $\sigma\in H_m$ for some $m\in D_n$.
This contradicts our choice of $H_m$.

We thus obtain $A\subseteq L\in\nn$.
Since $A$ is jump-right-total, $C=\emptyset$, and $L$ is c.e., we conclude that $A\leq_e L\leq_e\emptyset$.
\end{proof}

Choose $x\in\om^\om$ which is not $\emptyset''$-computably dominated.
Then, $\nbaseb{\rm co}{x}$ is cylinder-cototal as seen in Section \ref{section:cocylinder}, and quasi-minimal w.r.t.\ jump-right-total $e$-degrees by Lemma \ref{lem:dominated-jumprighttotal}.
In particular, $\nbaseb{\rm co}{x}$ is quasi-minimal w.r.t.\ telograph-cototal $e$-degree since every telograph-cototal $e$-degree is jump-right-total as seen in Example \ref{exa:telcototal-totallike}.
}

\subsubsection{$T_1$-degrees which are $T_2$-quasi-minimal}\label{sec:t1nott2quasiminimal}

We now give a proof of the main quasi-minimality result in Section \ref{sec:4-2-1}.

\thmproof{thm:T_2-quasiminimal}{
By Theorem \ref{thm:telophase-degree}, it suffices to construct sets $A,B\subseteq\om$ such that $A\cup B$ is co-c.e.\ and ${\rm Sep}(A,B)$ is quasi-minimal with respect to $\{S_i\}_{i\in\omega}$.

\subsubsection*{The requirements}
We construct disjoint sets $A,B\subseteq\om$ and a c.e. set $X$ such that $A\cup B=\omega-X$, $(A,B)$ are computably inseparable, and we have to satisfy the following requirements:
\begin{multline*}
R_e:\quad \forall D_0, D_1\in {\rm Sep}(A,B)~~ W_e^{D_0}=W_e^{D_1} \text{ and } W_e^{D_0}={\rm Nbase}(y) \\
\text{ for some point $y\in S_e$}\Longrightarrow\;\exists \text{ c.e. set }V=W_e^{D_0}.
\end{multline*}
Here $W_e^X$ is the $e^{th}$ c.e. set relative to $X\oplus (\omega-X)$. We will construct a path $Y\in \ab^\omega$. Given any path $Z\in \ab^\omega$ we define $Z^+, Z^-\in 2^\omega$ by $Z^+(n)=1$ iff $Z(n)=a$, and $Z^-(n)=1$ iff $Z(n)=b$. Clearly $Z^+$ and $Z^-$ are the characteristic functions of disjoint sets. (At the end we will take $A=Y^+$ and $B=Y^-$).

We also assume that all approximations of a $\Pi^0_1$-class $Q$ are slowed down such that for every $s$, the complement of $Q_s$ is presented by a finite set of strings of length $<s$.


\subsubsection*{The definition of $X$} We first describe how to construct the c.e. set $X$. We consider requirements $S_{e,\sigma}$ indexed by $e\in\omega$ and $\sigma\in \ab^{<\omega}$, and arrange the requirements in some order of priority. The c.e. set $X$ is constructed by a straightforward finite injury construction, and whenever a requirement $S_{e,\sigma}$ performs any action, all lower priority requirements are initialized and starts afresh with a value of $m$ much larger than before.

The basic activity of requirement $S_{e,\sigma}$ is as follows. First pick a very large number $m_{e,\sigma}$, and for each node $\tau\in \ab^{m_{e,\sigma}}$ do the following. Search for a pair $\eta_0, \eta_1\supset\tau$ and some $i,j\in\omega$ such that $B^{S_e}_i\cap B^{S_e}_j=\emptyset$ and $i\in W_e^{\eta_0^+}$, $j\in W_e^{\eta_1^+}$. For the first pair $\eta_0, \eta_1$ found at some stage $s$, where $|\eta_0|,|\eta_1|<s$, enumerate the interval $\left\{ m_{e,\sigma},\cdots, s\right\}$ into $X_{s+1}$. Note that $S_{e,\sigma}$ will change $X$ up to $3^{m_{e,\sigma}}$ many times (until it is initialized).

It is easy to see that $X$ is c.e. and co-infinite. Now define the closed set $C\subset \ab^\omega$ to consist of precisely the paths $Z\in \ab^\omega$ such that $Z^+ \cup Z^- = \mathbb{N}-X$. Also for each $s$ define $C_s$ to be the set of paths $Z$ such that $\left(Z^+ \cup Z^-\right)\rs{s} =(\mathbb{N}-X_s)\rs{s}$. Note that $C$ is a $\Pi^0_1$-class relative to $\emptyset'$ and each $C_s$ is a $\Pi^0_1$-class. However, $\{C_s\}$ is not monotone.
This non-computability feature of $C$ will cause various difficulties in our construction.
First, clearly we will need to take $Y\in C$, and second we will need to enumerate $V$ without oracles; however, $C_s$ are guesses about the true $C$ which will help us in the construction.

Given a string $Z\in \ab^\omega$ and $n,s\in\omega$, we say that \emph{$Z(n)$ is compatible with $C_s$} if $n\in X_s \Leftrightarrow Z(n)=0$. Similarly we say that \emph{$Z(n)$ is compatible with $C$} if $n\in X \Leftrightarrow Z(n)=0$. Thus the set of all $Z$ such that $Z(n)$ is compatible with $C_s$ for all $n<s$ is the set $C_s$ itself.

\subsubsection*{The class $\ex{Q,\rho}$} Our construction will define a sequence of approximations to  $P_0\supseteq P_1\supseteq P_2\supseteq\cdots$ of $\Pi^0_1$-classes and a sequence of nodes $\rho_0\subset\rho_1\subset\cdots$ with a unique $Y$ such that $\{Y\}=\bigcap_{k\in\omega} P_k$. In fact, we will ensure a little more. Given any $s$ and a node $\rho\in\ab^{<\omega}$, we define $\rho\boxplus C_s$ to contain all strings $Z\in\ab^\omega$ such that $Z\supset \rho$ and for every $n\geq |\rho|$, $Z(n)$ is compatible with $C_s$. We also define for each (approximation to) $Q$ and $\rho$, the class $\ex{Q,\rho}$ by the following.

We define the c.e. set of strings $E=\cup_s E_s$ by the following. First enumerate into $E_0$ all $\sigma$ incompatible with $\rho$. Then at stage $s+1$, enumerate all strings $\sigma$, where $|\sigma|<s$ and $\sigma\supseteq\rho$ into $E_{s+1}$ if $(\sigma\boxplus C_{s+1})\subseteq E_s\cup\left(\ab^\omega-Q[s]\right)$.

Now take $\ex{Q,\rho}=\ab^\omega-[E]$. Clearly $\ex{Q,\rho}\subseteq Q$ is still a $\Pi^0_1$-class. Note that $\ex{Q,\rho}$ depends on the approximation $\{Q[s]\}$ of $Q$; different approximations of $Q$ may give rise to different versions of $\ex{Q,\rho}$. Our aim in the construction is to define approximations to $P_k$ and ensure that $\{Y\}=\bigcap_{k\in\omega} P_k=\bigcap_{k\in\omega}\ex{P_k,\rho_k}\cap C$. For convenience, we will denote $[E_s]$ by $\exc{Q,\rho}[s]$ and $\mathtt{Ex}(Q,\rho)[s]=\ab^\omega-\exc{Q,\rho}[s]$.

We now fix some conventions about the approximations to $\Pi^0_1$-classes. Given an approximation $Q[s]$ of $Q$ and a string $\alpha$, define the natural approximation $(Q\cap[\alpha])[s]$ of $Q\cap[\alpha]$ by taking $\ab^\omega-(Q\cap[\alpha])[s]=\left(\ab^\omega-Q[s]\right)\cup\left(\ab^\omega-[\alpha]\right)$ for all $s$. If $P[s]$ and $Q[s]$ are approximations to $P$ and $Q$, then the natural approximation to $P\cap Q$ is given by $\ab^\omega-(P\cap Q)[s]=\left(\ab^\omega-P[s]\right)\cup\left(\ab^\omega-Q[s]\right)$ for all $s$.


\subsubsection*{The initial condition $P_0,\rho_0$}
It is easy to see that there is a $\Pi^0_1$-class $P_0\subset \ab^\omega$ such that for every $Z\in P_0\cap C$, $Z^+$ and $Z^-$ are computably inseparable. $P_0$ can be constructed from an effective approximation to $X$ and the knowledge that $X$ is coinfinite. For each $s$, let $x_{0,s}<x_{1,s}<\cdots$ be the elements of $\omega- X_s$ listed in increasing order. Notice that as $X$ is constructed by dumping an entire interval into $X$ at each stage, the approximation above has the property that if $x_{i,s+1}\neq x_{i,s}$ then $x_{i,s+1}>s$.

Now at stage $s$, for each $i<s$ such that $\varphi_i(x_{i,s})\downarrow=j$, enumerate $[\sigma*a]$ into $\ab^\omega- P_0$ if $j=1$ and enumerate $[\sigma*b]$ into $\ab^\omega- P_0$ if $j=0$ for each $\sigma\in 3^{x_{i,s}}$. It is easy to see that the definition of $P_0$ above ensures that for any $Z\in P_0\cap C$, there is no computable set $R$ such that $Z^+\subseteq R$ and $R\cap Z^-=\emptyset$. Furthermore, by the observation above that if $x_{i,s+1}\neq x_{i,s}$ then $x_{i,s+1}>s$, we see that each level only has nodes removed from $P_0$ at most once. Therefore $P_0\cap C$ and $P_0\cap C_s$ are nonempty for any $s$.
Note that $P_0$ and $C_s$ are both homogeneous, that is, given any string $\alpha$ which is extendible to an infinite path in $P_0$, and any $s$, there is always an infinite extension $\alpha*Y$ of $\alpha$ such that $\alpha*Y\in P_0\cap (\alpha\boxplus C_s)$.


We now check that $\ex{P_0,\langle\rangle}=P_0$: Suppose there is some $Z\in P_0\cap \exc{P_0,\langle\rangle}$. Let $s$ be the least stage such that there is some $Z\in P_0\cap \exc{P_0,\langle\rangle}[s]$. Suppose $Z\rs{k}\in\exc{P_0,\langle\rangle}[s]$ (here we make the obvious identification between finite strings and the open sets they generate). Now apply homogeneity of $P_0$ above to $\alpha=Z\rs{k}$, and we get a contradiction to the fact that $Z\rs{k}\in\exc{P_0,\langle\rangle}[s]$ (and the minimality of $s$). This shows that $\ex{P_0,\langle\rangle}=P_0$, and in particular,
\[\ex{P_0,\langle\rangle}\cap C\neq \emptyset.\]
Obviously we shall take $\rho_0=\langle\rangle$.

\subsubsection*{Forcing $R_e$ and the condition $P_e,\rho_e$}

Now assume that at step $e$ we are given (approximations to) a sequence $P_0\supset P_1\supset\cdots\supset P_{e-1}$ of $\Pi^0_1$-classes such that $\ex{P_i,\rho_i}\cap C\neq\emptyset$, and a sequence $\rho_0\subset \rho_1\subset\cdots \subset\rho_{e-1}$ such that for every $i<e$, we have $P_i\subseteq [\rho_i]$.
We assume $P_i\subset\ab^\omega$ and $\rho_i\in\ab^{<\omega}$  for every $i<e$. Our aim is to define $P_e$ and $\rho_e$ such that the following condition ($\star$) holds:
\begin{multline*}
(\star)~~~ P_e[s]\subset P_{e-1}[s]\text{ for every $s$}, ~~\ex{P_e,\rho_e}\cap C\neq\emptyset,  \\~~\ex{P_{e-1},\rho_{e-1}}\cap [\rho_e]\neq\emptyset,  \text{ and } \rho_e\supset \rho_{e-1}. \text{ Furthermore }\\
\text{if $Z\in \ex{P_e,\rho_e}\cap C$ then $R_e$ is satisfied along $Z$}.
\end{multline*}

We begin with a technical lemma.

\begin{lemma}\label{lem:compatible}
Let $Q_0$ and $Q_1$ be $\Pi^0_1$-classes with approximations such that $Q_0[s]\cap[\alpha_1]\subseteq Q_1[s]\cap[\alpha_1]$ for every $s$, and where $\alpha_0\subseteq\alpha_1$. Then
\[\ex{Q_0,\alpha_0}\cap[\alpha_1]\subseteq\ex{Q_1,\alpha_1}.\]
\end{lemma}
\begin{proof}
We prove by induction on $s$ that \[\exc{Q_1,\alpha_1}[s]\subseteq\exc{Q_0,\alpha_0}[s]\cup\left(\ab^\omega-[\alpha_1]\right).\]
For $s=0$, this is obviously true. Now suppose that $\sigma$ is enumerated in $\exc{Q_1,\alpha_1}[s+1]$ at that stage. This means that $\sigma\supseteq\alpha_1$, and
\begin{multline*}
(\sigma\boxplus C_{s+1}) \subseteq \exc{Q_1,\alpha_1}[s] \cup \left(\ab^\omega-Q_1[s]\right)\\
 \subseteq \exc{Q_0,\alpha_0}[s] \cup \left(\ab^\omega-Q_0[s]\right) \cup \left(\ab^\omega-[\alpha_1]\right).
\end{multline*}
As $\sigma\supseteq\alpha_1$, this implies that
\[ (\sigma\boxplus C_{s+1}) \subseteq \exc{Q_0,\alpha_0}[s] \cup \left(\ab^\omega-Q_0[s]\right).\]
Thus $\sigma\in \exc{Q_0,\alpha_0}[s+1]$.
\end{proof}

Let $\hat{m}$ be the final parameter used by $S_{e,\rho_{e-1}}$ during the construction of the c.e. set $X$. Since $\ex{P_{e-1},\rho_{e-1}}\cap C\neq\emptyset$, there must be some $\tau\in\ab^{\hat m}$ such that $\ex{P_{e-1},\rho_{e-1}}\cap C\cap[\tau]\neq\emptyset$ and $\tau\supset \rho_{e-1}$. Fix any such $\tau$. We will now meet requirement $R_e$ in this cone. For each $n\in\omega$ define
\[T_n=\left\{ Z\in\ab^\omega\mid n\not\in W_e^{Z^+} \right\}.\]
Then $T_n$ is a $\Pi^0_1$-class for each $n\in\omega$, and we fix an approximation $T_n[s]$ of $T_n$. There are two cases.

\emph{Case 1: Assume that there exists some $n\in\omega$ such that $\mathtt{Ex}(T_n\cap P_{e-1} \cap [\tau],\tau)\cap C\neq\emptyset$, and there exists some $\alpha\supseteq\tau$ such that $[\alpha]\cap T_n=\emptyset$ and $\ex{P_{e-1},\rho_{e-1}}\cap[\alpha]\cap C \neq\emptyset$.} Fix $n$ and $\alpha$ as in the case assumption. There are now three subcases.
\begin{description}
\item[Subcase 1.1] For every $Z\in \ex{[\alpha]\cap P_{e-1},\alpha}\cap C,~~ W_e^{Z^+} \neq {\rm Nbase}(y)$ for any point $y\in S_e$. In this case we take $P_e=[\alpha]\cap P_{e-1}$ and $\rho_e=\alpha$. Now note that by Lemma \ref{lem:compatible}, we have $\emptyset\neq\ex{P_{e-1},\rho_{e-1}}\cap[\alpha]\cap C \subseteq \ex{[\alpha]\cap P_{e-1},\alpha}\cap C$.

Then we clearly have condition $(\star)$ holds. Note that we have $Z^+\in{\rm Sep}(Z^+, Z^-)$.
\item[Subcase 1.2] For every $Z\in \mathtt{Ex}(T_n\cap P_{e-1} \cap [\tau],\tau)\cap C,~~ W_e^{Z^+} \neq {\rm Nbase}(y)$ for any point $y\in S_e$. In this case we take $P_e= T_n\cap P_{e-1}\cap[\tau]$ and take $\rho_e=\tau$. Then we clearly also have $(\star)$. Note that in this case $\ex{P_{e-1},\rho_{e-1}}\cap[\tau]\neq\emptyset$ by the choice of $\tau$.
\item[Subcase 1.3] Otherwise. This means that there exists $Z_1\in [\alpha]$ and $Z_2\in  T_n\cap[\tau]$ such that $W_e^{Z_1^+} = {\rm Nbase}(y_1)$ and $W_e^{Z_2^+} ={\rm Nbase}(y_2)$ for some  points $y_1, y_2$ in $S_e$. Since $n\in W_e^{Z_1^+}-W_e^{Z_2^+}$, hence $y_1\neq y_2$. Since $S_e$ is a $T_2$ space, this means that there are disjoint balls $B_k$ and $B_l$ such that $k\in W_e^{Z_1^+}$ and $l\in W_e^{Z_2^+}$.

    This means that during the construction of $X$, the requirement $S_{e,\rho_{e-1}}$ must have found a pair $\eta_0, \eta_1\supset\tau$ successfully where $\eta_0^+$ and $\eta_1^+$ enumerate disjoint $S_e$-balls. Hence we must have the entire interval $\{\hat{m}, \cdots, \max\{|\eta_0|,|\eta_1|\}\}\subset X$. Note that the pair $\eta_0,\eta_1$ found in the construction for $X$ might not be along $Z_1$ and $Z_2$, and in fact, they might not even be extendible in $P_{e-1}$ or $C$, but this will not matter, as we will soon explain.

In this subcase we take $P_e= P_{e-1}\cap [\tau]$ and $\rho_e=\tau$. Note that as $\ex{P_{e-1},\rho_{e-1}}\cap C\cap[\tau]\neq\emptyset$  by the choice of $\tau$ and by Lemma \ref{lem:compatible}, $\ex{P_{e-1},\rho_{e-1}}\cap[\tau]\subseteq\ex{P_{e-1}\cap[\tau],\tau}$, we have that $\ex{P_e,\rho_e}\cap C\neq\emptyset$.
	
		Take any $Z\in \ex{P_{e-1}\cap [\tau],\tau}\cap C$. Since $Z\supset \tau$ and $Z\in C$, this means that
    \begin{align*}
    \eta^+_0 * Z^+(|\eta_0|) Z^+(|\eta_0|+1) Z^+(|\eta_0|+2) \cdots &\in {\rm Sep}(Z^+, Z^-) \text{ and }\\
    \eta^+_1 * Z^+(|\eta_1|) Z^+(|\eta_1|+1) Z^+(|\eta_1|+2) \cdots &\in {\rm Sep}(Z^+, Z^-),
    \end{align*}
    which means that $R_e$ is met along any such $Z$. Hence condition $(\star)$ holds.
\end{description}

\emph{Case 2: No such $n$ in Case 1 exists.} This means that for every $n\in\omega$, one of the following holds:
\begin{description}
\item[(I)] $\mathtt{Ex}(T_n\cap P_{e-1} \cap [\tau],\tau)\cap C=\emptyset$, or
\item[(II)] For every $\alpha\supseteq\tau$ such that $[\alpha]\cap T_n=\emptyset$, we have $\ex{P_{e-1},\rho_{e-1}}\cap[\alpha]\cap C =\emptyset$.
\end{description}
In this case we take $P_e=P_{e-1}\cap[\tau]$ and $\rho_e=\tau$. Again by Lemma \ref{lem:compatible}, $\ex{P_e,\rho_e}\cap C\neq\emptyset$. It only remains to check that $R_e$ is met along all $Z\in \mathtt{Ex}(P_{e},\rho_e)\cap C$.

Define the c.e. set $V$ as follows:
\[V=\{n\in\om:(\exists s)\;\mathtt{Ex}(T_{n}\cap P_{e-1}\cap [\tau],\tau)[s]=\emptyset\}.\]
  We may assume that at stage $s$, $X_s\upharpoonright{|\tau|}=X\upharpoonright{|\tau|}$.

\begin{lemma}\label{Iholds}
For each $n$, if (I) holds, then $n\in V$.
\end{lemma}
\begin{proof}
Fix $n$ such that $\mathtt{Ex}(T_n\cap P_{e-1} \cap [\tau],\tau)\cap C=\emptyset$. By compactness, fix $l>|\tau|$ such that $\mathtt{Ex}(T_n\cap P_{e-1} \cap [\tau],\tau)[l-1]\cap C_l=\emptyset$. We want to verify that $\mathtt{Ex}(T_n\cap P_{e-1} \cap [\tau],\tau)[l]=\emptyset$. Suppose for a contradiction that there is some $Z\in \mathtt{Ex}(T_n\cap P_{e-1} \cap [\tau],\tau)[l]$.

We have $Z\not\in C_l$, and we let $k$ be the least number such that $Z(k)$ is not compatible with $C_l$. We know that $k<l$ since the complement of $C_l$ can be presented by strings of length less than $l$. Furthermore we also know that $k\geq|\tau|$: Since $Z\in \mathtt{Ex}(T_n\cap P_{e-1} \cap [\tau],\tau)[l]$, we have $Z\supset \tau$, but since $\tau$ is compatible with $C$ (as $[\tau]\cap C\neq\emptyset$) and we can assume $l$ is large enough so that $X_l\upharpoonright{|\tau|}=X\upharpoonright{|\tau|}$, we have $k\geq|\tau|$.

Now applying the definition of  $Z\in \mathtt{Ex}(T_n\cap P_{e-1} \cap [\tau],\tau)[l]$, we get that  \[\Big(\left(Z\rs{k}\boxplus C_l\right)\cap T_{n,l}\cap P_{e-1,l}\cap[\tau]\Big)-\exc{T_n\cap P_{e-1} \cap [\tau],\tau}[l-1] \neq\emptyset.\] Fix $\hat{Z}$ in the set above.
By the definition of $k$, we know that $\hat{Z}\in C_l$. At the same time, $\hat{Z}\not\in\mathtt{Ex}^c(T_n\cap P_{e-1} \cap [\tau],\tau)[l-1]$. This is a contradiction to the assumption that $\mathtt{Ex}(T_n\cap P_{e-1} \cap [\tau],\tau)[l-1]\cap C_l=\emptyset$.
\end{proof}

Now fix $Z\in \mathtt{Ex}(P_{e-1}\cap [\tau],\tau)\cap C$, and we want to argue that requirement $R_e$ is met along $Z$.
Obviously we begin by assuming that $ \forall D_0, D_1\in {\rm Sep}(Z^+,Z^-),~~ W_e^{D_0}=W_e^{D_1}$ and $W_e^{D_0}={\rm Nbase}(y)$ for some point $y\in S_e$. We wish to now verify that $V=W_e^{Z^+}$.

\begin{lemma} $V\subseteq W_e^{Z^+}$.
\end{lemma}
\begin{proof}
Suppose that $n\in V$. This means that there is a stage $s$ such that $\ex{T_{n}\cap P_{e-1}\cap [\tau],\tau}[s]=\emptyset$, and so $Z\not\in \ex{T_{n}\cap P_{e-1}\cap [\tau],\tau}[s]$. This means there is a $k\geq |\tau|$ and some $t\leq s$ such that
$Z\rs{k}$ is enumerated in $\exc{T_{n}\cap P_{e-1}\cap [\tau],\tau}[t]$ at stage $t$.
We fix $t$ to be the smallest stage which enumerates some initial segment of $Z$ this way.

Let $x_m=\min (X-X_t)\upharpoonright{t}$. First of all, if $x_m$ does not exist then $X\upharpoonright{t}=X_t\upharpoonright{t}$, and so $Z(l)$ is compatible with $C_t$ for all $l<t$. This means that $\left[Z\upharpoonright{t}\right]\cap T_{n,t}\cap P_{e-1,t}\cap[\tau]\cap \ex{T_{n}\cap P_{e-1}\cap [\tau],\tau}[t-1]=\emptyset$ and we apply the fact that
$\left[Z\upharpoonright{t}\right]\cap P_{e-1,t}\neq\emptyset$ and the minimality of $t$ to conclude that $\left[Z\upharpoonright{t}\right]\cap T_{n,t}=\emptyset$. Thus, $Z\not\in T_n$ and hence $n\in W_e^{Z^+}$.

So we will assume that $x_m<t$ exists. By the construction of $X$, since $x_m$ is enumerated after stage $t+1$, we observe that $\{x_m,\cdots,t\}\subseteq X$.
Since $Z(l)$ is compatible with $C_t$ for all $l<x_m$, we may assume that $k\geq x_m$, because otherwise we have $Z\rs{x_m}\boxplus C_t\subseteq Z\rs{k}\boxplus C_t$ and we can use $x_m$ in place of $k$. So we assume the order $x_m\leq k <t$.

Now suppose that there is some $\alpha\in\ab^{t-k}$ such that $\left[(Z\upharpoonright{k})*\alpha\right]\cap T_{n}=\emptyset$.
Take $D= \Big(Z\upharpoonright{k}*\alpha* Z(t)Z(t+1)Z(t+2)\cdots\Big)^+$, and as $x_m\leq k<t$ and $\{x_m,\dots,t\}\subseteq X$, we see that $D\in{\rm Sep}(Z^+,Z^-)$. So $n\in W_e^{D}=W_e^{Z^+}$. So we suppose that no such $\alpha$ exists.

Now we prove by induction on $v\leq t$ that for every $\sigma$ enumerated into $\exc{T_{n}\cap P_{e-1}\cap [\tau],\tau}[v]$ such that $\sigma\supseteq Z\rs{k}$, we have that $\sigma$ is also enumerated into $\exc{P_{e-1}\cap [\tau],\tau}[v]$. At $v=0$ this is trivially true. Now suppose that $\sigma$ is enumerated into $\exc{T_{n}\cap P_{e-1}\cap [\tau],\tau}[v]$ at stage $v$, and that $\sigma\supseteq Z\rs{k}$. This means that $\sigma\boxplus C_v$ is covered by $\exc{T_{n}\cap P_{e-1}\cap [\tau],\tau}[v-1]$ and the complements of $T_{n,v}$ and $P_{e-1,v}$. But since $\alpha$ above is assumed not to exist, and $v\leq t$, this means that $\sigma\boxplus C_v$ is covered by $\exc{T_{n}\cap P_{e-1}\cap [\tau],\tau}[v-1]$ and the complement of $P_{e-1,v}$. But any $\sigma'$ enumerated in $\exc{T_{n}\cap P_{e-1}\cap [\tau],\tau}[v-1]$ cannot have $\sigma'\subseteq Z\rs{k}$ by the minimality of $t$. Thus by induction hypothesis, we see that $\sigma\boxplus C_v$ is in fact covered by $\exc{P_{e-1}\cap [\tau],\tau}[v-1]$ and the complement of $P_{e-1,v}$. Thus, $\sigma$ is in $\exc{P_{e-1}\cap [\tau],\tau}[v]$. This concludes the induction.

By our choice of $t$ and $k$, we have that $Z\rs{k}$ is enumerated in $\exc{T_{n}\cap P_{e-1}\cap [\tau],\tau}[t]$. By our induction above, we see that $Z\rs{k}$ is also enumerated in $\exc{P_{e-1}\cap [\tau],\tau}[t]$. However, recall that we had assumed that $Z\in \mathtt{Ex}(P_{e-1}\cap [\tau],\tau)$, and thus we have a contradiction.
\end{proof}

\begin{lemma}\label{compwithex}
Let $Q_0$ and $Q_1$ be $\Pi^0_1$-classes with approximations such that $Q_0[s]\subseteq Q_1[s]$ for every $s$, and $\alpha_0\supseteq\alpha_1$ such that $[\alpha_0]\cap \ex{Q_1,\alpha_1}\neq\emptyset$. Then
\[\ex{Q_0,\alpha_0}\subseteq\ex{Q_1,\alpha_1}.\]
\end{lemma}
\begin{proof}
We proceed by induction on $s$, the statement
\[\exc{Q_1,\alpha_1}[s]\subseteq\exc{Q_0,\alpha_0}[s].\]
For $s=0$ it is surely trivial as $\alpha_1\subseteq\alpha_0$. At stage $s+1$, suppose that $\sigma$ is enumerated in $\exc{Q_1,\alpha_1}[s+1]$. If $\sigma\supseteq\alpha_0$ then we apply the induction hypothesis to get $\sigma\in\exc{Q_0,\alpha_0}[s+1]$. If $\sigma$ is incomparable with $\alpha_0$ then $[\sigma]\subseteq \exc{Q_0,\alpha_0}[0]$. If $\sigma\subset\alpha_0$ then we get $[\alpha_0]\cap\ex{Q_1,\alpha_1}=\emptyset$, contrary to the assumption.
\end{proof}


Finally we check that $W_e^{Z^+}\subseteq V$. Fix any $n$ and suppose that (II) holds for $n$. If $n\in W_e^{Z^+}$, i.e., $Z\not\in T_n$, then by taking $\alpha=Z\rs{k}$ for some appropriate $k$ in (II), we see that $\ex{P_{e-1}, \rho_{e-1}}\cap[Z\rs{k}]\cap C=\emptyset$. By Lemma \ref{compwithex} we see that
\[\ex{P_{e-1}\cap[\tau],\tau}\subseteq \ex{P_{e-1}, \rho_{e-1}}.\]
As $Z\in \ex{P_{e-1}\cap[\tau],\tau}\cap C$, we conclude that $Z\in \ex{P_{e-1}, \rho_{e-1}}\cap[Z\rs{k}]\cap C$, a contradiction.

Thus if $n\in W_e^{Z^+}$ then (II) does not hold for $n$, which means that (I) has to hold, and by Lemma \ref{Iholds} we see that $n\in V$. Thus $W_e^{Z^+}\subseteq V$. This shows that once again condition $(\star)$ is met in Case 2.

We produce a sequence $P_0\supseteq P_1\supseteq P_2\supseteq\cdots$ of $\Pi^0_1$-classes and a sequence of nodes $\rho_0\subset\rho_1\subset\cdots$. At the end we take $Y=\bigcup_k \rho_k$, and $A=Y^+$ and $B=Y^-$. By condition ($\star$) and Lemma \ref{compwithex} we have $\ex{P_k,\rho_k}\supseteq\ex{P_{k+1},\rho_{k+1}}$ for every $k$. We also see that $Y\in \ex{P_k,\rho_k}\cap C$ for every $k$. Thus requirement $R_k$ is met along $Y$ for every $k$.}

\subsubsection{Continuous degrees}

We show some results on continuous degrees mentioned in Section \ref{sec:4-2-2}.

\propproof{thm:continuous-cospec}{
The {\em cospectrum} of a point $x\in X$ is the set of all $z\in 2^\om$ such that $z\reduce x$ (cf.\ Kihara-Pauly \cite{KP}).
Equivalently, the cospectrum of $x\in X$ is the following set:
\[\{Z\subseteq\om:Z\oplus Z^{\sf c}\leq_e\nbaseb{X}{x}\}.\]
If the cospectrum is closed under the Turing jump, it is called a jump ideal.

\begin{lemma}
\label{lemma:telephoasejumpgap}
For every $x \in (\hat{\om}_{TP})^\om$ there is a $y \in 2^\om$ with $y \leq_T \pt{x}{(\hat{\om}_{TP})^\om} \leq_T y'$.
\begin{proof}
Given $x\in(\hat{\om}_{TP})^\om$, let $c(x)\in \hat{\om}^\om$ be its total information, that is, $c(x)(n)=x(n)$ if $x(n)\in\om$; otherwise $c(x)(n)=\infty$.
Since $c(x)$ is an element of the one-point compactification of $\om$, by Observation \ref{obs:one-point-compactification}, $c(x)$ is total.
It is clear that $c(x)\reduce x$.
Given $X$, let $X_{d}$ denote the space whose underlying space is the same as $X$, but its topology is endowed by the discrete topology.
By asking to the jump of $c(x)$, for each $n$, whether $c(x)(n)$ converges to $\infty$ or not, one can easily see that $c(x)'$ computes $\pt{c(x)}{(\hat{\om}_{d})^\om}$.
Then it is not hard to see that the pair $\pt{(x,c(x))}{(\hat{\om}_{TP})^\om\times(\hat{\om}_{d})^\om}$ computes $\pt{x}{((\hat{\om}_{TP})_d)^\om}$, which is total.
\end{proof}
\end{lemma}

\begin{lemma}\label{lem:telo-cospectrum}
There is no telograph-cototal $e$-degree whose cospectrum is a jump ideal.
\end{lemma}

\begin{proof}
If $x$ is total, then its cospectrum must be a principal Turing ideal, and thus it cannot be a jump ideal.
If $x$ is not total, then for the witness $y$ from Lemma \ref{lemma:telephoasejumpgap} we have $y \leq_Tx$ but $y'\not\leq_Tx$.
This implies that the cospectrum of $x$ is not closed under the Turing jump.
\end{proof}

\begin{lemma}\label{lem:cylin-cospectrum}
There is no cylinder-cototal $e$-degree whose cospectrum is a jump ideal.
\end{lemma}

\begin{proof}
Let $f\in \om^\om_{\rm co}$ be given.
If $f$ is not $\emptyset'$-computably dominated, then by Lemma \ref{lem:cocylinder-quasiminimal}, $f$ is quasi-minimal (see also the proof of Theorem \ref{thm:T-1-degree-NNN-quasiminimal}).
If $f$ is $\emptyset'$-dominated, by relativizing Observation \ref{obs:basic-cylinder-cototal}, $f\oplus \emptyset'$ computes $\pt{f}{\om^\om}$.
Hence, either $f$ is total or the cospectrum of $f$ does not contains $\emptyset'$.
In any case, the cospectrum of $f$ cannot be a jump ideal.
\end{proof}

Miller \cite{miller2} showed that every countable Scott ideal is realized as a cospectrum of a point in the Hilbert cube.
Thus, take a countable jump ideal $\mathcal{I}$, and choose $x\in[0,1]^\om$ whose cospectrum is $\mathcal{I}$.
Then, by Lemmas \ref{lem:telo-cospectrum} and \ref{lem:cylin-cospectrum}, the $e$-degree of $\nbase{x}$ is continuous, but neither telograph-cototal nor cylinder-cototal.
}


\subsection{$T_2$-degrees which are not $T_{2.5}$}\label{sec:T2-not-T25-qmin}

To prove Theorems \ref{thm:T-2-degree-not-T-25} and \ref{thm:T-2-degree-NNN-quasiminimal} mentioned in Section \ref{sec:T2-deg-nT25list}, we need a special property of the relatively prime integer topology.
We say that a space $\xx$ is {\em nowhere $T_{2.5}$} if for any open sets $U,V\subseteq\xx$, $\overline{U}\cap\overline{V}$ is nonempty.

\begin{fact}[see Steen-Seebach {\cite[II.60]{CTopBook}}]\label{fact:relative-prime-basic}
$b\mathbb{Z}\subseteq\overline{a+b\mathbb{Z}}$, and therefore ${\rm lcm}(b,d)\mathbb{Z}\subseteq\overline{a+b\mathbb{Z}}\cap\overline{c+d\mathbb{Z}}$ in the relatively prime integer topology.
In particular, the relatively prime integer topology is nowhere $T_{2.5}$.
\end{fact}

Instead of dealing with $\mathbb{N}_{\rm rp}$, we consider any countable, second-countable, nowhere $T_{2.5}$ space $\mathcal{H}$, and conclude that Theorems \ref{thm:T-2-degree-not-T-25} and \ref{thm:T-2-degree-NNN-quasiminimal} hold true for $\mathcal{H}^\om$ relative to some oracle.
Combining the argument in \cite{KP}, this, in particular, implies that the $\om$-power $\mathcal{H}^\om$ of a countable, second-countable, nowhere $T_{2.5}$ space cannot be written as a countable union of $T_{2.5}$ subspaces.

Now we modify the closure representation argument.
Given a network $\nn$ of a space $\xx$, we consider the following representation $\widetilde{\dnn}$ defined as follows.
We say that $p$ is a $\widetilde{\dnn}$-name of $x$ if and only if
\[\{N_{p(n)}:n\in\om\}\mbox{ is a network at $x$, and }(\forall m,n)\;\overline{N_{p(m)}}\cap\overline{N_{p(n)}}\not=\emptyset.\]

Here recall that an element of a network at $x$ does not need to contain $x$.

\begin{obs}\label{obs:T-25network}
Let $\nn$ be a network of $\xx$.
\begin{enumerate}
\item The identity map ${\rm id}:(\xx,\overline{\dnn})\to(\xx,\widetilde{\dnn})$ is always computable.
\item If $\nn$ is a regular-like network, then ${\rm id}:(\xx,\widetilde{\dnn})\to(\xx,\overline{\dnn})$ is computable.
\item If $\xx$ is $T_{2.5}$, then $\widetilde{\dnn}$ is single-valued.
\item If $\xx$ is Hausdorff, and $\nn$ is regular-like, then $\widetilde{\dnn}$ is single-valued.
\end{enumerate}
\end{obs}

\begin{proof}
For (2), let $p$ is a $\widetilde{\dnn}$-name of $x$.
We show that $p$ is also a $\overline{\dnn}$-name of $x$.
Suppose for the sake of contradiction that $x\not\in \overline{N_{p(k)}}$ for some $k\in\om$.
Since $p$ is a $\widetilde{\dnn}$-name of $x$, $\overline{N_{p(k)}}$ intersects with $\overline{N_{p(n)}}$ for all $n$.
Moreover, since $\nn$ is regular-like, $\{\overline{N_{p(n)}}:n\in\om\}$ is a network at $x$.
Therefore, $\overline{N_{p(k)}}$ must intersect with all open neighborhoods of $x$.
However, $\xx\setminus\overline{N_{p(k)}}$ is an open neighborhood of $x$ since $x\not\in \overline{N_{p(k)}}$.
Hence, $p$ is a $\overline{\dnn}$-name of $x$.

For (3), assume that $\xx$ is $T_{2.5}$, and $p$ is a $\widetilde{\dnn}$-name of $x$ and $y$.
If $x\not=y$, there are open sets $U,V\subseteq\xx$ such that $x\in U$, $y\in V$, and $\overline{U}\cap \overline{V}=\emptyset$.
Since $\{N_{p(n)}:n\in\om\}$ is a network at $x$ and $y$, there are $d,e\in\om$ such that $x\in N_{p(d)}\subseteq U$ and $y\in N_{p(e)}\subseteq V$.
However, this implies that $\overline{N_{p(d)}}\cap\overline{N_{p(e)}}=\emptyset$.
Then, $p$ cannot be a $\widetilde{\dnn}$-name.

For (4), assume that $\xx$ is Hausdorff and regular-like, and $p$ is a $\widetilde{\dnn}$-name of $x$ and $y$.
If $x\not=y$, there are open sets $U,V\subseteq\xx$ such that $x\in U$, $y\in V$, and $U\cap V=\emptyset$.
Since $\nn$ is regular-like, and $\{N_{p(n)}:n\in\om\}$ is a network at $x$ and $y$, $\{\overline{N_{p(n)}}:n\in\om\}$ is also a network at $x$ and $y$.
Therefore, there are $d,e\in\om$ such that $x\in \overline{N_{p(d)}}\subseteq U$ and $y\in \overline{N_{p(e)}}\subseteq V$.
However, this implies that $\overline{N_{p(d)}}\cap\overline{N_{p(e)}}=\emptyset$.
Then, $p$ cannot be a $\widetilde{\dnn}$-name.
\end{proof}

By Observation \ref{obs:T-25network}, if either $\xx$ is $T_{2.5}$ or $\xx$ is Hausdorff and $\nn$ is regular-like, then there are only countably many points $x$ such that $\pt{x}{\widetilde{\dnn}}$ is computable.

\begin{definition}
We say that a point $x\in\xx$ is {\em $\widetilde{\ast}$-nearly computable} if $x$ is $\widetilde{\dmm}$-computable, that is, there is a computable $p\in\om^\om$ such that $\widetilde{\dmm}(p)=x$.
\end{definition}

\begin{definition}
Let $\xx=(X,\nn)$ and $\yy=(Y,\mm)$ be topological spaces with countable cs-networks.
Then, we say that a point $x\in\xx$ is {\em $\widetilde{\ast}$-nearly $\yy$-quasi-minimal} if
\[(\forall y\in\yy)\;[\pt{y}{\yy}\reduce \pt{x}{\xx}\;\Longrightarrow\;y\mbox{ is $\widetilde{\ast}$-nearly computable}].\]
\end{definition}

Every nearly computable point is $\widetilde{\ast}$-nearly computable.
Similarly, if a point is nearly $\mathcal{Z}$-quasi-minimal, then it is $\widetilde{\ast}$-nearly $\mathcal{Z}$-quasi-minimal.
Let $\mathcal{H}=(\om,(H_e)_{e\in\om})$ be a represented second-countable space.
A {\em witness for being nowhere $T_{2.5}$} is a set $\Lambda\subseteq\om^3$ such that for any $d,e\in\om$, if both $H_d$ and $H_e$ are nonempty, then $\Lambda_{d,e}=\{n:(d,e,n)\in \Lambda\}$ is nonempty, and $\Lambda_{d,e}\subseteq\overline{H}_d\cap\overline{H_e}$.
For instance, Fact \ref{fact:relative-prime-basic} shows that $\mathbb{N}_{\rm rp}$ has a computable witness for being nowhere $T_{2.5}$, that is, $\Lambda_{d,e}={\rm lcm}(b,d)\mathbb{Z}$.
For a network $\nn$, we define the {\em strong disjointness diagram} as ${\rm Disj}^-_\nn={\rm Disj}_\nn\oplus\{\langle d,e\rangle:\overline{N_d}\cap\overline{N_e}=\emptyset\}$.

Recall that $x\in\om^\om$ is {\em $1$-generic} if it meets or avoids every c.e. open set. For an oracle $C$, a point $x\in\om^\om$ is {\em $1$-$C$-generic} if it meets or avoids every $C$-c.e. open set.



\begin{lemma}\label{lem:relprime-main}
Let $\mathcal{H}$ be a represented, countable, second-countable space with a $C$-c.e.\ witness for being nowhere $T_{2.5}$, and let $x\in\om^\om$ be $1$-$C$-generic.
For any topological space $\yy$ with a cs-network with a $C$-c.e.\ strong disjointness diagram, $\pt{x}{\mathcal{H}^\om}$ is $\widetilde{\ast}$-nearly $\yy$-quasi-minimal.
\end{lemma}

\begin{proof}
Since $\mathcal{H}$ is countable, we can assume that $\mathcal{H}$ is of the form $(\om,(H_e)_{e\in\om})$, where $(H_e)_{e\in\om}$ is an enumeration of countable basis of the space $\mathcal{H}$.
We code a basic open set in $\mathcal{H}^\om$ by a finite sequence $\alpha$, that is, $\alpha$ codes the open set
\[U_\alpha=\{x\in\mathcal{H}^\om:(\forall n<|\alpha|)\;x(n)\in H_{\alpha(n)}\}.\]
Note that $(U_{\alpha})_{\alpha<\om^{<\om}}$ forms a basis of $\mathcal{H}^\om$.
Hereafter we use $\nbase{x}$ to denote $\{\alpha:x\in U_\alpha\}$.

Now, assume that $\pt{y}{\yy}\reduce \pt{x}{\mathcal{H}^\om}$.
We will show that $\pt{y}{\yy}$ is $\widetilde{\ast}$-nearly computable.
By Observation \ref{obs:network-e-basis}, there is $J\leq_e\nbase{x}$ such that $\{N_e:e\in J\}$ forms a strict network at $y$.
Let $\Psi$ witness that $J\leq_e\nbase{x}$.
Since $(U_{\alpha})$ forms a basis, one can assume that $\Psi$ is a c.e.\ set of pairs of indices and singletons, that is,
\[e\in J\iff(\exists \alpha)\;[\alpha\in\nbase{x}\mbox{ and }\langle e,\alpha\rangle\in\Psi].\]

Consider the following three cases:

\medskip
\noindent
{\bf Case $1$.}
There is $\ell\in\om$ such that for any $d$, $e$, $\alpha$, and $\beta$,
\[U_{\alpha}\cap[x\upto\ell]\not=\emptyset,\;U_{\beta}\cap[x\upto\ell]\not=\emptyset\mbox{, and }\langle d,\alpha\rangle,\langle e,\beta\rangle\in\Psi\;\Rightarrow\;\overline{N_{d}}\cap\overline{N_e}\not=\emptyset.\]

Then let $p$ be a computable sequence such that $p(n)=e+1$ for some $n$ if and only if there is $\alpha$ such that $\langle e,\alpha\rangle\in\Psi$ and $U_\alpha\cap[x\upto\ell]\not=\emptyset$.
Then, it is easy to check that $p$ is a $\widetilde{\dnn}$-name of $y$.

\medskip
\noindent
{\bf Case $2$.}
For any $\ell\in\om$, there are $d$, $e$, $\alpha$, and $\beta$ such that
\[U_{\alpha}\cap U_{\beta}\cap[x\upto\ell]\not=\emptyset,\;\langle d,\alpha\rangle,\langle e,\beta\rangle\in\Psi\mbox{, and }{N_{d}}\cap{N_e}=\emptyset.\]

In this case, inconsistent $\Psi$-computations are dense along $x$, that is, consider
\[E=\{\langle\alpha,\beta\rangle:(\exists d,e)\;[\langle d,\alpha\rangle,\langle e,\beta\rangle\in\Psi\mbox{, and }{N_{d}}\cap{N_e}=\emptyset]\},\]
and then define $V_E=\bigcup\{U_{\alpha}\cap U_{\beta}:\langle\alpha,\beta\rangle\in E\}$.
Then, $\Psi$ is undefined on $V_E$, that is, for any $z\in V_E$, $\Psi(\nbase{z})$ is undefined.
Note that each $U_\alpha$ is clopen with respect to the standard Baire topology on $\om^\om$.
Therefore, since the disjointness diagram of $\nn$ is $C$-c.e., $V_E$ is $C$-c.e.~open and dense along $x$ with respect to the standard Baire topology on $\om^\om$.
Since $x$ is $1$-$C$-generic, we have $x\in V_E$.
Therefore, $\Psi(\nbase{x})$ is undefined.

\medskip
\noindent
{\bf Case $3$.}
Otherwise, let $\ell$ be a witness of the failure of Case $2$.
Then, since Case $1$ fails, there are $d$, $e$, $\alpha$, and $\beta$ such that
\[U_{\alpha}\cap[x\upto\ell]\not=\emptyset,\;U_{\beta}\cap[x\upto\ell]\not=\emptyset,\;\langle d,\alpha\rangle,\langle e,\beta\rangle\in\Psi\mbox{, and }\overline{N_{d}}\cap\overline{N_e}=\emptyset.\]
That is, there are splitting $\Psi$-computations above $x\upto\ell$ in a strong sense.
Consider
\[D=\{\langle\alpha,\beta\rangle:(\exists d,e)\;[\langle d,\alpha\rangle,\langle e,\beta\rangle\in\Psi\mbox{, and }\overline{N_{d}}\cap\overline{N_e}=\emptyset]\}.\]

Then, the set $D$ is $C$-c.e., since the strong disjointness diagram is $C$-c.e.\ by our assumption.
Let $\Lambda$ be a $C$-c.e.\ witness for being nowhere $T_{2.5}$.
Consider the following $\om^\om$-clopen set:
\begin{align*}
Q_{\alpha,\beta}=\{z\in \om^\om:(\exists\sigma\prec z)&\;\ell\leq |\sigma|\leq\max\{\ell,|\alpha|,|\beta|\},\;U_{\alpha}\cap[\sigma]\not=\emptyset,\;U_{\beta}\cap[\sigma]\not=\emptyset,\\
\mbox{and }&(\forall n)[|\sigma|\leq n<\max\{|\alpha|,|\beta|\})\;z(n)\in\Lambda_{\alpha(n),\beta(n)}\},
\end{align*}
where let $H_{\alpha(n)}$ be an index of the whole space $\om$ whenever $\alpha(n)$ is undefined.
Note that $Q_{\alpha,\beta}\subseteq\overline{U_{\alpha}}\cap\overline{U_{\beta}}$ since $\Lambda_{\alpha(n),\beta(n)}\subseteq\overline{H_{\alpha(n)}}\cap\overline{H_{\beta(n)}}$.
Define $Q_D=\bigcup\{Q_{\alpha,\beta}:\langle\alpha,\beta\rangle\in D\}$.
Then, $Q_D$ is $C$-c.e.~open, and dense along $x$ with respect to the standard Baire topology on $\om^\om$.
To see this, for any $m$, since Case 1 fails, there is $\langle\alpha,\beta\rangle\in D$ such that $U_\alpha\cap[x\upto m]\not=\emptyset$ and $U_\beta\cap[x\upto m]\not=\emptyset$.
By our assumption, we can always choose $z(k)\in H_{\alpha(k)}\cap H_{\beta(k)}$ for any $k\geq m$, and thus we can get some $z\in Q_{\alpha,\beta}$ extending $x\upto m$.
Therefore, by $1$-$C$-genericity of $x$, we have $x\in Q_D$.

We claim that $\Psi$ is undefined on $Q_D\cap[x\upto\ell]$.
Otherwise, $\Psi(\nbase{z})$ is defined for some $z\in Q_D\cap[x\upto\ell]$.
Since $z\in Q_D\cap[x\upto\ell]$, there is $\langle\alpha,\beta\rangle\in D$ such that $z\in Q_{\alpha,\beta}\subseteq\overline{U_{\alpha}}\cap\overline{U_{\beta}}$.
Let $\langle d,e\rangle$ be a pair witnessing $\langle\alpha,\beta\rangle\in D$, that is, $\langle d,\alpha\rangle,\langle e,\beta\rangle\in\Psi$, and $\overline{N_d}\cap\overline{N_e}=\emptyset$.
Clearly $\Psi(\nbase{z})\in\mathcal{Y}\setminus\overline{N_d}$ or $\Psi(\nbase{z})\in\mathcal{Y}\setminus\overline{N_e}$.
Without loss of generality, we may assume that $\Psi(\nbase{z})\in\mathcal{Y}\setminus\overline{N_d}$.
Then, there is $\langle c,\gamma\rangle\in\Psi$ such that
\[z\in U_{\gamma}\mbox{, and }\Psi(\nbase{z})\in N_c\subseteq\mathcal{Y}\setminus\overline{N_d}.\]
In particular, we have $N_c\cap N_d=\emptyset$.
Since $U_{\gamma}\cap[z\upto\ell]$ is an open neighborhood of $z$ and $z\in\overline{U_{\alpha}}$, $U_{\alpha}\cap U_{\gamma}\cap[z\upto\ell]$ is nonempty.
Since $x\upto\ell=z\upto\ell$, we conclude that
\[U_{\alpha}\cap U_{\gamma}\cap[x\upto\ell]\not=\emptyset,\;\langle d,\alpha\rangle,\langle c,\gamma\rangle\in\Psi\mbox{, and }N_d\cap N_c=\emptyset.\]

This contradicts our choice of $\ell$.
Consequently, $\Psi(\nbase{x})$ is undefined.
\end{proof}

\thmproof{thm:T-2-degree-not-T-25}{
Let $\xx$ be a topological space with a countable cs-network $\nn$.
Let $C$ be an oracle such that the strong disjointness diagram is $C$-c.e.
By Lemma \ref{lem:relprime-main}, for any $1$-$C$-generic point $x\in\om^\om$, $\pt{x}{(\mathbb{N}_{\rm rp})^\om}$ is $\widetilde{\ast}$-nearly $\xx$-quasi-minimal, that is, if $\pt{x}{\xx}\reduce \pt{f}{(\om^\om)_{\rm co}}$ then $\pt{x}{\xx}$ is $\widetilde{\ast}$-nearly computable.
By Observation \ref{obs:T-25network}, if $\xx$ is a $T_{2.5}$-space, then only countably many points in $\xx$ can be $\widetilde{\ast}$-nearly computable.
However, there are uncountably many points which are $1$-$C$-generic.
Thus, one can choose such a point which is not $\eqreduce$-equivalent to any $\widetilde{\ast}$-nearly computable points in $\xx$.}

\thmproof{thm:T-2-degree-NNN-quasiminimal}{
The canonical network $\nn$ of $\mathbb{N}^{\mathbb{N}^\mathbb{N}}$ has a computable strong disjointness diagram.
Since $\mathcal{N}$ is regular-like (see Example \ref{exa:Kleene-Kreisel}), by Observation \ref{obs:T-25network}, ${\rm id}\colon(\mathbb{N}^{\mathbb{N}^\mathbb{N}},\widetilde{\dnn})\to(\mathbb{N}^{\mathbb{N}^\mathbb{N}},\overline{\dnn})$ is computable.
Moreover, as seen in Example \ref{exa:Kleene-Kreisel}, ${\rm id}\colon(\mathbb{N}^{\mathbb{N}^\mathbb{N}},\overline{\dnn})\to(\mathbb{N}^{\mathbb{N}^\mathbb{N}},{\dnn})$ is computable, and so is ${\rm id}\colon(\mathbb{N}^{\mathbb{N}^\mathbb{N}},\widetilde{\dnn})\to(\mathbb{N}^{\mathbb{N}^\mathbb{N}},{\dnn})$ is computable.
Hence, $\widetilde{\ast}$-near $\mathbb{N}^{\mathbb{N}^\mathbb{N}}$-quasi-minimality is equivalent to $\mathbb{N}^{\mathbb{N}^\mathbb{N}}$-quasi-minimality by definition.
Therefore, by Lemma \ref{lem:relprime-main}, for any $1$-$C$-generic point $x$, $\pt{x}{(\mathbb{N}_{\rm rp})^\om}$ is $\mathbb{N}^{\mathbb{N}^\mathbb{N}}$-quasi-minimal.
}

\subsection{$T_{2.5}$-degrees which are not $T_3$}


We will show that if an admissibly represented space $\xx=(X,\nn)$ has an computably equivalent regular-like cs-network, then the Gandy-Harrington space has no point of $\xx$-degree.
To prove Theorems \ref{thm:Gandy-Harrington-closed-neighborhood} and \ref{thm:GH-not-NNN} mentioned in Section \ref{sec:4-4}, we need the following lemma.


\begin{lemma}\label{lem:Gandy-Harrington-closed-neighborhood}
Let $\xx$ be a topological space with a countable cs-network $\nn$.
If $\nn$ has a $\Sigma^1_1$ disjointness diagram, then for any $x\in(\om^\om)_{GH}$ and $z\in\xx$,
\[\pt{z}{\dnn}\reduce \pt{x}{(\om^\om)_{GH}}\;\Longrightarrow\;\pt{x}{(\om^\om)_{GH}}\not\reduce \pt{z}{\overline{\dnn}}.\]
\end{lemma}

\begin{proof}
For $x\in\om^\om$, consider $G_x:=\nbaseb{GH}{x}=\{e:x\in GH_e\}$, where recall that $GH_e$ is the $e$-th $\Sigma^1_1$ set in $\om^\om$.
Clearly, $G_x$ is a $\Sigma^1_1(x)$ subset of $\omega$.
Suppose that $\pt{z}{\dnn}\leq_T\pt{x}{(\om^\om)_{GH}}$ for $z\in\xx$, and $\nn$ is a countable cs-network for $\xx$ such that ${\rm Disj}_\nn$ is $\Sigma^1_1$.
By Observation \ref{obs:network-e-basis}, there is $J\leq_e G_x$ such that $\{N_e:e\in J\}$ forms a strict network at $z$.
Let $\Psi$ witness $J\leq_eG_x$, that is, $e\in J$ iff there is a finite set $D\subseteq G_x$ such that $\langle e,D\rangle\in\Psi$.
Then consider
\[L=\{n\in\om:(\forall \langle m,D\rangle \in\Psi)\;D\subseteq G_x\;\rightarrow\;N_m\cap N_n\not=\emptyset\}.\]

Note that $L$ is a $\Pi^1_1(x)$ subset of $\omega$ since ${\rm Disj}_\nn$ is $\Sigma^1_1$.
One can also see that $J\subseteq L$, since $n\in J$ implies that $z\in N_n$, and moreover, if $\langle m,D\rangle\in\Psi$ and $D\subseteq G_x$, then $m\in J$, and therefore, $z\in N_m\cap N_n$.
This implies that $\{N_e:e\in L\}$ forms a network at $z$.
We claim that $z\in\overline{N}_n$ for any $n\in L$.
This is because, if $z\not\in\overline{N}_n$ then there is an open set $U\subseteq\xx$ such that $z\in U$ and $U\cap N_n=\emptyset$.
By our choice of $\Psi$, there is $\langle e,D\rangle\in\Psi$ such that $D\subseteq G_x$ and $z\in N_e\subseteq U$.
Since $N_e\cap N_n=\emptyset$, we have $n\not\in L$.
This verifies the claim, and in particular, every enumeration of $L$ gives an $\overline{\dnn}$-name of $z$.

Suppose that $\pt{x}{(\om^\om)_{GH}}\reduce \pt{z}{\overline{\dnn}}$.
Then, in particular, $G_x$ is enumeration reducible to $L$, that is, there is a c.e.~set $\Gamma$ such that
\[e\in G_x\;\iff\;(\exists D\mbox{ finite})\;[(e,D)\in\Gamma\mbox{ and }D\subseteq L].\]

Since $L$ is $\Pi^1_1(x)$, this gives a $\Pi^1_1(x)$ definition of $G_x$.
However, $G_x$ is clearly a complete $\Sigma^1_1(x)$ subset of $\omega$, which implies a contradiction.
Consequently, $\pt{x}{(\om^\om)_{GH}}\not\reduce \pt{z}{\overline{\dnn}}$.
\end{proof}

\thmproof{thm:Gandy-Harrington-closed-neighborhood}{
Let $\xx=(X,\nn)$ be a regular Hausdorff space with a countable cs-network.
By Observation \ref{obs:regular-like-network}, $\nn$ is regular-like.
By Theorem \ref{thm:regular-like-network}, $\xx$ has a countable cs-network $\mm$ such that ${\rm id}:(\xx,\overline{\dmm})\to (\xx,\dmm)$ is continuous; hence, computable relative to some oracle $C_0$.
As mentioned in Section \ref{sec:intro-admissible-representation}, cs-networks induce admissible representations, that is, $\delta_\mm$ and $\delta_\nn$ are both $\leq$-maximal among continuous representations of $X$, and thus $\mm$ and $\nn$ are equivalent; hence, computably equivalent relative to some oracle $C_1$.
Moreover, ${\rm Disj}_\mm$ is $\Sigma^1_1$ relative to some oracle $C_2$.
We now put $C=C_0\oplus C_1\oplus C_2$.

Choose $x\in\om^\om$ such that $C\leq_Tx$.
Then, we have $\pt{C}{2^\om}\reduce \pt{x}{(\om^\om)_{GH}}$ by Proposition \ref{prop:gandy-harrington}.
Thus, the condition $\pt{z}{\dnn}\reduce \pt{x}{(\om^\om)_{GH}}$ is equivalent to saying that $\pt{z}{\dmm}\reduce \pt{x}{(\om^\om)_{GH}}$ since $C_1\leq_TC$.
By relativizing the proof of Lemma \ref{lem:Gandy-Harrington-closed-neighborhood}, since $C_2\leq_TC$ and we now have $\Sigma^1_1(x\oplus C)=\Sigma^1_1(x)$ and $\Pi^1_1(x\oplus C)=\Pi^1_1(x)$, we get the following.
\[\pt{z}{\dnn}\reduce \pt{x}{(\om^\om)_{GH}}\;\Longrightarrow\;\pt{x}{(\om^\om)_{GH}}\not\reduce (\pt{z}{\overline{\dmm}})\oplus C.\]

We now assume that $\pt{z}{\dnn}\reduce \pt{x}{(\om^\om)_{GH}}$.
Since $C_0\leq_TC$, we have $\pt{z}{\dmm}\reduce (\pt{z}{\overline{\dmm}})\oplus C$.
Combining this with the above implication, we get $\pt{x}{(\om^\om)_{GH}}\not\reduce (\pt{z}{{\dmm}})\oplus C$.
Since $C_1\leq_TC$, we have $\pt{z}{\dnn}\reduce (\pt{z}{{\dmm}})\oplus C$, and thus $\pt{x}{(\om^\om)_{GH}}\not\reduce \pt{z}{{\dnn}}$.
Hence, there is no $z\in\xx$ such that $\pt{x}{(\om^\om)_{GH}}\equiv_M\pt{z}{{\dnn}}$.
}

\thmproof{thm:GH-not-NNN}{
The canonical network $\nn$ of $\mathbb{N}^{\mathbb{N}^\mathbb{N}}$ has a computable disjointness diagram.
Moreover, as seen in Example \ref{exa:Kleene-Kreisel}, ${\rm id}\colon(\mathbb{N}^{\mathbb{N}^\mathbb{N}},\overline{\dnn})\to(\mathbb{N}^{\mathbb{N}^\mathbb{N}},\delta_\mathcal{N})$ is computable.
Therefore, by Lemma \ref{lem:Gandy-Harrington-closed-neighborhood}, for any $x\in(\om^\om)_{GH}$ and $z\in\mathbb{N}^{\mathbb{N}^\mathbb{N}}$,
\[\pt{z}{\mathbb{N}^{\mathbb{N}^\mathbb{N}}}\reduce \pt{x}{(\om^\om)_{GH}}\;\Longrightarrow\;\pt{x}{(\om^\om)_{GH}}\not\reduce \pt{z}{\mathbb{N}^{\mathbb{N}^\mathbb{N}}}.\]

This shows that there are no $x\in(\om^\om)_{GH}$ and $z\in\mathbb{N}^{\mathbb{N}^\mathbb{N}}$ such that $\pt{x}{(\om^\om)_{GH}}\equiv_M\pt{z}{\mathbb{N}^{\mathbb{N}^\mathbb{N}}}$.
Hence, the Gandy-Harrington degrees and the $\mathbb{N}^{\mathbb{N}^\mathbb{N}}$-degrees have no common element.
}

For an $\om$-parametrized pointclass $\Gamma$, the {\em $\Gamma$-Gandy-Harrington topology} is the topology $\tau_\Gamma$ on $\om^\om$ generated by the subbasis consisting of all $\Gamma$ subsets of $\om^\om$.
By $(\om^\om)_{GH(n)}$, we denote $\om^\om$ endowed with the $\Sigma^1_n$-Gandy-Harrington topology.
We show that there is a hierarchy of degree structures of Gandy-Harrington topologies.

\thmproof{thm:Gandy-Harrington-hierarchy}{It is easy to see that the disjointness diagram ${\rm Disj}_{GH(n)}=\{\langle d,e\rangle:S^n_e\cap S^n_d=\emptyset\}$ is $\Pi^1_n$, where $S^n_e$ is the $e$-th $\Sigma^1_n$ set in $\om^\om$, since $\langle d,e\rangle\in{\rm Disj}_{GH(n)}$ iff $x\not\in S^n_e\cap S^n_d$ for all $x\in\om^\om$.
Assume that $\pt{z}{GH(n)}\reduce\pt{x}{GH(n+1)}$, or equivalently $G^n_z\leq_eG^{n+1}_x$, where $G^m_y=\nbaseb{GH(m)}{y}=\{e\in\om:y\in S^m_e\}$.
Then we define $L$ as in the proof of Lemma \ref{lem:Gandy-Harrington-closed-neighborhood}.
Then $L$ is a $\Sigma^1_n(x)$ subset of $\om$.

Note that $z\in{\rm cl}_\beta(S^n_e)$ for any $e\in L$, where $\beta$ is the standard Baire topology on $\om^\om$.
Otherwise, there is an open set $\beta_j$ such that $z\in \beta_j$ and $\beta_j\cap S^n_e=\emptyset$.
Since $\beta_j$ is also open in the $\Sigma^1_n$-Gandy-Harrington topology, there is $k$ such that $S^n_k=\beta_e$.
Hence $k\in G^n_z$, and thus there is $D\subseteq G^{n+1}_z$ such that $\langle k,D\rangle\in\Psi$.
Since $S^n_e\cap S^n_k=\emptyset$, we have $k\not\in L$.
Let $V_e=\{d\in\om:S^n_e\cap B_d=\emptyset\}$, where $B_d$ is the $d$-th basic open set w.r.t.\ the standard Baire topology on $\om^\om$.
Then $V_e$ is a $\Pi^1_n$ subset of $\om$.
Note that ${\rm cl}_\beta({S^n_e})=\om^\om\setminus\bigcup_{d\in V_e}B_d$.
As in the proof of Lemma \ref{lem:Gandy-Harrington-closed-neighborhood}, one can see that $G^n_z\subseteq L$, and hence $\{z\}=\bigcap_{e\in L}{\rm cl}_{\beta}(S^n_e)$ since $\om^\om$ is Hausdorff.
Note that $\{z\}=\om^\om\setminus\bigcup\{B_d:e\in L\mbox{ and }d\in V_e\}$.

This shows that $\{z\}$ is a $\Pi^0_1$ singleton relative to the $\Sigma^1_n(L)$-complete set, $G^n_x$ say.
Therefore, $z$ is hyperarithmetic relative to $G^n_x$ (see Sacks \cite[Theorem I.1.6]{SacksBook}), and thus, $G^n_z$ is $\Sigma^1_n$ relative to $G^n_x$.
In particular, $G^n_z$ is $\Delta^1_{n+1}$ relative to $x$ since $\Delta^1_{n+1}$-reducibility is transitive (see Rogers \cite[Theorem 16.XXXIV]{RogersBook}).
Consequently, we obtain $G^{n+1}_x\not\leq_eG^n_z$ since $G^{n+1}_x$ is a complete $\Sigma^1_{n+1}$ set relative to $x$.}

\section{Open Questions}

Here we list the current open problems.

\subsubsection*{Major Questions}

We have shown that, in a certain sense, there are a $T_1$-quasi-minimal $e$-degree (Theorem \ref{thm:countable-T_1-quasiminimal}), and a $T_2$-quasi-minimal $T_1$-degree (Theorem \ref{thm:T_2-quasiminimal}).
Thus, whether there exist a $T_{2.5}$-quasi-minimal $T_2$-degree is the one of the most important open problems:
\index{quasi-minimal!$T_{2.5}$}
\begin{question}
Does there exist a represented Hausdorff space $\mathcal{X}$ such that given $T_{2.5}$ space $\mathcal{Y}$, there is $x\in\mathcal{X}$ which is $\mathcal{Y}$-quasi-minimal?
\end{question}

Currently we do not know if we can separate $T_{2.5}$ degrees and submetrizable degrees.
Hence, the following problem is also important:

\begin{question}\label{question-T25}
Does there exist a represented $T_{2.5}$-space $\mathcal{X}$ such that, given a submetrizable space $\mathcal{Y}$, there is $x\in\mathcal{X}$ which is not of $\mathcal{Y}$-degree?
\end{question}

We are also interested in whether we can show separation results in the category of effective quasi-Polish spaces.
For instance, co-$d$-CEA, chained (Arens) co-$d$-CEA, doubled co-$d$-CEA, telophase, and semirecursive $e$-degrees are realized as the degrees of points in effective quasi-Polish spaces.
\index{Degrees!submetrizable}
\begin{question}
Given a submetrizable space $\mathcal{Y}$, does there exist a Arens co-$d$-CEA (or Roy halfgraph-above) degree which is not a $\mathcal{Y}$-degree?
\end{question}

Note that the affirmative answer to the above question gives a quasi-Polish solution to Question \ref{question-T25}.
Similarly, we have found a non-$T_{2.5}$-degree in a $T_2$-space, namely, the product Golomb space $\mathbb{N}_{\rm rp}^\om$ (Theorem \ref{thm:T-2-degree-not-T-25}); however this space is not quasi-Polish.
We know that there is a quasi-Polish space Hausdorff space which is not $T_{2.5}$, e.g.\ the double origin space.
Therefore, one can ask the following:
\index{Degrees!doubled co-d-CEA}
\begin{question}
Given a $T_{2.5}$ space $\mathcal{Y}$, does there exist a doubled co-$d$-CEA degree which is not a $\mathcal{Y}$-degree?
\end{question}

Another big open problem is concerning graph-cototal degrees was raised by Joseph Miller:
\index{Degrees!continuous}
\index{Degrees!graph-cototal}
\begin{question}
Does there exist a continuous degree which is not graph-cototal?
\end{question}

With our framework, this is equivalent to asking whether there is a $\sigma$-embedding of the Hilbert cube $[0,1]^\om$ into the product cofinite space $(\om_{\rm cof})^\om$. Note that the continuous functions into $\om_{\rm cof}$ correspond to countable partitions into closed sets. A classic result by Sierpi\'nski shows that connected compact Polish space do not admit non-trivial countable partitions into closed sets (cf \cite[Theorem 6.1.27]{engelking}. In particular, there is no \emph{embedding} of $[0,1]^\om$ into $(\om_{\rm cof})^\om$. There exist, however, infinite dimensional spaces without any connected compact Polish subspaces (these are called punctiform), so an answer to the question is not immediate. A piece of the puzzle could that the result that the Hilbert cube cannot be decomposed into countably-many hereditarily disconnected spaces \cite{banakh2}.

To solve a question, one may examine the behavior of the co-spectrum of a space.
For instance, by an argument using the notion of co-spectrum, we have shown that there is a continuous degree which is neither telograph-cototal nor cylinder-cototal (by Proposition \ref{thm:continuous-cospec}).
However, we do not know even the following:
\index{Degrees!$2$-cylinder-cototal}
\begin{question}
Does there exist a continuous degree which is not $2$-cylinder-cototal?
\end{question}

In general, we are also interested in analyzing the behavior of the cospectrum of a given space.
For instance, it is important to ask the following:
\index{Degrees!graph-cototal}
\begin{question}
Is every countable Turing ideal realized as the cospectrum of a graph-cototal $e$-degree?
\end{question}

We next consider cototal $e$-degrees.
Recall from Proposition \ref{cototal-equal-twin} that a space is cototal if and only if it is computably $G_\delta$.
Since every computably $G_\delta$ space is effectively $T_1$ by Observation \ref{obs:twin-second-countable}, in particular, every point in a cototal space has a $T_1$-degree.
Then, can we separate cototal degrees and $T_1$-degrees?
\index{Degrees!cototal}
\begin{question}
Does there exist a point in an effective quasi-Polish $T_1$-space which has no cototal $e$-degree?
\end{question}

Recall that our universal (in the degree-theoretic sense) computably $G_\delta$ space $\mathcal{A}^{\rm co}_{\rm max}$, the maximal antichain space, is not quasi-Polish (Proposition \ref{prop:non-quasi-Polish-examples}).
One of the most important questions on cototal degrees is whether a universal computably $G_\delta$ quasi-Polish space exists:
\index{universal $G_\delta$ space}
\begin{question}
Does there exist a computably $G_\delta$, quasi-Polish, space which contains all cototal $e$-degrees?
\end{question}

In this article, we have also discussed $\N^{\N^\N}$-quasi-minimality.
However, currently we do not know whether quasi-minimality is different from $\N^{\N^\N}$-quasi-minimality.
\index{quasi-minimal!$\N^{\N^\N}$}
\begin{question}
Does there exist a quasi-minimal $e$-degree which is not $\N^{\N^\N}$-quasi-minimal?
\end{question}

Mariya Soskova raised the following question (see Subsection \ref{subsubsec:mariya} for context):\index{$\mathcal{K}$-pair}
\begin{question}
\label{question:kpairs}
Is there a topological characterization of the halves of non-trivial $\mathcal{K}$-pairs?
\end{question}

\subsubsection*{Minor Questions}

We also list some minor questions.
Recall that every telograph-cototal (double-origin) $e$-degree is graph-cototal (Propositions \ref{prop:telophase-is-graph-cototal} and  \ref{prop:doubleorigin-is-cototal}).
There is a graph-cototal (indeed cylinder-cototal) $e$-degree which is neither telograph-cototal nor doubled co-$d$-CEA (by Theorem \ref{telograph-quasiminimal2}).
Can every telograph-cototal (doubled co-$d$-CEA) $e$-degree be embedded into some level of the hierarchy of graph-cototal $e$-degrees?
\index{Degrees!telograph-cototal}
\index{Degrees!$n$-cylinder-cototal}
\begin{question}
Is every telograph-cototal $e$-degree $n$-cylinder-cototal for some $n\in\om$?
\end{question}

Recall from Proposition \ref{prop:non-cylinder-cototal} that there is a co-$d$-CEA $e$-degree which is not cylinder-cototal.
The following question is also open.
\index{Degrees!co-d-CEA}
\begin{question}
Does there exist a co-$d$-CEA $e$-degree which is not $2$-cylinder-cototal?
\end{question}

We also do not know the relationship among variations of co-$d$-CEA degrees.

\begin{question}
What is the relationship among doubled co-$d$-CEA degrees, Arens co-$d$-CEA degrees, and Roy halfgraph-above degrees?
\end{question}

Recall that every $3$-c.e.~$e$-degree is telograph-cototal while there is a $\Sigma^0_2$ $e$-degree which is not telograph-cototal (by Theorem \ref{telograph-quasiminimal1}).
\index{Degrees!telograph-cototal}
\begin{question}
For any $n$, is every $n$-c.e.\ $e$-degree telograph-cototal?
\end{question}

There is also a problem related to left-totality.
\index{Degrees!cylinder-cototal}
\index{Degrees!$G_\delta$-left-total}
\begin{question}
Is there a cylinder-cototal $e$-degree which is not $G_\delta$-left-total?
\end{question}

\begin{ack}
The first author was partially supported by JSPS KAKENHI Grant 19K03602 and 15H03634, and JSPS Core-to-Core Program (A. Advanced Research Networks). The first author also would like to thank the Young Scholars Overseas Visit Program in Nagoya University for supporting his visit to Nanyang Technological University. The second author is partially supported by the grants MOE2015-T2-2-055 and RG131/17. The third author has received funding from the European Union's Horizon 2020 research and innovation programme under the Marie Sklodowska-Curie grant agreement No 731143, \emph{Computing with Infinite Data}. 

The authors would like to thank Matthew de Brecht for discussions on the definitions of $T_i$-degrees and $T_i$-quasi-minimal degrees in the early stage of this work. The authors also thank Joseph Miller for helpful discussions pertaining to Sections \ref{sec:G-delta-space} and \ref{sec:cototalenumeration}. Special thanks go to Steffen Lempp for fruitful discussions: The results in Sections \ref{sec:telophasetopo}, \ref{sec:doubleorigintopology} and \ref{sec:t1nott2quasiminimal} were obtained when Lempp visited us. We are grateful to Mariya Soskova, and also to Josiah Jacobsen-Grocott, for their constructive critique of an earlier version of the paper; and to Taras Banakh for answering a question on punctiform spaces.
\end{ack}

\printindex

\bibliographystyle{plain}
\bibliography{../nonmetrizable}
\end{document}